\documentclass[onecolumn,12pt]{article}
\usepackage{amsmath,amssymb,amsfonts}
\usepackage{algorithmic}
\usepackage{graphicx}
\usepackage{textcomp}
\usepackage{xcolor}
\usepackage{here}
\usepackage{mathrsfs}
\usepackage{enumerate}
\usepackage{mathtools}
\usepackage{fullpage}
\usepackage{cite}
\usepackage{dsfont}

\usepackage{authblk}

\usepackage[pdfencoding=pdfdoc, unicode=false, pdfusetitle=false, bookmarks=true]{hyperref}



\hypersetup{
  pdftitle={Uniform Exponential Stability Analysis of Impulsive Linear Time-Invariant Infinite-Dimensional Systems in Banach and Hilbert spaces: Non-coercive and coercive conditions},
  pdfauthor={Corentin Briat, Francesco Ferrante, Christophe Prieur},
}

\newtheorem{theorem}{Theorem}[section]
\newtheorem{proposition}[theorem]{Proposition}
\newtheorem{lemma}[theorem]{Lemma}
\newtheorem{definition}[theorem]{Definition}
\newtheorem{assumption}[theorem]{Assumption}

\newtheorem{remark}[theorem]{Remark}

\definecolor{darkred}{rgb}{0.6,0,0}

\DeclareMathOperator*{\diag}{diag}

\def\d{\mathrm{d}}
\def\ds{\mathrm{d}s}
\def\dt{\mathrm{d}t}
\def\dr{\mathrm{d}r}
\def\du{\mathrm{d}u}

\def\dtheta{\mathrm{d}\theta}
\def\dz{\mathrm{d}z}
\newtheorem{corollary}[theorem]{Corollary}

\everymath{\displaystyle}

\def\Tmin{T_{\min}}

\newenvironment{proof}{{\it Proof :~}}{\hfill$\square$\\}

\date{}

\begin{document}


\title{Uniform Exponential Stability Analysis of Impulsive Linear Time-Invariant Systems on Banach and Hilbert Spaces: Non-Coercive and Coercive Stability Conditions\thanks{This document has not yet been peer-reviewed, and a portion of it is currently under review. A revised version, reflecting the outcome of that review and correcting any remaining typos or unclear statements, will be posted in due course. Comments and corrections from readers are warmly welcome in the meantime.}}

\author[1]{Corentin Briat\footnote{Corresponding author.}}
\author[2]{Francesco Ferrante}
\author[3]{Christophe Prieur}

\affil[1]{\footnotesize School of Life Sciences, University of Applied Sciences, Muttenz, Switzerland, \texttt{corentin@briat.info}, \url{http://www.briat.info}}
\affil[2]{Department of Engineering, University of
  Perugia, Perugia, Italy, \texttt{francesco.ferrante@unipg.it}, \url{http://www.fferrante.net/}}
\affil[3]{Universit\'e Grenoble Alpes, CNRS, Grenoble INP,
  GIPSA-lab, Grenoble, France, \texttt{christophe.prieur@gipsa-lab.fr}, \url{https://www.gipsa-lab.grenoble-inp.fr/~christophe.prieur/}}


\maketitle

\begin{abstract}
We consider the uniform exponential stability analysis of infinite-dimensional impulsive systems defined on a Banach or Hilbert space, whose flow is governed by a fixed $C_0$-semigroup generator and whose jumps occur at a prescribed time sequence. While the flow and jump maps are themselves time-invariant, the time-triggered impulses render the propagator a genuinely time-varying evolution family, which is the source of the analysis difficulty addressed here. We combine ideas from hybrid systems theory and infinite-dimensional systems to produce operator-based stability conditions, which can be analytically or numerically checked via convex programming. Necessary and sufficient conditions for the uniform exponential stability of impulsive systems on Banach spaces are obtained in the context of a fixed impulse-times sequence but also of arbitrary, constant, minimum, and range dwell-times using both non-coercive and coercive Lyapunov functionals. Some of those results are then adapted to systems on a Hilbert space and quadratic Lyapunov functionals. As an application, linear switched systems are shown to be an exact special case: reformulated as impulsive systems with unit-norm selector jumps, they inherit non-coercive and clock-dependent dwell-time stability conditions on both Banach and Hilbert spaces. Theoretical and numerical examples are given for illustration, notably on the sampled-data control of time-delay systems.\\

\end{abstract}

\section{Introduction}

Linear continuous-time infinite-dimensional systems provide {a} natural state-space framework for {linear partial differential equations and linear functional (e.g., time-delay) differential equations}, and their stability properties are now well-understood \cite{Pazy:83,Curtain:95,Bensoussan:06}. When the dynamics are time-invariant, the solution map is a strongly continuous ($C_0$) semigroup, and exponential stability is characterized through the spectrum of the generator, through Lyapunov operator equations, and through the integral criterion of Datko \cite{Datko:70}. A {characteristic} specific to the infinite-dimensional setting is that coercive \emph{quadratic} Lyapunov functionals need not exist for exponentially stable $C_0$-semigroups on general Banach or Hilbert spaces \cite{Mironchenko:19}: there are exponentially stable $C_0$-semigroups on Hilbert spaces for which no equivalent scalar product makes the semigroup a strict contraction. This phenomenon, which has no counterpart in the finite-dimensional setting, has motivated the development of non-coercive converse Lyapunov theorems for $C_0$-semigroups \cite{Hante:11,Mironchenko:19,Haidar:22} and, more broadly, of input-to-state stability theory for infinite-dimensional systems \cite{Mironchenko:18}. Coercive Lyapunov functionals do remain available through the so-called sup-type construction \cite{Mironchenko:18,Haidar:22}, but they need not be quadratic and so are not amenable to the operator-Lyapunov-equation analysis pursued here on Hilbert spaces.

In parallel, hybrid dynamical systems have emerged as a unifying language for processes that combine continuous flow with discrete jumps, encompassing switched and impulsive systems within a single formalism \cite{Goebel:12,Haddad:06,liu2015lyapunov}. Impulsive systems, in which the state is reset instantaneously at a prescribed sequence of time instants, are of particular interest because sampled-data and networked control systems admit an exact impulsive reformulation that, unlike a purely discrete-time model, preserves the intersample behavior \cite{Sivashankar:94,Naghshtabrizi:08,Hetel:13}. A closely related and widely used alternative models the sampling as a time-varying delay, the input-delay approach to sampled-data control \cite{Fridman:04,Fridman:10,Liu:12}, which connects the present setting to the analysis and control of time-delay systems \cite{Fridman:14}. This feature is essential whenever the intersample signal carries information, as in optimal control or in the presence of continuous-time exogenous inputs \cite{Bamieh:91,Khammash:93}. For finite-dimensional impulsive systems, a well-established stability theory is available, rooted in the classical theory of impulsive differential equations \cite{Lakshmikantham:89,Bainov:89,Samoilenko:95,Yang:01b} and comprising dwell-time and Lyapunov-based conditions \cite{Hespanha:08,Teel:14,Mancilla:20}, timer-dependent Lyapunov functions \cite{Goebel:12,Briat:13d,Briat:15i,Xiang:15a} or looped-functionals \cite{Seuret:12b,Briat:11l,Briat:12h, Briat:13b} that capture dwell-time constraints with low conservatism and may result in convex linear matrix inequality (LMI) formulations, suitable for both analysis and design. 

The stability analysis of infinite-dimensional impulsive systems, which sit at the intersection of these two theories, is considerably less developed. Input-to-state stability and Lyapunov methods for impulsive systems on Banach spaces have been investigated in \cite{Dashkovskiy:12,Dashkovskiy:13,Dashkhovskiy:23,Bivziuk:23}, mostly through dwell-time conditions and coercive functionals; the closest to the present work are \cite{Dashkhovskiy:23,Bivziuk:23}: the former derives dwell-time stability conditions for infinite-dimensional impulsive systems in which both the flow and the jumps may be individually unstable, and the latter establishes a comparison theorem for the linear case that reduces stability under an averaged dwell-time to a constant dwell-time problem. {Both works provide sufficient conditions: \cite{Dashkhovskiy:23} relies on coercive Lyapunov functions and addresses nonlinear systems, while \cite{Bivziuk:23} requires the semigroup to be analytic and concludes asymptotic stability. In contrast, we consider linear systems with general $C_0$-semigroups and obtain necessary and sufficient conditions for exponential stability, uniform with respect to the initial time, in terms of non-coercive functionals.} Analogously, infinite-dimensional switched systems were also considered in \cite{Haidar:22,Chitour:25}; we show in Section~\ref{sec:switched} that such systems are in fact an application of the present impulsive framework. Three features distinguish the problem both from its finite-dimensional and from its non-impulsive infinite-dimensional counterparts. First, although the flow generator and the jump operator are themselves time-invariant, a fixed sequence of impulse times pins the jumps to absolute instants, so the resulting propagator is a genuine two-parameter evolution family that is \emph{not} time-translation invariant. Second, the jump operator may not be invertible, resulting in a propagator that is not globally invertible. Those two properties together imply that the group and semigroup structure of propagators underpinning Datko's theorem is unavailable, and stability bounds valid from the initial time do not automatically extend to all initial times. Third, the absence of coercive functionals on general Banach spaces, already present without jumps, persists here, so a characterization of exponential stability that does not presuppose a coercive functional is required.

The objective of this paper is to develop such a theory for linear impulsive systems on Banach and Hilbert spaces whose jumps are time-triggered. The starting point is an impulsive analogue of Datko's theorem (Lemma~\ref{lem:datko}) that converts an integral-plus-sum boundedness condition into a \emph{hybrid} exponential bound, with explicit and separate decay contributions from the continuous flow and from the discrete jumps. Building on this lemma, we establish a complete, necessary and sufficient, \emph{non-coercive} Lyapunov characterization of global exponential stability for a fixed impulse sequence (Theorem~\ref{thm:main}), together with persistent-flowing and persistent-jumping counterparts (Theorems~\ref{thm:pf} and~\ref{thm:pj}) that cover the degenerate cases where only one of the two mechanisms is dissipative. The parallel coercive theory is developed in Section~\ref{sec:coercive}, where we provide an explicit construction of a coercive functional whenever exponential stability holds, alongside a discrete-time formulation obtained by sampling at the jump instants and embedded in the dwell-time subsections of Sections~\ref{sec:lyapunov}--\ref{sec:coercive}; on Hilbert spaces, these conditions specialize to quadratic ones expressed as operator Lyapunov equations and inequalities (Section~\ref{sec:hilbert}). Timer-dependent conditions are then derived for the constant, minimum and range dwell-time families (Sections~\ref{subsec:nc:cst}--\ref{subsec:nc:rng} and~\ref{subsec:coerc:cst}--\ref{subsec:coerc:rng}), together with convex design conditions obtained through a change of variables and a clarification of the role of the semigroup growth bound, which is shown to be necessary rather than restrictive for families admitting unbounded dwell-times (Proposition~\ref{prop:necessity:semigroup}). As an application, Section~\ref{sec:switched} shows that linear switched systems fall within this framework through an exact reformulation as impulsive systems on a product space: the switches become unit-norm selector jumps, placing the system in the persistent-flowing regime and yielding non-coercive Lyapunov conditions on Banach spaces and clock-dependent quadratic conditions on Hilbert spaces for fixed, minimum-dwell-time, and arbitrary switching, the last recovering the common non-coercive Lyapunov characterization of arbitrary switching \cite{Haidar:22,Liberzon:03}. The framework is finally illustrated through worked examples on transport equations and on the sampled-data control of time-delay systems, the latter treated within the Partial Integral Equation framework of \cite{Peet:21pie,Shivakumar:21}, where the hybrid conditions reduce to linear operator inequalities solvable with off-the-shelf software like PIETOOLS \cite{PIETOOLS:24} combined with a piecewise discretization approach for parametrizing the timer-dependent operators \cite{Allerhand:11,Allerhand:13,Xiang:15a,Briat:16c}.


The paper is organized as follows. Section~\ref{sec:prelim} defines the system class, the solution concept and the stability notions. Section~\ref{sec:useful} gathers the standing growth-bound assumption and the impulsive Datko lemma. Section~\ref{sec:lyapunov} establishes non-coercive Lyapunov characterizations, with a subsection for each class of impulse sequences (fixed, arbitrary, constant / minimum / range dwell-time); the fixed-sequence subsection presents an equivalence between exponential stability, a Datko condition and a continuous-time hybrid Lyapunov criterion; the arbitrary-sequence subsection presents an equivalence between exponential stability and a time-independent Lyapunov criterion; each dwell-time subsection adds a discrete-time criterion on the monodromy operator of the system. Section~\ref{sec:coercive} establishes the parallel coercive (Lipschitz, equivalent-norm) theory on Banach spaces, with the same subsection structure and full equivalences in the Lipschitz coercive class. Section~\ref{sec:hilbert} specializes to Hilbert spaces and the quadratic class, where each subsection states a non-coercive quadratic theorem and a coercive corollary; these are necessary and sufficient in the fixed-sequence and constant dwell-time cases, but only sufficient in the arbitrary, minimum, and range dwell-time cases, since the quadratic class is strictly smaller than the non-coercive Lipschitz class on Hilbert spaces whenever the family of admissible monodromies is not a singleton. Section~\ref{sec:switched} applies the framework to linear switched systems, recast as impulsive systems on a product space with unit-norm selector jumps. Section~\ref{sec:examples} illustrates the results on transport equations, time-delay systems, and the sampled-data boundary control of a reaction-diffusion equation.

\section{Definitions and preliminary results}\label{sec:prelim}

We introduce in this section fundamental concepts and results about infinite-dimensional and impulsive systems, the solutions of such systems via propagators, as well as various dwell-time concepts and appropriate exponential stability concepts for such systems. {In this article, the sets of nonnegative and positive real numbers are denoted by $\mathbb{R}_{\ge0}$ and $\mathbb{R}_{>0}$, respectively.}

\begin{definition}[\!\!\cite{Curtain:95}]
  A family of bounded linear operators $\{S(t)\}_{t\geq 0}$, on a Banach space $X$ is called a strongly continuous semigroup if
  \begin{enumerate}[(i)]
    \item $S(0)=I$
    \item $S(t+r)=S(t)S(r)$ for all $t, r\geq 0$
    \item $t\mapsto S(t)$ is strongly continuous. Namely for all $x\in X$
    $$
    \lim_{t \downarrow 0}\Vert S(t)x-x\Vert=0
    $$
    \end{enumerate}
\end{definition}

\begin{definition}[\!\!\cite{Pazy:83}]
  The infinitesimal generator $A:D(A)\subset X\mapsto X$ of a strongly continuous semigroup $S(t)$ defined on a Banach space $X$ is defined by
  \begin{equation}
    Ax=\lim_{t\downarrow 0}\dfrac{S(t)x-x}{t}\ \textnormal{for all }x\in D(A),
  \end{equation}
  where
    \begin{equation}
    D(A)=\left\{x\in X:\lim_{t\downarrow 0}\dfrac{S(t)x-x}{t}\ \textnormal{exists}\right\}.
  \end{equation}
  Moreover, we have that
  \begin{equation}
    \dfrac{\d}{\dt}S(t)x=AS(t)x
  \end{equation}
  for any $x\in D(A)$.
\end{definition}

We consider here the following class of infinite-dimensional impulsive systems
\begin{equation}\label{eq:syst}
\begin{array}{rcl}
  \dot{x}(t)&=&Ax(t),\ t\notin\mathbb{T}_\sigma\\
  x(t)&=&J x(t^-),\ t\in\mathbb{T}_\sigma\\
  x(t_0)&=&x_0
  \end{array}
\end{equation}
where $t_0$ is the initial time, $\sigma:=\{t_k\}_{k\ge1}$ is the sequence of impulse instants assumed to belong to the set
\begin{equation}
  \mathcal{S}_{0,\infty}:=\left\{\{t_k\}_{\ge1}:\begin{array}{l}
    t_{i+1}-t_i>0, t_0=0,i\ge0,\\
    t_j\to\infty\ \textnormal{as } j\to\infty
  \end{array}\right\},
\end{equation}
and $\mathbb{T}_\sigma$ is the set of impulse instants associated with a given sequence $\sigma$. The notation $x(t^-)$ stands for the limit from the left at $t$. For system \eqref{eq:syst}, we assume that $A:D(A)\subset X\mapsto X$ is the infinitesimal generator of a strongly continuous semigroup $S(t)$ on the Banach space $X$. We further assume $J\in L(X)$ is a nonzero bounded operator ($J\ne0$) such that $JD(A)\subset D(A)$.

For results that exploit additional structure on the timing of impulses, we will work with the following subfamilies of $\mathcal{S}_{0,\infty}$, parametrized by dwell-time bounds. Given $0<T_{\min}\le T_{\max}<\infty$:
\begin{align}
  \mathcal{S}_{\mathrm{cst}}(T)&:=\big\{\sigma\in\mathcal{S}_{0,\infty}:\ t_{k+1}-t_k=T\text{ for all }k\ge0\big\},\label{eq:family:cst}\\
  \mathcal{S}_{\min}(T_{\min})&:=\big\{\sigma\in\mathcal{S}_{0,\infty}:\ t_{k+1}-t_k\ge T_{\min}\text{ for all }k\ge0\big\},\label{eq:family:min}\\
  \mathcal{S}_{\mathrm{rng}}(T_{\min},T_{\max})&:=\big\{\sigma\in\mathcal{S}_{0,\infty}:\ T_{\min}\le t_{k+1}-t_k\le T_{\max}\text{ for all }k\ge0\big\}.\label{eq:family:rng}
\end{align}
We refer to $\mathcal{S}_{\mathrm{cst}}(T)$ as the \emph{constant dwell-time family} (a single periodic sequence), to $\mathcal{S}_{\min}(T_{\min})$ as the \emph{minimum dwell-time family} (inter-jump intervals bounded below), and to $\mathcal{S}_{\mathrm{rng}}(T_{\min},T_{\max})$ as the \emph{range dwell-time family} (inter-jump intervals in a bounded interval). 

Inspired by the notions of solutions to hybrid systems for finite-dimensional systems \cite{PrieurAstolfi03}, we propose the following definitions
\begin{definition}\label{def:existence}
A function $x: [t_0, \infty)\mapsto X$ is a solution to the system \eqref{eq:syst} if 
\begin{itemize}
  \item it is continuous on $\cup_{k\geq 0}[t_k, t_{k+1})$, 
  \item it has a left limit at $t_{k}$ for all $k\geq 1$ denoted $x(t_k^-)$,  and
  \item the following expressions hold true
\begin{equation*}
  \begin{array}{rcl}
   \int_{t_k}^t x(s) \ds&\in& D(A),\quad\forall t\in (t_k, t_{k+1}),\quad \forall k\geq 0,\\
    x(t)&=&x_0+A\int_{t_0}^t x(s) \ds,\quad\forall  t\in (t_0, t_1),\\
    x(t)&=&x(t_k)+A\int_{t_k}^t x(s) \ds,\quad\forall  t\in (t_k, t_{k+1}),\\
   x(t_k)&=&Jx(t_k^-), k\ge1.
  \end{array}
\end{equation*}
\end{itemize}
\end{definition}

\begin{proposition}
Assume that  $\sigma\in\mathcal{S}_{0,\infty}$. Then, there exists a unique solution to the system \eqref{eq:syst} {in the sense of Definition~\ref{def:existence} with $t_0=0$. This solution is defined for all $t\ge t_0$, i.e., it is forward complete.} In particular, this solution can be written as
\begin{equation}
  x(t)=U_\sigma(t,s)x(s),\ t\ge s\ge t_0=0
\end{equation}
where
\begin{equation}\label{eq:generator}
  \begin{array}{rcl}
    U_\sigma(t,\theta)U_\sigma(\theta,s)&=&U_\sigma(t,s), \theta\in[s,t]\\
   \dfrac{\partial}{\partial t} U_\sigma(t,s)&=&AU_\sigma(t,s),\quad t\notin\mathbb{T}_\sigma,\ s\le t\\
   \dfrac{\partial}{\partial s} U_\sigma(t,s)&=&-U_\sigma(t,s)A,\quad s\notin\mathbb{T}_\sigma,\ s\le t\\
    U_\sigma(t_k,s)&=&JU_\sigma(t_k^-,s),\quad k\ge1,\ s<t_k\\
    U_\sigma(t,t_k^-)&=&U_\sigma(t,t_k)J,\quad k\ge1,\ t\ge t_k\\
     U_\sigma(t,t)&=&I,\quad t\ge0\\
  \end{array}
\end{equation}
{where the second and third identities hold when both sides are applied to any $x\in D(A)$.}
\end{proposition}
Observe that the operator $U_\sigma(t,s)$ is not necessarily invertible for two reasons. The first one is that the semigroup $S(t)$ is not necessarily invertible. The second reason is that the jump operator $J$ is not necessarily invertible.
\begin{definition}
  A functional $V:\mathbb{R}_{\ge0}\times X\to\mathbb{R}$ is uniformly Lipschitz on bounded subsets of $X$ if for any $r>0$, there exists an $L_r>0$ such that
  \begin{equation}
    |V(t,x)-V(t,y)|\le L_r\|x-y\|
  \end{equation}
  holds for all $\|x\|<r,\|y\|<r$ and $t\ge0$.
\end{definition}

We define also the following jump counter:
\begin{definition}[Jump counter]
The counting variable
\begin{equation}
  \kappa_\sigma(t,s):=\#{(}\mathbb{T}_\sigma\cap(s,t]{)}
\end{equation}
counts the number of impulse instants on the half-open interval $(s,t]${, where $\#E$ denotes the cardinality of a set $E$}. The half-open convention is consistent with the propagator semantics: $U_\sigma(t,s)$ starts from the post-jump state at $s$, so a jump occurring exactly at $s$ is not counted, whereas $U_\sigma(t,s^-)$ (start from the pre-jump state) does count it. With a slight abuse of language, we also define $\kappa_\sigma(t):=\kappa_\sigma(t,0)$.
\end{definition}



\begin{definition}
For a functional $V:\mathbb{R}_{\ge0}\times X\ni (t,x)\mapsto\mathbb{R}$, piecewise (right) continuous in $t$ and continuous in $x$, its lower right-hand Dini derivative is given by 
\begin{equation}
  \underline{D}^+V(t,x)=\liminf_{h\downarrow0}\dfrac{V(t+h,S(h)x)-V(t,x)}{h}
\end{equation}
for all $t$ at which $V(t^-,x)=V(t,x)$. 
\end{definition}



\begin{definition} \label{def:stab}
  Let a sequence of impulse times $\sigma$ be given. Then, the origin of the hybrid system \eqref{eq:syst} is \textbf{globally exponentially stable} (GES) if there exist some constants $ M \ge1$, $\alpha\ge0$ and $\rho\in(0,1]$ verifying $\alpha-\log(\rho)>0$ such that
  \begin{equation}\label{eq:defstab}
         ||x(t)||_X\le M \rho^{\kappa_\sigma(t,s)}e^{-\alpha (t-s)}||x(s)||_X
  \end{equation}
  for all $t\ge s\ge0$ and all $x(s)\in X$, where $\kappa_\sigma(t,s)$ denotes the number of impulse instants in $(s,t]$; equivalently, $||U_\sigma(t,s)||\le  M \rho^{\kappa_\sigma(t,s)}e^{-\alpha (t-s)}$. It is \textbf{strongly GES} when $\alpha>0$ and $\rho\in(0,1)$.
\end{definition}

The constraint $\alpha-\log(\rho)>0$ requires at least one mechanism to be strictly decaying and admits three regimes: the \emph{hybrid} case $\alpha>0$, $\rho\in(0,1)$; the \emph{persistent-jumping} case $\alpha=0$, $\rho\in(0,1)$ (decay carried by the jumps alone); and the \emph{persistent-flowing} case $\alpha>0$, $\rho=1$ (carried by the flow alone). Strong GES is thus the hybrid regime.

In the above result, the bound is required to hold \emph{uniformly with respect to the initial time} $s\ge0$, and not merely at $s=0$. This uniformity is indispensable for the converse results of Section~\ref{sec:lyapunov}: since the impulsive propagator $U_\sigma(t,s)$ is not time-translation invariant (the jumps of a fixed sequence are pinned to absolute times, even though $A$ and $J$ are time-invariant), the bound at $s=0$ does \emph{not} imply the bound at $s>0$, and the latter is precisely what is produced by Lemma~\ref{lem:datko}. This definition naturally generalizes to families of impulse-times sequences:

\begin{definition} \label{def:stabu}
 Given a family $\mathcal{S}\subset\mathcal{S}_{0,\infty}$, the origin of \eqref{eq:syst} is \textbf{uniformly globally exponentially stable} (UGES) over $\mathcal{S}$ if there exist $M\ge1$, $\alpha\ge0$, $\rho\in(0,1]$ with $\alpha-\log(\rho)>0$ such that
  \begin{equation}\label{eq:defstab:uniform}
         ||x(t)||_X\le M \rho^{\kappa_\sigma(t,s)}e^{-\alpha (t-s)}||x(s)||_X
  \end{equation}
  for all $t\ge s\ge0$, $x(s)\in X$ and $\sigma\in \mathcal{S}$; it is \textbf{strongly UGES} when moreover $\alpha>0$ and $\rho\in(0,1)$.
\end{definition}

\begin{remark}\label{rem:stab}
{Since the system \eqref{eq:syst} combines a continuous flow and discrete impulses, it is natural for the stability notion to combine a continuous-time part and a discrete-time part. Accordingly, the bound \eqref{eq:defstab} involves a continuous-time decay rate $\alpha\ge0$ and a discrete-time decay rate $-\log(\rho)\ge0$, the latter being associated with the number of impulses. The extremal cases $\rho=1$ and $\alpha=0$ correspond to the usual notions of exponential stability for continuous-time and discrete-time systems, respectively, the discrete time being given here by the number of impulses, while strong GES corresponds to an exponential stability bound in the hybrid time $(t-s,\kappa_\sigma(t,s))$ as considered in the hybrid systems framework \cite{Goebel:12}. When one of the two mechanisms is destabilizing, the bound \eqref{eq:defstab} also connects with average dwell-time conditions, under which an upper bound on $\kappa_\sigma(t,s)$, affine in $t-s$, converts a bound with expansive jumps into \eqref{eq:defstab} with $\rho=1$ \cite{Hespanha:08}. Finally, as we shall see in Section~\ref{sec:lyapunov}, the factor $\rho^{\kappa_\sigma(t,s)}$ with $\rho\in(0,1)$ ensures the convergence of the sums over the impulse instants appearing in the Datko-type conditions, for any distribution of the impulse instants, which is not the case when $\rho=1$.}
\end{remark}
\section{Useful results}
\label{sec:useful}

We provide here the technical tools underpinning many of the stability results of the paper. The first is a standing growth-bound assumption on the propagator; the second is the main technical lemma, an impulsive analogue of Datko's theorem \cite{Datko:70}, which converts a hybrid integral-plus-sum boundedness condition into exponential decay.

\begin{assumption}[Growth bound]\label{ass:growth}
There exist a piecewise continuous function $g:\mathbb{R}_{\ge0}\to\mathbb{R}_{>0}$, bounded below by some $c_0>0$, and a scalar $\mu>1$ such that
\begin{equation}\label{eq:growth}
  \|U_\sigma(t,s)x\|\le g(t-s)  \mu^{\kappa_\sigma(t,s)}\|x\|
\end{equation}
holds for all $t\ge s\ge0$ and all $x\in X$. Let
\begin{equation}\label{eq:Tbar}
  \bar T:=\sup_{k\ge0}(t_{k+1}-t_k)\in(0,\infty]
\end{equation}
denote the \emph{supremal dwell-time} of the sequence $\sigma$ under consideration (with the convention $t_0=0$), to be replaced by $\sup_{\sigma\in\mathcal S}\sup_{k\ge0}(t_{k+1}-t_k)$ when a family $\mathcal S\subseteq\mathcal{S}_{0,\infty}$ is considered. We require
\begin{equation}\label{eq:G}
  G:=\sup_{0\le\tau<\bar T}g(\tau)<\infty .
\end{equation}
\end{assumption}

\begin{remark}\label{rem:growth:G}
The bound \eqref{eq:growth} always holds for some admissible $g$ and $\mu$: since $A$ generates a $C_0$-semigroup, there exist $M_0\ge1$ and $\omega_0\in\mathbb{R}$ with $\|S(t)\|\le M_0e^{\omega_0 t}$, and the factorization of $U_\sigma$ into semigroup arcs separated by jumps gives \eqref{eq:growth} with $g(t)=M_0\max(1,e^{\omega_0 t})$ (bounded below by $M_0>0$, as the assumption requires, while still dominating $\|S(t)\|$) and any $\mu\ge\|J\|M_0$; enlarging $\mu$ preserves \eqref{eq:growth} since $\kappa_\sigma\ge0$, so the normalization $\mu>1$ may always be assumed. Condition \eqref{eq:G} is governed entirely by the supremal dwell-time $\bar T$ and splits the theory into two regimes.
\begin{itemize}
  \item \emph{Bounded dwell-times} ($\bar T<\infty$), as for the constant and range dwell-time families of Sections~\ref{subsec:nc:cst} and \ref{subsec:nc:rng}. The choice $g(t)=M_0e^{\omega_0 t}$ with $\omega_0>0$ gives $G=M_0e^{\omega_0\bar T}$ finite, so \eqref{eq:G} places no restriction on the semigroup: the generator $A$ may be exponentially unstable. The same conclusion applies to an individual sequence whose dwell-times are bounded, even when it belongs to $\mathcal{S}_{0,\infty}$.
  \item \emph{Unbounded dwell-times} ($\bar T=\infty$), as for $\mathcal{S}_{0,\infty}$ and the minimum dwell-time family $\mathcal{S}_{\min}(T_{\min})$. Here \eqref{eq:G} holds only if $\omega_0\le0$, in which case $G=M_0$. By Proposition~\ref{prop:necessity:semigroup}, exponential stability under arbitrarily long jump-free intervals already requires $S$ to be exponentially stable. The situation parallels the finite-dimensional case, in which an impulsive or switched system admitting unbounded dwell-times is stable only if the underlying continuous dynamics is stable. The condition $\omega_0\le0$ is thus necessary rather than restrictive.
\end{itemize}
\end{remark}

\begin{proposition}[Necessity of semigroup stability under unbounded dwell-times]\label{prop:necessity:semigroup}
Suppose $\bar T=\infty$. If the system \eqref{eq:syst} is GES (Definition~\ref{def:stab}) for the sequence $\sigma$, or UGES with respect to a family $\mathcal S$ with $\bar T(\mathcal S)=\infty$, with constants $M\ge1$, $\alpha>0$, $\rho\in(0,1)$, then the semigroup is exponentially stable:
\begin{equation}\label{eq:necessity:S}
  \|S(\tau)\|\le M e^{-\alpha\tau}\qquad\text{for all }\tau\ge0 .
\end{equation}
In particular the growth bound of $S$ satisfies $\omega_0\le-\alpha<0$, and Assumption~\ref{ass:growth} then holds with $G=M_0$.
\end{proposition}
\begin{proof}
Fix $\tau\ge0$. Since $\bar T=\infty$, there exist a jump index $k$ (and, in the family case, a sequence $\sigma\in\mathcal S$) with $t_{k+1}-t_k>\tau$. The interval $[t_k,t_k+\tau]$ then contains no jump, so $U_\sigma(t_k+\tau,t_k)=S(\tau)$ and $\kappa_\sigma(t_k+\tau,t_k)=0$. The uniform-in-initial-time bound of Definition~\ref{def:stab}, applied at starting time $s=t_k$, yields $\|S(\tau)\|=\|U_\sigma(t_k+\tau,t_k)\|\le M\rho^{0}e^{-\alpha\tau}=Me^{-\alpha\tau}$. As $\tau\ge0$ was arbitrary, \eqref{eq:necessity:S} follows.
\end{proof}

%
%

The following is the central technical lemma of this section. It extends Datko's classical result \cite{Datko:70} to the impulsive setting, where the propagator is an evolutionary process rather than a semigroup, and neither time-translation invariance nor group structure is available. The key advance over the plain exponential bound is that the proof extracts a \emph{hybrid} decay rate, with explicit contributions from both the continuous flow and the discrete jumps, matching the parametrization of Definition~\ref{def:stab} exactly.

\begin{lemma}[Impulsive Datko Lemma]\label{lem:datko}
Let $\sigma\in\mathcal{S}_{0,\infty}$ be given. Under Assumption~\ref{ass:growth}, suppose that for some $p\ge1$ and $b>0$ the condition
\begin{equation}\label{eq:datko:all}
  \int_s^\infty\|U_\sigma(t,s)x\|^p  \dt
  \;+\sum_{k\ge1:  t_k>s}\|U_\sigma(t_k^-,s)x\|^p
  \;\le\; b^p\|x\|^p
\end{equation}
holds for all $s\ge0$ and all $x\in X$. Then there exist constants $M\ge1$, $\alpha>0$, and $\rho\in(0,1)$ such that
\begin{equation}\label{eq:ges:hybrid}
  \|U_\sigma(t,s)x\|\le M  e^{-\alpha(t-s)}\rho^{\kappa_\sigma(t,s)}\|x\|
\end{equation}
for all $t\ge s\ge0$ and all $x\in X$.
\end{lemma}

\begin{proof}
The proof proceeds in four steps. Steps~1--2 establish uniform boundedness and two independent polynomial decay bounds, one from each part of condition \eqref{eq:datko:all}. Steps~3A and~3B convert each polynomial bound into a geometric decay bound via iteration. Step~4 interpolates the two geometric bounds to produce the hybrid decay rate \eqref{eq:ges:hybrid}.

\medskip\noindent\textbf{Step 1 (Uniform boundedness).}
We show that $\|U_\sigma(T,s)x\|\le C_0\|x\|$ for all $T\ge s\ge0$ and all $x\in X$, where $C_0:=G\max(1,\|J\|  b)$.

Let $T\ge s\ge0$ and $x\in X$ be arbitrary. Let $t^*$ denote the last jump of $\sigma$ in the interval $(s,T]$ if such a jump exists, and write $t^*=\varnothing$ otherwise.

\emph{Case (a): no jump in $(s,T]$, i.e., $t^*=\varnothing$.} Then $U_\sigma(T,s)=S(T-s)$, so the growth bound gives $\|U_\sigma(T,s)x\|\le g(T-s)\|x\|\le G\|x\|\le C_0\|x\|$.

\emph{Case (b): at least one jump in $(s,T]$.} Since every term in the sum in \eqref{eq:datko:all} is nonnegative and their total does not exceed $b^p\|x\|^p$, each individual term satisfies
\begin{equation}
  \|U_\sigma(t_k^-,s)x\|^p
  \;\le\;\sum_{j\ge1:  t_j>s}\|U_\sigma(t_j^-,s)x\|^p
  \;\le\; b^p\|x\|^p,
\end{equation}
so in particular $\|U_\sigma(t^{*-},s)x\|\le b\|x\|$. Using the propagator identity $U_\sigma(T,s)=S(T-t^*)JU_\sigma(t^{*-},s)$ and the growth bound on $S(T-t^*)$:
\begin{equation}
  \|U_\sigma(T,s)x\|\le g(T-t^*)\|J\|  \|U_\sigma(t^{*-},s)x\|\le G\|J\|  b  \|x\|\le C_0\|x\|.
\end{equation}
In both cases $\|U_\sigma(T,s)x\|\le C_0\|x\|$, establishing uniform boundedness with constant $C_0$.

\medskip\noindent\textbf{Step 2A (Polynomial decay from the integral part of \eqref{eq:datko:all}).}
For any $s\ge0$, $T'>0$, and $x\in X$, insert the identity $U_\sigma(s+T',s)=U_\sigma(s+T',\tau)U_\sigma(\tau,s)$ and integrate over $\tau\in[s,s+T']$:
\begin{equation}
  T'  \|U_\sigma(s+T',s)x\|^p
  =\int_s^{s+T'}\|U_\sigma(s+T',\tau)  U_\sigma(\tau,s)x\|^p  \d\tau.
\end{equation}
Apply the uniform bound of Step~1 to $U_\sigma(s+T',\tau)$, and then the integral part of \eqref{eq:datko:all} at starting time $s$:
\begin{equation}
  T'  \|U_\sigma(s+T',s)x\|^p
  \;\le\; C_0^p\int_s^{s+T'}\|U_\sigma(\tau,s)x\|^p  \d\tau
  \;\le\; C_0^p b^p\|x\|^p.
\end{equation}
Dividing by $T'$:
\begin{equation}\label{eq:poly:int}
  \|U_\sigma(s+T',s)x\|\;\le\;\frac{C_0  b}{(T')^{1/p}}  \|x\|.
\end{equation}
This bound holds for every $s\ge0$ and $T'>0$, uniformly.

\medskip\noindent\textbf{Step 2B (Polynomial decay from the sum part of \eqref{eq:datko:all}).}
Let $t>s$ and $K:=\kappa_\sigma(t,s)\ge1$ be the number of jumps of $\sigma$ in $(s,t]$. Enumerate these jumps as $t_{k_1}<t_{k_2}<\cdots<t_{k_K}$. For each index $k_i$ with $s<t_{k_i}\le t$, use the propagator identity $U_\sigma(t,s)=U_\sigma(t,t_{k_i}^-)U_\sigma(t_{k_i}^-,s)$ together with $U_\sigma(t,t_{k_i}^-)=U_\sigma(t,t_{k_i})J$; since Step~1 bounds the propagator started at the \emph{regular} instant $t_{k_i}$, we get $\|U_\sigma(t,t_{k_i}^-)y\|=\|U_\sigma(t,t_{k_i})Jy\|\le C_0\|J\|\,\|y\|$ (the extra factor $\|J\|$ because $t_{k_i}^-$ is a pre-jump instant, where the uniform bound of Step~1 is not directly available), and hence
\begin{equation}
  \|U_\sigma(t,s)x\|^p\;\le\; (C_0\|J\|)^p\|U_\sigma(t_{k_i}^-,s)x\|^p.
\end{equation}
Sum this inequality over all $K$ jumps $t_{k_i}\in(s,t]$:
\begin{equation}
  K  \|U_\sigma(t,s)x\|^p
  \;\le\; (C_0\|J\|)^p\sum_{i=1}^{K}\|U_\sigma(t_{k_i}^-,s)x\|^p
  \;\le\; (C_0\|J\|)^p\sum_{j\ge1:  t_j>s}\|U_\sigma(t_j^-,s)x\|^p
  \;\le\; (C_0\|J\|)^p b^p\|x\|^p,
\end{equation}
where the last inequality uses the sum part of \eqref{eq:datko:all} at starting time $s$. Dividing by $K$:
\begin{equation}\label{eq:poly:sum}
  \|U_\sigma(t,s)x\|\;\le\;\frac{C_0\|J\|  b}{K^{1/p}}  \|x\|,
  \qquad K=\kappa_\sigma(t,s)\ge1.
\end{equation}
This bound holds for every $s\ge0$ and every $t>s$ with at least one jump in $(s,t]$, uniformly.

\medskip\noindent\textbf{Step 3A (Exponential time decay by iteration).}
Define $T_0:=(2C_0 b)^p$. Setting $T'=T_0$ in \eqref{eq:poly:int}:
\begin{equation}
  \|U_\sigma(s+T_0,s)x\|\;\le\;\frac{C_0 b}{T_0^{1/p}}  \|x\|=\frac{1}{2}\|x\|,
\end{equation}
uniformly in $s\ge0$ and $x\in X$. Fix $s\ge0$, $t\ge s$, and $x\in X$. Write $t-s=nT_0+r$ with integer $n=\lfloor(t-s)/T_0\rfloor\ge0$ and $r\in[0,T_0)$. By the propagator chain rule and the contraction at each step of length $T_0$, using the Datko condition \eqref{eq:datko:all} at each intermediate starting time:
\begin{equation}
  \|U_\sigma(s+nT_0,s)x\|
  =\|U_\sigma(s+nT_0,  s+(n-1)T_0)\cdots U_\sigma(s+T_0,s)x\|
  \;\le\;\Bigl(\tfrac{1}{2}\Bigr)^n\|x\|.
\end{equation}
For the residual interval of length $r<T_0$, apply the uniform bound of Step~1:
\begin{equation}
  \|U_\sigma(t,s)x\|\le C_0\Bigl(\tfrac{1}{2}\Bigr)^n\|x\|.
\end{equation}
Since $n\ge(t-s)/T_0-1$, one has $(1/2)^n\le2\cdot(1/2)^{(t-s)/T_0}=2e^{-\alpha_0(t-s)}$ where $\alpha_0:=\log2/T_0>0$. Therefore:
\begin{equation}\label{eq:exp:time}
  \|U_\sigma(t,s)x\|\;\le\;2C_0  e^{-\alpha_0(t-s)}\|x\|,
\end{equation}
for all $t\ge s\ge0$ and all $x\in X$.

\medskip\noindent\textbf{Step 3B (Geometric jump-count decay by iteration).}
Define $K_0:=\lceil(2C_0\|J\| b)^p\rceil$ (the smallest integer $\ge(2C_0\|J\|b)^p$). Setting $K=K_0$ in \eqref{eq:poly:sum}:
\begin{equation}
  \|U_\sigma(t,s)x\|\;\le\;\frac{C_0\|J\| b}{K_0^{1/p}}  \|x\|\le\frac{1}{2}\|x\|
\end{equation}
whenever $\kappa_\sigma(t,s)\ge K_0$ (using $K_0\ge(2C_0\|J\|b)^p$). Fix $s\ge0$, $t\ge s$, and $x\in X$. Write $K:=\kappa_\sigma(t,s)=nK_0+r$ with integer $n=\lfloor K/K_0\rfloor\ge0$ and $r\in\{0,\ldots,K_0-1\}$. Let $\tau$ denote the time of the $(nK_0)$-th jump of $\sigma$ after $s$ (with the convention $\tau=s$ if $n=0$), so that $\kappa_\sigma(\tau,s)=nK_0$ and $\kappa_\sigma(t,\tau)=r<K_0$.

We show $\|U_\sigma(\tau,s)x\|\le(1/2)^n\|x\|$ by induction on $n$. For $n=0$: $\tau=s$, so the bound is $\|x\|\le\|x\|$, which is trivially true. For the inductive step from $n$ to $n+1$: let $\tau'$ be the time of the $(nK_0)$-th jump after $s$ and $\tau''$ the time of the $((n+1)K_0)$-th jump after $s$, so $\kappa_\sigma(\tau'',\tau')=K_0$. By the inductive hypothesis $\|U_\sigma(\tau',s)x\|\le(1/2)^n\|x\|$. Since \eqref{eq:datko:all} holds at starting time $\tau'$ (it holds at all starting times by hypothesis), and $\kappa_\sigma(\tau'',\tau')=K_0\ge K_0$, bound \eqref{eq:poly:sum} applied from $\tau'$ to $\tau''$ gives:
\begin{equation}
  \|U_\sigma(\tau'',\tau')  y\|\;\le\;\frac{1}{2}\|y\|, \qquad y=U_\sigma(\tau',s)x.
\end{equation}
Hence $\|U_\sigma(\tau'',s)x\|=\|U_\sigma(\tau'',\tau')U_\sigma(\tau',s)x\|\le(1/2)^{n+1}\|x\|$, completing the induction.

For the residual $r<K_0$ jumps from $\tau$ to $t$, apply the uniform bound:
\begin{equation}
  \|U_\sigma(t,s)x\|\le C_0\Bigl(\tfrac{1}{2}\Bigr)^n\|x\|.
\end{equation}
Since $n\ge K/K_0-1$, one has $(1/2)^n\le2\cdot(1/2)^{K/K_0}=2\rho_0^K$ where $\rho_0:=(1/2)^{1/K_0}\in(0,1)$. Therefore:
\begin{equation}\label{eq:exp:jump}
  \|U_\sigma(t,s)x\|\;\le\;2C_0  \rho_0^{\kappa_\sigma(t,s)}\|x\|,
\end{equation}
for all $t\ge s\ge0$ and all $x\in X$.

\medskip\noindent\textbf{Step 4 (Hybrid bound by interpolation).}
Steps~3A and~3B have established two independent upper bounds that hold simultaneously:
\begin{equation}
  \|U_\sigma(t,s)x\|\le2C_0  e^{-\alpha_0(t-s)}\|x\|, \qquad  \|U_\sigma(t,s)x\|\le2C_0  \rho_0^{\kappa_\sigma(t,s)}\|x\|
\end{equation}
For any $\theta\in(0,1)$, multiply the first bound raised to the power $\theta$ by the second raised to the power $1-\theta$:
\begin{equation}
  \|U_\sigma(t,s)x\|
  =\|U_\sigma(t,s)x\|^{\theta}\cdot\|U_\sigma(t,s)x\|^{1-\theta}
  \;\le\; 2C_0  e^{-\alpha_0\theta(t-s)}\rho_0^{(1-\theta)\kappa_\sigma(t,s)}\|x\|.
\end{equation}
Setting $M:=2C_0\ge1$, $\alpha:=\alpha_0\theta>0$, and $\rho:=\rho_0^{1-\theta}\in(0,1)$ yields \eqref{eq:ges:hybrid} and completes the proof.
\end{proof}

\begin{remark}\label{rem:ges:rho}
The parameter $\theta\in(0,1)$ in Step~4 distributes the decay between the flow contribution $e^{-\alpha(t-s)}$ and the jump contribution $\rho^{\kappa_\sigma(t,s)}$. As $\theta\to1$ one recovers a purely exponential time bound with $\rho\to1$; as $\theta\to0$ one recovers a purely geometric jump-count bound with $\alpha\to0$. The thresholds $T_0=(2C_0 b)^p$ (flow part) and $K_0=\lceil(2C_0\|J\| b)^p\rceil$ (jump part) arise from the analogous polynomial-decay thresholds $C_0 b/T^{1/p}=1/2$ and $C_0\|J\| b/K^{1/p}=1/2$ of the two parts of \eqref{eq:datko:all}, the jump part carrying the extra factor $\|J\|$ from the pre-jump propagator (with the jump count rounded up to an integer). Any choice of $\theta$ yields a valid GES parametrization; the natural balanced choice is $\theta=1/2$, giving $\alpha=\alpha_0/2$ and $\rho=\rho_0^{1/2}$.
\end{remark}

The two extremal parametrizations described in Remark~\ref{rem:ges:rho} correspond to the degenerate cases in which only one of the two terms of \eqref{eq:datko:all} is assumed finite. They arise in the persistent-jumping and persistent-flowing results of Section~\ref{sec:lyapunov} and are isolated in the following two corollaries. Each retains only the steps of the proof of Lemma~\ref{lem:datko} that use the corresponding part of \eqref{eq:datko:all}; the interpolation of Step~4 is not invoked.

\begin{corollary}[Sum-only Datko condition]\label{cor:datko:sum}
Let $\sigma\in\mathcal{S}_{0,\infty}$ be given and let Assumption~\ref{ass:growth} hold. Suppose that for some $p\ge1$ and $b>0$ the condition
\begin{equation}\label{eq:datko:sum}
  \sum_{k\ge1:  t_k>s}\|U_\sigma(t_k^-,s)x\|^p\;\le\;b^p\|x\|^p
\end{equation}
holds for all $s\ge0$ and all $x\in X$. Then there exist constants $M\ge1$ and $\rho_0\in(0,1)$ such that
\begin{equation}
  \|U_\sigma(t,s)x\|\le M  \rho_0^{\kappa_\sigma(t,s)}\|x\|
\end{equation}
for all $t\ge s\ge0$ and all $x\in X$; that is, GES with $\alpha=0$ and $\rho=\rho_0$.
\end{corollary}

\begin{proof}
Condition \eqref{eq:datko:sum} is the second term of \eqref{eq:datko:all} alone. In the proof of Lemma~\ref{lem:datko}, Step~1 uses only this sum, its Case~(b) bounding the pre-jump state through $\|U_\sigma(t^{*-},s)x\|\le b\|x\|$ and its Case~(a) using neither part of \eqref{eq:datko:all}, so uniform boundedness holds with $C_0=G\max(1,\|J\|  b)$. Steps~2B and~3B likewise use only the sum, and therefore apply without change, yielding the geometric bound \eqref{eq:exp:jump}, which is the claim with $M=2C_0$ and $\rho_0=(1/2)^{1/K_0}$, $K_0=\lceil(2C_0\|J\|b)^p\rceil$. Steps~2A, 3A and~4, which require the integral, are not used.
\end{proof}

\begin{corollary}[Integral-only Datko condition]\label{cor:datko:int}
Let $\sigma\in\mathcal{S}_{0,\infty}$ have a positive infimal dwell-time, $T_\sigma:=\inf_{k\ge0}(t_{k+1}-t_k)>0$, and let Assumption~\ref{ass:growth} hold. Suppose that for some $p\ge1$ and $b>0$ the condition
\begin{equation}\label{eq:datko:int}
  \int_s^\infty\|U_\sigma(t,s)x\|^p  \dt\;\le\;b^p\|x\|^p
\end{equation}
holds for all $s\ge0$ and all $x\in X$. Then there exist constants $M\ge1$ and $\alpha_0>0$ such that
\begin{equation}
  \|U_\sigma(t,s)x\|\le M  e^{-\alpha_0(t-s)}\|x\|
\end{equation}
for all $t\ge s\ge0$ and all $x\in X$; that is, GES with $\alpha=\alpha_0$ and $\rho=1$.
\end{corollary}

\begin{proof}
Condition \eqref{eq:datko:int} is the first term of \eqref{eq:datko:all} alone. Steps~2A and~3A of the proof of Lemma~\ref{lem:datko} use only this integral and yield the exponential bound \eqref{eq:exp:time}, provided the uniform boundedness of Step~1 is available. Step~1 as written bounds the pre-jump state through the sum, which is now unavailable, and must be replaced by the following argument; this is the only place the hypothesis $T_\sigma>0$ enters.

Let $T\ge s\ge0$ and $x\in X$, and let $t^*$ be the last jump of $\sigma$ in $(s,T]$, if any. If there is none, $U_\sigma(T,s)=S(T-s)$ and $\|U_\sigma(T,s)x\|\le G\|x\|$. Otherwise, let $t_{\rm prev}$ be the jump preceding $t^*$, or $t_{\rm prev}:=s$ if $t^*$ is the first jump after $s$, so that $(t_{\rm prev},t^*)$ is jump-free. If $t^*-t_{\rm prev}<T_\sigma$ then, by definition of $T_\sigma$, no jump precedes $t^*$ in $(s,t^*)$, so $t_{\rm prev}=s$, $(s,t^*)$ is jump-free, and $\|U_\sigma(t^{*-},s)x\|=\|S(t^*-s)x\|\le G\|x\|$. Otherwise $t^*-t_{\rm prev}\ge T_\sigma$; for every $\tau\in(t_{\rm prev},t^*)$ the semigroup relation $U_\sigma(t^{*-},s)x=S(t^*-\tau)U_\sigma(\tau,s)x$ and the growth bound give $\|U_\sigma(t^{*-},s)x\|\le G\|U_\sigma(\tau,s)x\|$, the elapsed time $t^*-\tau$ not exceeding $\bar T$. Raising to the power $p$, integrating over $\tau\in(t_{\rm prev},t^*)$ and using \eqref{eq:datko:int} at starting time $s$,
\begin{equation}
  (t^*-t_{\rm prev})\|U_\sigma(t^{*-},s)x\|^p\le G^p\!\int_{t_{\rm prev}}^{t^*}\!\|U_\sigma(\tau,s)x\|^p  \d\tau\le G^p b^p\|x\|^p,
\end{equation}
whence $\|U_\sigma(t^{*-},s)x\|\le G  b  T_\sigma^{-1/p}\|x\|$. In all cases $\|U_\sigma(t^{*-},s)x\|\le\beta\|x\|$ with $\beta:=G\max(1,b  T_\sigma^{-1/p})$, and propagating through the last jump, $U_\sigma(T,s)=S(T-t^*)JU_\sigma(t^{*-},s)$ with $g(T-t^*)\le G$, gives $\|U_\sigma(T,s)x\|\le C_0\|x\|$ with $C_0:=G\max(1,\|J\|\beta)$.

With uniform boundedness established, Steps~2A and~3A apply without change and yield \eqref{eq:exp:time}, the claim with $M=2C_0$ and $\alpha_0=\log2/T_0$, $T_0=(2C_0b)^p$. Steps~2B, 3B and~4 are not used.
\end{proof}

\begin{remark}\label{rem:datko:int:dwell}
The dwell-time hypothesis $T_\sigma>0$ is required for Corollary~\ref{cor:datko:int} but not for Corollary~\ref{cor:datko:sum}. When only the jump sum is controlled it bounds each pre-jump state directly, whereas when only the integral is controlled the pre-jump state must be recovered from the flow preceding it, through $\|U_\sigma(t^{*-},s)x\|\le G  b  (t^*-t_{\rm prev})^{-1/p}\|x\|$, which degenerates as the preceding dwell-time tends to zero. This is intrinsic: the integral in \eqref{eq:datko:int} is blind to the instantaneous pre-jump values and no backward propagator is available to recover them, so an accumulation of arbitrarily short inter-jump intervals cannot be excluded by the integral alone. The hypothesis is not restrictive in the use made of this corollary, which concerns sequences of positive infimal dwell-time.
\end{remark}

\section{Non-coercive Lyapunov stability conditions}
\label{sec:lyapunov}

We now present a complete Lyapunov characterization of GES for the impulsive system \eqref{eq:syst} via non-coercive functionals. In contrast to the classical coercive Lyapunov approach, no lower bound on the functional is required; stability is established entirely through an upper bound combined with decrease conditions on flow and at jumps. Each subsection treats a different class of impulse sequences; within each subsection, a single theorem establishes the equivalence between (i) the relevant exponential stability notion, (ii) a hybrid Datko condition, (iii) a non-coercive Lyapunov criterion, and, where applicable, (iv) a discrete-time Lyapunov criterion on the monodromy operator.

\subsection{Fixed impulse sequence}\label{subsec:nc:fixed}

We address in this section the case of a fixed impulse sequence, that is the case where $\mathcal{S}$ is the singleton $\{\sigma\}$. The following result is the main result in this case:
\begin{theorem}[Fixed sequence, non-coercive N\&S]\label{thm:main}
Let $\sigma\in\mathcal{S}_{0,\infty}$ be given and let Assumption~\ref{ass:growth} hold. Then the following statements are equivalent:
\begin{enumerate}[(i)]
  \item The system \eqref{eq:syst} is strongly GES (Definition~\ref{def:stab}): there exist $M\ge1$, $\alpha>0$, and $\rho\in(0,1)$ such that
    \begin{equation}
      \|U_\sigma(t,s)\|\le M\rho^{\kappa_\sigma(t,s)}e^{-\alpha (t-s)} \qquad\forall  t\ge s\ge0.
    \end{equation}

  \item There exist $p\ge1$ and $b>0$ such that
    \begin{equation}\label{eq:datko:thm}
      \int_s^\infty\|U_\sigma(t,s)x\|^p  \dt
      \;+\sum_{k\ge1:  t_k>s}\|U_\sigma(t_k^-,s)x\|^p
      \;\le\;b^p\|x\|^p
    \end{equation}
    holds for all $s\ge0$ and all $x\in X$.

  \item There exist a functional $V:\mathbb{R}_{\ge0}\times X\to\mathbb{R}_{\ge0}$, Lipschitz on bounded subsets of $X$ uniformly in $t$, {such that, for every $x\in X$, the map $t\mapsto V(t,x)$ is continuous on $[t_k,t_{k+1})$ for all $k\ge0$ and admits a left limit $V(t_k^-,x)$ at $t_k$ for all $k\ge1$,} and constants $p\ge1$, $c>0$ such that
    \begin{align}
      V(t,x)&\le c\|x\|^p, &&\forall t\ge0,\label{eq:nc:upper}\\
      \underline{D}^+V(t,x)&\le -\|x\|^p, &&\forall t\ge0,\label{eq:nc:flow}\\
      V(t,Jx)-V(t^-,x)&\le -\|x\|^p, &&\forall t\in\mathbb{T}_\sigma,\label{eq:nc:jump}
    \end{align}
    hold for all $x\in X$.
\end{enumerate}
Moreover, when (i) holds, an explicit Lyapunov functional witnessing (iii) is given by
\begin{equation}\label{eq:V:def}
  V(t,x)\;:=\;\int_t^\infty\|U_\sigma(s,t)x\|^p  \ds+\sum_{k\ge1:  t_k>t}\|U_\sigma(t_k^-,t)x\|^p,\qquad t\ge0,\;x\in X.
\end{equation}
\end{theorem}

\begin{proof}
We prove the four implications $(i)\Rightarrow(ii)$, $(ii)\Rightarrow(iii)$, $(iii)\Rightarrow(ii)$, and $(ii)\Rightarrow(i)$.

\medskip\noindent\textbf{Proof that (i) implies (ii).}
Assume strong GES with constants $M$, $\alpha>0$, $\rho\in(0,1)$. By Definition~\ref{def:stab}, which requires the bound uniformly in the initial time, for any $s\ge0$ and $t\ge s$,
\begin{equation}
  \|U_\sigma(t,s)x\|\le M\rho^{\kappa_\sigma(t,s)}e^{-\alpha(t-s)}\|x\|.
\end{equation}

For the integral, since $\rho^{p\kappa_\sigma(t,s)}\le1$:
\begin{equation}
  \int_s^\infty\|U_\sigma(t,s)x\|^p  \dt
  \;\le\; M^p\|x\|^p\int_s^\infty e^{-\alpha p(t-s)}  \dt
  \;=\;\frac{M^p}{\alpha p}\|x\|^p.
\end{equation}

For the sum, let $K_s:=\kappa_\sigma(s,0)$ and index the jumps after $s$ as $t_{K_s+1}<t_{K_s+2}<\cdots$. Since $\kappa_\sigma(t_{K_s+j}^-,s)=j-1$ (there are $j-1$ complete jumps in $(s,t_{K_s+j})$) and $e^{-\alpha p(t_{K_s+j}-s)}\le1$:
\begin{equation}
  \sum_{k\ge1:  t_k>s}\|U_\sigma(t_k^-,s)x\|^p
  \;\le\; M^p\|x\|^p\sum_{j=1}^\infty\rho^{p(j-1)}e^{-\alpha p(t_{K_s+j}-s)}
  \;\le\; M^p\|x\|^p\sum_{j=1}^\infty\rho^{p(j-1)}
  \;=\;\frac{M^p}{1-\rho^p}\|x\|^p.
\end{equation}
Adding both bounds and setting $b^p:=M^p\!\left(\frac{1}{\alpha p}+\frac{1}{1-\rho^p}\right)$ establishes \eqref{eq:datko:thm} for all $s\ge0$ and $x\in X$.

\medskip\noindent\textbf{Proof that (ii) implies (iii).}
Assume (ii) holds. Define $V$ by \eqref{eq:V:def}. This is the left-hand side of \eqref{eq:datko:thm} evaluated at starting time $t$ with initial state $x$. We verify each condition in (iii) in turn.

\emph{Upper bound \eqref{eq:nc:upper}.} Applying \eqref{eq:datko:thm} at $s=t$:
\begin{equation}
  V(t,x)\;\le\; b^p\|x\|^p \qquad\forall  t\ge0,  x\in X.
\end{equation}
This is \eqref{eq:nc:upper} with $c=b^p$.

\emph{Flow decrease \eqref{eq:nc:flow}.} Let $t\ge0$ with $t\notin\mathbb{T}_\sigma$, and let $h>0$ be small enough so that $(t,t+h)\cap\mathbb{T}_\sigma=\varnothing$ (possible since the jump sequence has no accumulation points). Since there are no jumps in $(t,t+h)$, the propagator satisfies $U_\sigma(t+h,t)=S(h)$, and for any $s>t+h$ the chain rule gives $U_\sigma(s,t+h)S(h)=U_\sigma(s,t+h)U_\sigma(t+h,t)=U_\sigma(s,t)$. Moreover, $U_\sigma(s,t)=S(s-t)$ for $s\in(t,t+h)$ since no jump occurs in this interval. Therefore:
\begin{align}
  V(t+h,S(h)x)
  &=\int_{t+h}^\infty\|U_\sigma(s,t+h)S(h)x\|^p  \ds
   +\sum_{k\ge1:  t_k>t+h}\|U_\sigma(t_k^-,t+h)S(h)x\|^p\notag\\
  &=\int_{t+h}^\infty\|U_\sigma(s,t)x\|^p  \ds
   +\sum_{k\ge1:  t_k>t}\|U_\sigma(t_k^-,t)x\|^p,
\end{align}
where the jump sum is unchanged because $(t,t+h)$ contains no jump. Subtracting $V(t,x)$:
\begin{equation}
  V(t+h,S(h)x)-V(t,x)
  =-\int_t^{t+h}\|U_\sigma(s,t)x\|^p  \ds
  =-\int_t^{t+h}\|S(s-t)x\|^p  \ds.
\end{equation}
Dividing by $h>0$ and taking $h\downarrow0$, by strong continuity of $S(\cdot)$ at $0$:
\begin{equation}
  D^+V(t,x)
  =\lim_{h\downarrow0}\frac{V(t+h,S(h)x)-V(t,x)}{h}
  =-\lim_{h\downarrow0}\frac{1}{h}\int_t^{t+h}\|S(s-t)x\|^p  \ds
  =-\|x\|^p.
\end{equation}
In particular $\underline{D}^+V(t,x)=D^+V(t,x)=-\|x\|^p$, so \eqref{eq:nc:flow} holds with equality.

\emph{Jump decrease \eqref{eq:nc:jump}.} Let $t_k\in\mathbb{T}_\sigma$ be any jump time. From the propagator identity $U_\sigma(t,t_k^-)=U_\sigma(t,t_k)J$ (see \eqref{eq:generator}), it follows that $U_\sigma(\cdot,t_k)Jx=U_\sigma(\cdot,t_k^-)x$ for all $x\in X$ and all times after $t_k$. Therefore:
\begin{equation}
  V(t_k,Jx)
  =\int_{t_k}^\infty\|U_\sigma(s,t_k)Jx\|^p  \ds
  +\sum_{j\ge1:  t_j>t_k}\|U_\sigma(t_j^-,t_k)Jx\|^p
  =\int_{t_k}^\infty\|U_\sigma(s,t_k^-)x\|^p  \ds
  +\sum_{j\ge1:  t_j>t_k}\|U_\sigma(t_j^-,t_k^-)x\|^p.
\end{equation}
For $V(t_k^-,x)$, the sum includes the $j=k$ term, for which $\|U_\sigma(t_k^-,t_k^-)x\|=\|x\|$ (propagator at equal times is the identity):
\begin{equation}
  V(t_k^-,x)
  =\int_{t_k}^\infty\|U_\sigma(s,t_k^-)x\|^p  \ds
  +\|x\|^p
  +\sum_{j\ge1:  t_j>t_k}\|U_\sigma(t_j^-,t_k^-)x\|^p.
\end{equation}
Subtracting:
\begin{equation}
  V(t_k,Jx)-V(t_k^-,x)=-\|x\|^p,
\end{equation}
so \eqref{eq:nc:jump} holds with equality.

\emph{Lipschitz continuity on bounded subsets.} By condition (ii) and Lemma~\ref{lem:datko}, the system satisfies the hybrid bound \eqref{eq:ges:hybrid} with some $M\ge1$, $\alpha>0$, $\rho\in(0,1)$. Fix $r>0$ and let $\|x\|,\|y\|\le r$. For any $s\ge t$, the linearity of $U_\sigma(s,t)$ gives $U_\sigma(s,t)(x-y)=U_\sigma(s,t)x-U_\sigma(s,t)y$, and the estimate
\begin{equation}
  \bigl|\|U_\sigma(s,t)x\|^p-\|U_\sigma(s,t)y\|^p\bigr|
  \;\le\; p\max\bigl(\|U_\sigma(s,t)x\|,\|U_\sigma(s,t)y\|\bigr)^{p-1}\|U_\sigma(s,t)(x-y)\|.
\end{equation}
Using \eqref{eq:ges:hybrid} to bound $\max(\|U_\sigma(s,t)x\|,\|U_\sigma(s,t)y\|)\le M e^{-\alpha(s-t)}\rho^{\kappa_\sigma(s,t)}r$ and $\|U_\sigma(s,t)(x-y)\|\le Me^{-\alpha(s-t)}\rho^{\kappa_\sigma(s,t)}\|x-y\|$:
\begin{equation}
  \bigl|\|U_\sigma(s,t)x\|^p-\|U_\sigma(s,t)y\|^p\bigr|
  \;\le\; pM^p r^{p-1}e^{-\alpha p(s-t)}\rho^{p\kappa_\sigma(s,t)}\|x-y\|.
\end{equation}
For the integral part of $|V(t,x)-V(t,y)|$:
\begin{equation}
  \int_t^\infty pM^p r^{p-1}e^{-\alpha p(s-t)}\rho^{p\kappa_\sigma(s,t)}\|x-y\|  \ds
  \;\le\; pM^pr^{p-1}\|x-y\|\int_t^\infty e^{-\alpha p(s-t)}  \ds
  \;=\;\frac{M^p r^{p-1}}{\alpha}\|x-y\|.
\end{equation}
For the sum part, let $K:=\kappa_\sigma(t,0)$ and index jumps after $t$ as $t_{K+1}<t_{K+2}<\cdots$. Since $\kappa_\sigma(t_{K+j}^-,t)=j-1$ and $e^{-\alpha p(t_{K+j}-t)}\le1$:
\begin{equation}
  \sum_{j=1}^\infty pM^pr^{p-1}e^{-\alpha p(t_{K+j}-t)}\rho^{p(j-1)}\|x-y\|
  \;\le\; pM^pr^{p-1}\|x-y\|\sum_{j=1}^\infty\rho^{p(j-1)}
  \;=\;\frac{pM^pr^{p-1}}{1-\rho^p}\|x-y\|.
\end{equation}
The series $\sum_{j=1}^\infty\rho^{p(j-1)}=1/(1-\rho^p)$ converges since $\rho<1$, regardless of the density of the jump sequence. Combining:
\begin{equation}
  |V(t,x)-V(t,y)|\;\le\;M^pr^{p-1}\!\left(\frac{1}{\alpha}+\frac{p}{1-\rho^p}\right)\|x-y\|\;=:\;L_r\|x-y\|,
\end{equation}
establishing Lipschitz continuity on bounded subsets.

{\emph{Continuity in $t$.} Let $k\ge0$, $x\in X$, and $t,t+h\in[t_k,t_{k+1})$ with $h>0$. Since $(t,t+h)\cap\mathbb{T}_\sigma=\varnothing$, the computation carried out above to establish \eqref{eq:nc:flow}, which remains valid at $t=t_k$, gives $V(t+h,S(h)x)=V(t,x)-\int_0^{h}\|S(s)x\|^p\ds$. Let $r>0$ be such that $\|S(s)x\|<r$ for all $s\in[0,1]$. Then, for $h\le1$,
\begin{equation*}
\begin{aligned}
  |V(t+h,x)-V(t,x)|&\le|V(t+h,x)-V(t+h,S(h)x)|+\int_0^{h}\|S(s)x\|^p\ds\\
  &\le L_r\|x-S(h)x\|+\int_0^{h}\|S(s)x\|^p\ds.
\end{aligned}
\end{equation*}
The right-hand side of this inequality does not depend on $t$ and tends to $0$ as $h\downarrow0$ by strong continuity of $S(\cdot)$. Hence, $t\mapsto V(t,x)$ is uniformly continuous on $[t_k,t_{k+1})$; since a uniformly continuous function on an interval extends continuously to its closure, it admits a left limit at $t_{k+1}$, which coincides with the expression of $V(t_{k+1}^-,x)$ obtained above when establishing \eqref{eq:nc:jump}.}

\medskip\noindent\textbf{Proof that (iii) implies (ii).}
Assume (iii) holds. Fix any $t_0\ge0$ and any $y\in X$. Let $x(t):=U_\sigma(t,t_0)y$ denote the solution of \eqref{eq:syst} starting from $y$ at time $t_0$. We show that the Datko condition \eqref{eq:datko:thm} holds at starting time $t_0$ with $b^p=c$.

{Let $k\ge0$ with $t_k\ge t_0$, and define $v(t):=V(t,x(t))$ for $t\in[t_k,t_{k+1})$. The solution $x(\cdot)$ is continuous on $[t_k,t_{k+1})$ and admits the left limit $x(t_{k+1}^-)$, so there exists $r>0$ such that $\|x(t)\|<r$ for all $t\in[t_k,t_{k+1})$ and $\|x(t_{k+1}^-)\|<r$. For $t,t'\in[t_k,t_{k+1})$, the Lipschitz property of $V$ gives
\begin{equation*}
  |v(t')-v(t)|\le L_r\|x(t')-x(t)\|+|V(t',x(t))-V(t,x(t))|,
\end{equation*}
and both terms vanish as $t'\to t$, by continuity of $x(\cdot)$ and of $V(\cdot,x(t))$ on $[t_k,t_{k+1})$. Similarly,
\begin{equation*}
\begin{aligned}
  |v(t)-V(t_{k+1}^-,x(t_{k+1}^-))|&\le L_r\|x(t)-x(t_{k+1}^-)\|\\
  &\quad+|V(t,x(t_{k+1}^-))-V(t_{k+1}^-,x(t_{k+1}^-))|,
\end{aligned}
\end{equation*}
where both terms vanish as $t\uparrow t_{k+1}$.
Hence, $v$ is continuous on $[t_k,t_{k+1})$ and $v(t)\to V(t_{k+1}^-,x(t_{k+1}^-))$ as $t\uparrow t_{k+1}$. Moreover, since $x(t+h)=S(h)x(t)$ for $t,t+h\in[t_k,t_{k+1})$, we have
\begin{equation*}
  \liminf_{h\downarrow0}\frac{v(t+h)-v(t)}{h}=\underline{D}^+V(t,x(t))\le-\|x(t)\|^p,\qquad t\in[t_k,t_{k+1}),
\end{equation*}
by \eqref{eq:nc:flow}. As a result, the continuous function $w(t):=v(t)+\int_{t_k}^t\|x(s)\|^p\ds$ has a nonpositive lower right Dini derivative on $[t_k,t_{k+1})$, since $s\mapsto\|x(s)\|^p$ is continuous on this interval. A continuous function whose lower right Dini derivative is nonpositive on an interval is nonincreasing on that interval \cite{Hagood:06}. Hence, $w(t)\le w(t_k)$ for all $t\in[t_k,t_{k+1})$ and letting $t\uparrow t_{k+1}$ yields}
\begin{equation}\label{eq:flow:int}
  V(t_{k+1}^-,x(t_{k+1}^-))-V(t_k,x(t_k))
  \;\le\;-\int_{t_k}^{t_{k+1}}\|x(s)\|^p  \ds
  =-\int_{t_k}^{t_{k+1}}\|U_\sigma(s,t_0)y\|^p  \ds.
\end{equation}
At each jump time $t_{k+1}$ with $x(t_{k+1})=Jx(t_{k+1}^-)$, condition \eqref{eq:nc:jump} gives:
\begin{equation}\label{eq:jump:dec}
  V(t_{k+1},x(t_{k+1}))-V(t_{k+1}^-,x(t_{k+1}^-))
  \;\le\;-\|x(t_{k+1}^-)\|^p
  =-\|U_\sigma(t_{k+1}^-,t_0)y\|^p.
\end{equation}
Adding \eqref{eq:flow:int} and \eqref{eq:jump:dec} for $k=0,1,\ldots,n$ and telescoping (each $V(t_{k+1},\cdot)$ cancels with the $V(t_{k+1}^-,\cdot)$ from the next interval), for any $t\in(t_n,t_{n+1})$:
\begin{equation}
  V(t,x(t))
  \;\le\; V(t_0,y)
  -\int_{t_0}^t\|U_\sigma(s,t_0)y\|^p  \ds
  -\sum_{k\ge1:  t_0<t_k\le t}\|U_\sigma(t_k^-,t_0)y\|^p.
\end{equation}
Since $V\ge0$ everywhere and $V(t_0,y)\le c\|y\|^p$ by \eqref{eq:nc:upper}:
\begin{equation}
  \int_{t_0}^t\|U_\sigma(s,t_0)y\|^p  \ds
  +\sum_{k\ge1:  t_k>t_0}\|U_\sigma(t_k^-,t_0)y\|^p
  \;\le\;c\|y\|^p
\end{equation}
for all $t\ge t_0$. Letting $t\to\infty$ yields \eqref{eq:datko:thm} at starting time $t_0$ with $b^p=c$. Since $t_0\ge0$ and $y\in X$ were arbitrary, (ii) holds.

\medskip\noindent\textbf{Proof that (ii) implies (i).}
Condition (ii) and Assumption~\ref{ass:growth} satisfy the hypotheses of Lemma~\ref{lem:datko}. Lemma~\ref{lem:datko} therefore directly yields constants $M\ge1$, $\alpha>0$, and $\rho\in(0,1)$ such that $\|U_\sigma(t,s)\|\le Me^{-\alpha(t-s)}\rho^{\kappa_\sigma(t,s)}$ for all $t\ge s\ge0$. This is precisely strong GES (Definition~\ref{def:stab}; the bound being uniform in the initial time $s$), i.e.\ (i).
\end{proof}

The three conditions in Theorem~\ref{thm:main} admit the following interpretation. Condition~(ii) is an impulsive analogue of Datko's integral criterion \cite{Datko:70,Datko:72}: the $L^p$-in-time norm of the trajectory on $[s,\infty)$, augmented by the sum of pre-jump state norms, is bounded by the norm of the initial state. The integral term controls the continuous-time decay, while the sum over pre-jump states controls the jump-induced decay. Condition~(iii) is a non-coercive Lyapunov criterion: no lower bound on $V$ is required, and the explicit construction \eqref{eq:V:def} shows the natural candidate.


We close this subsection with two analogues of Theorem~\ref{thm:main} covering the degenerate cases where only one of the two mechanisms, flow or jumps, provides decay, and the other is neutral. Both results are stated for the same fixed sequence $\sigma\in\mathcal{S}_{0,\infty}$ as Theorem~\ref{thm:main}.

\begin{theorem}[Fixed sequence, persistent-jumping]\label{thm:pj}
Let $\sigma\in\mathcal{S}_{0,\infty}$ be given and let Assumption~\ref{ass:growth} hold. Then the following statements are equivalent:
\begin{enumerate}[(i)]
  \item The system \eqref{eq:syst} is GES with $\alpha=0$ and $\rho\in(0,1)$: there exist $M\ge1$ and $\rho\in(0,1)$ such that
    \begin{equation}
      \|U_\sigma(t,s)\|\le M\rho^{\kappa_\sigma(t,s)} \qquad\forall  t\ge s\ge0.
    \end{equation}

  \item There exist $p\ge1$ and $b>0$ such that
    \begin{equation}\label{eq:datko:pj}
      \sum_{k\ge1:  t_k>s}\|U_\sigma(t_k^-,s)x\|^p\;\le\;b^p\|x\|^p
    \end{equation}
    holds for all $s\ge0$ and all $x\in X$.

  \item There exist a functional $V:\mathbb{R}_{\ge0}\times X\to\mathbb{R}_{\ge0}$, Lipschitz on bounded subsets of $X$ uniformly in $t$, {such that, for every $x\in X$, the map $t\mapsto V(t,x)$ is continuous on $[t_k,t_{k+1})$ for all $k\ge0$ and admits a left limit $V(t_k^-,x)$ at $t_k$ for all $k\ge1$,} and constants $p\ge1$, $c>0$ such that
    \begin{align}
      V(t,x)&\le c\|x\|^p, &&\forall t\ge0,\label{eq:pj:upper}\\
      \underline{D}^+V(t,x)&\le 0, &&\forall t\ge0,\label{eq:pj:flow}\\
      V(t,Jx)-V(t^-,x)&\le -\|x\|^p, &&\forall t\in\mathbb{T}_\sigma,\label{eq:pj:jump}
    \end{align}
    hold for all $x\in X$.
\end{enumerate}
Moreover, when (i) holds, an explicit Lyapunov functional witnessing (iii) is given by
\begin{equation}\label{eq:V:pj:explicit}
  V(t,x)\;:=\;\sum_{k\ge1:  t_k>t}\|U_\sigma(t_k^-,t)x\|^p.
\end{equation}
\end{theorem}

\begin{proof}
The structure parallels Theorem~\ref{thm:main}, with the flow providing no decay and stability ensured entirely by the jumps. We prove $(i)\Rightarrow(ii)\Rightarrow(iii)\Rightarrow(ii)\Rightarrow(i)$.

\medskip\noindent\textbf{Proof that (i) implies (ii).}
For $s\ge0$ and $x\in X$, the GES bound gives $\|U_\sigma(t_k^-,s)x\|\le M\rho^{\kappa_\sigma(t_k^-,s)}\|x\|$. Since the $j$-th jump after $s$ satisfies $\kappa_\sigma(t_{K_s+j}^-,s)=j-1$ (where $K_s=\kappa_\sigma(s,0)$):
\begin{equation}
  \sum_{k\ge1:  t_k>s}\|U_\sigma(t_k^-,s)x\|^p
  \;\le\; M^p\|x\|^p\sum_{j=1}^\infty\rho^{p(j-1)}
  \;=\;\frac{M^p}{1-\rho^p}\|x\|^p.
\end{equation}
Setting $b^p:=M^p/(1-\rho^p)$ gives \eqref{eq:datko:pj} for all $s\ge0$.

\medskip\noindent\textbf{Proof that (ii) implies (iii).}
Define, for each $t\ge0$ and $x\in X$:
\begin{equation}\label{eq:V:pj}
  V(t,x)\;:=\;\sum_{k\ge1:  t_k>t}\|U_\sigma(t_k^-,t)x\|^p.
\end{equation}
The sum is bounded: $V(t,x)\le b^p\|x\|^p$ from \eqref{eq:datko:pj} at starting time $t$.

\emph{Flow decrease \eqref{eq:pj:flow}.} For $t\notin\mathbb{T}_\sigma$ and small $h>0$ with $(t,t+h)\cap\mathbb{T}_\sigma=\varnothing$: since $U_\sigma(t_{k}^-,t+h)S(h)=U_\sigma(t_k^-,t)$ for all $t_k>t+h$ (no jump in $(t,t+h)$), and the index set $\{k:t_k>t+h\}=\{k:t_k>t\}$ for small enough $h$:
\begin{equation}
  V(t+h,S(h)x)=\sum_{k\ge1:  t_k>t+h}\|U_\sigma(t_k^-,t+h)S(h)x\|^p
  =\sum_{k\ge1:  t_k>t}\|U_\sigma(t_k^-,t)x\|^p=V(t,x).
\end{equation}
Hence $D^+V(t,x)=0\le0$, so \eqref{eq:pj:flow} holds with equality (the flow is exactly neutral).

\emph{Jump decrease \eqref{eq:pj:jump}.} By the same propagator identity as in Theorem~\ref{thm:main}, $U_\sigma(\cdot,t_k)J=U_\sigma(\cdot,t_k^-)$, so:
\begin{equation}
  V(t_k,Jx)=\sum_{j\ge1:  t_j>t_k}\|U_\sigma(t_j^-,t_k^-)x\|^p, \qquad
  V(t_k^-,x)=\|x\|^p+\sum_{j\ge1:  t_j>t_k}\|U_\sigma(t_j^-,t_k^-)x\|^p.
\end{equation}
Subtracting: $V(t_k,Jx)-V(t_k^-,x)=-\|x\|^p$, so \eqref{eq:pj:jump} holds with equality.

\emph{Lipschitz continuity.} By condition (ii) and Corollary~\ref{cor:datko:sum}, the system satisfies $\|U_\sigma(t,s)\|\le M\rho_0^{\kappa_\sigma(t,s)}$ for $t\ge s\ge0$, with $M=2C_0$. A computation parallel to the Lipschitz argument of Theorem~\ref{thm:main} gives, for $\|x\|,\|y\|\le r$:
\begin{equation}
  |V(t,x)-V(t,y)|
  \;\le\;\sum_{j=1}^\infty pM^pr^{p-1}\rho_0^{p(j-1)}\|x-y\|
  =\frac{pM^pr^{p-1}}{1-\rho_0^p}\|x-y\|,
\end{equation}
establishing Lipschitz continuity (the sum converges since $\rho_0<1$, regardless of jump density).

{The continuity properties of $V$ in (iii) follow as in the proof of Theorem~\ref{thm:main} (Continuity in $t$), the argument applying with the absent term removed.}

\medskip\noindent\textbf{Proof that (iii) implies (ii).}
Fix $t_0\ge0$ and $y\in X$. Let $x(t)=U_\sigma(t,t_0)y$. {The integration of the Dini inequality on each inter-jump interval is justified as in the proof of $(iii)\Rightarrow(ii)$ of Theorem~\ref{thm:main}, using the continuity of $V(\cdot,x)$, the continuity of $x(\cdot)$ on $[t_k,t_{k+1})$, and the Lipschitz property of $V$. Since $\underline{D}^+V\le0$ by \eqref{eq:pj:flow}, the map $t\mapsto V(t,x(t))$ is therefore nonincreasing on each interval $[t_k,t_{k+1})$ with $t_k\ge t_0$, and letting $t\uparrow t_{k+1}$ yields}
\begin{equation}
  V(t_{k+1}^-,x(t_{k+1}^-))\;\le\;V(t_k,x(t_k)).
\end{equation}
At each jump $t_{k+1}$, condition \eqref{eq:pj:jump} gives $V(t_{k+1},x(t_{k+1}))\le V(t_{k+1}^-,x(t_{k+1}^-))-\|x(t_{k+1}^-)\|^p$. Combining and telescoping:
\begin{equation}
  V(t_n,x(t_n))\;\le\; V(t_0,y)-\sum_{k\ge1:  t_0<t_k\le t_n}\|U_\sigma(t_k^-,t_0)y\|^p.
\end{equation}
Since $V\ge0$ and $V(t_0,y)\le c\|y\|^p$, letting $n\to\infty$ yields \eqref{eq:datko:pj} at starting time $t_0$ with $b^p=c$.

\medskip\noindent\textbf{Proof that (ii) implies (i).}
Condition (ii) is the sum-only Datko condition \eqref{eq:datko:sum} at all starting times, so Corollary~\ref{cor:datko:sum} applies directly and yields $\|U_\sigma(t,s)\|\le M\rho_0^{\kappa_\sigma(t,s)}$ for $t\ge s\ge0$. This is GES with $\alpha=0$ and $\rho=\rho_0\in(0,1)$.
\end{proof}

\begin{theorem}[Fixed sequence, persistent-flowing]\label{thm:pf}
Let $\sigma\in\mathcal{S}_{0,\infty}$ have a positive infimal dwell-time $T_\sigma:=\inf_{k\ge0}(t_{k+1}-t_k)>0$, and let Assumption~\ref{ass:growth} hold. Then the following statements are equivalent:
\begin{enumerate}[(i)]
  \item The system \eqref{eq:syst} is GES with $\rho=1$ and $\alpha>0$: there exist $M\ge1$ and $\alpha>0$ such that
    \begin{equation}
      \|U_\sigma(t,s)\|\le Me^{-\alpha(t-s)} \qquad\forall  t\ge s\ge0.
    \end{equation}

  \item There exist $p\ge1$ and $b>0$ such that
    \begin{equation}\label{eq:datko:pf}
      \int_s^\infty\|U_\sigma(t,s)x\|^p  \dt\;\le\;b^p\|x\|^p
    \end{equation}
    holds for all $s\ge0$ and all $x\in X$.

  \item There exist a functional $V:\mathbb{R}_{\ge0}\times X\to\mathbb{R}_{\ge0}$, Lipschitz on bounded subsets of $X$ uniformly in $t$, {such that, for every $x\in X$, the map $t\mapsto V(t,x)$ is continuous on $[t_k,t_{k+1})$ for all $k\ge0$ and admits a left limit $V(t_k^-,x)$ at $t_k$ for all $k\ge1$,} and constants $p\ge1$, $c>0$ such that
    \begin{align}
      V(t,x)&\le c\|x\|^p, &&\forall t\ge0,\label{eq:pf:upper}\\
      \underline{D}^+V(t,x)&\le -\|x\|^p, &&\forall t\ge0,\label{eq:pf:flow}\\
      V(t,Jx)-V(t^-,x)&\le 0, &&\forall t\in\mathbb{T}_\sigma,\label{eq:pf:jump}
    \end{align}
    hold for all $x\in X$.
\end{enumerate}
Moreover, when (i) holds, an explicit Lyapunov functional witnessing (iii) is given by
\begin{equation}\label{eq:V:pf:explicit}
  V(t,x)\;:=\;\int_t^\infty\|U_\sigma(s,t)x\|^p  \ds.
\end{equation}
\end{theorem}

\begin{proof}
The structure parallels Theorem~\ref{thm:main}, with jumps providing no decay and stability ensured entirely by the flow. We prove $(i)\Rightarrow(ii)\Rightarrow(iii)\Rightarrow(ii)\Rightarrow(i)$.

\medskip\noindent\textbf{Proof that (i) implies (ii).}
For $s\ge0$ and $x\in X$, the GES bound (with $\rho=1$) gives $\|U_\sigma(t,s)x\|\le Me^{-\alpha(t-s)}\|x\|$. Therefore:
\begin{equation}
  \int_s^\infty\|U_\sigma(t,s)x\|^p  \dt
  \;\le\; M^p\|x\|^p\int_s^\infty e^{-\alpha p(t-s)}  \dt
  \;=\;\frac{M^p}{\alpha p}\|x\|^p.
\end{equation}
Setting $b^p:=M^p/(\alpha p)$ gives \eqref{eq:datko:pf} for all $s\ge0$.

\medskip\noindent\textbf{Proof that (ii) implies (iii).}
Define, for each $t\ge0$ and $x\in X$:
\begin{equation}\label{eq:V:pf}
  V(t,x)\;:=\;\int_t^\infty\|U_\sigma(s,t)x\|^p  \ds.
\end{equation}
The integral is bounded: $V(t,x)\le b^p\|x\|^p$ from \eqref{eq:datko:pf} at starting time $t$, giving \eqref{eq:pf:upper}.

\emph{Flow decrease \eqref{eq:pf:flow}.} For $t\notin\mathbb{T}_\sigma$ and small $h>0$ with $(t,t+h)\cap\mathbb{T}_\sigma=\varnothing$, the same computation as in Theorem~\ref{thm:main} gives $D^+V(t,x)=-\|x\|^p$, so \eqref{eq:pf:flow} holds with equality.

\emph{Jump neutrality \eqref{eq:pf:jump}.} At any jump time $t_k$, using $U_\sigma(\cdot,t_k)J=U_\sigma(\cdot,t_k^-)$:
\begin{equation}
  V(t_k,Jx)=\int_{t_k}^\infty\|U_\sigma(s,t_k)Jx\|^p  \ds=\int_{t_k}^\infty\|U_\sigma(s,t_k^-)x\|^p  \ds=V(t_k^-,x),
\end{equation}
where the last equality uses $V(t_k^-,x)=\int_{t_k}^\infty\|U_\sigma(s,t_k^-)x\|^p\ds$ (the lower limit of the integral does not change when passing from $t_k^-$ to $t_k$ since the integrand is unchanged). Hence $V(t_k,Jx)-V(t_k^-,x)=0\le0$, so \eqref{eq:pf:jump} holds with equality.

\emph{Lipschitz continuity.} By condition (ii) and Corollary~\ref{cor:datko:int} (whose dwell-time hypothesis is the standing assumption $T_\sigma>0$), the system satisfies $\|U_\sigma(t,s)\|\le Me^{-\alpha_0(t-s)}$ for $t\ge s\ge0$, with $M=2C_0$. For $\|x\|,\|y\|\le r$:
\begin{equation}
  |V(t,x)-V(t,y)|
  \;\le\;\int_t^\infty pM^pr^{p-1}e^{-\alpha_0 p(s-t)}\|x-y\|  \ds
  =\frac{M^pr^{p-1}}{\alpha_0}\|x-y\|,
\end{equation}
establishing Lipschitz continuity.

{The continuity properties of $V$ in (iii) follow as in the proof of Theorem~\ref{thm:main} (Continuity in $t$), the argument applying with the absent term removed.}

\medskip\noindent\textbf{Proof that (iii) implies (ii).}
Fix $t_0\ge0$ and $y\in X$. Let $x(t)=U_\sigma(t,t_0)y$. {The integration of the Dini inequality on each inter-jump interval is justified as in the proof of $(iii)\Rightarrow(ii)$ of Theorem~\ref{thm:main}, using the continuity of $V(\cdot,x)$, the continuity of $x(\cdot)$ on $[t_k,t_{k+1})$, and the Lipschitz property of $V$. Integrating \eqref{eq:pf:flow} in this way on each interval $[t_k,t_{k+1})$ with $t_k\ge t_0$ gives}
\begin{equation}
  V(t_{k+1}^-,x(t_{k+1}^-))-V(t_k,x(t_k))\;\le\;-\int_{t_k}^{t_{k+1}}\|U_\sigma(s,t_0)y\|^p  \ds.
\end{equation}
At each jump $t_{k+1}$, condition \eqref{eq:pf:jump} gives $V(t_{k+1},x(t_{k+1}))\le V(t_{k+1}^-,x(t_{k+1}^-))$, so jumps do not increase $V$. Telescoping over all inter-jump intervals up to time $t$:
\begin{equation}
  V(t,x(t))\;\le\; V(t_0,y)-\int_{t_0}^t\|U_\sigma(s,t_0)y\|^p  \ds.
\end{equation}
Since $V\ge0$ and $V(t_0,y)\le c\|y\|^p$, letting $t\to\infty$ yields \eqref{eq:datko:pf} at starting time $t_0$ with $b^p=c$.

\medskip\noindent\textbf{Proof that (ii) implies (i).}
Condition (ii) is the integral-only Datko condition \eqref{eq:datko:int} at all starting times, and $T_\sigma>0$ by hypothesis, so Corollary~\ref{cor:datko:int} applies directly and yields $\|U_\sigma(t,s)\|\le Me^{-\alpha_0(t-s)}$ for $t\ge s\ge0$. This is GES with $\rho=1$ and $\alpha=\alpha_0>0$.
\end{proof}

\subsection{Arbitrary impulse sequences ($\mathcal{S}_{0,\infty}$)}\label{subsec:arbitrary:banach}

We now address the family $\mathcal{S}_{0,\infty}$ in which no constraint whatsoever is placed on the impulse sequence: this is the impulsive-systems analogue of quadratic stability for uncertain or linear parameter-varying systems \cite{Barmish:85,Briat:book1} or of common-Lyapunov stability under arbitrary switching for switched systems \cite{Liberzon:03,Haidar:22}. Since no information on the timing of impulses is available, the Lyapunov functional cannot depend on the current time $t$, and the flow and jump conditions decouple. In contrast to the persistent-flowing and persistent-jumping versions of Theorems~\ref{thm:pf} and~\ref{thm:pj}, both mechanisms must here be strictly contractive.

\begin{proposition}[Necessary conditions]\label{prop:arbitrary:necessary}
  Assume \eqref{eq:syst} is strongly UGES over $\mathcal{S}_{0,\infty}$. Then:
  \begin{enumerate}[(i)]
    \item the $C_0$-semigroup $(S(t))_{t\ge0}$ generated by $A$ is exponentially stable, i.e., $\|S(t)\|\le M_0e^{-\alpha_0 t}$ for some $M_0\ge1$ and $\alpha_0>0$;
    \item the jump operator $J\in L(X)$ is geometrically stable, i.e., $\|J^k\|\le N_0\rho_0^k$ for some $N_0\ge1$ and $\rho_0\in(0,1)$.
  \end{enumerate}
\end{proposition}

\begin{proof}
For (i), any sequence whose first impulse occurs at $t_1\to+\infty$ exhibits arbitrarily long flow-only intervals; UGES on $\mathcal{S}_{0,\infty}$ therefore forces $\|S(t)x\|\le Me^{-\alpha t}\|x\|$ for some $M,\alpha>0$, which gives (i). For (ii), the sequence $t_k=k\epsilon$ with $\epsilon\downarrow0$ yields $U_\sigma(k\epsilon,0)=(JS(\epsilon))^k\to J^k$ in the strong sense, whence $\|J^k\|\le Me^{-\alpha k\epsilon}\rho^k\le M\rho^k$ in the limit, which gives (ii).
\end{proof}

Both properties are individually necessary. Neither is sufficient on its own, since UGES on $\mathcal{S}_{0,\infty}$ further requires a \emph{common} mechanism through which flow and jump dissipate the same Lyapunov functional.

\begin{theorem}[Arbitrary sequences, non-coercive N\&S]\label{thm:arbitrary:nc}
  Let Assumption~\ref{ass:growth} hold (which, over $\mathcal{S}_{0,\infty}$, requires $\omega_0\le0$ by Remark~\ref{rem:growth:G} and Proposition~\ref{prop:necessity:semigroup}). The following statements are equivalent:
  \begin{enumerate}[(i)]
    \item system \eqref{eq:syst} is strongly UGES over $\mathcal{S}_{0,\infty}$;
    \item there exist constants $c>0$, $p\ge1$, and a \emph{time-independent} functional $V:X\to\mathbb{R}_{\ge0}$, Lipschitz on bounded subsets of $X$, such that for all $x\in X$:
      \begin{subequations}\label{eq:arbitrary:nc}
        \begin{align}
          V(x)&\le c  \|x\|^p,\label{eq:arbitrary:nc:bd}\\
          \underline{D}^+V(x)&\le-\|x\|^p,\label{eq:arbitrary:nc:flow}\\
          V(Jx)-V(x)&\le-\|x\|^p.\label{eq:arbitrary:nc:jump}
        \end{align}
      \end{subequations}
  \end{enumerate}
Moreover, when (i) holds, an explicit time-independent Lyapunov functional witnessing (ii) is given by
\begin{equation}\label{eq:arbitrary:V}
  V(x):=\sup_{\sigma\in\mathcal{S}_{0,\infty}}\left[\int_0^\infty\|U_\sigma(t,0)x\|^p  \dt+\sum_{k\ge1}\|U_\sigma(t_k^-,0)x\|^p\right].
\end{equation}
\end{theorem}

\begin{proof}
We prove $(ii)\Rightarrow(i)$ and $(i)\Rightarrow(ii)$.

\medskip\noindent\textbf{Proof that $(ii)\Rightarrow(i)$.} Let $\sigma\in\mathcal{S}_{0,\infty}$, $s\ge0$, and $y\in X$ be given, and let $x(t)=U_\sigma(t,s)y$. Since $V$ is time-independent, the lower Dini condition \eqref{eq:arbitrary:nc:flow} along the flow integrates exactly as in the proof of Theorem~\ref{thm:main}$(iii)\Rightarrow(ii)$ on each inter-jump interval, yielding
  \begin{equation*}
    V(x(t_{k+1}^-))-V(x(t_k))\le -\int_{t_k}^{t_{k+1}}\|U_\sigma(s',s)y\|^p  \ds';
  \end{equation*}
  the jump condition \eqref{eq:arbitrary:nc:jump} gives $V(x(t_{k+1}))-V(x(t_{k+1}^-))\le -\|U_\sigma(t_{k+1}^-,s)y\|^p$. Telescoping over all jumps in $(s,t]$ and using $V\ge0$ with $V(y)\le c\|y\|^p$:
  \begin{equation*}
    \int_s^t\|U_\sigma(s',s)y\|^p  \ds'+\sum_{k\ge1:  t_k\in(s,t]}\|U_\sigma(t_k^-,s)y\|^p\le c\|y\|^p.
  \end{equation*}
  This is the Datko bound at starting time $s$, with constant uniform in $\sigma\in\mathcal{S}_{0,\infty}$ since $V$ does not depend on $\sigma$. Letting $t\to\infty$ and applying Lemma~\ref{lem:datko} uniformly over $\mathcal{S}_{0,\infty}$ (valid since Assumption~\ref{ass:growth} holds on $\mathcal{S}_{0,\infty}$) yields strong UGES (Lemma~\ref{lem:datko} returns $\alpha>0$, $\rho\in(0,1)$).

\medskip\noindent\textbf{Proof that $(i)\Rightarrow(ii)$.} Define $V$ by \eqref{eq:arbitrary:V}. We verify each property.

\emph{Finiteness and upper bound \eqref{eq:arbitrary:nc:bd}.} For any $\sigma\in\mathcal{S}_{0,\infty}$, UGES with constants $M,\alpha,\rho$ gives, by the estimates of $(i)\Rightarrow(ii)$ in Theorem~\ref{thm:main}, $$\int_0^\infty\|U_\sigma(t,0)x\|^p  \dt\le (M^p/(\alpha p))\|x\|^p\quad\text{and}\quad\sum_{k\ge1}\|U_\sigma(t_k^-,0)x\|^p\le (M^p/(1-\rho^p))\|x\|^p,$$ both uniform in $\sigma$. Hence $\Phi(\sigma,x)\le c\|x\|^p$ uniformly in $\sigma$ with $c:=M^p[1/(\alpha p)+1/(1-\rho^p)]$, where $\Phi(\sigma,x)$ denotes the bracketed quantity in \eqref{eq:arbitrary:V}, and taking the supremum over $\mathcal{S}_{0,\infty}$ gives $V(x)\le c\|x\|^p$.

\emph{Flow condition \eqref{eq:arbitrary:nc:flow}.} Fix $x\in X$ and $h>0$. For any $\sigma\in\mathcal{S}_{0,\infty}$ with $t_1^\sigma>h$ (so that no jump occurs on $[0,h]$), the orbit flows freely on $[0,h]$, hence $U_\sigma(t,0)x=S(t)x$ for $t\in[0,h]$ and $U_\sigma(t,0)x=U_\sigma(t,h)S(h)x$ for $t\ge h$. Splitting the integral at $h$ and reindexing the jump times of $\sigma$ by the backward shift $\sigma^{(h)}$, defined by $t_k^{\sigma^{(h)}}:=t_k^\sigma-h$ (admissible in $\mathcal{S}_{0,\infty}$ since $t_1^\sigma>h$), the propagator identity $U_{\sigma^{(h)}}(t,0)=U_\sigma(t+h,h)$ gives
\begin{equation*}
  \Phi(\sigma,x)=\int_0^h\|S(t)x\|^p  \dt+\Phi(\sigma^{(h)},S(h)x),
\end{equation*}
since the integral over $[h,\infty)$ becomes, after $t'=t-h$, $\int_0^\infty\|U_{\sigma^{(h)}}(t',0)S(h)x\|^p  \dt'$, and each pre-jump value satisfies $\|U_\sigma(t_k^{\sigma-},0)x\|=\|U_{\sigma^{(h)}}(t_k^{\sigma^{(h)}-},0)S(h)x\|$. Taking the supremum over all $\sigma\in\mathcal{S}_{0,\infty}$ with $t_1^\sigma>h$, and using that the backward shift $\sigma\mapsto\sigma^{(h)}$ maps this set onto all of $\mathcal{S}_{0,\infty}$,
\begin{equation*}
  V(x)\ge\sup_{\sigma:  t_1^\sigma>h}\Phi(\sigma,x)=\int_0^h\|S(t)x\|^p  \dt+\sup_{\sigma:  t_1^\sigma>h}\Phi(\sigma^{(h)},S(h)x)=\int_0^h\|S(t)x\|^p  \dt+V(S(h)x).
\end{equation*}
Hence $V(x)-V(S(h)x)\ge\int_0^h\|S(t)x\|^p  \dt$; dividing by $h$ and letting $h\downarrow0$, strong continuity of $S$ gives $\underline{D}^+V(x)\le-\|x\|^p$.

\emph{Jump condition \eqref{eq:arbitrary:nc:jump}.} Fix $x\in X$. {Since every sequence in $\mathcal{S}_{0,\infty}$ satisfies $t_1>t_0=0$, the prepended impulse is placed at a time $\varepsilon>0$, and $\varepsilon\downarrow0$ is taken at the end. For any $\sigma\in\mathcal{S}_{0,\infty}$ and $\varepsilon>0$, let $\widetilde\sigma_\varepsilon$ be defined by $t_1^{\widetilde\sigma_\varepsilon}:=\varepsilon$ and $t_{k+1}^{\widetilde\sigma_\varepsilon}:=t_k^\sigma+\varepsilon$ for $k\ge1$, so that $\widetilde\sigma_\varepsilon\in\mathcal{S}_{0,\infty}$. Then $U_{\widetilde\sigma_\varepsilon}(t,0)x=S(t)x$ for $t\in[0,\varepsilon)$ and $U_{\widetilde\sigma_\varepsilon}(t+\varepsilon,0)x=U_\sigma(t,0)JS(\varepsilon)x$ for $t\ge0$, the pre-jump value at the prepended instant is $S(\varepsilon)x$, and therefore
\begin{equation*}
\begin{aligned}
  \Phi(\widetilde\sigma_\varepsilon,x)&=\int_0^\varepsilon\|S(t)x\|^p\dt+\|S(\varepsilon)x\|^p+\Phi(\sigma,JS(\varepsilon)x)\\
  &\ge\|S(\varepsilon)x\|^p+\Phi(\sigma,JS(\varepsilon)x),
\end{aligned}
\end{equation*}
where the term $\|S(\varepsilon)x\|^p$ is the contribution of the prepended jump. Since $\widetilde\sigma_\varepsilon\in\mathcal{S}_{0,\infty}$ for every $\sigma\in\mathcal{S}_{0,\infty}$, taking the supremum over $\sigma$ yields
\begin{equation*}
  V(x)\ge\|S(\varepsilon)x\|^p+\sup_{\sigma\in\mathcal{S}_{0,\infty}}\Phi(\sigma,JS(\varepsilon)x)=\|S(\varepsilon)x\|^p+V(JS(\varepsilon)x).
\end{equation*}
Letting $\varepsilon\downarrow0$, the strong continuity of $S(\cdot)$ and the Lipschitz continuity of $V$ on bounded subsets of $X$, which is established below independently of the present step, give $V(x)\ge\|x\|^p+V(Jx)$,}
which is \eqref{eq:arbitrary:nc:jump}.

\emph{Lipschitz continuity.} As in Theorem~\ref{thm:main}, using $|  \|U_\sigma(\cdot,0)x\|^p-\|U_\sigma(\cdot,0)y\|^p  |\le pM^pr^{p-1}\|x-y\|$ on bounded $r$-balls and integrating/summing uniformly in $\sigma$.
\end{proof}

\begin{remark}\label{rem:arbitrary:decoupling}
The flow and jump conditions \eqref{eq:arbitrary:nc:flow}, \eqref{eq:arbitrary:nc:jump} are \emph{decoupled}: each must hold separately for the same $V$, with the \emph{same} strict decrement on the right-hand side. There is no trade-off between flow decay and jump expansion as in the persistent-jumping and persistent-flowing theorems (Theorems~\ref{thm:pj}, \ref{thm:pf}), where one mechanism may be neutral provided the other is strict. Indeed, $\mathcal{S}_{0,\infty}$ contains both sequences with $t_1\to\infty$ (only flow active) and sequences with $t_k=k\epsilon\to0$ (only jumps active), so each mechanism must individually drive $V$ down.
\end{remark}

\subsection{Constant dwell-time ($\mathcal{S}_{\mathrm{cst}}(T)$)}\label{subsec:nc:cst}

For the constant dwell-time family $\mathcal{S}_{\mathrm{cst}}(T)$, all inter-jump intervals equal $T>0$. The family contains a single signal $\sigma_T${, whose impulse instants are given by $t_k=kT$, $k\ge1$}, and the timer-dependent Lyapunov functional $\bar V(\tau,x)$ lives on $[0,T]\times X$.

\begin{theorem}[Constant dwell-time, non-coercive N\&S]\label{thm:cst:nc}
Let $T>0$ and let Assumption~\ref{ass:growth} hold. The following are equivalent:
\begin{enumerate}[(i)]
  \item the system \eqref{eq:syst} is strongly UGES over $\mathcal{S}_{\mathrm{cst}}(T)$;
  \item (\emph{Datko}) there exist $p\ge1$ and $b>0$ such that, along the unique trajectory $U_{\sigma_T}$ and for all $x\in X$ and $s\ge0$,
    \begin{equation}\label{eq:cst:datko}
      \int_s^\infty\|U_{\sigma_T}(t,s)x\|^p  \dt+\sum_{k\ge1:  t_k>s}\|U_{\sigma_T}(t_k^-,s)x\|^p\le b^p\|x\|^p;
    \end{equation}
  \item (\emph{continuous-time hybrid}) there exist $p\ge1$, $c>0$, and a functional $\bar V:[0,T]\times X\to\mathbb{R}_{\ge0}$, Lipschitz on bounded subsets of $X$ uniformly in $\tau$, {with $\tau\mapsto\bar V(\tau,x)$ continuous on $[0,T]$ for every $x\in X$,} such that
    \begin{subequations}\label{eq:cst:nc:hybrid}
    \begin{align}
      \bar V(\tau,x)&\le c\|x\|^p,\label{eq:cst:nc:upper}\\
      \underline{D}^+\bar V(\tau,x)&\le -\|x\|^p,\;&&\forall\tau\in[0,T),  x\in X,\label{eq:cst:nc:flow}\\
      \bar V(0,Jx)-\bar V(T,x)&\le -\|x\|^p,\;&&\forall x\in X;\label{eq:cst:nc:jump}
    \end{align}
    \end{subequations}
  \item (\emph{discrete-time}) the monodromy operator $S(T)J$\footnote{The monodromy is the one-period propagator of the periodic sequence $\sigma_T$; its form depends on the base instant, being $S(T)J$ when the period is measured between successive \emph{pre-jump} instants and $JS(T)$ when measured between \emph{post-jump} instants. As $S(T)J$ and $JS(T)$ share the same non-zero spectrum, the geometric-stability conclusion is the same for either choice.} is geometrically stable, i.e., there exist $p\ge1$, $\hat c>0$, and a functional $V_d:X\to\mathbb{R}_{\ge0}$ Lipschitz on bounded subsets of $X$ such that
    \begin{equation}\label{eq:cst:dt}
      V_d(x)\le \hat c\|x\|^p,\qquad V_d(S(T)J  x)-V_d(x)\le -\|x\|^p,\qquad\forall x\in X.
    \end{equation}
\end{enumerate}
\end{theorem}

\begin{proof}
We prove $(i)\Rightarrow(ii)\Rightarrow(iii)\Rightarrow(ii)\Rightarrow(i)$ and $(iii)\Leftrightarrow(iv)$.

\medskip\noindent\textbf{Proof that $(i)\Rightarrow(ii)$.} The fixed-sequence Datko bound \eqref{eq:cst:datko} for $\sigma=\sigma_T$ follows without change from the proof of $(i)\Rightarrow(ii)$ in Theorem~\ref{thm:main}, applied with $\sigma=\sigma_T$ and using the periodicity $\kappa_{\sigma_T}(t_{k}^-,s)=k-K_s-1$ for $t_k>s$ with $K_s=\kappa_{\sigma_T}(s,0)$. The resulting constant is $b^p=M^p(1/(\alpha p)+1/(1-\rho^p))$.

\medskip\noindent\textbf{Proof that $(ii)\Rightarrow(iii)$.} Define $\bar V$ on $[0,T]\times X$ by
\begin{equation}\label{eq:cst:V:explicit}
  \bar V(\tau,x):=\int_0^\infty\|U_{\sigma_T}(\tau+s,\tau)x\|^p  \ds+\sum_{k\ge1:  t_k^{\sigma_T}>\tau}\|U_{\sigma_T}(t_k^-,\tau)x\|^p.
\end{equation}
The change of variables $u=\tau+s$
{gives} $$\bar V(\tau,x)=\int_\tau^\infty\|U_{\sigma_T}(u,\tau)x\|^p  \du+\sum_{k\ge1:  t_k>\tau}\|U_{\sigma_T}(t_k^-,\tau)x\|^p,$$ which is the left-hand side of \eqref{eq:cst:datko} at starting time $s=\tau$. By (ii), $\bar V(\tau,x)\le b^p\|x\|^p$ for all $\tau\in[0,T)$ and all $x\in X$, giving \eqref{eq:cst:nc:upper} with $c=b^p$. The endpoint value is defined as the pre-jump left-limit $\bar V(T,x):=\lim_{\tau\uparrow T}\bar V(\tau,x)$, so that the imminent jump at timer $T$ is included; as $\tau\uparrow T$ the first flow arc $[\tau,T]$ vanishes while the pre-jump term $\|U_{\sigma_T}(T^-,\tau)x\|^p\to\|x\|^p$ is retained, whence
\begin{equation}\label{eq:cst:V:endpoint}
  \bar V(T,x)=\|x\|^p+\int_0^\infty\|U_{\sigma_T}(s,0)Jx\|^p\ds+\sum_{k\ge1}\|U_{\sigma_T}(t_k^-,0)Jx\|^p=\|x\|^p+\bar V(0,Jx),
\end{equation}
using $U_{\sigma_T}(T+s,T^-)=U_{\sigma_T}(T+s,T)J$ and periodicity. In particular $\bar V(T,x)\le(1+b^p\|J\|^p)\|x\|^p$, so \eqref{eq:cst:nc:upper} extends to $\tau=T$ after enlarging $c$.

\emph{Flow condition \eqref{eq:cst:nc:flow}.} For $\tau\in[0,T)$ and $h>0$ small enough that $(\tau,\tau+h)\cap\mathbb{T}_{\sigma_T}=\varnothing$, the same computation as in $(ii)\Rightarrow(iii)$ of Theorem~\ref{thm:main} (with starting time $\tau$ instead of arbitrary $t$) yields:
\begin{equation*}
  \bar V(\tau+h,S(h)x)-\bar V(\tau,x)=-\int_\tau^{\tau+h}\|U_{\sigma_T}(u,\tau)x\|^p  \du=-\int_0^h\|S(r)x\|^p  \dr,
\end{equation*}
since no jump occurs in $(\tau,\tau+h)$. Dividing by $h$ and using strong continuity of $S$, $\underline{D}^+\bar V(\tau,x)=-\|x\|^p$.

\emph{Jump condition \eqref{eq:cst:nc:jump}.} This is immediate from the endpoint identity \eqref{eq:cst:V:endpoint}: subtracting gives $\bar V(0,Jx)-\bar V(T,x)=-\|x\|^p$, which is \eqref{eq:cst:nc:jump} with equality.

\emph{Lipschitz continuity.} The argument from Theorem~\ref{thm:main} carries over without change with $\sigma=\sigma_T$, yielding $|\bar V(\tau,x)-\bar V(\tau,y)|\le L_r\|x-y\|$ on $\{\|x\|,\|y\|\le r\}$, uniform in $\tau\in[0,T]$.

{\emph{Continuity in $\tau$.} On $[0,T)=[t_0,t_1)$, $\bar V$ coincides with the functional \eqref{eq:V:def} associated with $\sigma=\sigma_T$. Hence, by the uniform continuity argument of Theorem~\ref{thm:main}, $\tau\mapsto\bar V(\tau,x)$ is continuous on $[0,T)$, and it is continuous at $\tau=T$ since $\bar V(T,x)$ is defined as its left limit.}

\medskip\noindent\textbf{Proof that $(iii)\Rightarrow(ii)$.} Fix $s\ge0$ and $y\in X$. Let $x(t):=U_{\sigma_T}(t,s)y$ and write $\bar V(t):=\bar V(\tau(t),x(t))$ where $\tau(t):=t-t_{k(t)}$ is the timer (with $t_{k(t)}$ the most recent jump time at or before $t$). On each inter-jump interval $(t_k,t_{k+1})=(t_k,t_k+T)$,
{the map $t\mapsto\bar V(t-t_k,x(t))$ is continuous on $[t_k,t_{k+1})$ and converges to $\bar V(T,x(t_{k+1}^-))$ as $t\uparrow t_{k+1}$, by continuity of $\bar V(\cdot,x)$ on $[0,T]$, continuity of $x(\cdot)$ on $[t_k,t_{k+1})$, and the Lipschitz property of $\bar V$. As in the proof of $(iii)\Rightarrow(ii)$ of Theorem~\ref{thm:main},} \eqref{eq:cst:nc:flow} integrates to:
\begin{equation*}
  \bar V(t_{k+1}^-,x(t_{k+1}^-))-\bar V(t_k,x(t_k))\le -\int_{t_k}^{t_{k+1}}\|U_{\sigma_T}(s',s)y\|^p  \ds',
\end{equation*}
where on the right we used $\tau(t)=t-t_k$ so that $\bar V$ on this interval is parametrized by $\tau\in[0,T]$. At the jump time $t_{k+1}$, \eqref{eq:cst:nc:jump} applied with $x=x(t_{k+1}^-)$ gives:
\begin{equation*}
  \bar V(0,x(t_{k+1}))=\bar V(0,Jx(t_{k+1}^-))\le \bar V(T,x(t_{k+1}^-))-\|x(t_{k+1}^-)\|^p.
\end{equation*}
Telescoping over $k=0,1,\ldots,n$ (with $t_0$ the first jump after $s$ if any, or $\tau(s)=s-t_{k(s)}$ as starting value) and noting $\bar V\ge0$ and $\bar V(\tau(s),y)\le c\|y\|^p$:
\begin{equation*}
  \int_s^t\|U_{\sigma_T}(s',s)y\|^p  \ds'+\sum_{k\ge1:  t_k\in(s,t]}\|U_{\sigma_T}(t_k^-,s)y\|^p\le c\|y\|^p\qquad\forall t\ge s,
\end{equation*}
which is \eqref{eq:cst:datko} with $b^p=c$.

\medskip\noindent\textbf{Proof that $(ii)\Rightarrow(i)$.} Condition (ii) at all starting times $s\ge0$ provides the hypothesis of Lemma~\ref{lem:datko} applied to the single sequence $\sigma_T$; the conclusion is {strong} UGES over $\mathcal{S}_{\mathrm{cst}}(T)=\{\sigma_T\}$.

\medskip\noindent\textbf{Proof that $(iii)\Rightarrow(iv)$.} Define $V_d(x):=\bar V(T,x)$, which inherits the upper bound $V_d(x)\le c\|x\|^p$ and Lipschitz continuity from (iii). Chaining the flow inequality \eqref{eq:cst:nc:flow} on $[0,T]$ along the unforced trajectory $r\mapsto S(r)Jx$ (starting from $Jx$ at $\tau=0$) and \eqref{eq:cst:nc:jump} at $\tau=T$ (with pre-jump state $x$):
\begin{align*}
  V_d(S(T)Jx)&=\bar V(T,S(T)Jx)\\
  &\le \bar V(0,Jx)-\int_0^T\|S(r)Jx\|^p  \dr\qquad\text{(flow on the period starting from $Jx$)}\\
  &\le \bar V(T,x)-\|x\|^p-\int_0^T\|S(r)Jx\|^p  \dr\qquad\text{(jump at $\tau=T$)}\\
  &\le V_d(x)-\|x\|^p,
\end{align*}
giving \eqref{eq:cst:dt} with $\hat c=c$.

\medskip\noindent\textbf{Proof that $(iv)\Rightarrow(iii)$.} Let $M_S:=M_0e^{|\omega_0|T}$ from Assumption~\ref{ass:growth} (so $\|S(r)\|\le M_S$ for $r\in[0,T]$) and set $\alpha:=1+M_S^p  T  \|J\|^p$. Define on $[0,T]\times X$:
\begin{equation}\label{eq:cst:Vbar:from:Vd}
  \bar V(\tau,x)\;:=\;\alpha  V_d(S(T-\tau)x)\;+\;\int_0^{T-\tau}\|S(r)x\|^p  \dr.
\end{equation}
Non-negativity of $\bar V$ is immediate from $V_d\ge0$ and the integrand being non-negative.

\emph{Upper bound.} Using $\|S(r)\|\le M_S$ on $r\in[0,T]$:
\begin{equation*}
  \bar V(\tau,x)\le \alpha  \hat c  M_S^p\|x\|^p+M_S^p  T\|x\|^p=[\alpha  \hat c  M_S^p+M_S^p  T]\|x\|^p,
\end{equation*}
which is \eqref{eq:cst:nc:upper} with $c:=\alpha  \hat c  M_S^p+M_S^p  T$.

\emph{Flow condition.} For $h>0$ with $\tau+h\le T$, using $S(T-\tau-h)S(h)=S(T-\tau)$ and the change of variable $r'=r+h$ in the integral:
\begin{align*}
  \bar V(\tau+h,S(h)x)&=\alpha  V_d(S(T-\tau)x)+\int_h^{T-\tau}\|S(r')x\|^p  \dr'=\bar V(\tau,x)-\int_0^h\|S(r')x\|^p  \dr'.
\end{align*}
Dividing by $h$ and letting $h\to0^+$, $h^{-1}\!\int_0^h\|S(r')x\|^p\,\dr'\to\|x\|^p$ by strong continuity of $S(\cdot)$ at $r'=0$, hence $\underline{D}^+\bar V(\tau,x)=-\|x\|^p$, which is \eqref{eq:cst:nc:flow}.

\emph{Jump condition.} At $\tau=T$, $\bar V(T,x)=\alpha  V_d(x)$. At $\tau=0$ with state $Jx$, $$\bar V(0,Jx)=\alpha  V_d(S(T)Jx)+\int_0^T\|S(r)Jx\|^p  \dr.$$ The difference is
\begin{equation*}
  \bar V(0,Jx)-\bar V(T,x)=\alpha\bigl[V_d(S(T)Jx)-V_d(x)\bigr]+\int_0^T\|S(r)Jx\|^p  \dr.
\end{equation*}
Applying \eqref{eq:cst:dt} (with the $S(T)J$ form) and the bound $\int_0^T\|S(r)Jx\|^p  \dr\le M_S^p  T  \|J\|^p\|x\|^p$:
\begin{equation*}
  \bar V(0,Jx)-\bar V(T,x)\le -\alpha\|x\|^p+M_S^p  T  \|J\|^p\|x\|^p=-(\alpha-M_S^p  T  \|J\|^p)\|x\|^p=-\|x\|^p,
\end{equation*}
which is \eqref{eq:cst:nc:jump}, exact at the choice $\alpha=1+M_S^p  T  \|J\|^p$. Lipschitz continuity of $\bar V$ on bounded sets is inherited from $V_d$, $S(\cdot)$, and the bounded integrand. {Finally, for every $x\in X$, the map $\tau\mapsto\bar V(\tau,x)$ is continuous on $[0,T]$, since $\tau\mapsto S(T-\tau)x$ is continuous by strong continuity of $S(\cdot)$, $V_d$ is Lipschitz on bounded subsets of $X$, and the integral in \eqref{eq:cst:Vbar:from:Vd} is continuous with respect to its upper limit.}
\end{proof}

\subsection{Minimum dwell-time ($\mathcal{S}_{\min}(T_{\min})$)}\label{subsec:nc:min}

For $\mathcal{S}_{\min}(T_{\min})$, inter-jump intervals satisfy $t_{k+1}-t_k\ge T_{\min}$ with no upper bound. The timer $\tau=t-t_k$ ranges over $[0,\infty)$, but the Lyapunov functional $\bar V(\tau,x)$ saturates beyond $\tau=T_{\min}$: it is timer-dependent on the mandatory phase $[0,T_{\min}]$ and timer-independent on the free phase $[T_{\min},\infty)$. Since $\mathcal{S}_{\min}(T_{\min})$ admits arbitrarily long jump-free intervals, exponential stability of the underlying semigroup is a necessary condition, as shown below.

\begin{proposition}[Necessity of semigroup stability, minimum dwell-time]\label{prop:min:necessary}
If the system \eqref{eq:syst} is strongly UGES over $\mathcal{S}_{\min}(T_{\min})$ with constants $M\ge1$, $\alpha>0$, $\rho\in(0,1)$, then the $C_0$-semigroup $S(t)$ is exponentially stable, $\|S(\tau)\|\le Me^{-\alpha\tau}$ for all $\tau\ge0$. In particular, each of the equivalent conditions of Theorem~\ref{thm:min:nc} forces $S$ to be exponentially stable.
\end{proposition}
\begin{proof}
The family $\mathcal{S}_{\min}(T_{\min})$ has $\bar T(\mathcal{S}_{\min}(T_{\min}))=\infty$, since inter-jump intervals are unbounded above. Proposition~\ref{prop:necessity:semigroup} applied to this family yields $\|S(\tau)\|\le Me^{-\alpha\tau}$ for all $\tau\ge0$.
\end{proof}

\begin{theorem}[Minimum dwell-time, non-coercive N\&S]\label{thm:min:nc}
Let $T_{\min}>0$ and let Assumption~\ref{ass:growth} hold. The following are equivalent:
\begin{enumerate}[(i)]
  \item the system \eqref{eq:syst} is strongly UGES over $\mathcal{S}_{\min}(T_{\min})$;
  \item (\emph{Datko}) there exist $p\ge1$ and $b>0$ such that \eqref{eq:cst:datko} holds for all $\sigma\in\mathcal{S}_{\min}(T_{\min})$ and all $s\ge0$;
  \item (\emph{continuous-time hybrid}) there exist $p\ge1$, $c>0$, and a functional $\bar V:[0,T_{\min}]\times X\to\mathbb{R}_{\ge0}$, Lipschitz on bounded subsets of $X$ uniformly in $\tau$, {with $\tau\mapsto\bar V(\tau,x)$ continuous on $[0,T_{\min}]$ for every $x\in X$,} such that
    \begin{subequations}\label{eq:min:nc:hybrid}
    \begin{align}
      \bar V(\tau,x)&\le c\|x\|^p,\;&&\forall\tau\in[0,T_{\min}],  x\in X,\label{eq:min:nc:upper}\\
      \underline{D}^+\bar V(\tau,x)&\le -\|x\|^p,\;&&\forall\tau\in[0,T_{\min}],  x\in X\quad\text{(mandatory phase)},\label{eq:min:nc:flow:mandatory}\\
      \underline{D}^+\bar V(T_{\min},x)&\le -\|x\|^p,\;&&\forall x\in X\quad\text{(free phase, at the boundary)},\label{eq:min:nc:flow:free}\\
      \bar V(0,Jx)-\bar V(T_{\min},x)&\le -\|x\|^p,\;&&\forall x\in X;\label{eq:min:nc:jump}
    \end{align}
    \end{subequations}
    (these conditions imply, in particular, that the $C_0$-semigroup $S(t)$ generated by $A$ is exponentially stable).
  \item (\emph{discrete-time}) the $C_0$-semigroup $S(t)$ is exponentially stable, and the family $\{S(\theta)  J:\theta\ge T_{\min}\}$ is uniformly geometrically stable, i.e., there exist $p\ge1$, $\hat c>0$, and a functional $V_d:X\to\mathbb{R}_{\ge0}$ Lipschitz on bounded subsets of $X$ such that
    \begin{equation}\label{eq:min:dt}
      V_d(x)\le \hat c\|x\|^p,\qquad V_d(S(\theta)  J  x)-V_d(x)\le -\|x\|^p,\qquad\forall x\in X,\;\theta\ge T_{\min}.
    \end{equation}
\end{enumerate}
Moreover, when (i) holds, an explicit timer-dependent Lyapunov functional witnessing (iii) is given by
\begin{equation}\label{eq:min:V:explicit}
  \bar V(\tau,x):=\sup_{\sigma\in\mathcal{S}^{(\tau)}}\left[\int_0^\infty\|U_\sigma(s,0)x\|^p  \ds+\sum_{k\ge1}\|U_\sigma(t_k^-,0)x\|^p\right],
\end{equation}
where, for $\tau\in[0,T_{\min}]$,
\begin{equation}\label{eq:min:family:tau}
  \mathcal{S}^{(\tau)}:=\big\{\sigma\in\mathcal{S}_{0,\infty}:\ t_1^\sigma\ge \max(T_{\min}-\tau,0)\ \text{ and }\ t_{k+1}^\sigma-t_k^\sigma\ge T_{\min}\ \text{ for all }k\ge1\big\},
\end{equation}
and $\mathcal{S}^{(\tau)}:=\mathcal{S}^{(T_{\min})}$ for $\tau\ge T_{\min}$. {By convention, we also include in $\mathcal{S}^{(T_{\min})}$ the sequences with $t_1^\sigma=0$ and $t_{k+1}^\sigma-t_k^\sigma\ge T_{\min}$ for all $k\ge1$, for which the impulse at $t_1^\sigma=0$ is applied to the initial state, that is, $U_\sigma(s,0):=U_{\sigma_+}(s,0)J$ for all $s\ge0$ with $\sigma_+:=\{t_{k+1}^\sigma\}_{k\ge1}$, and $U_\sigma(t_1^{\sigma-},0):=I$, so that the term $k=1$ of the sum in \eqref{eq:min:V:explicit} is equal to $\|x\|^p$.} {For such sequences, the jump counter also counts the impulse at $t_1^\sigma=0$, that is, $\kappa_\sigma(t,0):=1+\kappa_{\sigma_+}(t,0)$ for all $t\ge0$.} The set $\mathcal{S}^{(\tau)}$ collects the admissible \emph{future} signals when the elapsed time since the last jump equals $\tau$: the next jump must respect the remaining wait $\max(T_{\min}-\tau,0)$, while every later inter-jump interval obeys the minimum dwell-time. Hence, $\mathcal{S}^{(0)}=\mathcal{S}_{\min}(T_{\min})$, the sets $\mathcal{S}^{(\tau)}$ increase with $\tau$, and they saturate at $\tau=T_{\min}$, where $\mathcal{S}^{(T_{\min})}=\{\sigma\in\mathcal{S}_{0,\infty}:t_{k+1}^\sigma-t_k^\sigma\ge T_{\min}\ \forall k\ge1\}$ imposes no lower bound on $t_1^\sigma$.
\end{theorem}

\begin{proof}
We prove $(i)\Rightarrow(ii)\Rightarrow(iii)\Rightarrow(ii)\Rightarrow(i)$ and $(iii)\Leftrightarrow(iv)$.

\medskip\noindent\textbf{Proof that $(i)\Rightarrow(ii)$.} For each $\sigma\in\mathcal{S}_{\min}(T_{\min})$, the bound \eqref{eq:cst:datko} follows uniformly from the proof of $(i)\Rightarrow(ii)$ in Theorem~\ref{thm:main}, with constants $M,\alpha,\rho$ independent of $\sigma$ by UGES.

\medskip\noindent\textbf{Proof that $(ii)\Rightarrow(iii)$.} Define $\bar V$ by \eqref{eq:min:V:explicit}. Every $\sigma\in\mathcal{S}^{(\tau)}$ is the future, re-based to the origin, of some $\sigma'\in\mathcal{S}_{\min}(T_{\min})$ observed at a time whose elapsed dwell equals $\tau$: let $\sigma'$ have first inter-jump interval $t_1^\sigma+\tau\ge T_{\min}$ (using $t_1^\sigma\ge T_{\min}-\tau$) and the same subsequent jumps as $\sigma$, so that $\sigma'\in\mathcal{S}_{\min}(T_{\min})$ and $U_{\sigma'}(\tau+s',\tau)x=U_\sigma(s',0)x$ for all $s'\ge0$. When $t_1^\sigma>0$ the jump of $\sigma'$ at $t_1^\sigma+\tau>\tau$ is counted by the strict Datko sum, so, since \eqref{eq:cst:datko} holds uniformly over all starting times $s\ge0$ and all signals in $\mathcal{S}_{\min}(T_{\min})$, applying it to $\sigma'$ at $s=\tau$ gives $\int_0^\infty\|U_\sigma(s',0)x\|^p\ds'+\sum_{k\ge1}\|U_\sigma(t_k^{\sigma-},0)x\|^p\le b^p\|x\|^p$; hence $\bar V(\tau,x)\le b^p\|x\|^p$ for $\tau\in[0,T_{\min})$. At the saturation timer $\tau=T_{\min}$, $\mathcal{S}^{(T_{\min})}$ additionally contains
{the sequences with $t_1^\sigma=0$}, whose bracket equals $\|x\|^p+\int_0^\infty\|U_\sigma(s',0)Jx\|^p\ds'+\sum_{k\ge1}\|U_\sigma(t_k^{\sigma-},0)Jx\|^p\le(1+b^p\|J\|^p)\|x\|^p$, the leading $\|x\|^p$ being the pre-jump term omitted by the strict Datko sum. Hence $\bar V(T_{\min},x)\le(1+b^p\|J\|^p)\|x\|^p$, and \eqref{eq:min:nc:upper} holds with $c:=\max(b^p,\,1+b^p\|J\|^p)$.

\emph{Mandatory-phase flow \eqref{eq:min:nc:flow:mandatory}.} Fix $\tau\in[0,T_{\min})$, $x\in X$, $h>0$ with $\tau+h\le T_{\min}$. For any $\sigma\in\mathcal{S}^{(\tau+h)}$ (i.e., $t_1^\sigma\ge T_{\min}-\tau-h$), shift forward by $h$ to get $\sigma^{(h)}$ with $t_k^{\sigma^{(h)}}=t_k^\sigma+h$; then $t_1^{\sigma^{(h)}}\ge T_{\min}-\tau$, so $\sigma^{(h)}\in\mathcal{S}^{(\tau)}$. The propagator relation $U_{\sigma^{(h)}}(s+h,0)=U_\sigma(s,0)S(h)$ gives, after change of variable $s'=s+h$:
\begin{align*}
  \bar V(\tau+h,S(h)x)&=\sup_{\sigma\in\mathcal{S}^{(\tau+h)}}\Big[\int_h^\infty\|U_{\sigma^{(h)}}(s',0)x\|^p  \ds'+\sum_{k\ge1}\|U_{\sigma^{(h)}}(t_k^-,0)x\|^p\Big]\\
  &\le \sup_{\sigma^{(h)}\in\mathcal{S}^{(\tau)},  t_1\ge h}\Big[\int_0^\infty\|U_{\sigma^{(h)}}(s',0)x\|^p  \ds'+\sum_{k\ge1}\|U_{\sigma^{(h)}}(t_k^-,0)x\|^p\Big]\\
  &\quad-\int_0^h\|S(s')x\|^p  \ds',
\end{align*}
the last subtraction being valid because every $\sigma^{(h)}$ with $t_1\ge h$ has no jump on $[0,h]$ so $U_{\sigma^{(h)}}(s',0)=S(s')$ there, contributing the same $\int_0^h\|S(s')x\|^p  \ds'$ uniformly. The sup over the smaller subfamily is $\le \bar V(\tau,x)$, hence:
\begin{equation*}
  \bar V(\tau+h,S(h)x)-\bar V(\tau,x)\le -\int_0^h\|S(s')x\|^p  \ds'.
\end{equation*}
Dividing by $h$ and using strong continuity, $\underline{D}^+\bar V(\tau,x)\le -\|x\|^p$.

\emph{Free-phase flow \eqref{eq:min:nc:flow:free}.} At saturation $\tau=T_{\min}$, the set $\mathcal{S}^{(T_{\min})}=\{\sigma\in\mathcal{S}_{0,\infty}:t_{k+1}^\sigma-t_k^\sigma\ge T_{\min}\ \forall k\ge1\}$ imposes no lower bound on $t_1^\sigma$. For any $\sigma\in\mathcal{S}^{(T_{\min})}$, the prepended signal $\widetilde\sigma$ with $t_k^{\widetilde\sigma}=t_k^\sigma+h$ has $t_1^{\widetilde\sigma}=t_1^\sigma+h>0$ and the same subsequent inter-jump intervals, so $\widetilde\sigma\in\mathcal{S}^{(T_{\min})}$. The propagator identity $U_{\widetilde\sigma}(s+h,0)x=U_\sigma(s,0)S(h)x$ and the same calculation yield:
\begin{equation*}
  \bar V(T_{\min},S(h)x)-\bar V(T_{\min},x)\le -\int_0^h\|S(s')x\|^p  \ds',
\end{equation*}
which gives $\underline{D}^+\bar V(T_{\min},x)\le -\|x\|^p$.

\emph{Jump condition \eqref{eq:min:nc:jump}.} Fix $x\in X$. Here $\mathcal{S}^{(0)}=\mathcal{S}_{\min}(T_{\min})$. For any $\sigma\in\mathcal{S}_{\min}(T_{\min})$, define $\sigma^*$ by prepending a single jump at $t_1^{\sigma^*}={0}$ followed by $\sigma$, with subsequent jumps at $\{t_k^\sigma\}_{k\ge1}$. The inter-jump intervals of $\sigma^*$ following the prepended jump have length $t_1^\sigma\ge T_{\min}$ and then those of $\sigma$ ($\ge T_{\min}$), so $\sigma^*\in\mathcal{S}^{(T_{\min})}$
{by the convention following \eqref{eq:min:family:tau}}. Then $U_{\sigma^*}(s,0)x=U_\sigma(s,0)Jx$ for $s\ge0$ and:
\begin{align*}
  \int_0^\infty\|U_{\sigma^*}(s,0)x\|^p  \ds&=\int_0^\infty\|U_\sigma(s,0)Jx\|^p  \ds,\\
  \sum_{k\ge1}\|U_{\sigma^*}(t_k^{\sigma^*-},0)x\|^p&=\|x\|^p+\sum_{k\ge1}\|U_\sigma(t_k^{\sigma-},0)Jx\|^p,
\end{align*}
the $\|x\|^p$ term coming from the prepended jump at ${0}$. Since $\sigma^*\in\mathcal{S}^{(T_{\min})}$:
\begin{equation*}
  \|x\|^p+\int_0^\infty\|U_\sigma(s,0)Jx\|^p  \ds+\sum_{k\ge1}\|U_\sigma(t_k^{\sigma-},0)Jx\|^p\le \bar V(T_{\min},x).
\end{equation*}
Taking sup over $\sigma\in\mathcal{S}^{(0)}$ on the left: $\|x\|^p+\bar V(0,Jx)\le \bar V(T_{\min},x)$, i.e., \eqref{eq:min:nc:jump}.

\emph{Lipschitz continuity.} The argument of Theorem~\ref{thm:main} applies uniformly over $\sigma\in\mathcal{S}^{(\tau)}$, each element being the future of an $\mathcal{S}_{\min}(T_{\min})$ signal, hence covered by the uniform-in-$s$ UGES bound, giving $|\bar V(\tau,x)-\bar V(\tau,y)|\le L_r\|x-y\|$ uniform in $\tau$.

{\emph{Continuity in $\tau$.} Let $x\in X$, $0\le\tau<\tau'\le T_{\min}$ with $h:=\tau'-\tau\le1$, and let $r>0$ be such that $\|S(s)x\|<r$ for all $s\in[0,1]$. Since $\mathcal{S}^{(\tau)}\subseteq\mathcal{S}^{(\tau')}$, we have $\bar V(\tau,x)\le\bar V(\tau',x)$. Moreover, the inequality obtained above in the proof of \eqref{eq:min:nc:flow:mandatory} gives $\bar V(\tau',S(h)x)\le\bar V(\tau,x)-\int_0^h\|S(s')x\|^p\ds'$. Hence,
\begin{equation*}
  0\le\bar V(\tau',x)-\bar V(\tau,x)\le\bar V(\tau',x)-\bar V(\tau',S(h)x)\le L_r\|x-S(h)x\|,
\end{equation*}
where the right-hand side does not depend on $\tau$ and tends to $0$ as $h\downarrow0$ by strong continuity of $S(\cdot)$. Therefore, $\tau\mapsto\bar V(\tau,x)$ is continuous on $[0,T_{\min}]$.}

\medskip\noindent\textbf{Proof that $(iii)\Rightarrow(ii)$.} {Extend $\bar V$ to $[0,\infty)\times X$ by setting $\bar V(\tau,x):=\bar V(T_{\min},x)$ for all $\tau>T_{\min}$ and $x\in X$. The extended functional satisfies \eqref{eq:min:nc:upper} for all $\tau\ge0$, and, since $\bar V(\tau+h,S(h)x)=\bar V(T_{\min},S(h)x)$ for all $\tau\ge T_{\min}$ and $h>0$, its lower right Dini derivative at any $\tau\ge T_{\min}$ is equal to $\underline{D}^+\bar V(T_{\min},x)$.} Fix $\sigma\in\mathcal{S}_{\min}(T_{\min})$, $s\ge0$, $y\in X$. Let $x(t):=U_\sigma(t,s)y$. On each inter-jump interval $[t_k,t_{k+1})$ with $\theta_k:=t_{k+1}-t_k\ge T_{\min}$, decompose into mandatory phase $[t_k,t_k+T_{\min}]$ and free phase $[t_k+T_{\min},t_{k+1})$. {The integrations below are justified as in the proof of $(iii)\Rightarrow(ii)$ of Theorem~\ref{thm:main}, using the continuity of $\bar V(\cdot,x)$ on $[0,T_{\min}]$, the continuity of $x(\cdot)$ on $[t_k,t_{k+1})$, and the Lipschitz property of $\bar V$.} Integrating \eqref{eq:min:nc:flow:mandatory} on $[t_k,t_k+T_{\min}]$ with timer $\tau\in[0,T_{\min}]$:
\begin{equation*}
  \bar V(T_{\min},x(t_k+T_{\min}))-\bar V(0,x(t_k))\le -\int_{t_k}^{t_k+T_{\min}}\|x(s')\|^p  \ds'.
\end{equation*}
{On the free phase, the extended functional gives $\bar V(t-t_k,x(t))=\bar V(T_{\min},x(t))$ and, since $x(t+h)=S(h)x(t)$ for $t,t+h\in[t_k+T_{\min},t_{k+1})$, condition \eqref{eq:min:nc:flow:free} applied at $x=x(t)$ bounds the lower right Dini derivative of $t\mapsto\bar V(T_{\min},x(t))$ by $-\|x(t)\|^p$. Integrating this bound on $[t_k+T_{\min},t_{k+1})$:}
\begin{equation*}
  \bar V(T_{\min},x(t_{k+1}^-))-\bar V(T_{\min},x(t_k+T_{\min}))\le -\int_{t_k+T_{\min}}^{t_{k+1}}\|x(s')\|^p  \ds'.
\end{equation*}
At the jump, \eqref{eq:min:nc:jump} gives $\bar V(0,x(t_{k+1}))\le \bar V(T_{\min},x(t_{k+1}^-))-\|x(t_{k+1}^-)\|^p$. Summing and telescoping over $k$, with $\bar V\ge0$ and $\bar V(\tau(s),y)\le c\|y\|^p$, {where $\tau(s)\ge0$ denotes the time elapsed at time $s$ since the last impulse instant in $(0,s]$ (or since $t_0=0$ if there is none), possibly larger than $T_{\min}$,} yields \eqref{eq:cst:datko} at starting time $s$ with $b^p=c$, uniformly in $\sigma$.

\medskip\noindent\textbf{Proof that $(ii)\Rightarrow(i)$.} Lemma~\ref{lem:datko} applied uniformly over $\mathcal{S}_{\min}(T_{\min})$ yields {strong} UGES. Proposition~\ref{prop:min:necessary} then gives exponential stability of $S$, consistent with the requirement $\omega_0\le0$ from Assumption~\ref{ass:growth} on this family (Remark~\ref{rem:growth:G}).

\medskip\noindent\textbf{Proof that $(iii)\Rightarrow(iv)$.} First, the free-phase flow \eqref{eq:min:nc:flow:free} establishes the semigroup-stability clause of (iv): integrating it along $r\mapsto S(r)x$ gives $$\int_0^\infty\|S(r)x\|^p  \dr\le\bar V(T_{\min},x)\le c\|x\|^p,$$ and Datko's theorem for $C_0$-semigroups \cite{Datko:70} yields $\|S(t)\|\le M_0e^{-\alpha_0 t}$ for some $M_0\ge1$, $\alpha_0>0$. Set $V_d(x):=\bar V(T_{\min},x)$, which inherits the upper bound $V_d(x)\le c\|x\|^p$ (so $\hat c=c$) and Lipschitz continuity from (iii). For $\theta\ge T_{\min}$ and $x\in X$, chain the free-phase flow on $[T_{\min},\theta]$ (inequality~(a)), the mandatory-phase flow on $[0,T_{\min}]$ from $Jx$ (inequality~(b)), and the jump with pre-jump state $x$ (inequality~(c)):
\begin{align*}
  V_d(S(\theta)Jx)=\bar V(T_{\min},S(\theta)Jx)
  &\overset{\text{(a)}}{\le} \bar V(T_{\min},S(T_{\min})Jx)-\int_{T_{\min}}^{\theta}\|S(r)Jx\|^p  \dr\\
  &\overset{\text{(b)}}{\le} \bar V(0,Jx)-\int_0^{\theta}\|S(r)Jx\|^p  \dr\\
  &\overset{\text{(c)}}{\le} \bar V(T_{\min},x)-\|x\|^p-\int_0^{\theta}\|S(r)Jx\|^p  \dr\\
  &\le V_d(x)-\|x\|^p,
\end{align*}
where (a) uses the free-phase flow \eqref{eq:min:nc:flow:free}, (b) the mandatory flow \eqref{eq:min:nc:flow:mandatory}, (c) the jump \eqref{eq:min:nc:jump}, and the last step drops the non-negative integral. This is \eqref{eq:min:dt} with $\hat c=c$.

\medskip\noindent\textbf{Proof that $(iv)\Rightarrow(iii)$.} By the semigroup-stability clause of (iv), fix constants $M_0\ge1$ and $\alpha_0>0$ with
\begin{equation}\label{eq:min:S:expstable}
  \|S(t)\|\le M_0e^{-\alpha_0 t},\qquad t\ge0,
\end{equation}
and write $M_S:=\sup_{r\in[0,T_{\min}]}\|S(r)\|\le M_0$ for the (finite) bound of the semigroup over the mandatory phase. Define the discrete-time flow excess
\begin{equation}\label{eq:min:cS}
  c_S:=\int_0^\infty\|S(r)\|^p  \dr\le\frac{M_0^p}{p  \alpha_0}<\infty,
\end{equation}
which is finite by \eqref{eq:min:S:expstable}. We build the timer-dependent functional in two steps: first a timer-independent functional $W$ on $X$ that is dissipative along both the flow and the discrete map $S(\theta)J$, then a stitched timer-dependent extension $\bar V$.

\emph{Step 1: construction of $W$.} Set
\begin{equation}\label{eq:min:W:def}
  V_d^*(y):=\sup_{\tau\ge0}V_d(S(\tau)y),\qquad V_{\rm flow}(y):=\int_0^\infty\|S(r)y\|^p  \dr,\qquad W:=(1+c_S\|J\|^p)V_d^*+V_{\rm flow}.
\end{equation}
We record four properties of these functionals.
\begin{enumerate}[(a)]
  \item \emph{Upper bound on $V_d^*$.} Using \eqref{eq:min:dt} ($V_d(x)\le\hat c\|x\|^p$) and \eqref{eq:min:S:expstable}, we obtain $$V_d^*(y)\le\hat c\sup_{\tau\ge0}\|S(\tau)y\|^p\le\hat c  M_0^p\|y\|^p.$$
  \item \emph{Flow monotonicity of $V_d^*$.} For $t\ge0$, the following relationship holds $$V_d^*(S(t)y)=\sup_{\tau\ge0}V_d(S(\tau+t)y)\le\sup_{\tau'\ge0}V_d(S(\tau')y)=V_d^*(y),$$ since the shifted index set $\{\tau+t:\tau\ge0\}$ is contained in $\{\tau'\ge0\}$.
  \item \emph{Discrete decay of $V_d^*$.} For $\theta\ge T_{\min}$ and $y\in X$, every $\tau\ge0$ gives $\tau+\theta\ge T_{\min}$, so \eqref{eq:min:dt} applies to $$V_d(S(\tau+\theta)Jy)\le V_d(y)-\|y\|^p\le V_d^*(y)-\|y\|^p.$$ Taking the supremum over $\tau\ge0$ yields $$V_d^*(S(\theta)Jy)\le V_d^*(y)-\|y\|^p.$$
  \item \emph{Properties of $V_{\rm flow}$.} By \eqref{eq:min:S:expstable}, $V_{\rm flow}(y)\le c_S\|y\|^p$ and $V_{\rm flow}(Jy)\le c_S\|J\|^p\|y\|^p$. Along the flow, $$V_{\rm flow}(S(t)y)=\int_t^\infty\|S(u)y\|^p  \du=V_{\rm flow}(y)-\int_0^t\|S(u)y\|^p  \du,$$ so $\underline D^+V_{\rm flow}(y)=-\|y\|^p$ by strong continuity of $S(\cdot)$ at $u=0$. Moreover, for $\theta\ge0$, we have $$V_{\rm flow}(S(\theta)Jy)=V_{\rm flow}(Jy)-\int_0^\theta\|S(r)Jy\|^p  \dr.$$
\end{enumerate}
Combining (a) and (d), $W(y)\le\kappa\|y\|^p$ with
\begin{equation}\label{eq:min:kappa}
  \kappa:=(1+c_S\|J\|^p)\hat c  M_0^p+c_S.
\end{equation}
Combining the flow monotonicity (b) of $V_d^*$ with $\underline D^+V_{\rm flow}=-\|y\|^p$ from (d), $\underline D^+W(y)\le-\|y\|^p$. Finally, for the discrete map with $\theta\ge T_{\min}$,
\begin{align}
  W(S(\theta)Jy)-W(y)
  &=(1+c_S\|J\|^p)\bigl[V_d^*(S(\theta)Jy)-V_d^*(y)\bigr]+\bigl[V_{\rm flow}(S(\theta)Jy)-V_{\rm flow}(y)\bigr]\notag\\
  &\le(1+c_S\|J\|^p)(-\|y\|^p)+V_{\rm flow}(Jy)-V_{\rm flow}(y)-\int_0^\theta\|S(r)Jy\|^p  \dr\notag\\
  &\le-\|y\|^p-c_S\|J\|^p\|y\|^p+c_S\|J\|^p\|y\|^p-\int_0^\theta\|S(r)Jy\|^p  \dr\notag\\
  &=-\|y\|^p-\int_0^\theta\|S(r)Jy\|^p  \dr,\label{eq:min:W:jump}
\end{align}
where the second line uses (c) and the last identity of (d), and the third line uses $V_{\rm flow}(Jy)\le c_S\|J\|^p\|y\|^p$ and $V_{\rm flow}(y)\ge0$.

\emph{Step 2: stitched timer-dependent functional.} Set $\varsigma(\tau):=\max(T_{\min}-\tau,0)$ and
\begin{equation}\label{eq:min:Vbar:stitched}
  \bar V(\tau,x):=W(S(\varsigma(\tau))x)+\int_0^{\varsigma(\tau)}\|S(r)x\|^p  \dr,\qquad(\tau,x)\in[0,T_{\min}]\times X.
\end{equation}
Non-negativity is immediate from $W\ge0$ and the non-negative integrand. We verify each condition of (iii).

\noindent\emph{Upper bound \eqref{eq:min:nc:upper}.} Since $\varsigma(\tau)\in[0,T_{\min}]$, $\|S(\varsigma(\tau))x\|\le M_S\|x\|$, and using $W(z)\le\kappa\|z\|^p$ and $\|S(r)\|\le M_S$ on $[0,T_{\min}]$:
\begin{equation*}
  \bar V(\tau,x)\le\kappa  M_S^p\|x\|^p+M_S^p  \varsigma(\tau)\|x\|^p\le[\kappa  M_S^p+M_S^p  T_{\min}]\|x\|^p,
\end{equation*}
which is \eqref{eq:min:nc:upper} with $c:=\kappa  M_S^p+M_S^p  T_{\min}$.

\noindent\emph{Mandatory-phase flow \eqref{eq:min:nc:flow:mandatory} ($\tau+h\le T_{\min}$).} Here $\varsigma(\tau)=T_{\min}-\tau$ and $\varsigma(\tau+h)=T_{\min}-\tau-h$. Using $S(T_{\min}-\tau-h)S(h)=S(T_{\min}-\tau)$ and the change of variable $r'=r+h$ in the integral:
\begin{align*}
  \bar V(\tau+h,S(h)x)&=W(S(T_{\min}-\tau-h)S(h)x)+\int_0^{T_{\min}-\tau-h}\|S(r)S(h)x\|^p  \dr\\
  &=W(S(T_{\min}-\tau)x)+\int_h^{T_{\min}-\tau}\|S(r')x\|^p  \dr'=\bar V(\tau,x)-\int_0^h\|S(r')x\|^p  \dr'.
\end{align*}
Dividing by $h$ and letting $h\to0^+$, $h^{-1}\!\int_0^h\|S(r')x\|^p\,\dr'\to\|x\|^p$ by strong continuity of $S(\cdot)$ at $r'=0$, hence $\underline{D}^+\bar V(\tau,x)=-\|x\|^p$.

\noindent\emph{Free-phase flow \eqref{eq:min:nc:flow:free} ($\tau\ge T_{\min}$).} Here $\varsigma(\tau)=\varsigma(\tau+h)=0$, so $\bar V(\tau,x)=W(x)$ is timer-independent, and $$\bar V(\tau+h,S(h)x)-\bar V(\tau,x)=W(S(h)x)-W(x)\le-\int_0^h\|S(r)x\|^p  \dr$$ by the integrated flow inequality of $W$ established in Step 1 (combining the flow monotonicity (b) of $V_d^*$ with the identity in (d) for $V_{\rm flow}$). The same limit gives $\underline{D}^+\bar V(\tau,x)\le-\|x\|^p$.

\noindent\emph{Jump condition \eqref{eq:min:nc:jump} ($\tau=\theta\ge T_{\min}$).} Here $\varsigma(\theta)=0$, so $\bar V(\theta,z)=W(z)$, while $\varsigma(0)=T_{\min}$ gives $$\bar V(0,Jz)=W(S(T_{\min})Jz)+\int_0^{T_{\min}}\|S(r)Jz\|^p  \dr.$$ Therefore, applying \eqref{eq:min:W:jump} at $\theta=T_{\min}$:
\begin{align*}
  \bar V(0,Jz)-\bar V(\theta,z)&=W(S(T_{\min})Jz)-W(z)+\int_0^{T_{\min}}\|S(r)Jz\|^p  \dr\\
  &\le-\|z\|^p-\int_0^{T_{\min}}\|S(r)Jz\|^p  \dr+\int_0^{T_{\min}}\|S(r)Jz\|^p  \dr=-\|z\|^p,
\end{align*}
which is \eqref{eq:min:nc:jump}.

\noindent\emph{Lipschitz continuity.} $V_d$ is Lipschitz on bounded sets by (iv); hence so are $V_d^*$ (a supremum of $\|S(\tau)\|$-Lipschitz maps with $\sup_\tau\|S(\tau)\|\le M_0$), $V_{\rm flow}$ (integrand Lipschitz on bounded sets, uniformly integrable by \eqref{eq:min:S:expstable}), and thus $W$. Composition with the bounded maps $S(\varsigma(\tau))$ and the bounded integrand in \eqref{eq:min:Vbar:stitched} gives Lipschitz continuity of $\bar V$ on bounded subsets of $X$, uniformly in $\tau\in[0,T_{\min}]$.

{\noindent\emph{Continuity in $\tau$.} For every $x\in X$, the map $\tau\mapsto\bar V(\tau,x)$ in \eqref{eq:min:Vbar:stitched} is continuous on $[0,T_{\min}]$, since $\varsigma$ is continuous, $\tau\mapsto S(\varsigma(\tau))x$ is continuous by strong continuity of $S(\cdot)$, $W$ is Lipschitz on bounded subsets of $X$, and the integral is continuous with respect to its upper limit.}
\end{proof}

{Conditions \eqref{eq:min:nc:flow:mandatory} and \eqref{eq:min:nc:flow:free} play different roles. Condition \eqref{eq:min:nc:flow:mandatory} is a condition on the timer-dependent functional on the mandatory phase, and its Dini derivative accounts for the joint variation of the timer $\tau$ and of the state along the flow. Condition \eqref{eq:min:nc:flow:free} is a condition on the functional $\bar V(T_{\min},\cdot)$, which governs the free phase, where the value of the timer does not affect the functional. Its Dini derivative therefore only accounts for the variation of the state along the flow, and involves no derivative with respect to the timer. As we shall see in Section~\ref{subsec:hilbert:min}, this distinction is reflected in the quadratic conditions of Theorem~\ref{thm:hilbert:min}, where the mandatory-phase condition \eqref{eq:hilbert:min:flow} contains the derivative $\dot{\bar P}(\tau)$ while the free-phase condition \eqref{eq:hilbert:min:free} does not.}

\subsection{Range dwell-time ($\mathcal{S}_{\mathrm{rng}}(T_{\min},T_{\max})$)}\label{subsec:nc:rng}

For $\mathcal{S}_{\mathrm{rng}}(T_{\min},T_{\max})$, inter-jump intervals satisfy $T_{\min}\le t_{k+1}-t_k\le T_{\max}$. The timer ranges over $[0,T_{\max}]$, and the jump condition must hold for all admissible inter-jump lengths $\theta\in[T_{\min},T_{\max}]$.

\begin{theorem}[Range dwell-time, non-coercive N\&S]\label{thm:rng:nc}
Let $0<T_{\min}\le T_{\max}<\infty$ and let Assumption~\ref{ass:growth} hold. The following are equivalent:
\begin{enumerate}[(i)]
  \item the system \eqref{eq:syst} is strongly UGES over $\mathcal{S}_{\mathrm{rng}}(T_{\min},T_{\max})$;
  \item (\emph{Datko}) there exist $p\ge1$ and $b>0$ such that \eqref{eq:cst:datko} holds for all $\sigma\in\mathcal{S}_{\mathrm{rng}}(T_{\min},T_{\max})$ and all $s\ge0$;
  \item (\emph{continuous-time hybrid}) there exist $p\ge1$, $c>0$, and a functional $\bar V:[0,T_{\max}]\times X\to\mathbb{R}_{\ge0}$, Lipschitz on bounded subsets of $X$ uniformly in $\tau$, {with $\tau\mapsto\bar V(\tau,x)$ continuous on $[0,T_{\max}]$ for every $x\in X$,} such that
    \begin{subequations}\label{eq:rng:nc:hybrid}
    \begin{align}
      \bar V(\tau,x)&\le c\|x\|^p,\;&&\forall\tau\in[0,T_{\max}],  x\in X,\label{eq:rng:nc:upper}\\
      \underline{D}^+\bar V(\tau,x)&\le -\|x\|^p,\;&&\forall\tau\in{[0,T_{\max})},  x\in X,\label{eq:rng:nc:flow}\\
      \bar V(0,Jx)-\bar V(\theta,x)&\le -\|x\|^p,\;&&\forall\theta\in[T_{\min},T_{\max}],  x\in X;\label{eq:rng:nc:jump}
    \end{align}
    \end{subequations}
  \item (\emph{discrete-time}) the family $\{S(\theta)  J:\theta\in[T_{\min},T_{\max}]\}$ is uniformly geometrically stable, i.e., there exist $p\ge1$, $\hat c>0$, and a functional $V_d:X\to\mathbb{R}_{\ge0}$ Lipschitz on bounded subsets of $X$ such that
    \begin{equation}\label{eq:rng:dt}
      V_d(x)\le \hat c\|x\|^p,\qquad V_d(S(\theta)  J  x)-V_d(x)\le -\|x\|^p,\qquad\forall x\in X,\;\theta\in[T_{\min},T_{\max}].
    \end{equation}
\end{enumerate}
Moreover, when (i) holds, an explicit timer-dependent Lyapunov functional witnessing (iii) is given by
\begin{equation}\label{eq:rng:V:explicit}
  \bar V(\tau,x):=\sup_{\sigma\in\mathcal{S}^{(\tau)}_{\mathrm{rng}}}\left[\int_0^\infty\|U_\sigma(s,0)x\|^p  \ds+\sum_{k\ge1}\|U_\sigma(t_k^-,0)x\|^p\right],
\end{equation}
with, for $\tau\in[0,T_{\max}]$,
\begin{equation}\label{eq:rng:family:tau}
  \mathcal{S}^{(\tau)}_{\mathrm{rng}}:=\big\{\sigma\in\mathcal{S}_{0,\infty}:\ t_1^\sigma\in[\max(T_{\min}-\tau,0),\,T_{\max}-\tau]\ \text{ and }\ t_{k+1}^\sigma-t_k^\sigma\in[T_{\min},T_{\max}]\ \forall k\ge1\big\}.
\end{equation}
As for the minimum dwell-time, $\mathcal{S}^{(\tau)}_{\mathrm{rng}}$ collects the admissible future signals at elapsed dwell $\tau$: the next jump must fall in the remaining window $[\max(T_{\min}-\tau,0),\,T_{\max}-\tau]$ and every later inter-jump interval lies in $[T_{\min},T_{\max}]$. {As in Theorem~\ref{thm:min:nc}, for $\tau\in[T_{\min},T_{\max}]$, we also include in $\mathcal{S}^{(\tau)}_{\mathrm{rng}}$, by convention, the sequences with $t_1^\sigma=0$ and $t_{k+1}^\sigma-t_k^\sigma\in[T_{\min},T_{\max}]$ for all $k\ge1$, for which the impulse at $t_1^\sigma=0$ is applied to the initial state, with the same convention as in Theorem~\ref{thm:min:nc}.} Since $\max(T_{\min}-\tau,0)\le T_{\max}-\tau$ for all $\tau\in[0,T_{\max}]$, this set is nonempty throughout, with $\mathcal{S}^{(0)}_{\mathrm{rng}}=\mathcal{S}_{\mathrm{rng}}(T_{\min},T_{\max})$.
\end{theorem}

\begin{proof}
We prove $(i)\Rightarrow(ii)\Rightarrow(iii)\Rightarrow(ii)\Rightarrow(i)$ and $(iii)\Leftrightarrow(iv)$.

\medskip\noindent\textbf{Proof that $(i)\Rightarrow(ii)$.} Same as in Theorem~\ref{thm:min:nc}, uniformly in $\sigma\in\mathcal{S}_{\mathrm{rng}}(T_{\min},T_{\max})$.

\medskip\noindent\textbf{Proof that $(ii)\Rightarrow(iii)$.} Define $\bar V$ by \eqref{eq:rng:V:explicit}. Exactly as in Theorem~\ref{thm:min:nc}, every $\sigma\in\mathcal{S}^{(\tau)}_{\mathrm{rng}}$ with $t_1^\sigma>0$ is the future, re-based to the origin, of some $\sigma'\in\mathcal{S}_{\mathrm{rng}}(T_{\min},T_{\max})$ observed at elapsed dwell $\tau$ (take $\sigma'$ with first interval $t_1^\sigma+\tau\in[T_{\min},T_{\max}]$ and the same subsequent jumps as $\sigma$), whose jump at $t_1^\sigma+\tau>\tau$ is counted by the strict Datko sum; the uniform-in-$s$ Datko bound applied to $\sigma'$ at $s=\tau$ then gives $\bar V(\tau,x)\le b^p\|x\|^p$ from such signals. For $\tau\in[T_{\min},T_{\max}]$, however, $\mathcal{S}^{(\tau)}_{\mathrm{rng}}$ also admits
{the sequences with $t_1^\sigma=0$}, whose bracket equals $\|x\|^p+\int_0^\infty\|U_\sigma(s',0)Jx\|^p\ds'+\sum_{k\ge1}\|U_\sigma(t_k^{\sigma-},0)Jx\|^p\le(1+b^p\|J\|^p)\|x\|^p$ (the leading $\|x\|^p$ being the pre-jump term omitted by the strict sum). Hence \eqref{eq:rng:nc:upper} holds with $c:=\max(b^p,\,1+b^p\|J\|^p)$, uniformly in $\tau\in[0,T_{\max}]$.

\emph{Flow condition \eqref{eq:rng:nc:flow}.} Fix $\tau\in[0,T_{\max})$, $x\in X$, $h>0$ with $\tau+h\le T_{\max}$. For $\sigma\in\mathcal{S}^{(\tau+h)}_{\mathrm{rng}}$, the shift $\sigma^{(h)}$ with $t_k^{\sigma^{(h)}}=t_k^\sigma+h$ satisfies $t_1^{\sigma^{(h)}}=t_1^\sigma+h\in[\max(T_{\min}-\tau,h),\,T_{\max}-\tau]\subseteq[\max(T_{\min}-\tau,0),\,T_{\max}-\tau]$ and keeps its later intervals in $[T_{\min},T_{\max}]$, so $\sigma^{(h)}\in\mathcal{S}^{(\tau)}_{\mathrm{rng}}$. The propagator relation, change of variable, and uniform subtraction of $\int_0^h\|S(s')x\|^p  \ds'$ (valid because every $\sigma^{(h)}$ has $t_1^{\sigma^{(h)}}\ge h$, hence no jump on $[0,h]$) give:
\begin{equation*}
  \bar V(\tau+h,S(h)x)-\bar V(\tau,x)\le -\int_0^h\|S(s')x\|^p  \ds',
\end{equation*}
hence $\underline{D}^+\bar V(\tau,x)\le-\|x\|^p$ on $[0,T_{\max})$.

\emph{Jump condition \eqref{eq:rng:nc:jump}.} Fix $x\in X$ and $\theta\in[T_{\min},T_{\max}]$. Here $\mathcal{S}^{(0)}_{\mathrm{rng}}=\mathcal{S}_{\mathrm{rng}}(T_{\min},T_{\max})$. For any $\sigma\in\mathcal{S}^{(0)}_{\mathrm{rng}}$, prepend a single jump at ${0}$ to obtain $\sigma^*$. Its first future interval has length $0$, which lies in $[\max(T_{\min}-\theta,0),\,T_{\max}-\theta]$ because $\theta\ge T_{\min}$; its subsequent intervals are $t_1^\sigma\in[T_{\min},T_{\max}]$ followed by the inter-jump intervals of $\sigma$, all in $[T_{\min},T_{\max}]$. Hence $\sigma^*\in\mathcal{S}^{(\theta)}_{\mathrm{rng}}$, and:
\begin{equation*}
  \|x\|^p+\bar V(0,Jx)\le \bar V(\theta,x),
\end{equation*}
by the same calculation as in $(ii)\Rightarrow(iii)$ of Theorem~\ref{thm:min:nc}. This gives \eqref{eq:rng:nc:jump} for every $\theta\in[T_{\min},T_{\max}]$.

\emph{Lipschitz continuity.} Argument from Theorem~\ref{thm:min:nc}.

{\emph{Continuity in $\tau$.} For $\sigma\in\mathcal{S}_{0,\infty}$ and $y\in X$, let $F_\sigma(y)$ denote the bracket in \eqref{eq:rng:V:explicit}. Let $x\in X$, $0\le\tau<\tau'\le T_{\max}$ with $h:=\tau'-\tau\le1$, and let $r>0$ be such that $\|S(s)x\|<r$ and $\|JS(s)x\|<r$ for all $s\in[0,1]$. The Lipschitz argument applies to each $F_\sigma$ with the same constant $L_r$ for all the signals considered below. On the one hand, the inequality obtained above in the proof of \eqref{eq:rng:nc:flow} gives $\bar V(\tau',S(h)x)\le\bar V(\tau,x)-\int_0^h\|S(s)x\|^p\ds$, and therefore
\begin{equation*}
  \bar V(\tau',x)\le\bar V(\tau,x)+L_r\|x-S(h)x\|.
\end{equation*}
On the other hand, let $\sigma\in\mathcal{S}^{(\tau)}_{\mathrm{rng}}$. If $t_1^\sigma>h$, the signal $\widetilde\sigma$ with $t_k^{\widetilde\sigma}=t_k^\sigma-h$, $k\ge1$, belongs to $\mathcal{S}^{(\tau')}_{\mathrm{rng}}$, and $F_\sigma(x)=\int_0^h\|S(s)x\|^p\ds+F_{\widetilde\sigma}(S(h)x)$, so that
\begin{equation*}
  F_\sigma(x)\le\int_0^h\|S(s)x\|^p\ds+\bar V(\tau',x)+L_r\|S(h)x-x\|.
\end{equation*}
If $t_1^\sigma\le h$, then $T_{\min}-\tau\le t_1^\sigma\le h$, i.e., $\tau'\ge T_{\min}$. Let $\sigma_+\in\mathcal{S}_{\mathrm{rng}}(T_{\min},T_{\max})$ be the signal with impulse instants $t_{k+1}^\sigma-t_1^\sigma$, $k\ge1$, and let $\widehat\sigma$ be the signal consisting of a jump at $0$ followed by the impulse instants of $\sigma_+$. Then, $\widehat\sigma\in\mathcal{S}^{(\tau')}_{\mathrm{rng}}$ and
\begin{equation*}
  F_\sigma(x)=\int_0^{t_1^\sigma}\|S(s)x\|^p\ds+\|S(t_1^\sigma)x\|^p+F_{\sigma_+}(JS(t_1^\sigma)x),\qquad F_{\widehat\sigma}(x)=\|x\|^p+F_{\sigma_+}(Jx),
\end{equation*}
so that
\begin{equation*}
  F_\sigma(x)\le\bar V(\tau',x)+\int_0^h\|S(s)x\|^p\ds+\sup_{s\in[0,h]}\bigl|\|S(s)x\|^p-\|x\|^p\bigr|+L_r\|J\|\sup_{s\in[0,h]}\|S(s)x-x\|.
\end{equation*}
In both cases, $F_\sigma(x)\le\bar V(\tau',x)+\omega_x(h)$, where
\begin{equation*}
  \omega_x(h):=\int_0^h\|S(s)x\|^p\ds+\sup_{s\in[0,h]}\bigl|\|S(s)x\|^p-\|x\|^p\bigr|+L_r(1+\|J\|)\sup_{s\in[0,h]}\|S(s)x-x\|
\end{equation*}
does not depend on $\sigma$ and $\tau$, and tends to $0$ as $h\downarrow0$ by strong continuity of $S(\cdot)$. Taking the supremum over $\sigma\in\mathcal{S}^{(\tau)}_{\mathrm{rng}}$ and combining with the first bound yields $|\bar V(\tau',x)-\bar V(\tau,x)|\le\omega_x(h)$. Therefore, $\tau\mapsto\bar V(\tau,x)$ is continuous on $[0,T_{\max}]$.}

\medskip\noindent\textbf{Proof that $(iii)\Rightarrow(ii)$.} Telescoping on each inter-jump interval $[t_k,t_{k+1})$ of length $\theta_k\in[T_{\min},T_{\max}]$: integrate \eqref{eq:rng:nc:flow} on $[0,\theta_k]$ in the timer variable{, which is justified as in the proof of $(iii)\Rightarrow(ii)$ of Theorem~\ref{thm:main} using the continuity of $\bar V(\cdot,x)$ on $[0,T_{\max}]$, the continuity of $x(\cdot)$ on $[t_k,t_{k+1})$, and the Lipschitz property of $\bar V$,} to get $$\bar V(\theta_k,x(t_{k+1}^-))-\bar V(0,x(t_k))\le -\int_{t_k}^{t_{k+1}}\|x(s')\|^p  \ds'.$$ Applying then \eqref{eq:rng:nc:jump} with $\theta=\theta_k$ yields $\bar V(0,x(t_{k+1}))\le \bar V(\theta_k,x(t_{k+1}^-))-\|x(t_{k+1}^-)\|^p$. Summing and telescoping yields \eqref{eq:cst:datko} at any starting time $s$ with $b^p=c$, uniformly in $\sigma\in\mathcal{S}_{\mathrm{rng}}(T_{\min},T_{\max})$.

\medskip\noindent\textbf{Proof that $(ii)\Rightarrow(i)$.} Lemma~\ref{lem:datko} applied uniformly over $\mathcal{S}_{\mathrm{rng}}(T_{\min},T_{\max})$ yields {strong} UGES; Assumption~\ref{ass:growth} is automatic with $G=M_0e^{|\omega_0|T_{\max}}$ since inter-jump intervals are bounded above.

\medskip\noindent\textbf{Proof that $(iii)\Rightarrow(iv)$.} Set $V_d(x):=\inf_{\theta_0\in[T_{\min},T_{\max}]}\bar V(\theta_0,x)$, which inherits $V_d(x)\le c\|x\|^p$ (so $\hat c=c$) and Lipschitz continuity from (iii), the infimum of a uniformly Lipschitz family being Lipschitz. For $\theta\in[T_{\min},T_{\max}]$, $x\in X$, and any $\theta_0\in[T_{\min},T_{\max}]$, chain the flow on $[0,\theta]$ from $Jx$ (inequality~(a)) and the jump at timer value $\theta_0$ (inequality~(b)):
\begin{align*}
  V_d(S(\theta)Jx)\le\bar V(\theta,S(\theta)Jx)
  &\overset{\text{(a)}}{\le}\bar V(0,Jx)-\int_0^\theta\|S(r)Jx\|^p  \dr\\
  &\overset{\text{(b)}}{\le}\bar V(\theta_0,x)-\|x\|^p-\int_0^\theta\|S(r)Jx\|^p  \dr\\
  &\le\bar V(\theta_0,x)-\|x\|^p,
\end{align*}
where the first inequality uses $V_d\le\bar V(\theta,\cdot)$, (a) the flow \eqref{eq:rng:nc:flow} on $[0,\theta]$ from $Jx$, (b) the jump \eqref{eq:rng:nc:jump} with inter-jump length $\theta_0$, and the last step drops the non-negative integral. Taking the infimum over $\theta_0\in[T_{\min},T_{\max}]$ on the right gives $V_d(S(\theta)Jx)\le V_d(x)-\|x\|^p$, which is \eqref{eq:rng:dt} with $\hat c=c$.

\medskip\noindent\textbf{Proof that $(iv)\Rightarrow(iii)$.} Since inter-jump intervals are bounded above by $T_{\max}$, Assumption~\ref{ass:growth} gives the finite bound $M_S:=M_0e^{|\omega_0|T_{\max}}$, so that $\|S(r)\|\le M_S$ for all $r\in[0,T_{\max}]$. Set $\beta:=1+M_S^p  T_{\max}\|J\|^p$ and define, on $[0,T_{\max}]\times X$,
\begin{equation}\label{eq:rng:Vbar:sup}
  \bar V(\tau,x):=\sup_{\theta\in[\max(\tau,T_{\min}),T_{\max}]}\!\Big[\beta  V_d(S(\theta-\tau)x)+\int_0^{\theta-\tau}\|S(r)x\|^p  \dr\Big].
\end{equation}
The supremum is over a non-empty compact interval since $\max(\tau,T_{\min})\le T_{\max}$ for $\tau\in[0,T_{\max}]$, and $\bar V\ge0$ because $V_d\ge0$ and the integrand is non-negative. We verify each condition of (iii).

\noindent\emph{Upper bound \eqref{eq:rng:nc:upper}.} For any admissible $\theta$, using $V_d(z)\le\hat c\|z\|^p$ from \eqref{eq:rng:dt}, $\|S(\theta-\tau)\|\le M_S$, and $\theta-\tau\le T_{\max}$:
\begin{equation*}
  \beta  V_d(S(\theta-\tau)x)+\int_0^{\theta-\tau}\|S(r)x\|^p  \dr\le\beta\hat c  M_S^p\|x\|^p+M_S^p  T_{\max}\|x\|^p,
\end{equation*}
and taking the supremum over $\theta$ gives \eqref{eq:rng:nc:upper} with $c:=\beta\hat c  M_S^p+M_S^p  T_{\max}$.

\noindent\emph{Flow condition \eqref{eq:rng:nc:flow} ($\tau\in[0,T_{\max})$, $\tau+h\le T_{\max}$).} Fix $\theta\in[\max(\tau+h,T_{\min}),T_{\max}]$. Using $S(\theta-\tau-h)S(h)=S(\theta-\tau)$ and the change of variable $r'=r+h$, the summand at $(\tau+h,S(h)x)$ satisfies
\begin{align*}
  &\beta  V_d(S(\theta-\tau-h)S(h)x)+\int_0^{\theta-\tau-h}\|S(r)S(h)x\|^p  \dr
  =\beta  V_d(S(\theta-\tau)x)+\int_h^{\theta-\tau}\|S(r')x\|^p  \dr'\\
  &\qquad=\Big[\beta  V_d(S(\theta-\tau)x)+\int_0^{\theta-\tau}\|S(r')x\|^p  \dr'\Big]-\int_0^h\|S(r')x\|^p  \dr'.
\end{align*}
The bracket is the summand of $\bar V(\tau,x)$ at the same $\theta$. Since the supremation range $[\max(\tau+h,T_{\min}),T_{\max}]$ for $\bar V(\tau+h,S(h)x)$ is contained in the range $[\max(\tau,T_{\min}),T_{\max}]$ for $\bar V(\tau,x)$, taking the supremum over the smaller range:
\begin{equation*}
  \bar V(\tau+h,S(h)x)\le\bar V(\tau,x)-\int_0^h\|S(r)x\|^p  \dr.
\end{equation*}
Dividing by $h$ and letting $h\to0^+$, $h^{-1}\!\int_0^h\|S(r)x\|^p\,\dr\to\|x\|^p$ by strong continuity of $S(\cdot)$ at $r=0$, hence $\underline{D}^+\bar V(\tau,x)\le-\|x\|^p$.

\noindent\emph{Jump condition \eqref{eq:rng:nc:jump} ($\theta_0\in[T_{\min},T_{\max}]$).} On one hand, applying \eqref{eq:rng:dt} ($V_d(S(\theta)Jz)\le V_d(z)-\|z\|^p$ for $\theta\ge T_{\min}$) and $\int_0^\theta\|S(r)Jz\|^p  \dr\le M_S^p  T_{\max}\|J\|^p\|z\|^p$:
\begin{align*}
  \bar V(0,Jz)&=\sup_{\theta\in[T_{\min},T_{\max}]}\!\Big[\beta  V_d(S(\theta)Jz)+\int_0^{\theta}\|S(r)Jz\|^p  \dr\Big]\\
  &\le\sup_{\theta\in[T_{\min},T_{\max}]}\!\Big[\beta(V_d(z)-\|z\|^p)+M_S^p  T_{\max}\|J\|^p\|z\|^p\Big]\\
  &=\beta  V_d(z)-\beta\|z\|^p+M_S^p  T_{\max}\|J\|^p\|z\|^p=\beta  V_d(z)-\|z\|^p,
\end{align*}
the last equality by the choice $\beta=1+M_S^p  T_{\max}\|J\|^p$. On the other hand, choosing $\theta=\theta_0$ in \eqref{eq:rng:Vbar:sup} (admissible since $\theta_0\in[\max(\theta_0,T_{\min}),T_{\max}]=[\theta_0,T_{\max}]$) gives the lower bound $\bar V(\theta_0,z)\ge\beta  V_d(S(0)z)+0=\beta  V_d(z)$. Subtracting,
\begin{equation*}
  \bar V(0,Jz)-\bar V(\theta_0,z)\le[\beta  V_d(z)-\|z\|^p]-\beta  V_d(z)=-\|z\|^p,
\end{equation*}
which is \eqref{eq:rng:nc:jump}.

\noindent\emph{Lipschitz continuity.} $V_d$ is Lipschitz on bounded subsets by (iv); the maps $x\mapsto V_d(S(\theta-\tau)x)$ and $x\mapsto\int_0^{\theta-\tau}\|S(r)x\|^p  \dr$ are therefore Lipschitz on bounded subsets uniformly in $\theta,\tau\in[0,T_{\max}]$ (using $\|S(\cdot)\|\le M_S$), and a supremum of a family of uniformly Lipschitz functions is Lipschitz with the same constant. Hence $\bar V$ is Lipschitz on bounded subsets of $X$ uniformly in $\tau\in[0,T_{\max}]$.

{\noindent\emph{Continuity in $\tau$.} Fix $x\in X$ and set $g_x(u):=\beta V_d(S(u)x)+\int_0^{u}\|S(r)x\|^p\dr$ for $u\in[0,T_{\max}]$. The function $g_x$ is continuous, hence uniformly continuous, on $[0,T_{\max}]$, by strong continuity of $S(\cdot)$ and the Lipschitz property of $V_d$ on bounded subsets of $X$. With the change of variable $u=\theta-\tau$, \eqref{eq:rng:Vbar:sup} reads
\begin{equation*}
  \bar V(\tau,x)=\sup_{u\in I(\tau)}g_x(u),\qquad I(\tau):=[\max(T_{\min}-\tau,0),T_{\max}-\tau].
\end{equation*}
For $\tau,\tau'\in[0,T_{\max}]$, both endpoints of $I(\tau)$ and $I(\tau')$ differ by at most $|\tau-\tau'|$, so that every point of one of these nonempty intervals lies within distance $|\tau-\tau'|$ of the other interval. Hence, $|\bar V(\tau,x)-\bar V(\tau',x)|$ is bounded by the modulus of continuity of $g_x$ evaluated at $|\tau-\tau'|$, and $\tau\mapsto\bar V(\tau,x)$ is continuous on $[0,T_{\max}]$.}
\end{proof}

{\begin{remark}
Over the families $\mathcal{S}_{\mathrm{cst}}(T)$ and $\mathcal{S}_{\mathrm{rng}}(T_{\min},T_{\max})$, the notions of UGES and strong UGES coincide. Indeed, every sequence $\sigma$ in these families satisfies
\begin{equation*}
  \frac{t-s}{T_{\max}}-1\le\kappa_\sigma(t,s)\le\frac{t-s}{T_{\min}}+1,\qquad t\ge s\ge0,
\end{equation*}
with $T_{\min}=T_{\max}=T$ for $\mathcal{S}_{\mathrm{cst}}(T)$. Hence, if \eqref{eq:defstab} holds with $\alpha=0$ and $\rho\in(0,1)$, then
\begin{equation*}
  \rho^{\kappa_\sigma(t,s)}\le\rho^{-1/2}\left(\rho^{1/2}\right)^{\kappa_\sigma(t,s)}e^{-\frac{\ln(1/\rho)}{2T_{\max}}(t-s)},
\end{equation*}
and if \eqref{eq:defstab} holds with $\alpha>0$ and $\rho=1$, then
\begin{equation*}
  e^{-\alpha(t-s)}\le e^{\alpha T_{\min}/2}\left(e^{-\alpha T_{\min}/2}\right)^{\kappa_\sigma(t,s)}e^{-\alpha(t-s)/2}.
\end{equation*}
In both cases, the system is strongly UGES. Hence, strong UGES may be equivalently replaced by UGES in statement (i) of Theorems~\ref{thm:cst:nc} and~\ref{thm:rng:nc}, whereas this is not possible in Theorem~\ref{thm:min:nc}, for which inter-jump intervals are not bounded above.
\end{remark}}

{\begin{remark}\label{rem:eventual}
The results of Sections~\ref{subsec:nc:cst}--\ref{subsec:nc:rng} also apply to sequences that belong to one of the families \eqref{eq:family:cst}--\eqref{eq:family:rng} only after a finite number of impulses. Let $\mathcal{S}$ be one of these families, and assume that the system \eqref{eq:syst} is UGES over $\mathcal{S}$ with constants $M\ge1$, $\alpha>0$, and $\rho\in(0,1)$. Let $\sigma\in\mathcal{S}_{0,\infty}$ and $k_0\ge1$ be such that the sequence $\sigma_{k_0}:=\{t_{k_0+k}-t_{k_0}\}_{k\ge1}$ belongs to $\mathcal{S}$. Since $U_\sigma(t,s)=U_{\sigma_{k_0}}(t-t_{k_0},s-t_{k_0})$ and $\kappa_\sigma(t,s)=\kappa_{\sigma_{k_0}}(t-t_{k_0},s-t_{k_0})$ for all $t\ge s\ge t_{k_0}$, the bound \eqref{eq:defstab} holds for all $t\ge s\ge t_{k_0}$. Moreover, with $M_0\ge1$ and $\omega_0\in\mathbb{R}$ as in Remark~\ref{rem:growth:G}, we have
\begin{equation*}
  \|U_\sigma(r,s)\|\le G_{k_0}:=M_0\max(1,M_0\|J\|)^{k_0}e^{\max(\omega_0,0)t_{k_0}},\qquad 0\le s\le r\le t_{k_0}.
\end{equation*}
Hence, for $0\le s\le t_{k_0}\le t$, using $\kappa_\sigma(t_{k_0},s)\le k_0$ and $t_{k_0}-s\le t_{k_0}$, we get
\begin{equation*}
\begin{aligned}
  \|U_\sigma(t,s)\|&\le\|U_\sigma(t,t_{k_0})\|\|U_\sigma(t_{k_0},s)\|\le MG_{k_0}\rho^{\kappa_\sigma(t,t_{k_0})}e^{-\alpha(t-t_{k_0})}\\
  &\le MG_{k_0}\rho^{-k_0}e^{\alpha t_{k_0}}\rho^{\kappa_\sigma(t,s)}e^{-\alpha(t-s)},
\end{aligned}
\end{equation*}
and the same final bound holds for $0\le s\le t\le t_{k_0}$. Therefore, the system is strongly GES for the sequence $\sigma$, with $M$ replaced by $MG_{k_0}\rho^{-k_0}e^{\alpha t_{k_0}}$. This constant is uniform over any set of such sequences on which $k_0$ and $t_{k_0}$ are uniformly bounded, and the system is then UGES over this set.
\end{remark}}

\section{Coercive Lyapunov stability conditions}
\label{sec:coercive}

The non-coercive conditions of Section~\ref{sec:lyapunov} are both necessary and sufficient for UGES via upper-bounded Lyapunov functionals. In this section we establish the parallel theory in which the Lyapunov functional is additionally \emph{coercive}. We work throughout with the \emph{Lipschitz coercive} (equivalent-norm) class: $V$ is Lipschitz on bounded subsets of $X$ but otherwise unrestricted in structure. Such functionals always exist when
{strong (U)GES} holds, via the sup-type construction below; consequently, all theorems in this section are full equivalences (necessary and sufficient) in this class.


The structure of this section therefore fully mirrors Section~\ref{sec:lyapunov}: a subsection for each family of impulse sequences, with each subsection stating a single equivalence theorem connecting (i) UGES, (ii) a coercive continuous-time hybrid Lyapunov criterion, and (iii) a coercive discrete-time criterion.

\subsection{Fixed impulse sequence}\label{subsec:coerc:fixed}

\begin{theorem}[Fixed sequence, coercive N\&S]\label{thm:coerc:banach}
Let $\sigma\in\mathcal{S}_{0,\infty}$ be given and let Assumption~\ref{ass:growth} hold. The following are equivalent:
\begin{enumerate}[(i)]
  \item the system \eqref{eq:syst} is strongly GES (Definition~\ref{def:stab});
  \item there exist constants $p\ge1$, $c\ge1$, $\eta>0$, $\mu\in(0,1)$, and a functional $V:\mathbb{R}_{\ge0}\times X\to\mathbb{R}_{\ge0}$, Lipschitz on bounded subsets of $X$ uniformly in $t$, {such that, for every $x\in X$, the map $t\mapsto V(t,x)$ is continuous on $[t_k,t_{k+1})$ for all $k\ge0$ and admits a left limit $V(t_k^-,x)$ at $t_k$ for all $k\ge1$,} satisfying
    \begin{align}
      \|x\|^p\;\le\; V(t,x)&\;\le\; c\|x\|^p,\label{eq:coerc:bounds}\\
      \underline{D}^+V(t,x)&\;\le\;-\eta   V(t,x),\label{eq:coerc:flow}\\
      V(t_k,Jx)&\;\le\;\mu   V(t_k^-,x),\label{eq:coerc:jump}
    \end{align}
    for all $t\ge0$, $x\in X$, and all jump times $t_k\in\mathbb{T}_\sigma$.
\end{enumerate}
Moreover, when (i) holds with constants $M\ge1$, $\alpha>0$, $\rho\in(0,1)$, an explicit functional witnessing (ii) is given by
\begin{equation}\label{eq:sup:functional}
  V(t,x)\;:=\;\sup_{\tau\ge t}  e^{\eta(\tau-t)}\nu^{\kappa_\sigma(\tau,t)}  \|U_\sigma(\tau,t)x\|,\qquad t\ge0,\;x\in X,
\end{equation}
for any choice of $\eta\in(0,\alpha)$ and $\nu\in(1,1/\rho)$; this $V$ satisfies \eqref{eq:coerc:bounds}--\eqref{eq:coerc:jump} with $p=1$, $c=M$, the same $\eta$, and $\mu=1/\nu\in(0,1)$.
\end{theorem}

\begin{proof}
\medskip\noindent\textbf{Proof that $(ii)\Rightarrow(i)$.} Let $s\ge0$ and $x_0\in X$ be given, and let $x(t)=U_\sigma(t,s)x_0$.
{On each inter-jump interval $[t_k,t_{k+1})$, condition \eqref{eq:coerc:flow} is integrated along $x(\cdot)$. The integration is justified as in the proof of $(iii)\Rightarrow(ii)$ of Theorem~\ref{thm:main}, using the continuity of $V(\cdot,x)$, the continuity of $x(\cdot)$ on $[t_k,t_{k+1})$, and the Lipschitz property of $V$, the monotonicity argument being applied to the continuous function $t\mapsto e^{\eta(t-t_k)}V(t,x(t))$, whose lower right Dini derivative is nonpositive by \eqref{eq:coerc:flow}; the convergence of $V(t,x(t))$ to $V(t_{k+1}^-,x(t_{k+1}^-))$ as $t\uparrow t_{k+1}$ is obtained as in that proof. This yields}
\begin{equation*}
  V(t,x(t))\;\le\; V(t_k,x(t_k))  e^{-\eta(t-t_k)},\qquad t\in[t_k,t_{k+1}).
\end{equation*}
At each jump time $t_{k+1}$, $x(t_{k+1})=Jx(t_{k+1}^-)$ and \eqref{eq:coerc:jump} give:
\begin{equation*}
  V(t_{k+1},x(t_{k+1}))\;\le\;\mu   V(t_{k+1}^-,x(t_{k+1}^-))\;\le\;\mu  e^{-\eta(t_{k+1}-t_k)}  V(t_k,x(t_k)).
\end{equation*}
Iterating over all $n=\kappa_\sigma(t,s)$ jumps in $(s,t]$:
\begin{equation*}
  V(t,x(t))\;\le\; V(s,x_0)  \mu^{\kappa_\sigma(t,s)}  e^{-\eta(t-s)}.
\end{equation*}
Applying both bounds of \eqref{eq:coerc:bounds}:
\begin{equation*}
  \|U_\sigma(t,s)x_0\|^p\;\le\; c  \mu^{\kappa_\sigma(t,s)}  e^{-\eta(t-s)}  \|x_0\|^p.
\end{equation*}
Taking the $p$-th root, this is strongly GES with $M=c^{1/p}$, $\rho=\mu^{1/p}\in(0,1)$, and $\alpha=\eta/p>0$ (both strict, since $\eta>0$ and $\mu\in(0,1)$).

\medskip\noindent\textbf{Proof that $(i)\Rightarrow(ii)$, via the explicit construction \eqref{eq:sup:functional}.} Assume
{strong GES} with constants $M, \alpha, \rho$ and pick $\eta\in(0,\alpha)$, $\nu\in(1,1/\rho)$.

\emph{Finiteness and upper bound.} For each $\tau\ge t$,
{strong GES} gives:
\begin{equation*}
  e^{\eta(\tau-t)}\nu^{\kappa_\sigma(\tau,t)}\|U_\sigma(\tau,t)x\|\;\le\; M  e^{(\eta-\alpha)(\tau-t)}(\nu\rho)^{\kappa_\sigma(\tau,t)}\|x\|.
\end{equation*}
Since $\eta<\alpha$ and $\nu\rho<1$, both factors $e^{(\eta-\alpha)(\tau-t)}$ and $(\nu\rho)^{\kappa_\sigma(\tau,t)}$ are bounded by $1$, so $V(t,x)\le M\|x\|$, giving the upper bound in \eqref{eq:coerc:bounds} with $c=M$, $p=1$.

\emph{Lower bound.} Taking $\tau=t$ in the sup: $V(t,x)\ge e^0\nu^0\|U_\sigma(t,t)x\|=\|x\|$.

\emph{Flow condition \eqref{eq:coerc:flow}.} For $h>0$ small enough that $(t,t+h)\cap\mathbb{T}_\sigma=\varnothing$, $U_\sigma(t+h,t)=S(h)$, $U_\sigma(\tau,t+h)S(h)=U_\sigma(\tau,t)$ for $\tau\ge t+h$, and $\kappa_\sigma(\tau,t+h)=\kappa_\sigma(\tau,t)$ for $\tau\ge t+h$. Hence:
\begin{equation*}
  V(t+h,S(h)x)=\sup_{\tau\ge t+h}e^{\eta(\tau-t-h)}\nu^{\kappa_\sigma(\tau,t)}\|U_\sigma(\tau,t)x\|=e^{-\eta h}\sup_{\tau\ge t+h}e^{\eta(\tau-t)}\nu^{\kappa_\sigma(\tau,t)}\|U_\sigma(\tau,t)x\|\;\le\;e^{-\eta h}V(t,x).
\end{equation*}
Hence $\underline{D}^+V(t,x)\le -\eta V(t,x)$.

\emph{Jump condition \eqref{eq:coerc:jump}.} At $t_k\in\mathbb{T}_\sigma$, using $U_\sigma(\tau,t_k)J=U_\sigma(\tau,t_k^-)$ and $\kappa_\sigma(\tau,t_k)=\kappa_\sigma(\tau,t_k^-)-1$ for $\tau\ge t_k$:
\begin{align*}
  V(t_k,Jx)&=\sup_{\tau\ge t_k}e^{\eta(\tau-t_k)}\nu^{\kappa_\sigma(\tau,t_k)}\|U_\sigma(\tau,t_k^-)x\|\\
  &=\frac{1}{\nu}\sup_{\tau\ge t_k}e^{\eta(\tau-t_k)}\nu^{\kappa_\sigma(\tau,t_k^-)}\|U_\sigma(\tau,t_k^-)x\|{\le}\frac{1}{\nu}V(t_k^-,x).
\end{align*}
This gives \eqref{eq:coerc:jump} with $\mu=1/\nu<1${, where $V(t_k^-,x)$ denotes the left limit of $V(\cdot,x)$ at $t_k$ and the last inequality is established below (Continuity in $t$)}.

\emph{Lipschitz continuity.} The sup of linear-in-$x$ functions is subadditive, and combined with the
{strong GES} upper bound, $V$ is Lipschitz on $X$ with constant $M$.

{\emph{Continuity in $t$.} Let $k\ge0$, $x\in X$, and $t,t+h\in[t_k,t_{k+1})$ with $h>0$. Since $\kappa_\sigma(\tau,t)=0$ and $U_\sigma(\tau,t)=S(\tau-t)$ for $\tau\in[t,t+h]$, splitting the supremum in \eqref{eq:sup:functional} at $\tau=t+h$ and using the identity obtained above when establishing \eqref{eq:coerc:flow}, which remains valid at $t=t_k$, gives
\begin{equation*}
  V(t,x)=\max\left(\sup_{s\in[0,h]}e^{\eta s}\|S(s)x\|,e^{\eta h}V(t+h,S(h)x)\right).
\end{equation*}
Hence, since $\|x\|\le V(t+h,x)$ and $V(t+h,\cdot)$ is bounded by $M\|\cdot\|$ and Lipschitz with constant $M$,
\begin{equation*}
\begin{aligned}
  V(t+h,x)-V(t,x)&\le V(t+h,x)-V(t+h,S(h)x)\le M\|x-S(h)x\|,\\
  V(t,x)-V(t+h,x)&\le\max(a_h,b_h),
\end{aligned}
\end{equation*}
where $a_h:=\sup_{s\in[0,h]}e^{\eta s}\|S(s)x\|-\|x\|$ and $b_h:=(e^{\eta h}-1)M\|S(h)x\|+M\|S(h)x-x\|$. The right-hand sides of these inequalities do not depend on $t$ and tend to $0$ as $h\downarrow0$ by strong continuity of $S(\cdot)$. Hence, $t\mapsto V(t,x)$ is uniformly continuous on $[t_k,t_{k+1})$; since a uniformly continuous function on an interval extends continuously to its closure, it admits a left limit $V(t_{k+1}^-,x)$ at $t_{k+1}$. Finally, for $t\in[t_k,t_{k+1})$ and $\tau\ge t_{k+1}$, we have $U_\sigma(\tau,t)=U_\sigma(\tau,t_{k+1})JS(t_{k+1}-t)$ and $\kappa_\sigma(\tau,t)=\kappa_\sigma(\tau,t_{k+1})+1$, so that the part of the supremum in \eqref{eq:sup:functional} over $\tau\ge t_{k+1}$ is equal to $e^{\eta(t_{k+1}-t)}\nu V(t_{k+1},JS(t_{k+1}-t)x)$. Therefore, $V(t,x)\ge\nu V(t_{k+1},JS(t_{k+1}-t)x)$, and letting $t\uparrow t_{k+1}$ yields, by strong continuity of $S(\cdot)$ and the Lipschitz property of $V(t_{k+1},\cdot)$,
\begin{equation*}
\begin{aligned}
  V(t_{k+1}^-,x)&\ge\nu V(t_{k+1},Jx)\\
  &=\sup_{\tau\ge t_{k+1}}e^{\eta(\tau-t_{k+1})}\nu^{\kappa_\sigma(\tau,t_{k+1}^-)}\|U_\sigma(\tau,t_{k+1}^-)x\|,
\end{aligned}
\end{equation*}
which is the inequality used above when establishing \eqref{eq:coerc:jump}.}
\end{proof}

The theorem treats the strongly-GES (hybrid) case $\eta>0$, $\mu\in(0,1)$; relaxing it to $\eta>0$, $\mu=1$ recovers a persistent-flowing coercive criterion (no jump contraction required) and to $\eta=0$, $\mu<1$ a persistent-jumping one (no flow decay required), covering the two degenerate regimes of Definition~\ref{def:stab}. The coercive lower bound in \eqref{eq:coerc:bounds} allows the stability estimate to be read directly from $V$, in contrast to the non-coercive case of Section~\ref{sec:lyapunov} where only an upper bound is imposed.


\subsection{Arbitrary impulse sequences ($\mathcal{S}_{0,\infty}$)}\label{subsec:coerc:arbitrary}

For arbitrary impulse sequences, the coercive functional is time-independent, mirroring the structure of Theorem~\ref{thm:arbitrary:nc}.

\begin{theorem}[Arbitrary sequences, coercive N\&S]\label{thm:coerc:arbitrary}
Let Assumption~\ref{ass:growth} hold (requiring $\omega_0\le0$ on $\mathcal{S}_{0,\infty}$, by Remark~\ref{rem:growth:G}). The following are equivalent:
\begin{enumerate}[(i)]
  \item the system \eqref{eq:syst} is strongly UGES over $\mathcal{S}_{0,\infty}$;
  \item there exist $p\ge1$, $c\ge1$, $\eta>0$, $\mu\in(0,1)$, and a time-independent functional $V:X\to\mathbb{R}_{\ge0}$, Lipschitz on bounded subsets of $X$, such that for all $x\in X$:
    \begin{align}
      \|x\|^p\le V(x)&\le c\|x\|^p,\label{eq:coerc:arb:bounds}\\
      \underline{D}^+V(x)&\le -\eta V(x),\label{eq:coerc:arb:flow}\\
      V(Jx)&\le \mu V(x);\label{eq:coerc:arb:jump}
    \end{align}
  \item the same time-independent functional $V$ of (ii) is simultaneously a Lipschitz coercive Lyapunov function for the semigroup and for the discrete iteration: there exist $p\ge1$, $c\ge1$, $\eta>0$, $\mu\in(0,1)$, and $V:X\to\mathbb{R}_{\ge0}$ Lipschitz on bounded subsets of $X$ satisfying
  \begin{equation}\label{eq:coerc:arb:integrated}
    \|x\|^p\le V(x)\le c\|x\|^p,\qquad V(S(t)x)\le e^{-\eta t}V(x)\;\forall t\ge0,\qquad V(Jx)\le\mu V(x).
  \end{equation}
\end{enumerate}
\end{theorem}

\begin{proof}
$(ii)\Rightarrow(i)$ follows from Theorem~\ref{thm:coerc:banach} applied with the time-independent $V${, which trivially satisfies the continuity requirements in (ii) of that theorem}: the coercive bounds and flow / jump decrease conditions (with $\eta>0$, $\mu\in(0,1)$) are uniform in $\sigma\in\mathcal{S}_{0,\infty}$ since $V$ does not depend on $\sigma$, and the conclusion is strongly UGES over $\mathcal{S}_{0,\infty}$.

$(i)\Rightarrow(ii)$ uses the time-independent \emph{joint} sup-construction
\begin{equation*}
  V(x):=\sup_{\sigma\in\mathcal{S}_{0,\infty}}\Psi(\sigma,x),\qquad
  \Psi(\sigma,x):=\sup_{s\ge0}e^{\eta s}\nu^{\kappa_\sigma(s,0)}\|U_\sigma(s,0)x\|,
\end{equation*}
where $\eta\in(0,\alpha_0)$ and $\nu\in(1,1/\rho_0)$, with $M\ge1$, $\alpha_0>0$, $\rho_0\in(0,1)$ the
{strong UGES constants in (i).} This is the fixed-sequence functional \eqref{eq:sup:functional} evaluated at $t=0$ and further supremized over all admissible sequences; in contrast to an additive combination of separate flow and jump suprema, it contracts under \emph{both} mechanisms, because every orbit it retains is a genuine trajectory of the family along which flow and jumps act jointly.

\emph{Bounds.}
{Strong UGES} gives $\|U_\sigma(s,0)x\|\le M\rho_0^{\kappa_\sigma(s,0)}e^{-\alpha_0 s}\|x\|$ uniformly in $\sigma$, hence $e^{\eta s}\nu^{\kappa_\sigma(s,0)}\|U_\sigma(s,0)x\|\le M(\nu\rho_0)^{\kappa_\sigma(s,0)}e^{-(\alpha_0-\eta)s}\|x\|\le M\|x\|$ since $\nu\rho_0<1$ and $\alpha_0-\eta>0$; thus $V(x)\le M\|x\|$. Taking $s=0$ (so $\kappa_\sigma(0,0)=0$ and $U_\sigma(0,0)=I$) gives $V(x)\ge\|x\|$. This is \eqref{eq:coerc:arb:bounds} with $p=1$, $c=M$.

{\emph{Lipschitz continuity.} For $\sigma\in\mathcal{S}_{0,\infty}$ and $s\ge0$, the function $x\mapsto e^{\eta s}\nu^{\kappa_\sigma(s,0)}\|U_\sigma(s,0)x\|$, whose supremum over $\sigma$ and $s$ is $V$, is Lipschitz on $X$ with constant $e^{\eta s}\nu^{\kappa_\sigma(s,0)}\|U_\sigma(s,0)\|$. By the strong UGES bound used above, all these constants are bounded by $M$. Since $V$ is finite, it is Lipschitz on $X$ with constant $M$.}

\emph{Flow condition.} Fix $x\in X$ and $h>0$. For any $\sigma\in\mathcal{S}_{0,\infty}$ with $t_1^\sigma>h$, no jump occurs on $[0,h]$, so the backward shift $\sigma^{(h)}$ ($t_k^{\sigma^{(h)}}:=t_k^\sigma-h$, admissible since $t_1^\sigma>h$) satisfies $U_{\sigma^{(h)}}(s',0)=U_\sigma(s'+h,h)$ and $\kappa_\sigma(s,0)=\kappa_{\sigma^{(h)}}(s-h,0)$ for $s\ge h$. The substitution $s'=s-h$ turns the portion of $\Psi(\sigma,x)$ with $s\ge h$ into $e^{\eta h}\Psi(\sigma^{(h)},S(h)x)$, so $\Psi(\sigma,x)\ge e^{\eta h}\Psi(\sigma^{(h)},S(h)x)$. Since $\sigma\mapsto\sigma^{(h)}$ maps $\{\sigma:t_1^\sigma>h\}$ onto all of $\mathcal{S}_{0,\infty}$, taking the supremum gives $V(x)\ge e^{\eta h}V(S(h)x)$, i.e.\ $V(S(h)x)\le e^{-\eta h}V(x)$. Dividing $V(S(h)x)-V(x)\le(e^{-\eta h}-1)V(x)$ by $h$ and letting $h\downarrow0$ yields $\underline{D}^+V(x)\le-\eta V(x)$, which is \eqref{eq:coerc:arb:flow}.

\emph{Jump condition.} Fix $x\in X$. {Since every sequence in $\mathcal{S}_{0,\infty}$ satisfies $t_1>t_0=0$, the prepended impulse is placed at a time $\varepsilon>0$, and $\varepsilon\downarrow0$ is taken at the end. For any $\sigma\in\mathcal{S}_{0,\infty}$ and $\varepsilon>0$, let $\widetilde\sigma_\varepsilon$ be defined by $t_1^{\widetilde\sigma_\varepsilon}:=\varepsilon$ and $t_{k+1}^{\widetilde\sigma_\varepsilon}:=t_k^\sigma+\varepsilon$ for $k\ge1$, so that $\widetilde\sigma_\varepsilon\in\mathcal{S}_{0,\infty}$. For $s'\ge0$, we have $U_{\widetilde\sigma_\varepsilon}(s'+\varepsilon,0)x=U_\sigma(s',0)JS(\varepsilon)x$ and $\kappa_{\widetilde\sigma_\varepsilon}(s'+\varepsilon,0)=\kappa_\sigma(s',0)+1$. Restricting the supremum in $\Psi(\widetilde\sigma_\varepsilon,x)$ to $s=s'+\varepsilon$ with $s'\ge0$ and using $e^{\eta\varepsilon}\ge1$, we get
\begin{equation*}
  V(x)\ge\Psi(\widetilde\sigma_\varepsilon,x)\ge e^{\eta\varepsilon}\nu\sup_{s'\ge0}e^{\eta s'}\nu^{\kappa_\sigma(s',0)}\|U_\sigma(s',0)JS(\varepsilon)x\|\ge\nu\Psi(\sigma,JS(\varepsilon)x).
\end{equation*}
Taking the supremum over $\sigma\in\mathcal{S}_{0,\infty}$ yields $V(x)\ge\nu V(JS(\varepsilon)x)$. Since $\|JS(\varepsilon)x-Jx\|\le\|J\|\|S(\varepsilon)x-x\|$, letting $\varepsilon\downarrow0$, the strong continuity of $S(\cdot)$ and the Lipschitz continuity of $V$, which is established above independently of the present step, give $V(x)\ge\nu V(Jx)$,} i.e.\ $V(Jx)\le(1/\nu)V(x)$, which is \eqref{eq:coerc:arb:jump} with $\mu=1/\nu\in(0,1)$.

$(ii)\Leftrightarrow(iii)$ is the integrated form of (ii). $(ii)\Rightarrow(iii)$: along the continuous trajectory $t\mapsto S(t)x$, set $g(t):=V(S(t)x)$. Then $g$ is continuous (since $V$ is Lipschitz on bounded subsets and $S(\cdot)x$ is continuous) and the Dini inequality \eqref{eq:coerc:arb:flow} reads $\underline D^+g(t)\le-\eta g(t)$ for all $t\ge0$. By the standard Dini comparison lemma for continuous functions \cite{Lakshmikantham:91}, $g(t)\le g(0)e^{-\eta t}$, which is the flow inequality of \eqref{eq:coerc:arb:integrated}. The jump inequality is identical to \eqref{eq:coerc:arb:jump}. $(iii)\Rightarrow(ii)$: from $V(S(h)x)\le e^{-\eta h}V(x)$, $V(S(h)x)-V(x)\le(e^{-\eta h}-1)V(x)$; dividing by $h$ and taking $\liminf$ as $h\to0^+$ gives $\underline D^+V(x)\le-\eta V(x)$, which is \eqref{eq:coerc:arb:flow}.
\end{proof}

\subsection{Constant dwell-time ($\mathcal{S}_{\mathrm{cst}}(T)$)}\label{subsec:coerc:cst}

For the constant dwell-time family, the coercive functional is timer-dependent on the interval $[0,T]$, mirroring the structure of Theorem~\ref{thm:cst:nc}.

\begin{theorem}[Constant dwell-time, coercive N\&S]\label{thm:coerc:cst}
Let $T>0$ and Assumption~\ref{ass:growth} hold. The following are equivalent:
\begin{enumerate}[(i)]
  \item the system \eqref{eq:syst} is strongly UGES over $\mathcal{S}_{\mathrm{cst}}(T)$;
  \item there exist $p\ge1$, $c\ge1$, $\eta\ge0$, $\mu\in(0,1]$ with $\eta+(-\log\mu)>0$, and a functional $\bar V:[0,T]\times X\to\mathbb{R}_{\ge0}$, Lipschitz on bounded subsets of $X$ uniformly in $\tau$, {with $\tau\mapsto\bar V(\tau,x)$ continuous on $[0,T]$ for every $x\in X$,} satisfying
    \begin{subequations}\label{eq:coerc:cst:hybrid}
    \begin{align}
      \|x\|^p\le\bar V(\tau,x)&\le c\|x\|^p,\label{eq:coerc:cst:bounds}\\
      \underline{D}^+\bar V(\tau,x)&\le -\eta  \bar V(\tau,x){\text{ for all }\tau\in[0,T)},\label{eq:coerc:cst:flow}\\
      \bar V(0,Jx)&\le \mu  \bar V(T,x);\label{eq:coerc:cst:jump}
    \end{align}
    \end{subequations}
  \item the monodromy $M_T=JS(T)$ (post-jump sampling, cf.\ Theorem~\ref{thm:cst:nc}) admits a coercive Lipschitz Lyapunov function: there exist $\hat c\ge1$, $\hat\mu\in(0,1)$, and $V_d:X\to\mathbb{R}_{\ge0}$ Lipschitz on bounded subsets such that $\|x\|^p\le V_d(x)\le\hat c\|x\|^p$ and $V_d(M_T x)\le\hat\mu V_d(x)$.
\end{enumerate}
\end{theorem}

\begin{proof}
$(ii)\Rightarrow(i)$: integrating the flow $\eta$-decay \eqref{eq:coerc:cst:flow} over each period $[t_k,t_k+T]${, which is justified as in the proof of $(iii)\Rightarrow(ii)$ of Theorem~\ref{thm:main}, using the continuity of $\bar V(\cdot,x)$ on $[0,T]$, the continuity of $x(\cdot)$ on $[t_k,t_{k+1})$, and the Lipschitz property of $\bar V$, the monotonicity argument being applied to the continuous function $t\mapsto e^{\eta(t-t_k)}\bar V(t-t_k,x(t))$,} and applying the jump $\mu$-contraction \eqref{eq:coerc:cst:jump} give $\bar V(0,x(t_{k+1}))\le\mu e^{-\eta T}\bar V(0,x(t_k))$, with $\mu e^{-\eta T}=e^{-(\eta T-\log\mu)}<1$ since $\eta+(-\log\mu)>0$; as the period $T$ is fixed, telescoping together with the coercive lower bound $\|x\|^p\le\bar V$ gives geometric decay in $t$, i.e.\ UGES over $\mathcal{S}_{\mathrm{cst}}(T)$ (the within-period bound following from the flow decay){, which is strong UGES by the remark following the proof of Theorem~\ref{thm:rng:nc}}.

$(i)\Rightarrow(ii)$ uses the explicit construction \eqref{eq:sup:functional} of Theorem~\ref{thm:coerc:banach} applied to $\sigma_T$ and projected onto the timer:
\begin{equation}\label{eq:coerc:cst:V:explicit}
  \bar V(\tau,x):=\sup_{s\ge0}e^{\eta s}\nu^{\kappa_{\sigma_T}(\tau+s,\tau)}\|U_{\sigma_T}(\tau+s,\tau)x\|,\qquad\tau\in{[0,T)},  x\in X.
\end{equation}
{On $[0,T)=[t_0,t_1)$, $\bar V$ coincides with the functional \eqref{eq:sup:functional} associated with $\sigma=\sigma_T$. As in Theorem~\ref{thm:cst:nc}, the endpoint value is defined as the pre-jump left limit $\bar V(T,x):=\lim_{\tau\uparrow T}\bar V(\tau,x)$, which exists and satisfies $\bar V(T,x)\ge\nu\bar V(0,Jx)$ by the continuity in $t$ established in the proof of Theorem~\ref{thm:coerc:banach}, since $\bar V(0,Jx)=V(T,Jx)$ by periodicity of $\sigma_T$. Hence, $\tau\mapsto\bar V(\tau,x)$ is continuous on $[0,T]$ for every $x\in X$, and the coercive bounds and the Lipschitz property extend to $\tau=T$ by passing to the limit.} The coercive bounds, flow $\eta$-decay, and jump $\mu$-contraction follow from the proof of Theorem~\ref{thm:coerc:banach} applied with starting timer $\tau$ in place of starting time $t$, giving constants $c=M$, $\eta\in(0,\alpha)$, $\mu=1/\nu\in(0,1)$ for any $\nu\in(1,1/\rho)$
{(by (i), the UGES constants satisfy $\alpha>0$ and $\rho\in(0,1)$)}.

$(ii)\Leftrightarrow(iii)$. From continuous time to discrete time: set $V_d(x):=\bar V(0,x)$. Chaining \eqref{eq:coerc:cst:flow} on $[0,T]$ along the unforced trajectory $r\mapsto S(r)x${, which is justified as in the proof of $(iii)\Rightarrow(ii)$ of Theorem~\ref{thm:main}, using the continuity of $\bar V(\cdot,x)$ on $[0,T]$, the continuity of $r\mapsto S(r)x$, and the Lipschitz property of $\bar V$, the monotonicity argument being applied to the continuous function $r\mapsto e^{\eta r}\bar V(r,S(r)x)$ on $[0,T]$,} and applying \eqref{eq:coerc:cst:jump} at $\tau=T$:
\begin{equation*}
  V_d(M_T x)=\bar V(0,JS(T)x)\le \mu  \bar V(T,S(T)x)\le\mu  e^{-\eta T}\bar V(0,x)=\mu e^{-\eta T}V_d(x),
\end{equation*}
which is the coercive discrete contraction with $\hat\mu:=\mu e^{-\eta T}\in(0,1)$ and $\hat c:=c$. Conversely, suppose $V_d$ satisfies $\|x\|\le V_d(x)\le\hat c\|x\|$ and $V_d(M_T x)\le\hat\mu V_d(x)$ with $\hat\mu\in(0,1)$. Iterating the contraction gives $\|M_T^k x\|\le V_d(M_T^k x)\le\hat c\,\hat\mu^k\|x\|$, i.e.\ geometric decay of the sampled monodromy sequence. Choose $\nu\in(1,1/\hat\mu)$ and $\eta\in\bigl(0,-\tfrac1T\log(\nu\hat\mu)\bigr)$, so that $e^{\eta T}\nu\hat\mu<1$. The functional witnessing (ii) is then the \emph{forward} construction \eqref{eq:coerc:cst:V:explicit}: its coercive lower bound (the $s=0$ term), flow $\eta$-decay, and jump $\mu$-contraction with $\mu=1/\nu$ hold structurally, exactly as in $(i)\Rightarrow(ii)$ above{, as does its continuity in $\tau$ on $[0,T]$ (with the constant $M$ replaced by $c$ in the continuity argument of Theorem~\ref{thm:coerc:banach})}; and its finiteness with upper bound $c\|x\|$ follows from the geometric decay, since on the $k$-th future period $\|U_{\sigma_T}(\tau+s,\tau)x\|\le C\hat\mu^{k}\|x\|$ (with $C$ depending only on $G$, $\|J\|$, $\hat c$, by the growth bound of Assumption~\ref{ass:growth} with $\bar T=T<\infty$), so the supremum terms are dominated by $C(e^{\eta T}\nu\hat\mu)^{k}\|x\|$, a convergent geometric series.
\end{proof}

\subsection{Minimum dwell-time ($\mathcal{S}_{\min}(T_{\min})$)}\label{subsec:coerc:min}

For the minimum dwell-time family, the coercive functional is timer-dependent on the mandatory phase $[0,T_{\min}]$ and frozen on the free phase beyond $T_{\min}$, mirroring the structure of Theorem~\ref{thm:min:nc}.

\begin{theorem}[Minimum dwell-time, coercive N\&S]\label{thm:coerc:min}
Let $T_{\min}>0$ and Assumption~\ref{ass:growth} hold. The following are equivalent:
\begin{enumerate}[(i)]
  \item the system \eqref{eq:syst} is strongly UGES over $\mathcal{S}_{\min}(T_{\min})$;
  \item there exist $p\ge1$, $c\ge1$, $\eta>0$, $\mu\in(0,1)$, and a functional $\bar V:[0,T_{\min}]\times X\to\mathbb{R}_{\ge0}$, Lipschitz on bounded subsets uniformly in $\tau$, {with $\tau\mapsto\bar V(\tau,x)$ continuous on $[0,T_{\min}]$ for every $x\in X$,} satisfying \eqref{eq:coerc:cst:bounds}, and
    \begin{subequations}\label{eq:coerc:min:hybrid}
    \begin{align}
    \underline{D}^+\bar V(\tau,x)&\le-\eta\bar V(\tau,x),\; \forall\tau\in[0,T_{\min}]\\
    \underline{D}^+\bar V(T_{\min},x)&\le-\eta\bar V(T_{\min},x)\\
    \bar V(0,Jx)&\le\mu\bar V(T_{\min},x)
  \end{align}
  \end{subequations}
  (these conditions imply, in particular, that the $C_0$-semigroup $S(t)$ generated by $A$ is exponentially stable)
  \item the $C_0$-semigroup $S(t)$ is exponentially stable and the family $\mathcal{M}=\{JS(\theta):\theta\ge T_{\min}\}$ admits a coercive Lipschitz Lyapunov function uniformly: there exist $\hat c\ge1$, $\hat\mu\in(0,1)$, and $V_d:X\to\mathbb{R}_{\ge0}$ such that $\|x\|^p\le V_d(x)\le\hat c\|x\|^p$ and $V_d(Mx)\le\hat\mu V_d(x)$ for all $M\in\mathcal{M}$.
\end{enumerate}
\end{theorem}

\begin{proof}
$(ii)\Rightarrow(i)$. The hybrid coercive conditions translate to
{strong UGES} on $\mathcal{S}_{\min}(T_{\min})$ by the same telescoping argument as in the proof of $(iii)\Rightarrow(ii)$ in Theorem~\ref{thm:min:nc}, but with the coercive lower bound $\|x\|^p\le \bar V(\tau,x)$ now giving a direct exponential decay reading: integrating $\eta$-decay on each inter-jump interval $[t_k,t_k+\theta_k]$ with $\theta_k\ge T_{\min}${, with $\bar V$ extended by $\bar V(T_{\min},\cdot)$ to timer values larger than $T_{\min}$ as in the proof of Theorem~\ref{thm:min:nc} (this integration is justified as in the proof of $(iii)\Rightarrow(ii)$ of Theorem~\ref{thm:main}, using the continuity of $\bar V(\cdot,x)$ on $[0,T_{\min}]$, the continuity of $x(\cdot)$ on $[t_k,t_{k+1})$, and the Lipschitz property of $\bar V$, the monotonicity argument being applied to the continuous function $t\mapsto e^{\eta(t-t_k)}\bar V(t-t_k,x(t))$),} and applying the jump $\mu$-contraction gives $\bar V(0,x(t_{k+1}))\le \mu e^{-\eta\theta_k}\bar V(0,x(t_k))\le \mu e^{-\eta T_{\min}}\bar V(0,x(t_k))$, hence $\|x(t_{k+1})\|^p\le c  (\mu e^{-\eta T_{\min}})^{k+1}\|x_0\|^p$. The free-phase decay ensures the analogous bound between jumps. {More precisely, since $\eta>0$ and $\mu\in(0,1)$, keeping the factors $e^{-\eta\theta_k}$ in the telescoping, together with the flow decay on the current inter-jump interval and the coercive bounds, gives $\|x(t)\|^p\le c\mu^{\kappa_\sigma(t,s)}e^{-\eta(t-s)}\|x(s)\|^p$ for all $t\ge s\ge0$, which is strong UGES with $\alpha=\eta/p$ and $\rho=\mu^{1/p}$.}

$(i)\Rightarrow(ii)$. The explicit construction is
\begin{equation}\label{eq:coerc:min:V:explicit}
  \bar V(\tau,x):=\sup_{\sigma\in\mathcal{S}^{(\tau)}}\sup_{s\ge0}e^{\eta s}\nu^{\kappa_\sigma(s,0)}\|U_\sigma(s,0)x\|,
\end{equation}
with $\mathcal{S}^{(\tau)}$ as in the proof of Theorem~\ref{thm:min:nc}. Coercive bounds follow from the uniform-in-$s$ UGES bound applied to the $\mathcal{S}_{\min}(T_{\min})$ signals of which the elements of $\mathcal{S}^{(\tau)}$ are futures, exactly as in the upper-bound step of Theorem~\ref{thm:min:nc}. The mandatory-phase flow decay, free-phase flow decay, and jump contraction follow from the same shift / prepend arguments as in the proof of $(ii)\Rightarrow(iii)$ in Theorem~\ref{thm:min:nc}, transposed to the multiplicative (coercive) form via the calculations of Theorem~\ref{thm:coerc:banach}{, with the convention on $\kappa_\sigma$ for the sequences with $t_1^\sigma=0$ stated after \eqref{eq:min:family:tau}, which provides the factor $\nu$ of the prepended impulse}. The necessary exponential stability of $S(t)$ follows from the persistent-flowing sub-case (any $\sigma\in\mathcal{S}_{\min}(T_{\min})$ with $t_1\to\infty$ exhibits arbitrarily long flow-only intervals).

{\emph{Continuity in $\tau$.} Let $x\in X$ and $0\le\tau<\tau'\le T_{\min}$, and set $h:=\tau'-\tau$. Since $\mathcal{S}^{(\tau)}\subseteq\mathcal{S}^{(\tau')}$, we have $\bar V(\tau,x)\le\bar V(\tau',x)$, and the shift argument establishing the mandatory-phase flow decay gives $\bar V(\tau',S(h)x)\le e^{-\eta h}\bar V(\tau,x)\le\bar V(\tau,x)$. Moreover, $\bar V(\tau',\cdot)$ is a supremum of seminorms bounded by $c\|\cdot\|$, and is therefore Lipschitz on $X$ with constant $c$. Hence,
\begin{equation*}
  0\le\bar V(\tau',x)-\bar V(\tau,x)\le\bar V(\tau',x)-\bar V(\tau',S(h)x)\le c\|x-S(h)x\|,
\end{equation*}
where the right-hand side does not depend on $\tau$ and tends to $0$ as $h\downarrow0$ by strong continuity of $S(\cdot)$. Therefore, as in the proof of Theorem~\ref{thm:min:nc}, $\tau\mapsto\bar V(\tau,x)$ is continuous on $[0,T_{\min}]$.}

$(ii)\Leftrightarrow(iii)$. From continuous time to discrete time: $V_d(x):=\bar V(0,x)$. For any $\theta\ge T_{\min}$, chain the mandatory flow on $[0,T_{\min}]$, the free flow on $[T_{\min},\theta]$, and the jump at $\tau=T_{\min}$ (after which the timer resets){, the integration of the flow conditions along $r\mapsto S(r)x$ being justified as in the proof of $(iii)\Rightarrow(ii)$ of Theorem~\ref{thm:main}, using the continuity of $\bar V(\cdot,x)$ on $[0,T_{\min}]$, the continuity of $r\mapsto S(r)x$, and the Lipschitz property of $\bar V$, the monotonicity argument being applied to the continuous functions $r\mapsto e^{\eta r}\bar V(r,S(r)x)$ on $[0,T_{\min}]$ and $r\mapsto e^{\eta r}\bar V(T_{\min},S(r)x)$ on $[T_{\min},\theta]$}:
\begin{equation*}
  V_d(JS(\theta)x)=\bar V(0,JS(\theta)x)\le \mu  \bar V(T_{\min},S(\theta)x)\le\mu e^{-\eta(\theta-T_{\min})}\bar V(T_{\min},S(T_{\min})x)\le\mu e^{-\eta\theta}\bar V(0,x).
\end{equation*}
Hence $V_d(Mx)\le\hat\mu V_d(x)$ uniformly over $M=JS(\theta)\in\mathcal{M}$ with $\hat\mu:=\mu e^{-\eta T_{\min}}<1$; together with the exponential stability of $S(t)$ implied by (ii) (the free-phase decay, cf.\ the statement), this establishes (iii).

Conversely, assume (iii): $\|S(t)\|\le M_0 e^{-\alpha_0 t}$ for some $M_0\ge1$, $\alpha_0>0$, and $V_d(Mx)\le\hat\mu V_d(x)$ uniformly over $\mathcal{M}$ with $\|x\|\le V_d(x)\le\hat c\|x\|$. Two propagator bounds hold along any $\sigma\in\mathcal{S}_{\min}(T_{\min})$, uniformly in the starting time $s\ge0$. \emph{First (jump-geometric).} From post-jump to post-jump, $U_\sigma$ is a composition of maps $JS(\theta_i)\in\mathcal{M}$ ($\theta_i\ge T_{\min}$), each contracting $V_d$ by $\hat\mu$. If $s$ is a regular (post-jump) instant this gives $\|U_\sigma(t,s)x\|\le M_0\hat c\,\hat\mu^{\kappa_\sigma(t,s)}\|x\|$. If instead $s$ lies in the interior of an inter-jump interval, the first arc $U_\sigma(t_{\rm first},s)=JS(t_{\rm first}-s)$ has length $t_{\rm first}-s$ possibly below $T_{\min}$, so it is \emph{not} an $\mathcal{M}$-contraction and is only bounded, $\|U_\sigma(t_{\rm first},s)x\|\le M_0\|J\|\,\|x\|$; the remaining $\kappa_\sigma(t,s)-1$ jumps are genuine $\mathcal{M}$-maps, so $\|U_\sigma(t,s)x\|\le (M_0^2\|J\|\hat c/\hat\mu)\,\hat\mu^{\kappa_\sigma(t,s)}\|x\|$. In either case $\|U_\sigma(t,s)x\|\le C_1\hat\mu^{\kappa_\sigma(t,s)}\|x\|$ with $C_1:=\max(M_0\hat c,\,M_0^2\|J\|\hat c/\hat\mu)$. \emph{Second (time-exponential).} Factorising $U_\sigma(t,s)$ into its arcs and jumps and using $\|S(\cdot)\|\le M_0e^{-\alpha_0\cdot}$ and $\|J\|$ gives $\|U_\sigma(t,s)x\|\le M_0(M_0\|J\|)^{\kappa_\sigma(t,s)}e^{-\alpha_0(t-s)}\|x\|$. Interpolating these two bounds as in Step~4 of Lemma~\ref{lem:datko}, the first to the power $\theta$, the second to $1-\theta$, with $\theta\in(0,1)$ close enough to $1$ that $\rho:=\hat\mu^{\theta}(M_0\|J\|)^{1-\theta}<1$, yields $\|U_\sigma(t,s)x\|\le M\rho^{\kappa_\sigma(t,s)}e^{-\alpha(t-s)}\|x\|$ with $\rho\in(0,1)$ and $\alpha:=\alpha_0(1-\theta)>0$, uniformly over $\mathcal{S}_{\min}(T_{\min})$; this is {strong} UGES, i.e.\ (i). The functional witnessing (ii) is then the forward construction \eqref{eq:coerc:min:V:explicit}, which is finite and satisfies the coercive, flow, and jump conditions by the already-established implication $(i)\Rightarrow(ii)$. The additional hypothesis that $S(t)$ be exponentially stable is genuinely needed here and cannot be dropped: taking $A=\mathrm{diag}(1,-1)$ (so $S(t)=\mathrm{diag}(e^{t},e^{-t})$, unstable) and $J=\mathrm{diag}(0,1)$ gives $JS(\theta)=\mathrm{diag}(0,e^{-\theta})$, for which $V_d(x)=\|x\|$ satisfies $V_d(JS(\theta)x)\le e^{-T_{\min}}V_d(x)$; yet a sequence with $t_1\to\infty$ makes $\|S(t)x\|$ diverge, so the system is not UGES.
\end{proof}

\subsection{Range dwell-time ($\mathcal{S}_{\mathrm{rng}}(T_{\min},T_{\max})$)}\label{subsec:coerc:rng}

For the range dwell-time family, the coercive functional is timer-dependent on $[0,T_{\max}]$ with the jump criterion imposed over all admissible periods, mirroring the structure of Theorem~\ref{thm:rng:nc}.

\begin{theorem}[Range dwell-time, coercive N\&S]\label{thm:coerc:rng}
Let $0<T_{\min}\le T_{\max}<\infty$ and Assumption~\ref{ass:growth} hold. The following are equivalent:
\begin{enumerate}[(i)]
  \item the system \eqref{eq:syst} is strongly UGES over $\mathcal{S}_{\mathrm{rng}}(T_{\min},T_{\max})$;
  \item there exist $p\ge1$, $c\ge1$, $\eta\ge0$, $\mu\in(0,1]$ with $\eta+(-\log\mu)>0$, and a functional $\bar V:[0,T_{\max}]\times X\to\mathbb{R}_{\ge0}$, Lipschitz on bounded subsets uniformly in $\tau$, {with $\tau\mapsto\bar V(\tau,x)$ continuous on $[0,T_{\max}]$ for every $x\in X$,} satisfying
    \begin{subequations}\label{eq:coerc:rng:hybrid}
    \begin{align}
      \|x\|^p\le\bar V(\tau,x)&\le c\|x\|^p,\\
      \underline{D}^+\bar V(\tau,x)&\le -\eta  \bar V(\tau,x),\;\forall\tau\in
{[0,T_{\max})},\\
      \bar V(0,Jx)&\le\mu  \bar V(\theta,x),\;\forall\theta\in[T_{\min},T_{\max}];
    \end{align}
    \end{subequations}
  \item the family $\mathcal{M}=\{JS(\theta):\theta\in[T_{\min},T_{\max}]\}$ admits a coercive Lipschitz Lyapunov function uniformly.
\end{enumerate}
\end{theorem}

\begin{proof}
$(ii)\Rightarrow(i)$. The hybrid coercive conditions translate to UGES on $\mathcal{S}_{\mathrm{rng}}(T_{\min},T_{\max})$ by the same telescoping argument as in $(iii)\Rightarrow(ii)$ of Theorem~\ref{thm:rng:nc}: for $\sigma\in\mathcal{S}_{\mathrm{rng}}(T_{\min},T_{\max})$ with inter-jump lengths $\theta_k\in[T_{\min},T_{\max}]$, integrating $\eta$-decay on $[0,\theta_k]$ in the timer{, which is justified as in the proof of $(iii)\Rightarrow(ii)$ of Theorem~\ref{thm:main}, using the continuity of $\bar V(\cdot,x)$ on $[0,T_{\max}]$, the continuity of $x(\cdot)$ on $[t_k,t_{k+1})$, and the Lipschitz property of $\bar V$, the monotonicity argument being applied to the continuous function $t\mapsto e^{\eta(t-t_k)}\bar V(t-t_k,x(t))$,} and applying the jump $\mu$-contraction at $\tau=\theta_k$ give $\bar V(0,x(t_{k+1}))\le\mu e^{-\eta\theta_k}\bar V(0,x(t_k))\le\mu e^{-\eta T_{\min}}\bar V(0,x(t_k))$, with $\mu e^{-\eta T_{\min}}=e^{-(\eta T_{\min}-\log\mu)}<1$ since $\eta+(-\log\mu)>0$. {The resulting UGES is strong UGES by the remark following the proof of Theorem~\ref{thm:rng:nc}.}

$(i)\Rightarrow(ii)$. The explicit construction is
\begin{equation}\label{eq:coerc:rng:V:explicit}
  \bar V(\tau,x):=\sup_{\sigma\in\mathcal{S}^{(\tau)}_{\mathrm{rng}}}\sup_{s\ge0}e^{\eta s}\nu^{\kappa_\sigma(s,0)}\|U_\sigma(s,0)x\|,\qquad\tau\in{[0,T_{\max})},
\end{equation}
with $\mathcal{S}^{(\tau)}_{\mathrm{rng}}$ as in the proof of Theorem~\ref{thm:rng:nc}{, and with the endpoint value defined as the pre-jump left limit
\begin{equation}\label{eq:coerc:rng:Vbar:Tmax}
  \bar V(T_{\max},x):=\lim_{\tau\uparrow T_{\max}}\bar V(\tau,x),\qquad x\in X,
\end{equation}
whose existence is established below (Continuity in $\tau$). The formula \eqref{eq:coerc:rng:V:explicit} is not used at $\tau=T_{\max}$, since $\mathcal{S}^{(T_{\max})}_{\mathrm{rng}}$ only contains sequences with $t_1^\sigma=0$, for which the supremum in \eqref{eq:coerc:rng:V:explicit} has no pre-jump term and may fail to dominate $\|x\|$}. Coercive bounds, flow $\eta$-decay, and the jump $\mu$-contraction for every $\theta\in[T_{\min},T_{\max}]$ follow from the shift / prepend arguments of $(ii)\Rightarrow(iii)$ in Theorem~\ref{thm:rng:nc}, transposed to the multiplicative coercive form via Theorem~\ref{thm:coerc:banach}
{(by (i), the UGES constants satisfy $\alpha>0$ and $\rho\in(0,1)$), with the convention on $\kappa_\sigma$ for the sequences with $t_1^\sigma=0$ stated after \eqref{eq:min:family:tau}, which provides the factor $\nu$ of the prepended impulse. These arguments apply for $\tau\in[0,T_{\max})$ and $\theta\in[T_{\min},T_{\max})$, and the case $\theta=T_{\max}$ of the jump contraction is obtained below from \eqref{eq:coerc:rng:Vbar:Tmax}. We denote by $c$ the constant of the coercive upper bound, so that $\bar V(\tau,y)\le c\|y\|$ for all $\tau\in[0,T_{\max})$ and $y\in X$}.

{\emph{Lipschitz continuity and lower bound on $[0,T_{\max})$.} For $\sigma\in\mathcal{S}_{0,\infty}$ and $y\in X$, set $\Psi_\sigma(y):=\sup_{s\ge0}e^{\eta s}\nu^{\kappa_\sigma(s,0)}\|U_\sigma(s,0)y\|$, so that $\bar V(\tau,y)=\sup_{\sigma\in\mathcal{S}^{(\tau)}_{\mathrm{rng}}}\Psi_\sigma(y)$ for $\tau\in[0,T_{\max})$. For $\sigma\in\mathcal{S}^{(\tau)}_{\mathrm{rng}}$, $\Psi_\sigma$ is a supremum of seminorms and satisfies $\Psi_\sigma\le\bar V(\tau,\cdot)\le c\|\cdot\|$; it is therefore a seminorm, Lipschitz on $X$ with constant $c$, and so is $\bar V(\tau,\cdot)$. Moreover, for $\tau\in[0,T_{\max})$, $\mathcal{S}^{(\tau)}_{\mathrm{rng}}$ contains sequences with $t_1^\sigma=T_{\max}-\tau>0$, whose term $s=0$ is equal to $\|x\|$, so that $\bar V(\tau,x)\ge\|x\|$.}

{\emph{Continuity in $\tau$.} Let $x\in X$, $0\le\tau<\tau'<T_{\max}$, and $h:=\tau'-\tau$. Set
\begin{equation*}
  d_h:=\sup_{s\in[0,h]}\|S(s)x-x\|,\qquad a_h:=\sup_{s\in[0,h]}e^{\eta s}\|S(s)x\|-\|x\|,
\end{equation*}
which do not depend on $\tau$, are nonnegative, and tend to $0$ as $h\downarrow0$ by strong continuity of $S(\cdot)$. On the one hand, the shift argument establishing the flow decay gives $\bar V(\tau',S(h)x)\le e^{-\eta h}\bar V(\tau,x)\le\bar V(\tau,x)$, and therefore
\begin{equation*}
  \bar V(\tau',x)-\bar V(\tau,x)\le\bar V(\tau',x)-\bar V(\tau',S(h)x)\le c\,d_h.
\end{equation*}
On the other hand, let $\sigma\in\mathcal{S}^{(\tau)}_{\mathrm{rng}}$. If $t_1^\sigma>h$, the sequence $\widetilde\sigma$ with $t_k^{\widetilde\sigma}=t_k^\sigma-h$, $k\ge1$, belongs to $\mathcal{S}^{(\tau')}_{\mathrm{rng}}$, and splitting the supremum in $\Psi_\sigma(x)$ at $s=h$ gives
\begin{equation*}
  \Psi_\sigma(x)=\max\left(\sup_{s\in[0,h]}e^{\eta s}\|S(s)x\|,e^{\eta h}\Psi_{\widetilde\sigma}(S(h)x)\right).
\end{equation*}
Since $\|x\|\le\bar V(\tau',x)\le c\|x\|$ and $\Psi_{\widetilde\sigma}(S(h)x)\le\Psi_{\widetilde\sigma}(x)+c\,d_h\le\bar V(\tau',x)+c\,d_h$, we obtain
\begin{equation*}
  \Psi_\sigma(x)\le\bar V(\tau',x)+\max\left(a_h,(e^{\eta h}-1)c\|x\|+e^{\eta h}c\,d_h\right).
\end{equation*}
If $t_1^\sigma\le h$, then $T_{\min}-\tau\le t_1^\sigma\le h$, that is, $\tau'\ge T_{\min}$. Let $\sigma_+\in\mathcal{S}_{\mathrm{rng}}(T_{\min},T_{\max})$ be the sequence with impulse instants $t_{k+1}^\sigma-t_1^\sigma$, $k\ge1$, and let $\widehat\sigma$ be the sequence consisting of an impulse at $t_1^{\widehat\sigma}=0$ followed by the impulse instants of $\sigma_+$. Since $\tau'\in[T_{\min},T_{\max})$, $\widehat\sigma\in\mathcal{S}^{(\tau')}_{\mathrm{rng}}$ by the convention of Theorem~\ref{thm:rng:nc}, and the convention on $\kappa_\sigma$ stated after \eqref{eq:min:family:tau} gives $\Psi_{\widehat\sigma}(x)=\nu\Psi_{\sigma_+}(Jx)$. Moreover, splitting the supremum in $\Psi_\sigma(x)$ at $s=t_1^\sigma$ gives
\begin{equation*}
  \Psi_\sigma(x)=\max\left(\sup_{s\in[0,t_1^\sigma)}e^{\eta s}\|S(s)x\|,e^{\eta t_1^\sigma}\nu\Psi_{\sigma_+}(JS(t_1^\sigma)x)\right),
\end{equation*}
where the first term is absent when $t_1^\sigma=0$. Since $\Psi_{\sigma_+}$ is Lipschitz with constant $c$, we have $\Psi_{\sigma_+}(JS(t_1^\sigma)x)\le\Psi_{\sigma_+}(Jx)+c\|J\|d_h$, and therefore
\begin{equation*}
  e^{\eta t_1^\sigma}\nu\Psi_{\sigma_+}(JS(t_1^\sigma)x)\le e^{\eta h}\bar V(\tau',x)+e^{\eta h}\nu c\|J\|d_h\le\bar V(\tau',x)+(e^{\eta h}-1)c\|x\|+e^{\eta h}\nu c\|J\|d_h.
\end{equation*}
In both cases, $\Psi_\sigma(x)\le\bar V(\tau',x)+\omega_x(h)$, where
\begin{equation*}
  \omega_x(h):=a_h+(e^{\eta h}-1)c\|x\|+e^{\eta h}c(1+\nu\|J\|)d_h
\end{equation*}
does not depend on $\sigma$ and $\tau$, and tends to $0$ as $h\downarrow0$. Taking the supremum over $\sigma\in\mathcal{S}^{(\tau)}_{\mathrm{rng}}$ and combining with the first bound, which is also bounded by $\omega_x(h)$, yields $|\bar V(\tau',x)-\bar V(\tau,x)|\le\omega_x(h)$. Hence, $\tau\mapsto\bar V(\tau,x)$ is uniformly continuous on $[0,T_{\max})$, the limit in \eqref{eq:coerc:rng:Vbar:Tmax} exists, and $\tau\mapsto\bar V(\tau,x)$ is continuous on $[0,T_{\max}]$. Passing to the limit $\tau\uparrow T_{\max}$, the coercive bounds $\|x\|\le\bar V(T_{\max},x)\le c\|x\|$ and the Lipschitz property with constant $c$ extend to $\tau=T_{\max}$.}

{\emph{Jump contraction at $\theta=T_{\max}$.} Let $\tau\in[0,T_{\max})$ and $\sigma_+\in\mathcal{S}^{(0)}_{\mathrm{rng}}=\mathcal{S}_{\mathrm{rng}}(T_{\min},T_{\max})$, and let $\sigma$ be the sequence with $t_1^\sigma:=T_{\max}-\tau>0$ and $t_{k+1}^\sigma:=t_1^\sigma+t_k^{\sigma_+}$ for $k\ge1$, so that $\sigma\in\mathcal{S}^{(\tau)}_{\mathrm{rng}}$. For $s\ge t_1^\sigma$, we have $U_\sigma(s,0)=U_{\sigma_+}(s-t_1^\sigma,0)JS(t_1^\sigma)$ and $\kappa_\sigma(s,0)=1+\kappa_{\sigma_+}(s-t_1^\sigma,0)$, so that restricting the supremum in $\Psi_\sigma(x)$ to $s\ge t_1^\sigma$ gives $\bar V(\tau,x)\ge\Psi_\sigma(x)\ge e^{\eta t_1^\sigma}\nu\Psi_{\sigma_+}(JS(t_1^\sigma)x)\ge\nu\Psi_{\sigma_+}(JS(T_{\max}-\tau)x)$. Taking the supremum over $\sigma_+$ yields $\bar V(\tau,x)\ge\nu\bar V(0,JS(T_{\max}-\tau)x)$. Letting $\tau\uparrow T_{\max}$, the strong continuity of $S(\cdot)$ and the Lipschitz property of $\bar V(0,\cdot)$ give
\begin{equation*}
  \bar V(T_{\max},x)\ge\max\left(\|x\|,\nu\bar V(0,Jx)\right),
\end{equation*}
which is the jump contraction $\bar V(0,Jx)\le\mu\bar V(T_{\max},x)$ with $\mu=1/\nu$.}

$(ii)\Leftrightarrow(iii)$. From continuous time to discrete time: $V_d(x):=\bar V(0,x)$ satisfies $V_d(JS(\theta)x)\le\mu e^{-\eta\theta}V_d(x)\le\mu e^{-\eta T_{\min}}V_d(x)$ for every $\theta\in[T_{\min},T_{\max}]$, by chaining flow on $[0,\theta]$ and jump at $\tau=\theta${, the integration of the flow condition along $r\mapsto S(r)x$ being justified as in the proof of $(iii)\Rightarrow(ii)$ of Theorem~\ref{thm:main}, using the continuity of $\bar V(\cdot,x)$ on $[0,T_{\max}]$, the continuity of $r\mapsto S(r)x$, and the Lipschitz property of $\bar V$, the monotonicity argument being applied to the continuous function $r\mapsto e^{\eta r}\bar V(r,S(r)x)$ on $[0,\theta]$}; write $\hat\mu:=\mu e^{-\eta T_{\min}}\in(0,1)$. Conversely, suppose the uniform contraction $V_d(Mx)\le\hat\mu V_d(x)$ holds over all $M\in\mathcal{M}=\{JS(\theta):\theta\in[T_{\min},T_{\max}]\}$, with $\|x\|\le V_d(x)\le\hat c\|x\|$. Any composition of $k$ maps from $\mathcal{M}$ contracts $V_d$ by $\hat\mu^k$, so for every $\sigma\in\mathcal{S}_{\mathrm{rng}}(T_{\min},T_{\max})$ the post-jump states satisfy $\|U_\sigma(t_k,0)x\|\le\hat c\,\hat\mu^k\|x\|$, i.e.\ geometric decay in the jump count $k$. Choose $\nu\in(1,1/\hat\mu)$ and $\eta\in\bigl(0,-\tfrac{1}{T_{\max}}\log(\nu\hat\mu)\bigr)$, so that $e^{\eta T_{\max}}\nu\hat\mu<1$. The functional witnessing (ii) is then the forward construction \eqref{eq:coerc:rng:V:explicit}{, extended to $\tau=T_{\max}$ by \eqref{eq:coerc:rng:Vbar:Tmax}}: its coercive lower bound (the $s=0$ term), flow $\eta$-decay, and jump $\mu$-contraction with $\mu=1/\nu$ hold structurally, exactly as in $(i)\Rightarrow(ii)${, as do the existence of the limit in \eqref{eq:coerc:rng:Vbar:Tmax}, the continuity of $\tau\mapsto\bar V(\tau,x)$ on $[0,T_{\max}]$, and the jump contraction at $\theta=T_{\max}$, since the corresponding arguments above only use the coercive upper bound $\bar V(\tau,\cdot)\le c\|\cdot\|$ on $[0,T_{\max})$, which is established below}; and its finiteness with upper bound $c\|x\|$ follows from the geometric decay, since a state $s$ ahead lying in the $k$-th future inter-jump interval obeys $\|U_\sigma(\tau+s,\tau)x\|\le G\,\hat c\,\hat\mu^{k}\|x\|$ (the growth bound of Assumption~\ref{ass:growth}, finite as $\bar T=T_{\max}<\infty$, applied to the mid-interval arc of length $\le T_{\max}$), whence the supremand is dominated by $G\hat c\,(e^{\eta T_{\max}}\nu\hat\mu)^{k}\|x\|$, a convergent geometric series.
\end{proof}

\section{Stability conditions on Hilbert spaces}
\label{sec:hilbert}

When the state space is a Hilbert space $H$ (with inner product $\langle\cdot,\cdot\rangle_H$ and induced norm $\|\cdot\|_H$), the Lyapunov functional takes the explicit quadratic form $V(t,x)=\langle x,P(t)x\rangle_H$, reducing the abstract conditions of the previous sections to operator inequalities. We denote by $L(H)$ the algebra of bounded linear operators on $H$, by $A^*$ the Hilbert adjoint of $A$, and write $P\succeq Q$ (resp.\ $P\succ Q$) if $P-Q$ is positive semidefinite (resp.\ positive definite) in the sense that $\langle h,(P-Q)h\rangle_H\ge0$ (resp.\ $>0$) for all $h\in H$ (resp.\ $h\ne0$).


\begin{definition}\label{def:operator:classes}
  Let $\sigma\in\mathcal{S}_{0,\infty}$ be given. Define the following classes of operator-valued functions:
\begin{align*}
  \mathcal{P}_\sigma&:=\left\{P:\mathbb{R}_{\ge0}\to L(H)\;\middle|\;\begin{array}{l}
    P(t)=P(t)^*,\;\exists  c_P>0:\; 0\prec P(t)\preceq c_P I,\\
    {\langle x,P(\cdot)x\rangle_H\text{ continuous on each }[t_k,t_{k+1}),}\\
    {\text{with a left limit at }t_{k+1},\text{ for all }x\in H,}\\
    {\langle d,P(\cdot)d\rangle_H\text{ right differentiable on each }[t_k,t_{k+1}),}\\
    {\text{for all }d\in D(A)}
    \end{array}\right\},\\
  {\mathcal{Q}(\mathcal{I})}&{:=\left\{Q:\mathcal{I}\to L(H) \middle| \begin{array}{l}
    Q(t)=Q(t)^*, \exists  \alpha_q,\alpha_Q>0: \alpha_q I\preceq Q(t)\preceq\alpha_Q I,\\
    Q\text{ piecewise strongly continuous,}\\
    Q\text{ strongly right-continuous}
    \end{array}\right\},}\\
  \mathcal{R}&:=\left\{R:\mathbb{Z}_{\ge1}\to L(H)\;\middle|\;\begin{array}{l}
    R(k)=R(k)^*,\;\exists  \alpha_r,\alpha_R>0:\;\alpha_r I\preceq R(k)\preceq\alpha_R I
    \end{array}\right\}{,}
\end{align*}
{where $\mathcal{I}\subseteq\mathbb{R}_{\ge0}$ is an interval. We also define $\mathcal{Q}:=\mathcal{Q}(\mathbb{R}_{\ge0})$.}
\end{definition}

{For $P\in\mathcal{P}_\sigma$ and $k\ge1$, the left limit $P(t_k^-)$ denotes the self-adjoint operator defined by $\langle x,P(t_k^-)x\rangle_H:=\lim_{t\uparrow t_k}\langle x,P(t)x\rangle_H$, $x\in H$. For $d\in D(A)$, $\langle d,\dot P(t)d\rangle_H$ denotes the right derivative of $t\mapsto\langle d,P(t)d\rangle_H$.
When $P$ is strongly differentiable, $\langle d,\dot P(t)d\rangle_H$ coincides with the value at $d$ of the quadratic form of the strong derivative of $P$.}

\subsection{Fixed impulse sequence}\label{subsec:hilbert:fixed}

The following theorem provides an N\&S characterization of strong GES on Hilbert spaces in terms of an operator Lyapunov equation, where the free parameters $Q$ and $R$ play the role of the weighting operators in the functional.

\begin{theorem}[Fixed sequence, quadratic N\&S]\label{thm:hilbert:nc}
Let $\sigma\in\mathcal{S}_{0,\infty}$ be given, let Assumption~\ref{ass:growth} hold, and suppose the system \eqref{eq:syst} acts on a Hilbert space $H$ with $A$ generating a $C_0$-semigroup $S(t)$ on $H$. The following statements are equivalent:
\begin{enumerate}[(i)]
  \item The system \eqref{eq:syst} is strongly GES (Definition~\ref{def:stab}).

  \item For every choice of $Q\in\mathcal{Q}$ and $R\in\mathcal{R}$, there exists $P\in\mathcal{P}_\sigma$ such that
    \begin{align}
      \langle d,(\dot{P}(t)+A^*P(t)+P(t)A+Q(t))d\rangle_H&=0,\label{eq:hilbert:flow:eq}\quad d\in D(A),\; t\notin\mathbb{T}_\sigma,\\
      J^*P(t_k)J-P(t_k^-)+R(k)&=0,\label{eq:hilbert:jump:eq}\quad k\ge1.
    \end{align}
\end{enumerate}
Moreover, the explicit solution $P=\bar{P}$ is given by
\begin{equation}\label{eq:explicit:P}
  \bar{P}(t):=\int_t^\infty U_\sigma(s,t)^*Q(s)U_\sigma(s,t)  \ds+\sum_{{k\ge1:  t_k>t}}U_\sigma(t_k^-,t)^*R(k)U_\sigma(t_k^-,t).
\end{equation}
\end{theorem}

\begin{proof}
\emph{(i) $\Rightarrow$ (ii).} Assume GES. Fix $Q\in\mathcal{Q}$ and $R\in\mathcal{R}$, and define $\bar P(t)$ by \eqref{eq:explicit:P}.

\emph{Positivity.} For any $x\in H$, $x\ne0$: the integral contains $\|Q(s)^{1/2}U_\sigma(s,t)x\|_H^2\ge0$ and the integrand at $s=t$ is $\|Q(t)^{1/2}x\|_H^2>0$ (since $Q(t)\succ0$), so $\langle x,\bar P(t)x\rangle_H>0$, giving $\bar P(t)\succ0$.

\emph{Boundedness.} By GES and Theorem~\ref{thm:main}(i)$\Rightarrow$(ii) applied with $p=2$, the Datko condition holds at all starting times with some uniform $b>0$. With $\|Q(s)^{1/2}\|_{L(H)}^2\le\alpha_Q$ and $\|R(k)^{1/2}\|_{L(H)}^2\le\alpha_R$:
\begin{align*}
  \langle x,\bar P(t)x\rangle_H
  &\le\alpha_Q\int_t^\infty\|U_\sigma(s,t)x\|_H^2  \ds+\alpha_R\sum_{{k\ge1:  t_k>t}}\|U_\sigma(t_k^-,t)x\|_H^2 \le\max(\alpha_Q,\alpha_R)  b^2\|x\|_H^2,
\end{align*}
so $\bar P(t)\preceq\max(\alpha_Q,\alpha_R)b^2  I=:c_P I$.

{\emph{Regularity and flow equation \eqref{eq:hilbert:flow:eq}.} For $x\in H$, $\langle x,\bar P(t)x\rangle_H$ is the functional \eqref{eq:V:def} with $p=2$, in which the integrand and the summand are weighted by $Q(s)$ and $R(k)$, respectively. The uniform continuity argument in the proof of Theorem~\ref{thm:main} applies, with the weights bounded by $\alpha_Q$ and $\alpha_R$, and shows that $t\mapsto\langle x,\bar P(t)x\rangle_H$ is continuous on each $[t_k,t_{k+1})$ with a left limit at $t_{k+1}$. By polarization, $t\mapsto\bar P(t)$ is then weakly continuous on each $[t_k,t_{k+1})$. Let now $d\in D(A)$, $k\ge0$, $t\in[t_k,t_{k+1})$, and $h>0$ with $t+h<t_{k+1}$. Since $U_\sigma(s,t+h)S(h)=U_\sigma(s,t)$ for all $s\ge t+h$, the computation carried out in the proof of Theorem~\ref{thm:main} to establish \eqref{eq:nc:flow} gives
\begin{equation*}
  \langle S(h)d,\bar P(t+h)S(h)d\rangle_H=\langle d,\bar P(t)d\rangle_H-\int_t^{t+h}\langle S(s-t)d,Q(s)S(s-t)d\rangle_H\ds.
\end{equation*}
Since $\bar P(t+h)$ is self-adjoint, this yields
\begin{equation*}
\begin{aligned}
  \frac{\langle d,\bar P(t+h)d\rangle_H-\langle d,\bar P(t)d\rangle_H}{h}
  &=-\Big\langle\frac{S(h)d-d}{h},\bar P(t+h)S(h)d\Big\rangle_H\\
  &\quad-\Big\langle\bar P(t+h)d,\frac{S(h)d-d}{h}\Big\rangle_H\\
  &\quad-\frac1h\int_t^{t+h}\langle S(s-t)d,Q(s)S(s-t)d\rangle_H\ds.
\end{aligned}
\end{equation*}
As $h\downarrow0$, $(S(h)d-d)/h\to Ad$ in $H$, both $\bar P(t+h)S(h)d$ and $\bar P(t+h)d$ converge weakly to $\bar P(t)d$, and the last term converges to $\langle d,Q(t)d\rangle_H$ by strong right-continuity of $Q$ and strong continuity of $S(\cdot)$. Therefore, $t\mapsto\langle d,\bar P(t)d\rangle_H$ is right differentiable on $[t_k,t_{k+1})$ with
\begin{equation*}
  \langle d,\dot{\bar P}(t)d\rangle_H=-\langle Ad,\bar P(t)d\rangle_H-\langle\bar P(t)d,Ad\rangle_H-\langle d,Q(t)d\rangle_H,
\end{equation*}
which is \eqref{eq:hilbert:flow:eq}. Together with the positivity and boundedness properties, this shows that $\bar P\in\mathcal{P}_\sigma$.}

\emph{Jump equation \eqref{eq:hilbert:jump:eq}.} Evaluating $\bar P$ at $t_k^-$ and $t_k$ via $U_\sigma(s,t_k^-)=U_\sigma(s,t_k)J$ gives $\bar P(t_k^-)=J^*\bar P(t_k)J+R(k)$, i.e.\ \eqref{eq:hilbert:jump:eq}.

\emph{(ii) $\Rightarrow$ (i).} Suppose $P\in\mathcal{P}_\sigma$ satisfies \eqref{eq:hilbert:flow:eq}--\eqref{eq:hilbert:jump:eq}. Define $V(t,x)=\langle x,P(t)x\rangle_H$. Then $V(t,x)\le c_P\|x\|_H^2$ (upper bound from $P\preceq c_P I$).
{Since $\|P(t)\|\le c_P$, $V$ is Lipschitz on bounded subsets of $H$ uniformly in $t$, and, by definition of $\mathcal{P}_\sigma$, $t\mapsto V(t,x)$ has the continuity properties required in Theorem~\ref{thm:main}(iii). For the flow condition, let $x\in D(A)$, $k\ge0$, $t\in[t_k,t_{k+1})$, and $h>0$ with $t+h<t_{k+1}$. Let $y(s):=S(s)x\in D(A)$ and $g(s):=\langle y(s),P(t+s)y(s)\rangle_H$, $s\in[0,h]$. Since $P(t+s+\varepsilon)$ is self-adjoint, we have, for $s\in[0,h)$ and $\varepsilon>0$,
\begin{equation*}
\begin{aligned}
  g(s+\varepsilon)-g(s)&=\langle y(s+\varepsilon)-y(s),P(t+s+\varepsilon)y(s+\varepsilon)\rangle_H\\
  &\quad+\langle P(t+s+\varepsilon)y(s),y(s+\varepsilon)-y(s)\rangle_H\\
  &\quad+\langle y(s),(P(t+s+\varepsilon)-P(t+s))y(s)\rangle_H.
\end{aligned}
\end{equation*}
Since $y$ is differentiable with $\dot y(s)=Ay(s)$, and $P$ is bounded and weakly continuous on $[t_k,t_{k+1})$, dividing by $\varepsilon$ and letting $\varepsilon\downarrow0$ shows that $g$ is right differentiable on $[0,h)$ with right derivative
\begin{equation*}
  \langle y(s),(\dot P(t+s)+A^*P(t+s)+P(t+s)A)y(s)\rangle_H=-\langle y(s),Q(t+s)y(s)\rangle_H\le-\alpha_q\|y(s)\|_H^2
\end{equation*}
for $t+s\notin\mathbb{T}_\sigma$, by \eqref{eq:hilbert:flow:eq} and the coercivity of $Q\in\mathcal{Q}$. The same decomposition shows that $g$ is continuous on $[0,h]$. Hence, the continuous function $s\mapsto g(s)+\alpha_q\int_0^s\|y(r)\|_H^2\dr$ has a nonpositive right derivative on $(0,h)$ and is therefore nonincreasing on $[0,h]$ \cite{Hagood:06}. This gives
\begin{equation*}
  V(t+h,S(h)x)-V(t,x)\le-\alpha_q\int_0^h\|S(s)x\|_H^2\ds
\end{equation*}
for $x\in D(A)$, which extends to all $x\in H$ by density of $D(A)$ and continuity of both sides in $x$. Dividing by $h$ and letting $h\downarrow0$ gives $\underline{D}^+V(t,x)\le-\alpha_q\|x\|_H^2$.} At jumps: $$V(t_k,Jx)-V(t_k^-,x)=-\langle x,R(k)x\rangle_H\le-\alpha_r\|x\|_H^2$$ with $\alpha_r>0$ from the coercivity of $R\in\mathcal{R}$. So $V$ satisfies the non-coercive conditions of Theorem~\ref{thm:main}(iii) with $\|x\|^p$ replaced by $\min(\alpha_q,\alpha_r)\|x\|^2$, and Theorem~\ref{thm:main}(iii)$\Rightarrow$(i) yields GES.
\end{proof}

\begin{remark}\label{rem:hilbert:gap}
The functional $\bar P(t)$ in \eqref{eq:explicit:P} is the quadratic realization of the non-coercive Lyapunov functional \eqref{eq:V:def} for $p=2$. By construction, $\bar P(t)\succ0$ pointwise whenever $Q\succ0$ and $R\succ0$, but a \emph{uniform} lower bound $\bar P(t)\succeq\alpha_m I$ with $\alpha_m>0$ is \emph{not} guaranteed by the construction. This is the impulsive analogue of the phenomenon identified in \cite{Mironchenko:19}: there exist exponentially stable $C_0$-semigroups on Hilbert spaces for which the integral operator is bounded and pointwise positive definite yet not uniformly coercive. The uniform lower bound is the additional structural requirement of the (sufficient-only) coercive corollary below.
\end{remark}

The coercive case is now a corollary of Theorem~\ref{thm:hilbert:nc}: strengthening the equalities to inequalities and adding a coercive lower bound on $P$ yields a sufficient condition.

\begin{corollary}[Fixed sequence, coercive quadratic]\label{cor:hilbert:coerc}
Let $\sigma\in\mathcal{S}_{0,\infty}$ be given. If there exist $P\in\mathcal{P}_\sigma$, $Q\in\mathcal{Q}$, $R\in\mathcal{R}$, and a scalar $\alpha_m>0$ such that
\begin{align}
  P(t)&\;\succeq\;\alpha_m I,\label{eq:hilbert:coerc:lower}\\
  \dot{P}(t)+A^*P(t)+P(t)A+Q(t)&\;\preceq\;0,\label{eq:hilbert:coerc:flow}\\
  J^*P(t_k)J-P(t_k^-)+R(k)&\;\preceq\;0,\label{eq:hilbert:coerc:jump}
\end{align}
hold for all $t\notin\mathbb{T}_\sigma$ (in the sense of quadratic forms on $D(A)$) and all $k\ge1$, then the system \eqref{eq:syst} is strongly GES (for the fixed sequence $\sigma$).
\end{corollary}

\begin{proof}
Define $V(t,x):=\langle x,P(t)x\rangle_H$; since $P\in\mathcal{P}_\sigma$, $V$ is Lipschitz on bounded subsets of $H$ (uniformly in $t$) and $V(t,x)\le c_P\|x\|_H^2${, and $t\mapsto V(t,x)$ has the continuity properties required in Theorem~\ref{thm:main}(iii)}. For $x\in D(A)$ and $t\notin\mathbb{T}_\sigma$, the flow inequality \eqref{eq:hilbert:coerc:flow} gives $\langle x,(\dot P(t)+A^*P(t)+P(t)A)x\rangle_H\le-\langle x,Q(t)x\rangle_H\le-\alpha_q\|x\|_H^2$, whence, integrating along the flow {by means of the right-derivative and monotonicity argument in the proof of Theorem~\ref{thm:hilbert:nc}, (ii)$\Rightarrow$(i), applied to $s\mapsto V(t+s,S(s)x)$ for $t+h<t_{k+1}$},
\begin{equation*}
  V(t+h,S(h)x)-V(t,x)\le-\alpha_q\int_0^h\|S(s)x\|_H^2\,\ds .
\end{equation*}
This integrated inequality extends from $D(A)$ to all of $H$ by density and continuity in $x$; dividing by $h$ and letting $h\downarrow0$, strong continuity of $S(\cdot)$ gives $\underline{D}^+V(t,x)\le-\alpha_q\|x\|_H^2$. The jump inequality \eqref{eq:hilbert:coerc:jump} gives $V(t_k,Jx)-V(t_k^-,x)\le-\langle x,R(k)x\rangle_H\le-\alpha_r\|x\|_H^2$, valid on all of $H$ since $J$ and $P$ are bounded. These are exactly the non-coercive conditions of Theorem~\ref{thm:main}(iii) with $p=2$ (the strict decrement $\|x\|^p$ replaced by $\min(\alpha_q,\alpha_r)\|x\|_H^2$), so Theorem~\ref{thm:main}(iii)$\Rightarrow$(i) yields strong GES. The argument uses the weights $Q,R$ only through their coercive \emph{lower} bounds; in particular $\dot P+A^*P+PA$, a form on $D(A)$ that is generally unbounded, need not define a bounded operator, so no appeal to membership of $-(\dot P+A^*P+PA)$ in $\mathcal{Q}$ is made.

The coercive lower bound \eqref{eq:hilbert:coerc:lower} is not needed for GES per se (which follows already from the non-coercive theorem above), but is the structural assumption distinguishing this corollary: it forces $V(t,x)=\langle x,P(t)x\rangle$ to be a coercive functional with $\alpha_m\|x\|^2\le V(t,x)\le c_P\|x\|^2$, which is more restrictive than the non-coercive existence result. The converse direction (i.e., GES $\Rightarrow$ existence of $P$ satisfying \eqref{eq:hilbert:coerc:lower} together with \eqref{eq:hilbert:coerc:flow}--\eqref{eq:hilbert:coerc:jump}) does \emph{not} hold in general, for the reason discussed in Remark~\ref{rem:hilbert:gap}.
\end{proof}

\subsection{Arbitrary impulse sequences ($\mathcal{S}_{0,\infty}$)}\label{subsec:hilbert:arbitrary}

On a Hilbert space $H$, the time-independent functional $V$ of Theorem~\ref{thm:arbitrary:nc} can be taken to be quadratic, leading to operator-equation sufficient conditions. In contrast to the dwell-time settings, the equivalence with UGES fails: there exist UGES systems on Hilbert spaces, even with exponentially stable $S(t)$, for which no quadratic $V$ satisfying the operator equations below exists.

\begin{theorem}[Arbitrary sequences, quadratic sufficient]\label{thm:hilbert:arbitrary:nc}
Let $H$ be a Hilbert space, $A$ generate a $C_0$-semigroup on $H$, and $J\in L(H)$.
If there exist self-adjoint operators $Q,R\in L(H)$ \emph{coercive}, i.e., $Q\succeq\alpha_q I$, $R\succeq\alpha_r I$ for some $\alpha_q,\alpha_r>0$, and a self-adjoint $P\in L(H)$ with $0\prec P\preceq c_P I$ for some $c_P>0$, such that
    \begin{subequations}\label{eq:hilbert:arbitrary:eq}
    \begin{align}
      \langle d,(A^*P+PA+Q)d\rangle_H&=0,\qquad d\in D(A),\label{eq:hilbert:arbitrary:flow:eq}\\
      J^*PJ-P+R&=0,\label{eq:hilbert:arbitrary:jump:eq}
    \end{align}
    \end{subequations}
then the system \eqref{eq:syst} is strongly UGES over $\mathcal{S}_{0,\infty}$ (Definition~\ref{def:stabu}). Assumption~\ref{ass:growth} is automatic and need not be required a priori: \eqref{eq:hilbert:arbitrary:flow:eq} forces the semigroup $S(t)$ to be exponentially stable, hence Assumption~\ref{ass:growth} holds with $\omega_0<0$ a posteriori.
\end{theorem}

\begin{proof}
Set $V(x):=\langle x,Px\rangle_H$. Then $V$ is Lipschitz on bounded subsets of $H$ and $0\le V(x)\le c_P\|x\|^2$. From \eqref{eq:hilbert:arbitrary:flow:eq}, for $x\in D(A)$,
\[
  \tfrac{\d}{\dt}V(S(t)x)\big|_{t=0^+}=\langle Ax,Px\rangle_H+\langle x,PAx\rangle_H=-\langle x,Qx\rangle_H\le-\alpha_q\|x\|^2,
\]
and integration along $\dot x=Ax$ gives, for $x\in D(A)$ and $h\ge0$,
\[
  V(S(h)x)-V(x)\le-\alpha_q\!\int_0^h\!\|S(s)x\|^2  \ds.
\]
By density of $D(A)$ in $H$ and continuity in $x$, this extends to all $x\in H$. Dividing by $h$ and letting $h\to0^+$, strong continuity of $S(\cdot)$ at $s=0$ gives $\underline D^+V(x)\le-\alpha_q\|x\|^2$. From \eqref{eq:hilbert:arbitrary:jump:eq}, $V(Jx)-V(x)=-\langle x,Rx\rangle_H\le-\alpha_r\|x\|^2$.

\emph{Assumption~\ref{ass:growth} holds a priori.} Letting $h\to\infty$ in the integrated flow inequality above and using $V\ge0$ gives $\alpha_q\int_0^\infty\|S(s)x\|^2\ds\le V(x)\le c_P\|x\|^2$ for every $x\in H$. By Datko's theorem for $C_0$-semigroups \cite{Datko:70}, $S(t)$ is therefore exponentially stable, so $\omega_0<0$ and Assumption~\ref{ass:growth} holds on $\mathcal{S}_{0,\infty}$ with $G=M_0$ (Remark~\ref{rem:growth:G}); this justification uses only \eqref{eq:hilbert:arbitrary:flow:eq}, independently of any impulsive result, so the invocation of Theorem~\ref{thm:arbitrary:nc} below is legitimate.

Set $\gamma:=\min(\alpha_q,\alpha_r)>0$. Then $\gamma^{-1}V$ satisfies $0\le\gamma^{-1}V(x)\le(c_P/\gamma)\|x\|^2$, $\underline D^+(\gamma^{-1}V)(x)\le-\|x\|^2$, and $(\gamma^{-1}V)(Jx)-(\gamma^{-1}V)(x)\le-\|x\|^2$; these are the bounds \eqref{eq:arbitrary:nc:bd}--\eqref{eq:arbitrary:nc:jump} of Theorem~\ref{thm:arbitrary:nc} with $p=2$. Since Assumption~\ref{ass:growth} has just been established a priori, Theorem~\ref{thm:arbitrary:nc}(ii)$\Rightarrow$(i) applies and yields strong UGES over $\mathcal{S}_{0,\infty}$.
\end{proof}


The coercive case is a corollary in which a uniform lower bound on $P$ is additionally assumed.

\begin{corollary}[Arbitrary sequences, coercive quadratic]\label{cor:hilbert:arbitrary:coerc}
If there exist $P\in L(H)$ self-adjoint and uniformly coercive with $\alpha_m I\preceq P\preceq c_P I$ for some $c_P\ge\alpha_m>0$, and $\epsilon>0$, such that
\begin{subequations}\label{eq:arbitrary:coerc}
\begin{align}
  A^*P+PA&\preceq -\epsilon I,\label{eq:arbitrary:coerc:flow}\\
  J^*PJ-P&\preceq -\epsilon I,\label{eq:arbitrary:coerc:jump}
\end{align}
\end{subequations}
then the system \eqref{eq:syst} is strongly UGES over $\mathcal{S}_{0,\infty}$.
\end{corollary}

\begin{proof}
Define the time-independent $V(x):=\langle x,Px\rangle_H$; it is Lipschitz on bounded subsets and $V(x)\le c_P\|x\|_H^2$. For $x\in D(A)$ the flow inequality \eqref{eq:arbitrary:coerc:flow} gives $\langle x,(A^*P+PA)x\rangle_H\le-\epsilon\|x\|_H^2$, hence $V(S(h)x)-V(x)\le-\epsilon\int_0^h\|S(s)x\|_H^2\,\ds$; this integrated bound extends from $D(A)$ to all of $H$ by density and continuity in $x$, and letting $h\downarrow0$ gives $\underline{D}^+V(x)\le-\epsilon\|x\|_H^2$; moreover, letting $h\to\infty$ in the same integrated inequality and using $V\ge0$ gives $\epsilon\int_0^\infty\|S(s)x\|_H^2\ds\le V(x)\le c_P\|x\|_H^2$, so $S(t)$ is exponentially stable (Datko's theorem \cite{Datko:70}) and Assumption~\ref{ass:growth} holds a priori on $\mathcal{S}_{0,\infty}$. The jump inequality \eqref{eq:arbitrary:coerc:jump} gives $V(Jx)-V(x)=\langle x,(J^*PJ-P)x\rangle_H\le-\epsilon\|x\|_H^2$ on all of $H$ ($J,P$ bounded). These are the (time-independent) non-coercive conditions of Theorem~\ref{thm:arbitrary:nc}(ii), so Theorem~\ref{thm:arbitrary:nc}(ii)$\Rightarrow$(i) yields strong UGES over $\mathcal{S}_{0,\infty}$. Only the coercive lower bounds are used; the generally unbounded form $A^*P+PA$ need not define a bounded operator, so no class-membership of $-(A^*P+PA)$ is claimed. The coercive lower bound $P\succeq\alpha_m I$ further forces $V$ to be coercive, which need not be available in general (Remark~\ref{rem:hilbert:gap}).
\end{proof}




\subsection{Constant dwell-time ($\mathcal{S}_{\mathrm{cst}}(T)$)}\label{subsec:hilbert:cst}

The following theorem specializes the timer-dependent conditions of Theorem~\ref{thm:cst:nc} to the quadratic class, where they become an operator Lyapunov ODE on $[0,T]$ coupled to a jump equation; as in the fixed-sequence case, the characterization is necessary and sufficient.

\begin{theorem}[Constant dwell-time, quadratic N\&S]\label{thm:hilbert:cst}
Let $T>0$ and let Assumption~\ref{ass:growth} hold. The following are equivalent:
\begin{enumerate}[(i)]
  \item the system \eqref{eq:syst} is strongly UGES over $\mathcal{S}_{\mathrm{cst}}(T)$;
  \item for every choice of $Q\in\mathcal{Q}([0,T])$ and $R\succ0$ in $L(H)$, there exists a self-adjoint 
$\bar P:[0,T]\to L(H)${, with $\tau\mapsto\langle x,\bar P(\tau)x\rangle_H$ continuous on $[0,T]$ for all $x\in H$ and $\tau\mapsto\langle d,\bar P(\tau)d\rangle_H$ right differentiable on $[0,T)$ for all $d\in D(A)$,} with $0\prec\bar P(\tau)\preceq c_{\bar P} I$ for some $c_{\bar P}>0$, satisfying the operator Lyapunov ODE
    \begin{equation}\label{eq:hilbert:cst:flow:ode}
      \langle d,(\dot{\bar P}(\tau)+A^*\bar P(\tau)+\bar P(\tau)A+Q(\tau))d\rangle_H=0,\qquad d\in D(A),\;\tau\in(0,T),
    \end{equation}
    and the jump equation
    \begin{equation}\label{eq:hilbert:cst:jump:eq}
      J^*\bar P(0)J-\bar P(T)+R=0;
    \end{equation}
  \item the monodromy $M_T=S(T)J$ (pre-jump sampling, as in Theorem~\ref{thm:cst:nc}) admits an operator Lyapunov function: for every coercive $R_d\in L(H)$, there exists $P_d\in L(H)$ self-adjoint with $0\prec P_d\preceq c_{P_d} I$ for some $c_{P_d}>0$ such that
    \begin{equation}\label{eq:hilbert:cst:dt}
      M_T^*P_d M_T - P_d + R_d = 0,
    \end{equation}
    with associated discrete-time quadratic Lyapunov function $V_d(x):=\langle x,P_d x\rangle_H$.
\end{enumerate}
\end{theorem}

\begin{proof}
$(i)\Leftrightarrow(ii)$ is the quadratic specialization of Theorem~\ref{thm:cst:nc}: take $\bar V(\tau,x)=\langle x,\bar P(\tau)x\rangle_H$ with $\bar P$ the explicit construction \eqref{eq:cst:V:explicit} adapted to $p=2$. {By (ii), the map $\tau\mapsto\bar V(\tau,x)$ is continuous on $[0,T]$ for every $x\in H$, as required in Theorem~\ref{thm:cst:nc}(iii).} The flow and jump computations specialize as in the proof of Theorem~\ref{thm:hilbert:nc}{, which also provides the regularity of $\bar P$ required in (ii) and, in the sufficiency part, the flow condition \eqref{eq:cst:nc:flow} from \eqref{eq:hilbert:cst:flow:ode}}.

$(ii)\Rightarrow(iii)$: given any coercive $R_d\in L(H)$, choose any $Q\in\mathcal{Q}([0,T])$ small enough that $$R:=R_d-J^*\!\int_0^T\!S(s)^*Q(s)S(s)  \ds\,J\succ0$$ is coercive (e.g.\ $Q=\epsilon I$ with $\epsilon\in(0,\alpha_{R_d}/(\|J\|^2  \int_0^T\!\|S(s)\|^2  \ds))$, where $\alpha_{R_d}>0$ denotes the coercivity constant of $R_d$). Apply (ii) with this $Q$ and $R$ to obtain $\bar P:[0,T]\to L(H)$ self-adjoint, {with the regularity stated in (ii)} 
with $0\prec\bar P(\tau)\preceq c_{\bar P}I$, satisfying \eqref{eq:hilbert:cst:flow:ode}--\eqref{eq:hilbert:cst:jump:eq}. Set $P_d:=\bar P(T)$. Integrating the flow equation along the unforced trajectory $r\mapsto S(r)x$ {(as in the proof of Theorem~\ref{thm:hilbert:nc}, the continuous function $s\mapsto\langle S(s)x,\bar P(s)S(s)x\rangle_H+\int_0^s\langle S(r)x,Q(r)S(r)x\rangle_H\dr$ has a zero right derivative on $(0,T)$, and is therefore constant on $[0,T]$ \cite{Hagood:06})} yields, for $x\in D(A)$,
\[
  \langle S(T)x,\bar P(T)S(T)x\rangle_H-\langle x,\bar P(0)x\rangle_H=-\!\int_0^T\!\langle S(s)x,Q(s)S(s)x\rangle_H  \ds,
\]
i.e.\ $S(T)^*\bar P(T)S(T)=\bar P(0)-\int_0^T S(s)^*Q(s)S(s)  \ds$ in the sense of quadratic forms on $D(A)$, hence on $H$ by density and boundedness of both sides. Combined with the jump equation $J^*\bar P(0)J=P_d-R$ (using $\bar P(T)=P_d$):
\begin{align*}
   M_T^*P_d M_T=J^*S(T)^*P_d S(T)J&=J^*\!\left[\bar P(0)-\!\int_0^T\!S(s)^*Q(s)S(s)  \ds\right]\!J\\
   &=(P_d-R)-J^*\!\int_0^T\!S(s)^*Q(s)S(s)  \ds\,J,
\end{align*}
so $M_T^*P_d M_T-P_d+R_d=0$ holds by the choice of $R$. The bound $0\prec P_d=\bar P(T)\preceq c_{\bar P}I$ gives $c_{P_d}:=c_{\bar P}$.

$(iii)\Rightarrow(ii)$: given any $Q\in\mathcal{Q}([0,T])$ and $R\succ0$ coercive in $L(H)$, set $R_d:=R+J^*\!\int_0^T\!S(s)^*Q(s)S(s)  \ds\,J$, which is coercive since $R$ and $Q$ are. Since $Q\in\mathcal{Q}([0,T])$ is piecewise {strongly} continuous and uniformly bounded, and $S(\cdot)$ is strongly continuous and uniformly bounded on $[0,T]$, the integral $\int_0^T S(s-\tau)^*Q(s)S(s-\tau)  \ds$ is well-defined as a bounded operator for each $\tau\in[0,T]$. Apply (iii) with this $R_d$ to obtain $P_d\in L(H)$ self-adjoint with $0\prec P_d\preceq c_{P_d}I$ satisfying $M_T^*P_d M_T-P_d+R_d=0$. Define
\[
  \bar P(\tau):=S(T-\tau)^*P_d S(T-\tau)+\!\int_\tau^T\!S(s-\tau)^*Q(s)S(s-\tau)  \ds.
\]
Then $\bar P(T)=P_d$, 
{$\tau\mapsto\langle x,\bar P(\tau)x\rangle_H$ is continuous on $[0,T]$ for all $x\in H$ (by strong continuity of $S(\cdot)$ and dominated convergence, $Q$ being bounded and piecewise strongly continuous), $\tau\mapsto\langle d,\bar P(\tau)d\rangle_H$ is right differentiable on $[0,T)$ for all $d\in D(A)$ (by strong right-continuity of $Q$),} and $\bar P(0)=S(T)^*P_d S(T)+\int_0^T\!S(s)^*Q(s)S(s)  \ds$. Direct differentiation in $\tau$, using $\frac{\d}{\d\tau}S(T-\tau)^*X S(T-\tau)=-A^*S(T-\tau)^*X S(T-\tau)-S(T-\tau)^*X S(T-\tau)A$ on $D(A)$ and $\frac{\d}{\d\tau}\int_\tau^T S(s-\tau)^*Q(s)S(s-\tau)  \ds=-Q(\tau)-A^*\!\int_\tau^T\!S(s-\tau)^*Q(s)S(s-\tau)  \ds-\int_\tau^T\!S(s-\tau)^*Q(s)S(s-\tau)  \ds\,A$ on $D(A)$, gives \eqref{eq:hilbert:cst:flow:ode}. The jump equation:
\[
  J^*\bar P(0)J-\bar P(T)+R=M_T^*P_d M_T+J^*\!\int_0^T\!S(s)^*Q(s)S(s)  \ds\,J-P_d+R=M_T^*P_d M_T-P_d+R_d=0
\]
by the choice of $R_d$. The bound $\bar P(\tau)\preceq c_{\bar P}I$ with $c_{\bar P}:=M_S^2(c_{P_d}+T\alpha_Q)$, where $M_S:=\sup_{s\in[0,T]}\|S(s)\|<\infty$ and $\alpha_Q$ is the upper-bound constant of $Q\in\mathcal{Q}([0,T])$, follows from the construction.
\end{proof}

\begin{corollary}[Constant dwell-time, coercive quadratic]\label{cor:hilbert:cst:coerc}
If there exist constants $\alpha_m,c_P,\epsilon>0$ and a self-adjoint, strongly differentiable $\bar P:[0,T]\to L(H)$ with $\alpha_m I\preceq\bar P(\tau)\preceq c_P I$ satisfying
\begin{equation}\label{eq:hilbert:cst:flow:LMI}
  \dot{\bar P}(\tau)+A^*\bar P(\tau)+\bar P(\tau)A\preceq -\epsilon I,\quad\tau\in(0,T),\qquad J^*\bar P(0)J - \bar P(T)\preceq -\epsilon I,
\end{equation}
then the system \eqref{eq:syst} is strongly UGES over $\mathcal{S}_{\mathrm{cst}}(T)$.
\end{corollary}

\begin{proof}
Define the timer functional $\bar V(\tau,x):=\langle x,\bar P(\tau)x\rangle_H$ on $[0,T]\times X$; it is Lipschitz on bounded subsets (uniformly in $\tau$) and $\bar V(\tau,x)\le c_P\|x\|_H^2$. {Moreover, since $\bar P$ is strongly differentiable, and hence strongly continuous, on $[0,T]$, the map $\tau\mapsto\bar V(\tau,x)$ is continuous on $[0,T]$ for every $x\in H$, as required in Theorem~\ref{thm:cst:nc}(iii).} For $x\in D(A)$, the flow inequality \eqref{eq:hilbert:cst:flow:LMI} gives $\langle x,(\dot{\bar P}(\tau)+A^*\bar P(\tau)+\bar P(\tau)A)x\rangle_H\le-\epsilon\|x\|_H^2$, hence $\bar V(\tau+h,S(h)x)-\bar V(\tau,x)\le-\epsilon\int_0^h\|S(s)x\|_H^2\,\ds$; this integrated bound extends from $D(A)$ to all of $H$ by density and continuity in $x$, and $h\downarrow0$ gives $\underline{D}^+\bar V(\tau,x)\le-\epsilon\|x\|_H^2$. The jump inequality gives $\bar V(0,Jx)-\bar V(T,x)=\langle x,(J^*\bar P(0)J-\bar P(T))x\rangle_H\le-\epsilon\|x\|_H^2$ on all of $H$ ($J,\bar P$ bounded). These are the non-coercive timer conditions of Theorem~\ref{thm:cst:nc}(iii), so Theorem~\ref{thm:cst:nc}(iii)$\Rightarrow$(i) yields UGES over $\mathcal{S}_{\mathrm{cst}}(T)$. Only the coercive lower bounds are used; the generally unbounded form $\dot{\bar P}+A^*\bar P+\bar P A$ need not define a bounded operator. The uniform lower bound $\bar P\succeq\alpha_m I$ may not exist in general (Remark~\ref{rem:hilbert:gap}).
\end{proof}

\subsection{Minimum dwell-time ($\mathcal{S}_{\min}(T_{\min})$)}\label{subsec:hilbert:min}

For the minimum dwell-time family, the quadratic specialization of Theorem~\ref{thm:min:nc} carries an operator Lyapunov ODE on the mandatory phase $[0,T_{\min}]$, a Lyapunov inequality on the free phase, and a jump equation. Unlike the fixed-sequence and constant dwell-time cases, the condition is only sufficient, the quadratic class being strictly smaller than the non-coercive Lipschitz one.

\begin{theorem}[Minimum dwell-time, quadratic sufficient]\label{thm:hilbert:min}
Let $T_{\min}>0$, $H$ a Hilbert space, $A$ generate a $C_0$-semigroup on $H$, $J\in L(H)$, and let Assumption~\ref{ass:growth} hold. If there exist a self-adjoint, strongly differentiable $\bar P:[0,T_{\min}]\to L(H)$ with $0\prec\bar P(\tau)\preceq c_P I$ for some $c_P>0$, 
{$Q\in\mathcal{Q}([0,T_{\min}])$, and self-adjoint operators $Q_\infty,R\in L(H)$ satisfying $Q_\infty\succeq\alpha_q^\infty I$, $R\succeq\alpha_r I$ for some $\alpha_q^\infty,\alpha_r>0$,} such that
    \begin{subequations}\label{eq:hilbert:min:eq}
    \begin{align}
      \langle d,(\dot{\bar P}(\tau)+A^*\bar P(\tau)+\bar P(\tau)A+Q(\tau))d\rangle_H&=0,\quad d\in D(A),\;\tau\in(0,T_{\min}),\label{eq:hilbert:min:flow}\\
      \langle d,(A^*\bar P(T_{\min})+\bar P(T_{\min})A+Q_\infty)d\rangle_H&\le0,\quad d\in D(A),\label{eq:hilbert:min:free}\\
      J^*\bar P(0)J-\bar P(T_{\min})+R&=0,\label{eq:hilbert:min:jump}
    \end{align}
    \end{subequations}
then the system \eqref{eq:syst} is strongly UGES over $\mathcal{S}_{\min}(T_{\min})$.
\end{theorem}

\begin{proof}
Define $\bar V(\tau,x):=\langle x,\bar P(\tau)x\rangle_H$ on $[0,T_{\min}]\times H$ and extend by $\bar V(\tau,x):=\langle x,\bar P(T_{\min})x\rangle_H$ for $\tau\ge T_{\min}$. Then $0\le\bar V(\tau,x)\le c_P\|x\|^2$ and $\bar V$ is Lipschitz on bounded subsets of $H$. {Moreover, since $\bar P$ is strongly differentiable, and hence strongly continuous, on $[0,T_{\min}]$, the map $\tau\mapsto\bar V(\tau,x)$ is continuous on $[0,T_{\min}]$ for every $x\in H$.}

\emph{Mandatory flow \eqref{eq:min:nc:flow:mandatory}.} For $\tau+h\le T_{\min}$ and $x\in D(A)$, $S(s)x\in D(A)$ for $s\in[0,h]$ and
\[
\begin{aligned}
  \tfrac{\d}{\ds}\bar V(\tau+s,S(s)x)&=\langle S(s)x,[\dot{\bar P}(\tau+s)+A^*\bar P(\tau+s)+\bar P(\tau+s)A]S(s)x\rangle_H\\
  &=-\langle S(s)x,Q(\tau+s)S(s)x\rangle_H.
\end{aligned}
\]
Integration from $0$ to $h$ and \eqref{eq:hilbert:min:flow} give $\bar V(\tau+h,S(h)x)-\bar V(\tau,x)\le-\alpha_q\!\int_0^h\!\|S(s)x\|^2\,\ds$ for $x\in D(A)$, which extends to all $x\in H$ by density and continuity in $x$. Dividing by $h$ and letting $h\to0^+$, strong continuity of $S(\cdot)$ at $s=0$ gives $\underline D^+\bar V(\tau,x)\le-\alpha_q\|x\|^2$.

\emph{Free-phase flow \eqref{eq:min:nc:flow:free}.} For $\tau\ge T_{\min}$, $\bar V(\tau,x)=\langle x,\bar P(T_{\min})x\rangle_H$. For $x\in D(A)$,
\[
  \tfrac{\d}{\ds}\bar V(T_{\min},S(s)x)\big|_{s=0}=\langle x,[A^*\bar P(T_{\min})+\bar P(T_{\min})A]x\rangle_H\le-\langle x,Q_\infty x\rangle_H\le-\alpha_q^\infty\|x\|^2,
\]
using \eqref{eq:hilbert:min:free}. The same density and limit argument gives $\underline D^+\bar V(T_{\min},x)\le-\alpha_q^\infty\|x\|^2$ on $H$.

\emph{Jump \eqref{eq:min:nc:jump}.} For $\theta\ge T_{\min}$ and $z\in H$, $\bar V(\theta,z)=\langle z,\bar P(T_{\min})z\rangle_H$ and \eqref{eq:hilbert:min:jump} gives
\[
  \bar V(0,Jz)-\bar V(\theta,z)=\langle z,[J^*\bar P(0)J-\bar P(T_{\min})]z\rangle_H=-\langle z,Rz\rangle_H\le-\alpha_r\|z\|^2.
\]

Setting $\gamma:=\min(\alpha_q,\alpha_q^\infty,\alpha_r)>0$, the rescaled $\gamma^{-1}\bar V$ satisfies \eqref{eq:min:nc:upper}--\eqref{eq:min:nc:jump} of Theorem~\ref{thm:min:nc} with $p=2$ and $c=c_P/\gamma$. Theorem~\ref{thm:min:nc}(iii)$\Rightarrow$(i) yields UGES over $\mathcal{S}_{\min}(T_{\min})$.
\end{proof}

{The conditions of Theorem~\ref{thm:hilbert:min} are sufficient but not necessary. Indeed, for $p=2$, the explicit functional \eqref{eq:min:V:explicit} is a supremum of quadratic forms over the admissible signals, which is generally not a quadratic form. This is already the case in finite dimension: quadratic Lyapunov functions are known to be conservative for switched linear systems under dwell-time constraints \cite{Chesi:12,Wirth:05}, and a switched system with two modes $\dot x=A_ix$, $i\in\{1,2\}$, on $\mathbb{R}^n$ is represented by the impulsive system \eqref{eq:syst} on $\mathbb{R}^{2n}$ with $A=\diag(A_1,A_2)$ and $J=\left[\begin{smallmatrix}0&I_n\\I_n&0\end{smallmatrix}\right]$ (see \cite{Briat:15i}; a related reformulation on a product space is developed in Section~\ref{subsec:sw:lift}).}

The following monotonicity lemma, in the spirit of the dwell-time analysis of Geromel and Colaneri \cite{Geromel:06b} and \cite{Briat:13d,Briat:16c}, shows that for a \emph{constant} positive operator $P$ a single jump inequality at $\theta=T_{\min}$ propagates to all admissible dwell-times $\theta\ge T_{\min}$, provided the flow inequality $A^*P+PA\preceq-\epsilon I$ holds.

\begin{lemma}[Monotonicity of the dwell-time jump inequality]\label{lem:gc:monotone}
Let $P\in L(H)$ be self-adjoint with $\langle d,(A^*P+PA)d\rangle_H\le0$ for all $d\in D(A)$. Then for every $x\in H$ the map
\[
  \theta\longmapsto\langle x,J^*S(\theta)^*PS(\theta)Jx\rangle_H
\]
is non-increasing on $[0,\infty)$. Hence, if in addition $J^*S(T_{\min})^*PS(T_{\min})J-P\preceq-\delta I$ for some $\delta>0$, then
\[
  J^*S(\theta)^*PS(\theta)J-P\preceq-\delta I\qquad\text{for all }\theta\ge T_{\min}.
\]
\end{lemma}
\begin{proof}
Fix $x\in H$ and set $y:=Jx$. For $y\in D(A)$, the orbit $S(\theta)y\in D(A)$ and, by strong differentiability of the semigroup on $D(A)$,
\[
  \tfrac{\d}{\dtheta}\langle S(\theta)y,PS(\theta)y\rangle_H=\langle S(\theta)y,(A^*P+PA)S(\theta)y\rangle_H\le0,
\]
so $\theta\mapsto\langle S(\theta)y,PS(\theta)y\rangle_H=\langle x,J^*S(\theta)^*PS(\theta)Jx\rangle_H$ is non-increasing for $y\in D(A)$; the bound extends to all $y\in H$, hence all $x\in H$, by density of $D(A)$ and continuity of $S(\cdot)$ and $P$. In particular, for $\theta\ge T_{\min}$,
\[
  \langle x,J^*S(\theta)^*PS(\theta)Jx\rangle_H\le\langle x,J^*S(T_{\min})^*PS(T_{\min})Jx\rangle_H,
\]
so $J^*S(\theta)^*PS(\theta)J\preceq J^*S(T_{\min})^*PS(T_{\min})J$. Subtracting $P$ and using the hypothesis at $\theta=T_{\min}$ gives $J^*S(\theta)^*PS(\theta)J-P\preceq J^*S(T_{\min})^*PS(T_{\min})J-P\preceq-\delta I$.
\end{proof}

\begin{corollary}[Minimum dwell-time, coercive quadratic]\label{cor:hilbert:min:coerc}
If there exist constants $\alpha_m,c_P,\epsilon>0$ and a self-adjoint, strongly differentiable $\bar P:[0,T_{\min}]\to L(H)$ with $\alpha_m I\preceq\bar P(\tau)\preceq c_P I$ satisfying \eqref{eq:hilbert:cst:flow:LMI} on $[0,T_{\min}]$ together with the free-phase inequality $A^*\bar P(T_{\min})+\bar P(T_{\min})A\preceq -\epsilon I$, then the system \eqref{eq:syst} is strongly UGES over $\mathcal{S}_{\min}(T_{\min})$.
\end{corollary}

\begin{proof}
Define $\bar V(\tau,x):=\langle x,\bar P(\tau)x\rangle_H$ on $[0,T_{\min}]\times X$; it is Lipschitz on bounded subsets and $\bar V(\tau,x)\le c_P\|x\|_H^2$. {Moreover, since $\bar P$ is strongly differentiable, and hence strongly continuous, on $[0,T_{\min}]$, the map $\tau\mapsto\bar V(\tau,x)$ is continuous on $[0,T_{\min}]$ for every $x\in H$, as required in Theorem~\ref{thm:min:nc}(iii).} For $x\in D(A)$, the mandatory- and free-phase inequalities give $\langle x,(\dot{\bar P}+A^*\bar P+\bar P A)x\rangle_H\le-\epsilon\|x\|_H^2$ for $\tau\in(0,T_{\min})$ and $\langle x,(A^*\bar P(T_{\min})+\bar P(T_{\min})A)x\rangle_H\le-\epsilon\|x\|_H^2$ at $\tau=T_{\min}$; integrating each along the flow gives $\bar V(\tau+h,S(h)x)-\bar V(\tau,x)\le-\epsilon\int_0^h\|S(s)x\|_H^2\,\ds$, which extends from $D(A)$ to all of $H$ by density and continuity in $x$, so that $h\downarrow0$ yields $\underline{D}^+\bar V(\tau,x)\le-\epsilon\|x\|_H^2$ on $[0,T_{\min}]$. In particular, letting $h\to\infty$ in the free-phase integrated inequality $\bar V(T_{\min},S(h)x)-\bar V(T_{\min},x)\le-\epsilon\int_0^h\|S(s)x\|_H^2\ds$ and using $\bar V\ge0$ gives $\epsilon\int_0^\infty\|S(s)x\|_H^2\ds\le\bar V(T_{\min},x)\le c_P\|x\|_H^2$, so $S(t)$ is exponentially stable (Datko's theorem \cite{Datko:70}) and Assumption~\ref{ass:growth} holds on $\mathcal{S}_{\min}(T_{\min})$, as required by Theorem~\ref{thm:min:nc}. The jump inequality gives $\bar V(0,Jx)-\bar V(T_{\min},x)=\langle x,(J^*\bar P(0)J-\bar P(T_{\min}))x\rangle_H\le-\epsilon\|x\|_H^2$ on all of $H$ ($J,\bar P$ bounded). These are the non-coercive timer conditions of Theorem~\ref{thm:min:nc}(iii), so Theorem~\ref{thm:min:nc}(iii)$\Rightarrow$(i) yields UGES over $\mathcal{S}_{\min}(T_{\min})$. Only the coercive lower bounds are used; the generally unbounded forms $\dot{\bar P}+A^*\bar P+\bar P A$ and $A^*\bar P(T_{\min})+\bar P(T_{\min})A$ need not define bounded operators. The coercive lower bound $\bar P\succeq\alpha_m I$ further forces $\bar V$ to be coercive, which need not be available in general (Remark~\ref{rem:hilbert:gap}).
\end{proof}

Lemma~\ref{lem:gc:monotone} yields a certificate based on a single \emph{constant} operator $P$, namely the discrete-time Lyapunov operator $V_d(x)=\langle x,Px\rangle_H$ for the map $S(\theta)J$. This is condition (iv) of Theorem~\ref{thm:min:nc}, not the timer-dependent hybrid operator $\bar P(\tau)$ of (iii): the latter still varies with the timer, with $\bar P(0)\ne\bar P(T_{\min})$ in general. The point of the lemma is that, for this constant $P$, the discrete-time decrease need only be tested at the single dwell-time $\theta=T_{\min}$, since the flow inequality $A^*P+PA\preceq-\epsilon I$ then propagates it to all $\theta\ge T_{\min}$.

\begin{corollary}[Minimum dwell-time, constant coercive certificate]\label{cor:hilbert:min:const}
If there exist a self-adjoint $P\in L(H)$ with $\alpha_m I\preceq P\preceq c_P I$ for some $\alpha_m,c_P>0$, and $\epsilon>0$, such that
\begin{equation}\label{eq:hilbert:min:const:lmi}
  \langle d,(A^*P+PA)d\rangle_H\le-\epsilon\|d\|^2\;\;\forall d\in D(A),\qquad J^*S(T_{\min})^*PS(T_{\min})J-P\preceq-\epsilon I,
\end{equation}
then the system \eqref{eq:syst} is strongly UGES over $\mathcal{S}_{\min}(T_{\min})$.
\end{corollary}
\begin{proof}
The flow inequality $\langle d,(A^*P+PA)d\rangle_H\le-\epsilon\|d\|^2\le0$ is the hypothesis of Lemma~\ref{lem:gc:monotone}; together with the jump inequality at $\theta=T_{\min}$ it gives $J^*S(\theta)^*PS(\theta)J-P\preceq-\epsilon I$ for all $\theta\ge T_{\min}$, that is, $V_d(x):=\langle x,Px\rangle_H$ satisfies $V_d(S(\theta)Jx)-V_d(x)\le-\epsilon\|x\|^2$ for all $\theta\ge T_{\min}$. Moreover $A^*P+PA\preceq-\epsilon I$ makes the semigroup exponentially stable (integrating $\frac{\d}{\dt}\langle S(t)x,PS(t)x\rangle_H\le-\epsilon\|S(t)x\|^2$ {for $x\in D(A)$, extending the integrated inequality to all $x\in H$ by density of $D(A)$,} and using $P\succeq\alpha_m I$ gives $\|S(t)x\|^2\le(c_P/\alpha_m)e^{-(\epsilon/c_P)t}\|x\|^2$). Hence the augmented discrete-time condition (iv) of Theorem~\ref{thm:min:nc} holds with $p=2$, $V_d(x)=\langle x,Px\rangle_H$ rescaled by $\epsilon^{-1}$, and $\hat c=c_P/\epsilon$; Theorem~\ref{thm:min:nc}(iv)$\Rightarrow$(i) yields UGES over $\mathcal{S}_{\min}(T_{\min})$.
\end{proof}

\subsection{Range dwell-time ($\mathcal{S}_{\mathrm{rng}}(T_{\min},T_{\max})$)}\label{subsec:hilbert:rng}

For the range dwell-time family, the quadratic specialization of Theorem~\ref{thm:rng:nc} imposes an operator Lyapunov ODE along the flow on $[0,T_{\max}]$ together with a jump equation for every admissible period $\theta\in[T_{\min},T_{\max}]$; as in the minimum dwell-time case, it is only sufficient.

\begin{theorem}[Range dwell-time, quadratic sufficient]\label{thm:hilbert:rng}
Let $0<T_{\min}\le T_{\max}<\infty$, $H$ a Hilbert space, $A$ generate a $C_0$-semigroup on $H$, $J\in L(H)$, and let Assumption~\ref{ass:growth} hold. If there exist a self-adjoint, strongly differentiable $\bar P:[0,T_{\max}]\to L(H)$ with $0\prec\bar P(\tau)\preceq c_P I$ for some $c_P>0$, 
{$Q\in\mathcal{Q}([0,T_{\max}])$, and a self-adjoint $R:[T_{\min},T_{\max}]\to L(H)$ satisfying $R(\theta)\succeq\alpha_r I$ for some $\alpha_r>0$,} such that
    \begin{subequations}\label{eq:hilbert:rng:eq}
    \begin{align}
      \langle d,(\dot{\bar P}(\tau)+A^*\bar P(\tau)+\bar P(\tau)A+Q(\tau))d\rangle_H&=0,\quad d\in D(A),\;\tau\in(0,T_{\max}),\label{eq:hilbert:rng:flow}\\
      J^*\bar P(0)J-\bar P(\theta)+R(\theta)&=0,\quad\theta\in[T_{\min},T_{\max}],\label{eq:hilbert:rng:jump}
    \end{align}
    \end{subequations}
then the system \eqref{eq:syst} is strongly UGES over $\mathcal{S}_{\mathrm{rng}}(T_{\min},T_{\max})$.
\end{theorem}

\begin{proof}
Define $\bar V(\tau,x):=\langle x,\bar P(\tau)x\rangle_H$ on $[0,T_{\max}]\times H$. Then $0\le\bar V(\tau,x)\le c_P\|x\|^2$ and $\bar V$ is Lipschitz on bounded subsets of $H$. {Moreover, since $\bar P$ is strongly differentiable, and hence strongly continuous, on $[0,T_{\max}]$, the map $\tau\mapsto\bar V(\tau,x)$ is continuous on $[0,T_{\max}]$ for every $x\in H$.}

\emph{Flow \eqref{eq:rng:nc:flow}.} For $\tau\in[0,T_{\max})$, $\tau+h\le T_{\max}$, and $x\in D(A)$,
\[
  \tfrac{\d}{\ds}\bar V(\tau+s,S(s)x)=-\langle S(s)x,Q(\tau+s)S(s)x\rangle_H\le-\alpha_q\|S(s)x\|^2.
\]
Integrating from $0$ to $h$, extending to $H$ by density and continuity in $x$, dividing by $h$ and letting $h\to0^+$, strong continuity of $S(\cdot)$ at $s=0$ gives $\underline D^+\bar V(\tau,x)\le-\alpha_q\|x\|^2$.

\emph{Jump \eqref{eq:rng:nc:jump}.} For $\theta\in[T_{\min},T_{\max}]$ and $z\in H$, \eqref{eq:hilbert:rng:jump} gives
\[
  \bar V(0,Jz)-\bar V(\theta,z)=\langle z,[J^*\bar P(0)J-\bar P(\theta)]z\rangle_H=-\langle z,R(\theta)z\rangle_H\le-\alpha_r\|z\|^2.
\]

Setting $\gamma:=\min(\alpha_q,\alpha_r)>0$, the rescaled $\gamma^{-1}\bar V$ satisfies \eqref{eq:rng:nc:upper}--\eqref{eq:rng:nc:jump} of Theorem~\ref{thm:rng:nc} with $p=2$ and $c=c_P/\gamma$. Theorem~\ref{thm:rng:nc}(iii)$\Rightarrow$(i) yields UGES over $\mathcal{S}_{\mathrm{rng}}(T_{\min},T_{\max})$.
\end{proof}


\begin{corollary}[Range dwell-time, coercive quadratic]\label{cor:hilbert:rng:coerc}
If there exist constants $\alpha_m,c_P,\epsilon>0$ and a self-adjoint, strongly differentiable $\bar P:[0,T_{\max}]\to L(H)$ with $\alpha_m I\preceq\bar P(\tau)\preceq c_P I$ satisfying
\begin{equation*}
  \dot{\bar P}(\tau)+A^*\bar P(\tau)+\bar P(\tau)A\preceq -\epsilon I,\quad\tau\in(0,T_{\max}),\qquad J^*\bar P(0)J-\bar P(\theta)\preceq -\epsilon I,\quad\theta\in[T_{\min},T_{\max}],
\end{equation*}
then the system \eqref{eq:syst} is strongly UGES over $\mathcal{S}_{\mathrm{rng}}(T_{\min},T_{\max})$.
\end{corollary}

\begin{proof}
Define $\bar V(\tau,x):=\langle x,\bar P(\tau)x\rangle_H$ on $[0,T_{\max}]\times X$; it is Lipschitz on bounded subsets and $\bar V(\tau,x)\le c_P\|x\|_H^2$. {Moreover, since $\bar P$ is strongly differentiable, and hence strongly continuous, on $[0,T_{\max}]$, the map $\tau\mapsto\bar V(\tau,x)$ is continuous on $[0,T_{\max}]$ for every $x\in H$, as required in Theorem~\ref{thm:rng:nc}(iii).} For $x\in D(A)$, the flow inequality gives $\langle x,(\dot{\bar P}+A^*\bar P+\bar P A)x\rangle_H\le-\epsilon\|x\|_H^2$, hence $\bar V(\tau+h,S(h)x)-\bar V(\tau,x)\le-\epsilon\int_0^h\|S(s)x\|_H^2\,\ds$; this integrated bound extends from $D(A)$ to all of $H$ by density and continuity in $x$, and $h\downarrow0$ gives $\underline{D}^+\bar V(\tau,x)\le-\epsilon\|x\|_H^2$. The jump inequality gives, for every $\theta\in[T_{\min},T_{\max}]$, $\bar V(0,Jx)-\bar V(\theta,x)=\langle x,(J^*\bar P(0)J-\bar P(\theta))x\rangle_H\le-\epsilon\|x\|_H^2$ on all of $H$ ($J,\bar P$ bounded). These are the non-coercive timer conditions of Theorem~\ref{thm:rng:nc}(iii), so Theorem~\ref{thm:rng:nc}(iii)$\Rightarrow$(i) yields UGES over $\mathcal{S}_{\mathrm{rng}}(T_{\min},T_{\max})$. Only the coercive lower bounds are used; the generally unbounded form $\dot{\bar P}+A^*\bar P+\bar P A$ need not define a bounded operator.
\end{proof}

\begin{remark}[Coercivity and convex synthesis]\label{rem:coercive:synthesis}
Beyond guaranteeing an equivalent-norm Lyapunov functional, coercivity of the certificate $\bar P(\tau)$ is what makes the conditions amenable to \emph{controller design}. A coercive operator, satisfying $\bar P(\tau)\succeq\alpha_m I$ with $\alpha_m>0$, is boundedly invertible, so one may perform a congruence transformation by $\bar P(\tau)^{-1}$ (equivalently, work with the dual variable $\bar Q(\tau):=\bar P(\tau)^{-1}$) without leaving the space of bounded operators. When the jump operator depends affinely on a controller gain $K$, say $J(K)$, the analysis inequalities are bilinear in $(\bar P,K)$; the congruence by $\bar Q$ together with the linearizing change of variables $U:=K\bar Q$ renders them \emph{jointly affine} in the transformed variables $(\bar Q,U)$, with the gain recovered a posteriori as $K=U\bar Q^{-1}$. This dualization is the operator-space counterpart of the linearizing changes of variables that are standard in finite-dimensional LMI-based synthesis \cite{Bernussou:89,Boyd:94a,Scherer:97a,Scherer:05a,Oliveira:99}, and it is precisely the bounded invertibility furnished by coercivity, absent in the merely pointwise-positive non-coercive certificates discussed above, that permits it. The non-coercive conditions, by contrast, remain the sharp tool for \emph{analysis}.
\end{remark}

\section{Application to switched systems}\label{sec:switched}

Linear switched systems form an important subclass to which the present framework applies \emph{unchanged}, once they are recast as impulsive systems on a product space. We make this reduction explicit, show that the reformulated jumps are \emph{selectors} of unit operator norm, and derive from the persistent-flowing and non-coercive results of Sections~\ref{sec:lyapunov}--\ref{sec:hilbert} stability conditions for the fixed-signal, minimum-dwell-time, and arbitrary-switching cases, on both Banach and Hilbert spaces. Because the reformulated jump operator is not a single fixed map but ranges over a finite family contained in the closed unit ball, the conditions are stated \emph{robustly}, uniformly over that family. As the selectors have unit norm and the switched state is continuous across switches, the switches are non-expansive and all decay is supplied by the flow; the certificates below are accordingly \emph{persistent-flowing} (flow strict, discrete part non-increasing), and we state the conclusions as GES/UGES (Definitions~\ref{def:stab}--\ref{def:stabu}).

\subsection{Reformulation as an impulsive system}\label{subsec:sw:lift}

Consider the linear switched system
\begin{equation}\label{eq:sw:syst}
  \dot x(t)=A_{\eta(t)}x(t),\qquad x(0)=x_0,
\end{equation}
on a Banach space $X$, where the \emph{switching signal} $\eta:\mathbb{R}_{\ge0}\to\{1,\dots,N\}$ is piecewise constant and right-continuous, with the set of its discontinuities (the \emph{switching instants}) forming a sequence $\{t_k\}_{k\ge1}\in\mathcal{S}_{0,\infty}$, and $\eta\equiv i_k$ on $[t_k,t_{k+1})$ with $i_{k-1}\ne i_k$. Each $A_i$ ($i=1,\dots,N$) generates a $C_0$-semigroup $S_i(t)$ on $X$, and we assume the modes share a common domain
\begin{equation}\label{eq:sw:common:domain}
  \mathcal{D}:=D(A_1)=\cdots=D(A_N),
\end{equation}
which holds in particular whenever $A_i=A_0+B_i$ with a common generator $A_0$ and bounded $B_i\in L(X)$. We denote by $\Phi_\eta(t,s)$ the (jump-free) evolution family of \eqref{eq:sw:syst}, so that $x(t)=\Phi_\eta(t,s)x(s)$; on each mode-interval it coincides with the active semigroup, and it is continuous across switches. A \emph{switching family} $\Sigma$ prescribes the admissible signals: we write $\Sigma_{0,\infty}$ for \emph{arbitrary switching} (switching instants in $\mathcal{S}_{0,\infty}$, no further constraint) and $\Sigma_{\min}(\Tmin)$ for the \emph{minimum-dwell-time family} (consecutive switching instants at least $\Tmin$ apart).

Throughout this section, $\underline{D}^+_i$ denotes the Dini derivative along the flow of the $i$-th subsystem $\dot x=A_ix$: for a mode $i\in\{1,\dots,N\}$ and a functional $W$ carrying a time or timer argument,
\begin{equation}\label{eq:sw:dini}
  \underline{D}^+_i W(\cdot,x):=\liminf_{h\downarrow0}\tfrac1h\big(W(\cdot+h,S_i(h)x)-W(\cdot,x)\big),
\end{equation}
which is the lower right-hand Dini derivative of Section~\ref{sec:prelim} with the semigroup $S$ replaced by the $i$-th subsystem semigroup $S_i$; for a time-independent $W=W(x)$ this reads $\underline{D}^+_i W(x)=\liminf_{h\downarrow0}h^{-1}(W(S_i(h)x)-W(x))$. Hence, $\underline{D}^+_i W$ is ``the derivative of $W$ for subsystem $i$''.

\medskip\noindent\emph{Lifting.} Let $\mathcal{X}:=X^N$ carry the norm $\|z\|_{\mathcal{X}}:=\big(\sum_{i=1}^N\|z_i\|_X^p\big)^{1/p}$ for some $p\ge1$ (on a Hilbert space $H$ we take $p=2$ and $\mathcal{H}:=H^N$ with $\langle z,w\rangle:=\sum_{i=1}^N\langle z_i,w_i\rangle_H$). Define the block-diagonal generator and semigroup
\begin{equation}\label{eq:sw:lifted:gen}
  \mathcal{A}:=\diag(A_1,\dots,A_N),\quad D(\mathcal{A})=\mathcal{D}^N,\qquad \mathcal{S}(t)=\diag(S_1(t),\dots,S_N(t)),
\end{equation}
so that $\|\mathcal{S}(t)\|=\max_i\|S_i(t)\|$, and, for each ordered pair $(i,j)$ with $i\ne j$, the \emph{selector}
\begin{equation}\label{eq:sw:selector}
  J_{ij}:=(e_j e_i^{\!\top})\otimes I_X\in L(\mathcal{X}),\qquad (J_{ij}z)_l=\begin{cases} z_i,& l=j,\\ 0,& l\ne j,\end{cases}
\end{equation}
where $e_i$ is the $i$-th canonical basis vector of $\mathbb{R}^N$. The selector transfers the content of block $i$ into block $j$ and annihilates the remaining blocks; it satisfies
\begin{equation}\label{eq:sw:selector:norm}
  \|J_{ij}\|_{L(\mathcal{X})}=1\qquad\text{and}\qquad J_{ij}\,D(\mathcal{A})\subseteq D(\mathcal{A}),
\end{equation}
the inclusion by the common-domain hypothesis \eqref{eq:sw:common:domain}. We collect the selectors into
\begin{equation}\label{eq:sw:Jset}
  \mathcal{J}:=\{J_{ij}:1\le i,j\le N,\ i\ne j\}\subseteq \overline{\mathbb{B}},\qquad \overline{\mathbb{B}}:=\{J\in L(\mathcal{X}):\|J\|\le1\}.
\end{equation}
Given a signal $\eta$ with switching instants $\{t_k\}$ and modes $(i_k)_{k\ge0}$, its \emph{associated impulsive system} is \eqref{eq:syst} on $\mathcal{X}$ with $A$ replaced by $\mathcal{A}$ and, at each $t_k$, the jump operator $J$ replaced by $J_{i_{k-1}i_k}\in\mathcal{J}$, initialized at $z(0)=e_{i_0}\otimes x_0$.

\begin{proposition}[Lifting]\label{prop:sw:lift}
Let $x$ solve \eqref{eq:sw:syst} under a signal $\eta$ and let $z$ solve the associated impulsive system with $z(0)=e_{i_0}\otimes x_0$. Then
\begin{equation}\label{eq:sw:aligned}
  z(t)=e_{\eta(t)}\otimes x(t)\qquad\text{and}\qquad \|z(t)\|_{\mathcal{X}}=\|x(t)\|_X\qquad\text{for all }t\ge0.
\end{equation}
As a result, for any switching family $\Sigma$ with switching instants in $\mathcal{S}_{0,\infty}$, the switched system \eqref{eq:sw:syst} is (uniformly) GES over $\Sigma$ if and only if the associated impulsive system is (uniformly) GES over the corresponding family of impulse-sequence/selector pairs, with identical constants $M,\alpha,\rho$. Since $\|J_{ij}\|=1$, the reformulated system lies in the \emph{persistent-flowing regime} ($\rho=1$); whenever arbitrarily long mode-intervals are admissible each $S_i(t)$ is necessarily exponentially stable (Proposition~\ref{prop:necessity:semigroup}).
\end{proposition}
\begin{proof}
We induct over the intervals $[t_k,t_{k+1})$. Assume $z(t_k)=e_{i_k}\otimes x(t_k)$ (the base case $k=0$ holds by initialization). For $t\in[t_k,t_{k+1})$, since $\mathcal{S}(\cdot)$ is block-diagonal, $z(t)=\mathcal{S}(t-t_k)z(t_k)=e_{i_k}\otimes S_{i_k}(t-t_k)x(t_k)=e_{i_k}\otimes x(t)$, the remaining blocks mapping $0$ to $0$. At the switch $t_{k+1}$,
\[
  z(t_{k+1})=J_{i_ki_{k+1}}z(t_{k+1}^-)=J_{i_ki_{k+1}}\big(e_{i_k}\otimes x(t_{k+1}^-)\big)=e_{i_{k+1}}\otimes x(t_{k+1}^-)=e_{i_{k+1}}\otimes x(t_{k+1}),
\]
using continuity of $x$ at the switch. Exactly one block of $z(t)$ is nonzero, so $\|z(t)\|_{\mathcal{X}}=\|x(t)\|_X$, and the stability equivalence with identical constants follows from Definition~\ref{def:stab}. The persistent-flowing statement is \eqref{eq:sw:selector:norm} together with Proposition~\ref{prop:necessity:semigroup}.
\end{proof}

\medskip\noindent\emph{Robust certificates and the aligned reduction.} The reformulation produces, at each impulse, one selector from $\mathcal{J}$, determined by the realized mode transition. A Lyapunov certificate agnostic to the transition is one whose \emph{jump} condition holds for every $J\in\mathcal{J}$ simultaneously, the flow condition being unchanged; every proof in Sections~\ref{sec:lyapunov}--\ref{sec:hilbert} invokes the jump inequality only at the realized jump operator, so replacing ``$J$'' by ``for all $J\in\mathcal{J}$'' in the jump condition (and $\|J\|$ by $\sup_{J\in\mathcal{J}}\|J\|=1$ in the growth bound, so that Assumption~\ref{ass:growth} holds for the lifted system with $\mu=1$ whenever the modes are non-expansive) leaves each argument valid without change and yields stability uniformly over all selector sequences drawn from $\mathcal{J}$. Moreover, by \eqref{eq:sw:aligned} the reformulated trajectories evolve on the \emph{aligned set} $\bigcup_{i=1}^N(e_i\otimes X)$, which the selectors leave invariant; since the converse implications (iii)$\Rightarrow$(i) of the theorems below evaluate the certificate only along trajectories, it suffices to test the Lyapunov conditions on aligned states $z=e_i\otimes x$. On aligned states a \emph{block-diagonal} certificate reduces to mode-wise conditions, which is what makes the operator inequalities below decouple across modes.

\subsection{Fixed switching signal}\label{subsec:sw:fixed}

For a single, fully specified switching signal the lifted certificate collapses onto the original space. Indeed, along the reformulated dynamics all but one block of $z$ vanish, the nonzero one carrying the switched state $x$: applying Proposition~\ref{prop:sw:lift} from the running time $t$, the aligned state $z=e_{\eta(t)}\otimes x$ propagates as $U_\eta(s,t)z=e_{\eta(s)}\otimes\Phi_\eta(s,t)x$, so that
\begin{equation}\label{eq:sw:fixed:collapse}
  \|U_\eta(s,t)z\|_{\mathcal{X}}=\|\Phi_\eta(s,t)x\|_X\qquad\text{for all }s\ge t.
\end{equation}
Consequently the explicit persistent-flowing witness $\int_t^\infty\|U_\eta(s,t)z\|^p\,\ds$ depends on $z$ only through the active content $x$, and the selector at a switch, which merely \emph{relabels} the active block without altering $x$, reduces to a non-increase condition on a functional of $x$ alone. This yields a necessary and sufficient characterization stated directly on $X$, with a $t$-dependent (single) functional rather than one on the product space $\mathcal{X}$.

\begin{theorem}[Fixed signal, non-coercive N\&S]\label{thm:sw:fixed}
Let $\eta$ be a switching signal whose switching instants have positive infimal dwell-time $T_\eta:=\inf_{k\ge0}(t_{k+1}-t_k)>0$, let each $S_i$ be exponentially stable, and let Assumption~\ref{ass:growth} hold for the lifted system. The following statements are equivalent:
\begin{enumerate}[(i)]
  \item the switched system \eqref{eq:sw:syst} is GES along $\eta$;
  \item (\emph{Datko}) there exist $p\ge1$, $b>0$ such that $\int_s^\infty\|\Phi_\eta(t,s)x\|^p\,\dt\le b^p\|x\|^p$ for all $s\ge0$ and $x\in X$;
  \item there exist $p\ge1$, $c>0$ and a functional $V:\mathbb{R}_{\ge0}\times X\to\mathbb{R}_{\ge0}$, Lipschitz on bounded subsets of $X$ uniformly in $t$, {with $t\mapsto V(t,x)$ continuous on $[t_k,t_{k+1})$ for all $k\ge0$ and admitting a left limit at $t_k$ for all $k\ge1$, for every $x\in X$,} such that
  \begin{subequations}\label{eq:sw:fixed:cond}
  \begin{align}
    V(t,x)&\le c\|x\|^p, &&\forall t\ge0,\label{eq:sw:fixed:upper}\\
    \underline{D}^+_{\eta(t)}V(t,x)&\le-\|x\|^p, &&\forall t\notin\mathbb{T}_\eta,\label{eq:sw:fixed:flow}\\
    V(t_k,x)-V(t_k^-,x)&\le0, &&\forall t_k\in\mathbb{T}_\eta,\label{eq:sw:fixed:jump}
  \end{align}
  \end{subequations}
  hold for all $x\in X$, where $V(t_k^-,x):=\lim_{t\uparrow t_k}V(t,x)$.
\end{enumerate}
An explicit functional witnessing (iii) is $V(t,x)=\int_t^\infty\|\Phi_\eta(s,t)x\|^p\,\ds$.
\end{theorem}
\begin{proof}
By the collapse \eqref{eq:sw:fixed:collapse}, $\|U_\eta(s,t)(e_{\eta(t)}\otimes x)\|=\|\Phi_\eta(s,t)x\|$, so on the invariant aligned set the conditions \eqref{eq:sw:fixed:cond} and (ii) are exactly the persistent-flowing conditions \eqref{eq:pf:upper}--\eqref{eq:pf:jump} and the Datko condition \eqref{eq:datko:pf} of Theorem~\ref{thm:pf} for the lifted system $(\mathcal{X},\mathcal{A},\{J_{i_{k-1}i_k}\})$, whose dwell-time hypothesis is the standing $T_\eta>0$: the flow condition \eqref{eq:pf:flow} is \eqref{eq:sw:fixed:flow} for the active mode, and the neutral jump \eqref{eq:pf:jump} at the realized selector $J_{i_{k-1}i_k}$ is \eqref{eq:sw:fixed:jump} (the selector relabels the block without changing $x$). Since the lifted trajectories issued from aligned states remain aligned, Theorem~\ref{thm:pf} applies along them and Proposition~\ref{prop:sw:lift} transfers the conclusion to \eqref{eq:sw:syst}, giving (i)$\Leftrightarrow$(ii)$\Leftrightarrow$(iii); the witness \eqref{eq:V:pf:explicit} at $z=e_{\eta(t)}\otimes x$ is $\int_t^\infty\|\Phi_\eta(s,t)x\|^p\,\ds$. {The continuity requirement in (iii) is the one of Theorem~\ref{thm:pf}(iii), and, by \eqref{eq:sw:fixed:collapse}, the witness inherits it from \eqref{eq:V:pf:explicit}, for which it follows from the uniform continuity argument in the proof of Theorem~\ref{thm:main}.}
\end{proof}

On a Hilbert space the functional is quadratic and the reduction produces a \emph{single} operator-valued function $P(t)$ (not one operator per mode) the mode entering only through the active generator $A_{\eta(t)}$. No coercive jump weight appears: $P$ obeys the flow Lyapunov equation with a coercive $Q$ and merely does not increase across switches, $P(t_k)\preceq P(t_k^-)$, the natural certificate being in fact continuous there.

\begin{corollary}[Fixed signal, quadratic N\&S]\label{cor:sw:fixed:hilbert}
On a Hilbert space $H$, under the hypotheses of Theorem~\ref{thm:sw:fixed}, the switched system \eqref{eq:sw:syst} is GES along $\eta$ if and only if, for every $Q\in\mathcal{Q}$, there exists a self-adjoint $P:\mathbb{R}_{\ge0}\to L(H)$, 
{with $t\mapsto\langle x,P(t)x\rangle_H$ continuous on each $[t_k,t_{k+1})$ with a left limit at $t_{k+1}$ for all $x\in H$, and $t\mapsto\langle d,P(t)d\rangle_H$ right differentiable on each $[t_k,t_{k+1})$ for all $d\in\mathcal{D}$,} and $0\prec P(t)\preceq c_P I$ for some $c_P>0$, such that
\begin{subequations}\label{eq:sw:fixed:hilbert}
\begin{align}
  \langle d,\big(\dot P(t)+A_{\eta(t)}^*P(t)+P(t)A_{\eta(t)}+Q(t)\big)d\rangle_H&=0,\quad d\in D(A_{\eta(t)}),\ t\notin\mathbb{T}_\eta,\label{eq:sw:fixed:hilbert:flow}\\
  P(t_k)&\preceq P(t_k^-),\quad k\ge1.\label{eq:sw:fixed:hilbert:jump}
\end{align}
\end{subequations}
The explicit solution $P(t)=\int_t^\infty\Phi_\eta(s,t)^*Q(s)\Phi_\eta(s,t)\,\ds$ is continuous across the switches, so \eqref{eq:sw:fixed:hilbert:jump} then holds with equality.
\end{corollary}
\begin{proof}
This is the quadratic specialization of Theorem~\ref{thm:sw:fixed}. \emph{Sufficiency:} with $V(t,x)=\langle x,P(t)x\rangle_H$, the flow equation \eqref{eq:sw:fixed:hilbert:flow} gives $\underline{D}^+_{\eta(t)}V(t,x)\le-\alpha_q\|x\|^2$ (coercivity $Q\succeq\alpha_q I$; extension from $D(A_{\eta(t)})$ to $H$ by density{, as in the proof of Theorem~\ref{thm:hilbert:nc}, (ii)$\Rightarrow$(i), with $S$ replaced by $S_{\eta(t)}$}){, the continuity requirement of Theorem~\ref{thm:sw:fixed}(iii) being provided by the assumed regularity of $P$} and the non-increase \eqref{eq:sw:fixed:hilbert:jump} gives $V(t_k,x)\le V(t_k^-,x)$; these are the conditions of Theorem~\ref{thm:sw:fixed}(iii) with $p=2$, whence GES. \emph{Necessity:} given GES ($\|\Phi_\eta(t,s)\|\le Me^{-\alpha(t-s)}$), the operator $P(t)=\int_t^\infty\Phi_\eta(s,t)^*Q(s)\Phi_\eta(s,t)\,\ds$ is self-adjoint, satisfies $0\prec P(t)\preceq\frac{\alpha_Q M^2}{2\alpha}I$ (positivity from $Q\succeq\alpha_q I$ on a right-neighbourhood of $s=t$, boundedness from the integral bound), 
{has the required regularity and solves \eqref{eq:sw:fixed:hilbert:flow}, by the argument in the proof of Theorem~\ref{thm:hilbert:nc}, (i)$\Rightarrow$(ii), with $U_\sigma$ replaced by $\Phi_\eta$, $S$ by $S_{\eta(t)}$, and no summand,} and is continuous across the switches (dominated convergence), so \eqref{eq:sw:fixed:hilbert:jump} holds with equality.
\end{proof}

\subsection{Minimum dwell-time}\label{subsec:sw:min}

The minimum-dwell-time family is the richest dwell-time case: each mode must be exponentially stable, and a mode-dependent timer certificate is coupled across modes at each switch, the dwell-time $\Tmin>0$ supplying, through Corollary~\ref{cor:datko:int}, the uniform boundedness that turns the integral Datko bound into exponential decay. The Banach characterization is necessary and sufficient; the Hilbert one is sufficient, the quadratic class being strictly smaller.

\begin{theorem}[Minimum dwell-time, non-coercive N\&S]\label{thm:sw:min}
Let $\Tmin>0$, each $S_i$ be exponentially stable, and let Assumption~\ref{ass:growth} hold for the lifted system. The following statements are equivalent:
\begin{enumerate}[(i)]
  \item the switched system \eqref{eq:sw:syst} is UGES over $\Sigma_{\min}(\Tmin)$;
  \item there exist $p\ge1$, $c>0$ and mode-dependent timer functionals $\bar V_i:[0,\Tmin]\times X\to\mathbb{R}_{\ge0}$, $i=1,\dots,N$, each Lipschitz on bounded subsets of $X$ uniformly in $\tau$, {with $\tau\mapsto\bar V_i(\tau,x)$ continuous on $[0,\Tmin]$ for every $x\in X$,} such that for all $x\in X$, every mode $i$, and every $j\ne i$,
  \begin{subequations}\label{eq:sw:min:cond}
  \begin{align}
    \bar V_i(\tau,x)&\le c\|x\|^p, &&\tau\in[0,\Tmin],\label{eq:sw:min:upper}\\
    \underline{D}^+_i\bar V_i(\tau,x)&\le-\|x\|^p, &&\tau\in[0,\Tmin],\label{eq:sw:min:flow:nc}\\
    \bar V_j(0,x)&\le\bar V_i(\Tmin,x),\label{eq:sw:min:jump:nc}
  \end{align}
  \end{subequations}
  where the free phase $\tau\ge\Tmin$ is covered by the timer-independent extension $\bar V_i(\tau,\cdot):=\bar V_i(\Tmin,\cdot)$.
\end{enumerate}
An explicit witness for (ii) is the flow functional
\begin{equation}\label{eq:sw:min:V:explicit}
  \bar V_i(\tau,x):=\sup_{\eta\in\Sigma^{(\tau)}_i}\int_0^\infty\|\Phi_\eta(s,0)x\|^p\,\ds,
\end{equation}
where $\Sigma^{(\tau)}_i$ collects the admissible future signals starting in mode $i$ with elapsed dwell $\tau$ (first switch after $\max(\Tmin-\tau,0)$, subsequent dwells $\ge\Tmin$).
\end{theorem}
\begin{proof}
\emph{$(ii)\Rightarrow(i)$.} Telescoping the flow decrease \eqref{eq:sw:min:flow:nc} (mandatory and free phases) together with the switch non-increase \eqref{eq:sw:min:jump:nc} along any $\eta\in\Sigma_{\min}(\Tmin)$ (state continuous at switches{; the integrations of the Dini inequalities are justified as in the proof of Theorem~\ref{thm:min:nc}, using the continuity of $\bar V_i(\cdot,x)$ on $[0,\Tmin]$}) yields the integral Datko bound $\int_s^\infty\|\Phi_\eta(u,s)y\|^p\,\du\le c\|y\|^p$, uniformly in $\eta$ and $s$; since $T_\eta\ge\Tmin>0$, Corollary~\ref{cor:datko:int} then gives $\|\Phi_\eta(t,s)\|\le Me^{-\alpha(t-s)}$ with $M,\alpha$ depending only on $c,G,\Tmin$ (hence uniform in $\eta$), exactly as in Theorem~\ref{thm:sw:fixed}(ii)$\Rightarrow$(i). This is UGES over $\Sigma_{\min}(\Tmin)$. \emph{$(i)\Rightarrow(ii)$.} The flow witness \eqref{eq:sw:min:V:explicit} satisfies \eqref{eq:sw:min:cond}: finiteness and the upper bound $\bar V_i(\tau,x)\le\frac{M^p}{\alpha p}\|x\|^p$ follow from UGES; the mandatory/free flow from prefixing a mode-$i$ arc of length $h$ (which maps $\Sigma^{(\tau+h)}_i$ into $\Sigma^{(\tau)}_i$ and gives $\bar V_i(\tau,x)\ge\int_0^h\|S_i(u)x\|^p\,\du+\bar V_i(\tau+h,S_i(h)x)$); and the switch non-increase from prefixing an immediate switch $i\to j$ (which maps $\Sigma^{(0)}_j$ into $\Sigma^{(\Tmin)}_i$ leaving the flow integral unchanged, so $\bar V_i(\Tmin,x)\ge\bar V_j(0,x)$). {Finally, since $\Sigma^{(\tau)}_i\subseteq\Sigma^{(\tau')}_i$ for $\tau\le\tau'$, the continuity of $\tau\mapsto\bar V_i(\tau,x)$ on $[0,\Tmin]$ follows from the prefixing inequality above, exactly as in the proof of Theorem~\ref{thm:min:nc}, with $S$ replaced by $S_i$.}
\end{proof}

\begin{theorem}[Minimum dwell-time, quadratic]\label{thm:sw:min:hilbert}
Let $\Tmin>0$ and $H$ a Hilbert space. If there exist self-adjoint, strongly differentiable $\bar P_i:[0,\Tmin]\to L(H)$ with $0\prec\bar P_i(\tau)\preceq c_P I$, and self-adjoint coercive $Q_i:[0,\Tmin]\to L(H)$, $Q_i^\infty\in L(H)$ with $Q_i(\tau)\succeq\alpha_q I$, $Q_i^\infty\succeq\alpha_q^\infty I$ ($\alpha_q,\alpha_q^\infty>0$), such that for every mode $i$ and every $j\ne i$,
\begin{subequations}\label{eq:sw:min:hilbert}
\begin{align}
  \dot{\bar P}_i(\tau)+A_i^*\bar P_i(\tau)+\bar P_i(\tau)A_i+Q_i(\tau)&\preceq0,\quad \tau\in(0,\Tmin),\label{eq:sw:min:hflow}\\
  A_i^*\bar P_i(\Tmin)+\bar P_i(\Tmin)A_i+Q_i^\infty&\preceq0,\label{eq:sw:min:hfree}\\
  \bar P_j(0)&\preceq\bar P_i(\Tmin),\label{eq:sw:min:hjump}
\end{align}
\end{subequations}
(the flow inequalities holding as quadratic forms on $D(A_i)$), then \eqref{eq:sw:syst} is UGES over $\Sigma_{\min}(\Tmin)$.
\end{theorem}
\begin{proof}
With $\bar V_i(\tau,x)=\langle x,\bar P_i(\tau)x\rangle_H$ (extended by $\bar V_i(\Tmin,\cdot)$ for $\tau\ge\Tmin$){, which is continuous in $\tau$ on $[0,\Tmin]$ for every $x\in H$ since $\bar P_i$ is strongly differentiable, and hence strongly continuous}, the flow inequalities \eqref{eq:sw:min:hflow}--\eqref{eq:sw:min:hfree} give $\underline{D}^+_i\bar V_i(\tau,x)\le-\min(\alpha_q,\alpha_q^\infty)\|x\|^2$ (coercivity of $Q_i,Q_i^\infty$; extension from $D(A_i)$ to $H$ by density), and the non-increase \eqref{eq:sw:min:hjump} gives $\bar V_j(0,x)-\bar V_i(\Tmin,x)=\langle x,(\bar P_j(0)-\bar P_i(\Tmin))x\rangle_H\le0$; these are the conditions \eqref{eq:sw:min:cond} of Theorem~\ref{thm:sw:min}(ii) with $p=2$, whence UGES over $\Sigma_{\min}(\Tmin)$. The standing Assumption~\ref{ass:growth} holds since \eqref{eq:sw:min:hfree} forces each $S_i$ exponentially stable (integrate along $S_i$ and let $h\to\infty$: $\alpha_q^\infty\int_0^\infty\|S_i(r)x\|^2\,\dr\le\langle x,\bar P_i(\Tmin)x\rangle_H\le c_P\|x\|^2$; Datko's theorem \cite{Datko:70}).
\end{proof}

\begin{remark}\label{rem:sw:geromel}
Conditions \eqref{eq:sw:min:hilbert} are the infinite-dimensional, operator-valued counterpart of the clock-dependent minimum-dwell-time conditions of Geromel and Colaneri \cite{Geromel:06b} and Briat \cite{Briat:13d,Briat:16c}: a mode-dependent timer certificate $\bar P_i(\tau)$ \emph{strictly} decreasing along mode $i$ over the mandatory phase $[0,\Tmin]$ and the free phase (coercive $Q_i,Q_i^\infty$), coupled across modes at a switch by the non-increase \eqref{eq:sw:min:hjump}. Consistent with the state being continuous at switches, the discrete step is non-expansive; all strictness resides in the flow, the accumulated mandatory-phase contraction making $\bar P_i(\Tmin)$ dominate $\bar P_j(0)$. On $H=\mathbb{R}^n$ they are a finite family of LMIs in $(\bar P_i,Q_i,Q_i^\infty)$, amenable to the convex synthesis dualization of Remark~\ref{rem:coercive:synthesis} when the modes depend affinely on a controller gain.
\end{remark}

\subsection{Arbitrary switching}\label{subsec:sw:arb}

Under arbitrary switching the state is continuous at the switches, and stability is governed by a \emph{common} non-coercive Lyapunov functional decreasing along every mode. We assume throughout this subsection that the modes are uniformly non-expansive after a common equivalent renorming,
\begin{equation}\label{eq:sw:nonexpansive}
  \|S_i(t)\|\le1\qquad\text{for all }i\in\{1,\dots,N\}\text{ and }t\ge0,
\end{equation}
which is the specialization of Assumption~\ref{ass:growth} to $\mu=1$ (Remark~\ref{rem:growth:G}). For linear systems this entails no loss of generality for the arbitrary-switching UGES problem: UGES over $\Sigma_{0,\infty}$ is equivalent to the existence of a common norm renormalizing every mode into a non-expansive semigroup, the infinite-dimensional analogue of a Barabanov/extremal norm \cite{Chitour:25}.

\begin{theorem}[Arbitrary switching, non-coercive N\&S]\label{thm:sw:arb}
Under \eqref{eq:sw:common:domain} and \eqref{eq:sw:nonexpansive}, the following statements are equivalent:
\begin{enumerate}[(i)]
  \item the switched system \eqref{eq:sw:syst} is UGES over $\Sigma_{0,\infty}$ (arbitrary switching);
  \item there exist $p\ge1$, $c>0$ and a common functional $V:X\to\mathbb{R}_{\ge0}$, Lipschitz on bounded subsets, such that
  \begin{equation}\label{eq:sw:arb:cond}
    V(x)\le c\|x\|^p\quad\text{and}\quad \underline{D}^+_i V(x)\le-\|x\|^p\quad\text{for every mode }i.
  \end{equation}
\end{enumerate}
An explicit common functional witnessing (ii) is
\begin{equation}\label{eq:sw:arb:V}
  V(x):=\sup_{\eta}\int_0^\infty\|\Phi_\eta(t,0)x\|^p\,\dt,
\end{equation}
the supremum ranging over all admissible signals.
\end{theorem}
\begin{proof}
\emph{(ii)$\Rightarrow$(i).} Fix a signal $\eta$, $s\ge0$, $x_0\in X$, and set $x(t)=\Phi_\eta(t,s)x_0$. The state is continuous at the switches, and 
{the function $t\mapsto V(x(t))$ is continuous on $[s,\infty)$, since $V$ is Lipschitz on bounded subsets. On each mode-interval, its lower right Dini derivative is $\underline{D}^+_{\eta(t)}V(x(t))\le-\|x(t)\|^p$, so the continuous function $t\mapsto V(x(t))+\int_s^t\|x(r)\|^p\dr$ is nonincreasing \cite{Hagood:06}. Hence,} using $V\ge0$ and $V(x_0)\le c\|x_0\|^p$,
\begin{equation}\label{eq:sw:arb:datko}
  \int_s^\infty\|\Phi_\eta(t,s)x_0\|^p\,\dt\le c\|x_0\|^p,
\end{equation}
uniformly in $\eta$ and $s$. By \eqref{eq:sw:nonexpansive} the propagator, a composition of the maps $S_i(\cdot)$, satisfies $\|\Phi_\eta(t,s)\|\le1$. Steps~2A and~3A of the proof of Lemma~\ref{lem:datko}, which use only the integral bound \eqref{eq:sw:arb:datko} together with the uniform bound of Step~1 (here $C_0=1$) and involve no jump term, apply to the jump-free evolution family $\Phi_\eta$ (for which $\kappa_\eta\equiv0$) and yield $\|\Phi_\eta(t,s)\|\le M e^{-\alpha_0(t-s)}$ for some $M\ge1$, $\alpha_0>0$ independent of $\eta$; equivalently, this is the Datko theorem for the evolution family $\Phi_\eta$ \cite{Datko:70}. Hence \eqref{eq:sw:syst} is UGES over $\Sigma_{0,\infty}$.

\emph{(i)$\Rightarrow$(ii).} Define $V$ by \eqref{eq:sw:arb:V}. UGES gives $\int_0^\infty\|\Phi_\eta(t,0)x\|^p\,\dt\le(M^p/(\alpha p))\|x\|^p$ uniformly in $\eta$, so $V$ is finite with $V(x)\le c\|x\|^p$, $c:=M^p/(\alpha p)$; Lipschitz continuity on bounded subsets follows as in Theorem~\ref{thm:arbitrary:nc} from $|\,\|\Phi_\eta(t,0)x\|^p-\|\Phi_\eta(t,0)y\|^p\,|\le pM^p r^{p-1}\|x-y\|$ on $r$-balls, integrated in $t$ and taken to the supremum in $\eta$. For the flow condition, fix a mode $i$ and $h>0$. The admissible family is invariant under prepending a mode-$i$ arc of length $h$: for any $\eta$ let $\eta^{[i,h]}$ denote $\eta$ preceded by mode $i$ on $[0,h]$, so that $\Phi_{\eta^{[i,h]}}(t,0)=\Phi_\eta(t-h,0)S_i(h)$ for $t\ge h$ and $=S_i(t)$ for $t\le h$, whence
\[
  \int_0^\infty\|\Phi_{\eta^{[i,h]}}(t,0)x\|^p\,\dt=\int_0^h\|S_i(t)x\|^p\,\dt+\int_0^\infty\|\Phi_\eta(t,0)S_i(h)x\|^p\,\dt.
\]
Taking the supremum over $\eta$, and using that $\eta\mapsto\eta^{[i,h]}$ maps admissible signals into admissible signals,
\[
  V(x)\ge\int_0^h\|S_i(t)x\|^p\,\dt+V(S_i(h)x),
\]
so $V(x)-V(S_i(h)x)\ge\int_0^h\|S_i(t)x\|^p\,\dt$; dividing by $h$ and letting $h\downarrow0$, strong continuity of $S_i$ gives $\underline{D}^+_iV(x)\le-\|x\|^p$, for every mode $i$.
\end{proof}

\begin{remark}\label{rem:sw:haidar}
Theorem~\ref{thm:sw:arb} is the switched-system, common non-coercive Lyapunov characterization of arbitrary-switching UGES \cite{Haidar:22,Liberzon:03}, recovered here through the impulsive reformulation; the state being continuous at the switches, no jump condition appears, and the certificate need not be coercive. This is the arbitrary-switching analogue of the non-coercive necessary and sufficient conditions of Sections~\ref{sec:lyapunov}, in the same spirit as the fixed-sequence Theorem~\ref{thm:main}.
\end{remark}

On a Hilbert space the common functional may be taken quadratic, and the aligned collapse of Sections~\ref{subsec:sw:fixed}--\ref{subsec:sw:min} then \emph{forces} it to be governed by a single operator. Indeed, a block-diagonal quadratic certificate $\mathcal{V}(z)=\langle z,\mathcal{P}z\rangle$ with $\mathcal{P}=\diag(P_1,\dots,P_N)\succ0$ meets the persistent-flowing conditions on the lifted system (flow-strict along $\mathcal{A}$ and non-increasing at the selectors) only if the selector condition $J_{ij}^*\mathcal{P}J_{ij}-\mathcal{P}\preceq0$, evaluated at the aligned state $z=e_i\otimes x$, gives
\begin{equation}\label{eq:sw:arb:collapse}
  \langle x,P_j x\rangle_H\le\langle x,P_i x\rangle_H\quad\text{for all }i\ne j,\qquad\text{i.e.}\quad P_j\preceq P_i,
\end{equation}
so that applying \eqref{eq:sw:arb:collapse} to both orderings $(i,j)$ and $(j,i)$ forces $P_1=\cdots=P_N=:P$. The block certificate thus collapses to a \emph{single common} operator $P$ on $H$, and the per-mode flow-strict condition becomes \eqref{eq:sw:arb:hilbert}; the selector coupling is precisely what makes the common quadratic Lyapunov function unavoidable in this class, recovering the classical arbitrary-switching condition.

\begin{corollary}[Arbitrary switching, quadratic]\label{cor:sw:arb:hilbert}
On a Hilbert space $H$, suppose there exist a self-adjoint $P\in L(H)$ with $0\prec P\preceq c_P I$ and $\epsilon>0$ such that
\begin{equation}\label{eq:sw:arb:hilbert}
  \langle d,(A_i^*P+PA_i)d\rangle_H\le-\epsilon\|d\|^2\qquad\forall d\in D(A_i),\ \forall i.
\end{equation}
Then \eqref{eq:sw:syst} is UGES over $\Sigma_{0,\infty}$; if moreover $P\succeq\alpha_m I$ for some $\alpha_m>0$ (a coercive common quadratic Lyapunov function), the conclusion holds without the non-expansiveness hypothesis \eqref{eq:sw:nonexpansive}.
\end{corollary}
\begin{proof}
Set $V(x)=\langle x,Px\rangle_H$. Inequality \eqref{eq:sw:arb:hilbert} gives, along each mode, $\underline{D}^+_iV(x)\le-\epsilon\|x\|^2$ (extending from $D(A_i)$ to $H$ by density and continuity, as in Corollary~\ref{cor:hilbert:arbitrary:coerc}), and $V(x)\le c_P\|x\|^2$; Theorem~\ref{thm:sw:arb}(ii)$\Rightarrow$(i) with $p=2$ (after dividing by $\epsilon$) yields UGES. If $P\succeq\alpha_m I$, then $V$ is coercive, $\alpha_m\|x\|^2\le V(x)\le c_P\|x\|^2$, and the flow inequality integrates to $\frac{\d}{\dt}V(x(t))\le-\epsilon\|x(t)\|^2\le-(\epsilon/c_P)V(x(t))$ along any switching signal{, for $x(s)\in\mathcal{D}$, which is invariant under each $S_i$ by \eqref{eq:sw:common:domain}}; Gr\"onwall gives $V(x(t))\le e^{-(\epsilon/c_P)(t-s)}V(x(s))${, which extends to all $x(s)\in H$ by density of $\mathcal{D}$ and continuity of both sides}, whence $\|\Phi_\eta(t,s)x\|^2\le(c_P/\alpha_m)e^{-(\epsilon/c_P)(t-s)}\|x\|^2$ uniformly in $\eta$, i.e.\ UGES, without recourse to \eqref{eq:sw:nonexpansive}.
\end{proof}

\section{Examples}
\label{sec:examples}

We illustrate the abstract theory through four examples. Examples~\ref{ex:jump} and~\ref{ex:flow} apply the timer-dependent Hilbert-space conditions of Sections~\ref{subsec:hilbert:cst}--\ref{subsec:hilbert:rng} to two PDE systems involving a transport equation: one where stability is driven by the jumps, and one where it is driven by the flow. The third example in Section~\ref{ex:pie} introduces a new application, the sampled-data control of a linear time-delay system, which is reformulated as an impulsive system on an infinite-dimensional state space; the analysis conditions from Sections~\ref{subsec:hilbert:cst}--\ref{subsec:hilbert:rng} are then applied explicitly, and the coercive certificates lend themselves to the convex synthesis dualization of Remark~\ref{rem:coercive:synthesis}. Example~\ref{ex:boundary} exploits the same Partial Integral Equation machinery in a more demanding setting: the sampled-data \emph{boundary} control of a reaction-diffusion equation, where the input operator is unbounded in the conventional state space but becomes bounded in the PIE representation, the conversion being made rigorous by the PIE lift of Proposition~\ref{prop:pie:lift}.

\subsection{Example 1: Jump-induced stability of a transport equation}\label{ex:jump}

Consider the advection equation on $[0,1]$ with periodic boundary condition and periodic multiplicative jumps:
\begin{equation}\label{eq:transport1}
  \partial_t x(t,z)=-\partial_z x(t,z),\quad z\in(0,1),\quad x(t,0)=x(t,1),\quad x(kT^+,z)=\vartheta(z)  x(kT^-,z),
\end{equation}
where $\vartheta\in W^{1,\infty}(0,1)$ with $\vartheta(0)=\vartheta(1)$ and $\vartheta^\star:=\operatorname{ess\,sup}_{z\in[0,1]}|\vartheta(z)|$. The state space is $H=L^2(0,1)$, with the transport operator $A:=-\partial_z$ on $D(A)=\{x\in H^1(0,1):x(0)=x(1)\}$ generating the periodic-shift semigroup $S(t)x(z)=x(z-t\bmod 1)$. The jump operator is $Jx:=\vartheta(\cdot)  x$; the requirements $\vartheta\in W^{1,\infty}$ and $\vartheta(0)=\vartheta(1)$ ensure $J\in L(H)$ with $JD(A)\subseteq D(A)$, as assumed in \eqref{eq:syst}.

The unforced flow ($J=I$) is marginally stable: $\|S(t)\|=1$ for all $t\ge0$, since the shift is an isometry on $L^2(0,1)$. Stability must therefore be induced entirely by the jumps.

\paragraph{Analysis via Theorem~\ref{thm:hilbert:cst}.} Choose the timer-dependent operator $\bar P(\tau):=e^{-\alpha\tau}I$ for some $\alpha>0$.

\emph{Bounds:} $e^{-\alpha T}I\preceq\bar P(\tau)\preceq I$ for $\tau\in[0,T]$, giving $\alpha_m=e^{-\alpha T}$ and $c_P=1$.

\emph{Flow equation \eqref{eq:hilbert:cst:flow:ode}:} For $d\in D(A)$,
\begin{align}
  \langle d,(\dot{\bar P}(\tau)+A^*\bar P(\tau)+\bar P(\tau)A)d\rangle
  &=e^{-\alpha\tau}\left(-\alpha\|d\|^2+\langle d,Ad\rangle+\langle Ad,d\rangle\right)\notag\\
  &=e^{-\alpha\tau}\left(-\alpha\|d\|^2-2  \langle d,\partial_z d\rangle\right)\notag\\
  &=e^{-\alpha\tau}\left(-\alpha\|d\|^2+|d(0)|^2-|d(1)|^2\right)=-\alpha e^{-\alpha\tau}\|d\|^2
\end{align}
where we used integration by parts and the periodic boundary condition $d(0)=d(1)$. Hence \eqref{eq:hilbert:cst:flow:ode} holds with the bounded coercive weight $Q(\tau)=\alpha e^{-\alpha\tau}I\succeq\alpha e^{-\alpha T}I$.

\emph{Jump equation \eqref{eq:hilbert:cst:jump:eq}:} For $x\in H$,
\begin{align}
  \langle Jx,\bar P(0)Jx\rangle-\langle x,\bar P(T^-)x\rangle
  &=\int_0^1\vartheta(z)^2x(z)^2\dz-e^{-\alpha T}\int_0^1x(z)^2\dz\notag\\
  &=\int_0^1(\vartheta(z)^2-e^{-\alpha T})x(z)^2\dz.
\end{align}
This is $\le-\delta\|x\|^2$ with $\delta:=e^{-\alpha T}-(\vartheta^\star)^2$ provided
\begin{equation}\label{eq:ex1:cond}
  e^{-\alpha T}>(\vartheta^\star)^2,\qquad\text{i.e.,}\quad T<\frac{-2\ln\vartheta^\star}{\alpha}.
\end{equation}
Since the flow boundary terms cancel here (periodic condition $d(0)=d(1)$), the weight $Q(\tau)=\alpha e^{-\alpha\tau}I$ and the jump residual $R=(e^{-\alpha T}-\vartheta(\cdot)^2)I$ are \emph{bounded} and coercive, i.e.\ $Q\in\mathcal{Q}{([0,T])}$ and $R\in\mathcal{R}$; the necessary-and-sufficient Theorem~\ref{thm:hilbert:cst} therefore applies directly, and the system \eqref{eq:transport1} is GES whenever \eqref{eq:ex1:cond} holds. {The same stability condition $\vartheta^\star<1$ is obtained from \cite[Theorem~1 and Remark~2]{Dashkhovskiy:23}, which reduces here to \cite[Theorem~3]{Dashkovskiy:13}, with the coercive Lyapunov function $V(x)=\|x\|^2$, which is constant along the flow and satisfies $V(Jx)\le(\vartheta^\star)^2V(x)$.}

\begin{remark}
The certificate $\bar P(\tau)=e^{-\alpha\tau}I$ is conservative: the parameter $\alpha>0$ is a free weighting rate, not a property of the dynamics, so condition \eqref{eq:ex1:cond} should be read by optimizing over $\alpha$. Since $\vartheta^\star$ is fixed, taking $\alpha\to0^+$ makes the bound $-2\ln\vartheta^\star/\alpha\to\infty$, so \emph{any} period $T>0$ is admitted as soon as $\vartheta^\star<1$; conversely a larger $\alpha$ tightens the admissible $T$. This matches the exact stability mechanism: the shift semigroup is an isometry ($\|S(t)\|=1$) and each period contracts the state by exactly $\vartheta^\star$, so the true necessary and sufficient condition is simply $\vartheta^\star<1$, \emph{independently of $T$}. The $T$--$\alpha$ trade-off in \eqref{eq:ex1:cond} is therefore an artifact of fixing $\alpha$ rather than a genuine restriction. In the degenerate limit $\vartheta^\star\to0$ (fully contractive jump) the conclusion is immediate for any $T>0$.
\end{remark}

Figure~\ref{fig:ex1} shows the space-time evolution of the state for two values of the jump gain, computed by exact integration along the characteristics with periodic boundary feeding and a multiplicative reset $x(kT^+,z)=\vartheta^\star x(kT^-,z)$ at each period $T=0.3$. For $\vartheta^\star=0.70<1$ the solution decays to zero, while for $\vartheta^\star=1.15>1$ it grows without bound, in agreement with the condition $\vartheta^\star<1$ identified above and independently of the period; for the constant reset gain used in the simulation, this condition is also necessary, since $\|U_{\sigma_T}(kT,0)\|=(\vartheta^\star)^k$.

\begin{figure}[t]
\centering
\includegraphics[width=0.48\linewidth]{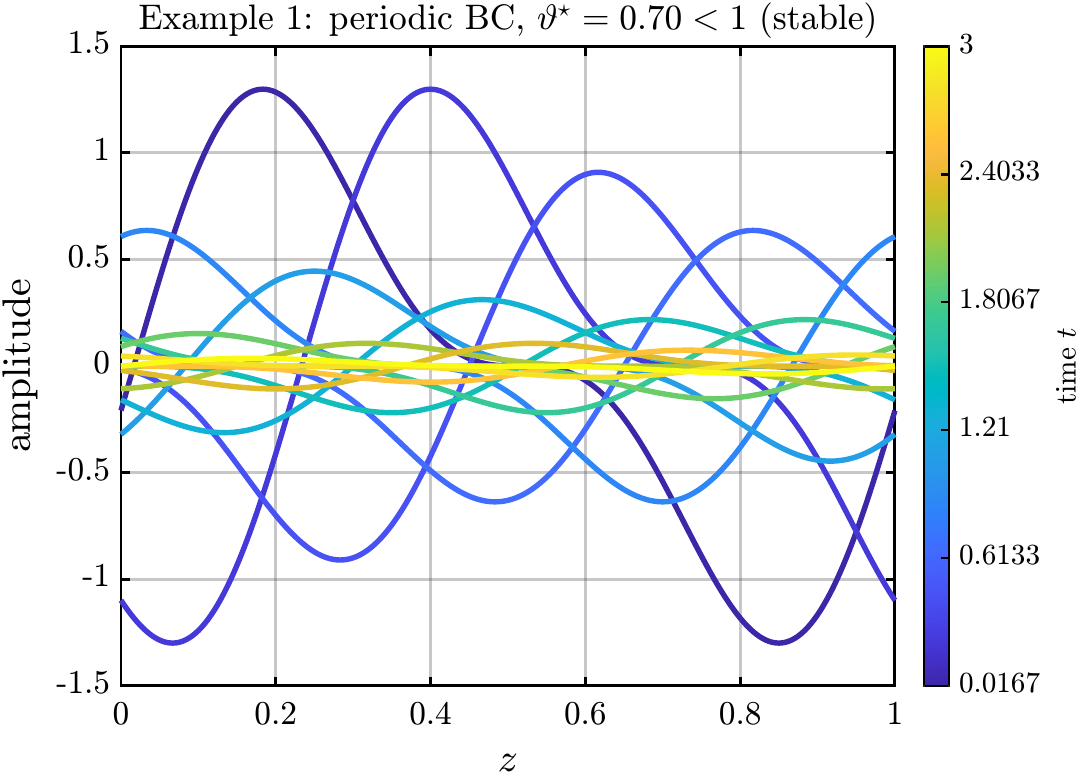}\hfill
\includegraphics[width=0.48\linewidth]{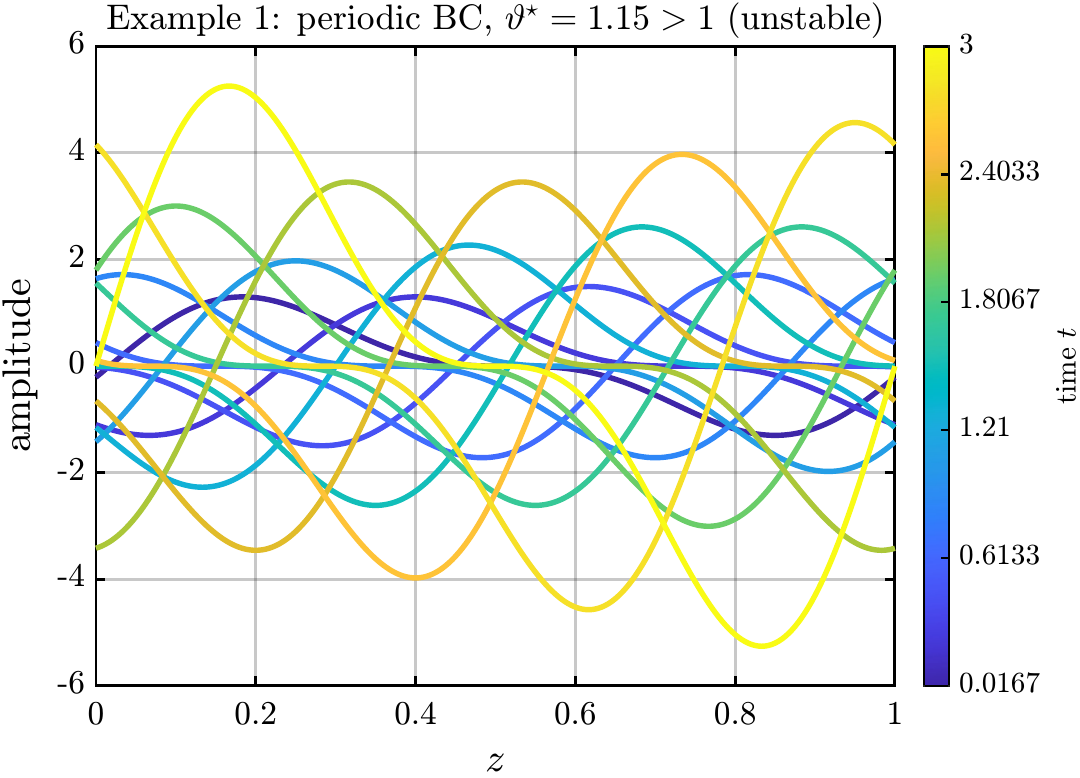}
\caption{Example~\ref{ex:jump} (jump-induced stability, periodic boundary condition, $T=0.3$). Spatial profiles $x(\cdot,z)$ at successive times, colored by time from early (dark) to late (bright), the time scale being given by the colorbar. Left: $\vartheta^\star=0.70<1$, the profile decays to zero (UGES). Right: $\vartheta^\star=1.15>1$, it grows. The isometric transport flow neither contracts nor expands; stability is decided entirely by the jumps.}\label{fig:ex1}
\end{figure}

\subsection{Example 2: Flow-induced stability with potentially unstable jumps}\label{ex:flow}

Modify \eqref{eq:transport1} by replacing the periodic boundary condition with a dissipative one:
\begin{equation}\label{eq:transport2}
  \partial_t x(t,z)=-\partial_z x(t,z),\quad x(t,0)=\nu   x(t,1),\quad x(kT^+,z)=\vartheta(z)  x(kT^-,z),
\end{equation}
where $\nu\in(-1,1)\setminus\{0\}$ (boundary absorption), $\vartheta\in W^{1,\infty}(0,1)$ with $\vartheta(0)=\vartheta(1)$ (so that $JD(A)\subseteq D(A)$ with $D(A)=\{x\in H^1(0,1):x(0)=\nu x(1)\}$), and $\vartheta^\star:=\operatorname{ess\,sup}_{z}|\vartheta(z)|$ \emph{may exceed} $1$ (potentially destabilizing jumps). We work with the weighted inner product $\langle x,y\rangle_\mu:=\int_0^1e^{-\mu z}x(z)y(z)\dz$ for $\mu>0$, whose induced norm $\|\cdot\|_\mu$ is equivalent to the $L^2$-norm.

\paragraph{Analysis via Theorem~\ref{thm:hilbert:rng} (range dwell-time).} Choose the timer-dependent weight $\bar P(\tau):=e^{\beta\tau}I$, where $\beta>0$ is a rate to be selected (the operator is a scalar multiple of the identity; the weighting lives in the inner product $\langle\cdot,\cdot\rangle_\mu$).

\emph{Flow condition \eqref{eq:hilbert:rng:flow}.} {For $d\in D(A)$, we have}
\begin{equation}
  {\langle d,(\dot{\bar P}(\tau)+A^*\bar P(\tau)+\bar P(\tau)A)d\rangle_\mu=\beta e^{\beta\tau}\|d\|_\mu^2+2e^{\beta\tau}\langle d,-\partial_z d\rangle_\mu.}
\end{equation}
For the boundary term, integration by parts against the weight gives
\begin{align}
  2\langle d,-\partial_z d\rangle_\mu
  &= -\int_0^1 e^{-\mu z}\partial_z(d^2)\dz
  = -\bigl[e^{-\mu z}d^2\bigr]_0^1 - \mu\int_0^1 e^{-\mu z}d^2\dz\\
  &= \bigl(d(0)^2-e^{-\mu}d(1)^2\bigr) - \mu\|d\|_\mu^2 .
\end{align}
Using the dissipative boundary condition $d(0)=\nu d(1)${, satisfied by any $d\in D(A)$}, so that $d(0)^2=\nu^2d(1)^2$,
\begin{equation}\label{eq:ex2:bdy}
  2\langle d,-\partial_z d\rangle_\mu
  = \bigl(\nu^2-e^{-\mu}\bigr)d(1)^2 - \mu\|d\|_\mu^2 .
\end{equation}
We now \emph{tune the weight} by choosing $\mu:=-2\ln|\nu|=2\ln\tfrac{1}{|\nu|}>0$, so that $e^{-\mu}=\nu^2$ and the boundary coefficient $\nu^2-e^{-\mu}$ vanishes. The boundary point-evaluation then disappears identically from \eqref{eq:ex2:bdy}:
\begin{equation}\label{eq:ex2:murange}
  2\langle d,-\partial_z d\rangle_\mu=-\mu\|d\|_\mu^2,
\end{equation}
a \emph{bounded} coercive relation: the weight is chosen so that the boundary outflow exactly balances the transport. As a result, for any $\beta\in(0,\mu)$,
\begin{equation}
  {\langle d,(\dot{\bar P}(\tau)+A^*\bar P(\tau)+\bar P(\tau)A)d\rangle_\mu=(\beta-\mu)e^{\beta\tau}\|d\|_\mu^2,\qquad d\in D(A),}
\end{equation}
and the associated weight $Q(\tau):=-(\dot{\bar P}+A^*\bar P+\bar P A)=(\mu-\beta)e^{\beta\tau}I$ is a bounded, coercive multiple of the identity, i.e.\ $Q\in\mathcal{Q}{([0,T_{\max}])}$; no boundary point-evaluation term survives. The bounded-weight hypotheses of the (sufficient) Theorem~\ref{thm:hilbert:rng} therefore hold directly. The flow supplies a contraction margin of order $\mu-\beta$; the bulk of the decay is produced at the jumps.

\emph{Jump condition \eqref{eq:hilbert:rng:jump}.} Since $\bar P(0)=I$ and $\bar P(\theta_0)=e^{\beta\theta_0}I$, for any dwell-time $\theta_0\in[T_{\min},T_{\max}]$,
\begin{equation}
  \langle Jx,\bar P(0)Jx\rangle_\mu - \langle x,\bar P(\theta_0)x\rangle_\mu
  =\int_0^1 e^{-\mu z}\bigl(\vartheta(z)^2-e^{\beta\theta_0}\bigr)x(z)^2\dz .
\end{equation}
This is $\le-\delta\|x\|_\mu^2$ provided $\vartheta(z)^2-e^{\beta\theta_0}\le-\delta$ for a.e.\ $z$ and all $\theta_0\in[T_{\min},T_{\max}]$. As $e^{\beta\theta_0}$ increases in $\theta_0$, the worst case is $\theta_0=T_{\min}$, so it suffices that $(\vartheta^\star)^2<e^{\beta T_{\min}}$, with margin $\delta=e^{\beta T_{\min}}-(\vartheta^\star)^2>0$. Equivalently,
\begin{equation}\label{eq:ex2:cond}
  (\vartheta^\star)^2 < e^{\beta T_{\min}},\qquad\text{i.e.,}\quad T_{\min}>\frac{2\ln\vartheta^\star}{\beta}.
\end{equation}
Taking $\beta\to\mu^-$ (any $\beta$ strictly less than $\mu=-2\ln|\nu|$) makes the bound least restrictive, approaching the dwell-time requirement $T_{\min}>\dfrac{\ln\vartheta^\star}{\ln(1/|\nu|)}$. Under \eqref{eq:ex2:cond} with any $\beta\in(0,-2\ln|\nu|)$, Theorem~\ref{thm:hilbert:rng} gives UGES for all $\sigma\in\mathcal{S}_{\mathrm{rng}}(T_{\min},T_{\max})$ and any $T_{\max}<\infty$. {The same dwell-time condition is obtained from \cite[Theorem~1 and Remark~2]{Dashkhovskiy:23}, which reduces here to \cite[Theorem~1]{Dashkovskiy:13}, with the coercive Lyapunov function $V(x)=\|x\|_\mu^2$, which satisfies $V(Jx)\le(\vartheta^\star)^2V(x)$ and, by \eqref{eq:ex2:murange}, decreases along the flow with rate $\mu$.}

\begin{remark}
Condition \eqref{eq:ex2:cond} captures the balance between the flow dissipation (boundary absorption $|\nu|<1$, quantified by the admissible rate $\beta<\mu=-2\ln|\nu|$) and the jump expansion ($\vartheta^\star$, possibly $>1$): the minimum dwell-time $T_{\min}$ must be large enough for the flow to overcome the destabilizing jumps, $T_{\min}>2\ln\vartheta^\star/\beta$. When $\vartheta^\star\le1$ the jumps are already non-expansive and any $T_{\min}>0$ works. This is a minimum-dwell-time condition, dual to the short-period condition of Example~\ref{ex:jump}.
\end{remark}

Figure~\ref{fig:ex2} illustrates the dual mechanism with dissipative boundary $\nu=0.5$ and expansive jumps $\vartheta^\star=1.2>1$, for which the threshold predicted by \eqref{eq:ex2:cond} is $T_{\min}^\star=\ln\vartheta^\star/\ln(1/|\nu|)\approx0.263$. With period $T=0.50>T_{\min}^\star$ the boundary absorption has enough time between jumps to overcome the expansion and the state decays, whereas with $T=0.15<T_{\min}^\star$ the jumps win and the state grows, confirming the minimum-dwell-time nature of the condition.

\begin{figure}[t]
\centering
\includegraphics[width=0.48\linewidth]{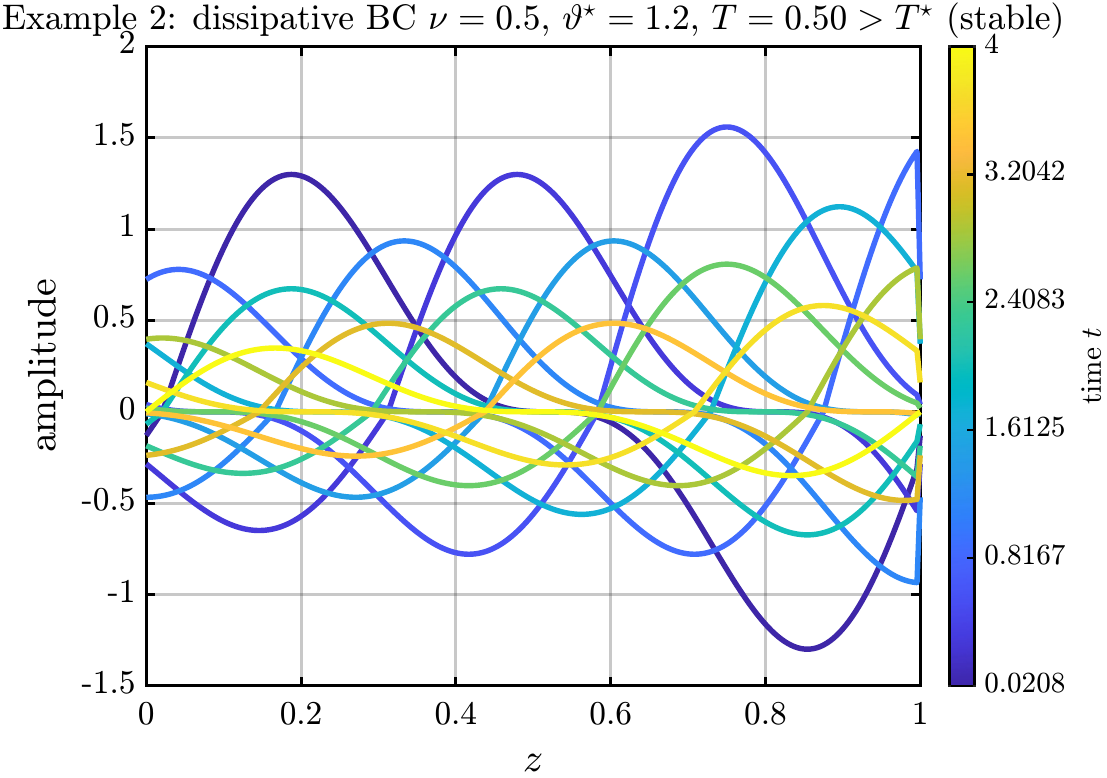}\hfill
\includegraphics[width=0.48\linewidth]{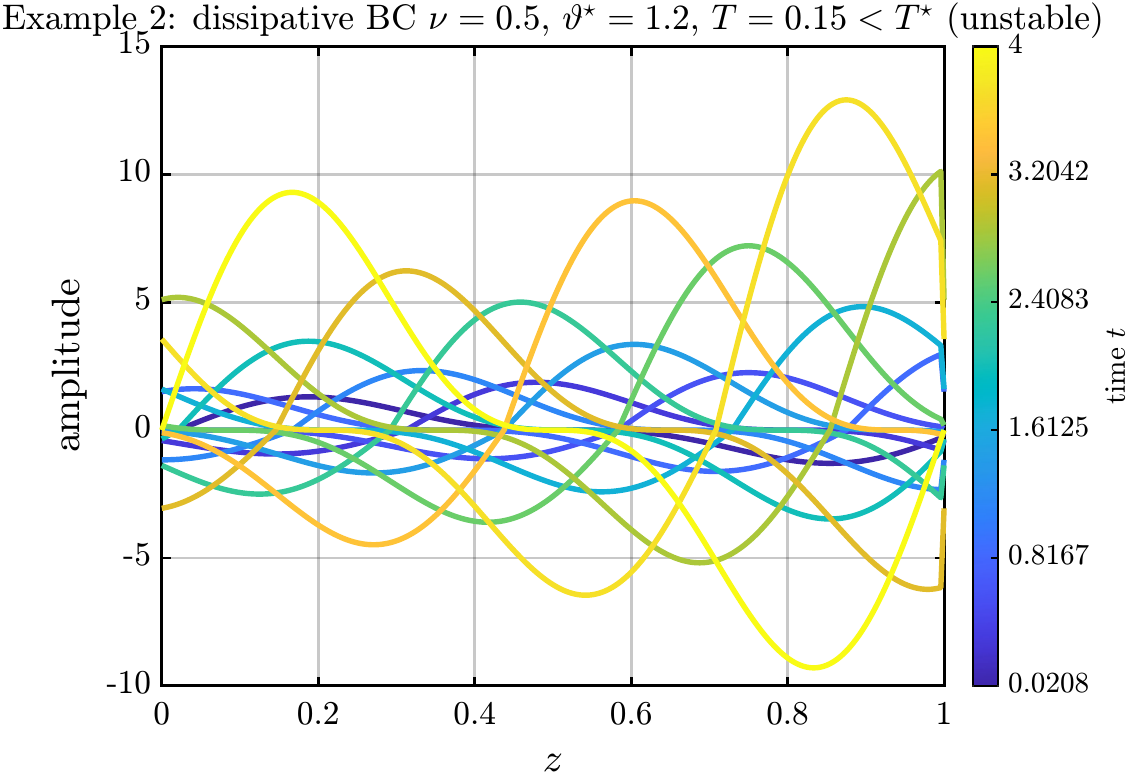}
\caption{Example~\ref{ex:flow} (flow-induced stability, dissipative boundary $\nu=0.5$, expansive jumps $\vartheta^\star=1.2$). Spatial profiles $x(\cdot,z)$ at successive times, colored by time from early (dark) to late (bright), the time scale being given by the colorbar. Left: $T=0.50>T_{\min}^\star\approx0.263$, the flow dissipation dominates and the profile converges. Right: $T=0.15<T_{\min}^\star$, the expansive jumps dominate and it diverges.}\label{fig:ex2}
\end{figure}

\subsection{Example 3: Impulsive time-delay systems in the Partial Integral Equation (PIE) framework}\label{ex:pie}

We now show that the hybrid conditions of Sections~\ref{subsec:nc:cst}--\ref{subsec:nc:rng} admit an exact operator-algebraic formulation in the Partial Integral Equation (PIE) framework of Peet and co-workers \cite{Peet:21pie,Shivakumar:21,Peet:20}, in which the Lyapunov-Krasovskii functional is parametrized by a positive \emph{Partial Integral (PI) operator} and both the flow and the jump conditions become Linear PI Inequalities (LPIs) amenable to \texttt{PIETOOLS} \cite{PIETOOLS:24}. The contribution of this example is to extend that framework, developed for non-impulsive systems, to the impulsive (hybrid) setting through a \emph{timer-dependent} PI operator.

\paragraph{State space and notation for partial integral operators.}
Throughout this example we work on the Hilbert space
\begin{equation*}
  H = \mathbb{R}^n \times L^2([-\tau,0];\mathbb{R}^n),
\end{equation*}
whose elements are pairs $\mathbf{x} = \mathrm{col}(v,\phi)$ with finite-dimensional component $v \in \mathbb{R}^n$ and infinite-dimensional component $\phi \in L^2([-\tau,0];\mathbb{R}^n)$. The latter plays the role of a \emph{history function}: $\phi(\theta)$ is the value of the function at offset $\theta \in [-\tau,0]$ from the current time.

A \emph{four-parameter partial integral (4-PI) operator} on $H$ is a bounded linear operator specified by four data:
\begin{itemize}
  \item a constant matrix $P \in \mathbb{R}^{n \times n}$ acting on the finite-dimensional component;
  \item a measurable map $Q_1 : [-\tau,0] \to \mathbb{R}^{n \times n}$ feeding the distributed component into the finite-dimensional output through an integral;
  \item a measurable map $Q_2 : [-\tau,0] \to \mathbb{R}^{n \times n}$ feeding the finite-dimensional component into the distributed output as a multiplication;
  \item a triple of measurable maps $\{R_0, R_1, R_2\}$ acting on the distributed component, with $R_0 : [-\tau,0] \to \mathbb{R}^{n \times n}$ a pointwise multiplier, and $R_1, R_2 : [-\tau,0]^2 \to \mathbb{R}^{n \times n}$ two kernels.
\end{itemize}
With these data, the 4-PI operator
\begin{equation}\label{eq:pie:op}
  \mathcal{P}\!\left[\begin{smallmatrix}P & Q_1\\ Q_2 & \{R_0,R_1,R_2\}\end{smallmatrix}\right]
\end{equation}
is defined to send $\mathbf{x}=\mathrm{col}(v,\phi)$ to the element of $H$ whose finite-dimensional output is
\begin{equation*}
  Pv + \int_{-\tau}^0 Q_1(\theta)\,\phi(\theta)\dtheta
\end{equation*}
and whose distributed output, evaluated at $s \in [-\tau,0]$, is
\begin{equation*}
  Q_2(s)v \;+\; R_0(s)\phi(s) \;+\; \int_{-\tau}^s R_1(s,\theta)\,\phi(\theta)\dtheta \;+\; \int_s^0 R_2(s,\theta)\,\phi(\theta)\dtheta.
\end{equation*}
The integration variable in both integrals is $\theta \in [-\tau,0]$; the upper integration limit is the running point $s$ for the $R_1$-term and the right endpoint $0$ for the $R_2$-term. The triple $\{R_0,R_1,R_2\}$ thus encodes \emph{three independent actions} on the distributed component: $R_0$ is a pointwise (no-integral) multiplier, $R_1$ integrates over $[-\tau,s]$ (history below the running point), and $R_2$ integrates over $[s,0]$ (history above the running point). Concretely:
\begin{itemize}
  \item the symbol $\{I,0,0\}$ means $R_0 = I$, $R_1 = 0$, $R_2 = 0$, so the action on the distributed component is the identity $\phi \mapsto \phi$;
  \item the symbol $\{0,0,-I\}$ means $R_0 = 0$, $R_1 = 0$, $R_2 = -I$, so the action is $\phi \mapsto -\int_s^0 \phi(\theta)\dtheta$, the running integral from $s$ to $0$ with a minus sign.
\end{itemize}
The PI operators form a $*$-algebra: sums, compositions and adjoints of PI operators are PI operators with explicitly computable parameters \cite{Shivakumar:21}.

\paragraph{The PIE representation: from delay equation to bounded operators.}
For the linear delay system
\begin{equation}\label{eq:pie:delay}
  \dot{x}(t)=A_0 x(t)+A_1 x(t-\tau),\qquad x(t)\in\mathbb{R}^n,
\end{equation}
the state at time $t$ in the classical Krasovskii sense is the pair $\mathrm{col}(x(t), x_t)$, where $x_t \in L^2([-\tau,0];\mathbb{R}^n)$ is defined by $x_t(\theta) := x(t+\theta)$. The generator of this dynamics on $\mathbb{R}^n \times L^2([-\tau,0];\mathbb{R}^n)$ is an unbounded operator (because differentiation in $\theta$ on the distributed part is unbounded). The PIE approach circumvents this by changing the state variable.

The key identity is: for any $s \in [-\tau,0]$,
\begin{equation}\label{eq:pie:identity}
  x_t(s) = x(t+s) = x(t) - \int_s^0 \frac{\d}{\dr}\big[x(t+r)\big]\dr = x(t) - \int_s^0 \dot{x}(t+r)\dr,
\end{equation}
where $r$ is the integration variable running from $s$ up to $0$. Writing $\dot{x}_t(r) := \dot{x}(t+r)$ for the derivative history, the identity rewrites as $x_t(s) = x(t) - \int_s^0 \dot{x}_t(r)\dr$. Applied with $s = -\tau$ it gives the delayed value
\begin{equation}\label{eq:pie:delay:int}
  x(t-\tau) = x(t) - \int_{-\tau}^0 \dot{x}_t(r)\dr,
\end{equation}
again with $r \in [-\tau,0]$ the integration variable. Substituting \eqref{eq:pie:delay:int} into \eqref{eq:pie:delay} yields
\begin{equation}\label{eq:pie:delay:rewritten}
  \dot{x}(t) = (A_0 + A_1)\,x(t) - A_1 \int_{-\tau}^0 \dot{x}_t(r)\dr,
\end{equation}
an equation in which the delay no longer appears as an evaluation $x(t-\tau)$ but as a bounded integral of the derivative history $\dot{x}_t$.

The PIE adopts the \emph{fundamental state}
\begin{equation*}
  \mathbf{x}(t) := \mathrm{col}\!\big(x(t),\,\dot{x}_t\big) \in H,
\end{equation*}
i.e.\ the pair (current value, derivative history). In these coordinates the original Krasovskii state is recovered by the PI operator
\begin{equation}\label{eq:pie:T}
  \mathcal{T} = \mathcal{P}\!\left[\begin{smallmatrix}I & 0\\ I & \{0,0,-I\}\end{smallmatrix}\right],
\end{equation}
which acts as $\mathcal{T}\mathbf{x} = \mathrm{col}\!\big(v,\,\phi^{(2)}\big)$ with
\begin{equation*}
  v = I\cdot x(t) + 0 = x(t),\qquad \phi^{(2)}(s) = I\cdot x(t) + 0\cdot\dot{x}_t(s) + 0 + \int_s^0 (-I)\dot{x}_t(r)\dr = x(t) - \int_s^0 \dot{x}_t(r)\dr = x_t(s)
\end{equation*}
by \eqref{eq:pie:identity}. Hence $\mathcal{T}\mathbf{x} = \mathrm{col}(x(t),\, x_t)$: the operator $\mathcal{T}$ reconstructs the original Krasovskii state from the fundamental state. The dynamics in the fundamental coordinates are described by the operator
\begin{equation}\label{eq:pie:A}
  \mathcal{A} = \mathcal{P}\!\left[\begin{smallmatrix}A_0 + A_1 & -A_1\\ 0 & \{I,0,0\}\end{smallmatrix}\right],
\end{equation}
whose action on $\mathbf{x} = \mathrm{col}(x(t),\dot{x}_t)$ is, by definition \eqref{eq:pie:op},
\begin{equation*}
  \mathcal{A}\mathbf{x} = \mathrm{col}\!\Big(\underbrace{(A_0+A_1)x(t) - A_1\!\int_{-\tau}^0 \dot{x}_t(r)\dr}_{=\,\dot{x}(t)\text{ by \eqref{eq:pie:delay:rewritten}}},\;\;\underbrace{0 + I\cdot \dot{x}_t(s) + 0 + 0}_{=\,\dot{x}_t(s)}\Big),
\end{equation*}
i.e.\ $\mathcal{A}\mathbf{x} = \mathrm{col}(\dot{x}(t),\,\dot{x}_t)$. The first block of $\{I,0,0\}$ here is the pointwise multiplier $R_0 = I$ that returns $\dot{x}_t(s)$ from itself; the zeros mean no integral over $[-\tau,s]$ and no integral over $[s,0]$. With these definitions the delay system \eqref{eq:pie:delay} is equivalent to the \emph{partial integral equation}
\begin{equation}\label{eq:pie:flow}
  \mathcal{T}\dot{\mathbf{x}}(t)=\mathcal{A}\mathbf{x}(t),
\end{equation}
whose left- and right-hand sides are now governed by the \emph{bounded} operators $\mathcal{T}$ and $\mathcal{A}$ on $H$, in contrast with the unbounded generator $A = \mathcal{T}^{-1}\mathcal{A}$ that would arise in the original Krasovskii formulation of Section~\ref{sec:hilbert}. The operator $\mathcal{T}$ is bounded and injective but not boundedly invertible, which is what allows \eqref{eq:pie:flow} to carry the dynamics through bounded operators.

\paragraph{Hybrid PIE and the timer-dependent PI Lyapunov functional.}
We endow the delay system with periodic jumps at $t_k=kT$: at each $t_k$ the fundamental state is reset by a bounded PI operator $\mathcal{J}$, so that the closed loop is the hybrid PIE
\begin{equation}\label{eq:pie:hybrid}
  \left\{
  \begin{aligned}
    \mathcal{T}\dot{\mathbf{x}}(t)&=\mathcal{A}\mathbf{x}(t), && t\ne t_k,\\
    \mathbf{x}(t_k)&=\mathcal{J}\mathbf{x}(t_k^-), && t_k=kT,
  \end{aligned}
  \right.
\end{equation}
with the jump operator
\begin{equation}\label{eq:pie:jump}
  \mathcal{J}=\mathcal{P}\!\left[\begin{smallmatrix}J_0 & 0\\ 0 & \{I,0,0\}\end{smallmatrix}\right],
\end{equation}
whose action on the fundamental state $\mathbf{x}=\mathrm{col}(x,\dot x_t)$ is $\mathcal{J}\mathbf{x}=\mathrm{col}(J_0 x,\dot x_t)$: the matrix $J_0\in\mathbb{R}^{n\times n}$ resets the instantaneous state (e.g.\ $J_0=I+BK$ for a sampled actuation update) while the distributed component $\dot x_t$ is left unchanged. The triple $\{I,0,0\}$ here means: identity multiplier on the distributed component, no $R_1$ or $R_2$ integrals, hence $\dot x_t\mapsto\dot x_t$. Because the reconstruction $\mathcal{T}$ rebuilds the physical history $x_t$ from the (unchanged) instantaneous state and derivative history, and $J_0$ acts trivially on the current value $x$ in the sampled-data setting, the physical Krasovskii state $\mathcal{T}\mathbf{x}=\mathrm{col}(x,x_t)$ jumps consistently as $\mathcal{T}\mathbf{x}(t_k)=\mathcal{T}\mathcal{J}\mathbf{x}(t_k^-)$, resetting only the finite actuation channel while the delay buffer varies continuously. Writing the jump on the fundamental state $\mathbf{x}$, rather than on the physical state $\mathcal{T}\mathbf{x}$, is what keeps the conditions expressed through the bounded operators $\mathcal{T},\mathcal{A},\mathcal{J}$ and matches the abstract impulsive model \eqref{eq:syst}. The system \eqref{eq:pie:hybrid} is an instance of \eqref{eq:syst} on $H$ with $\mathbb{T}_\sigma=\{kT\}$, hence $\sigma\in\mathcal{S}_{\mathrm{cst}}(T)$. We take the timer-dependent Lyapunov-Krasovskii functional
\begin{equation}\label{eq:pie:LKF}
  V(\tau,\mathbf{x}):=\big\langle \mathcal{T}\mathbf{x},\,\mathcal{P}(\tau)\,\mathcal{T}\mathbf{x}\big\rangle_H,\qquad \tau=t-t_k\in[0,T],
\end{equation}
where $\mathcal{P}(\tau)$ is, for each $\tau$, a positive PI operator. In contrast with the non-impulsive PIE stability theory, where a single constant $\mathcal{P}$ suffices, the timer dependence of $\mathcal{P}(\tau)$ is what enables the flow phase to be non-contractive (or even expansive, when $A_0$ is unstable) while the jumps restore stability, exactly the mechanism of Theorem~\ref{thm:hilbert:cst}.

\paragraph{Hybrid LPI conditions.}
Along \eqref{eq:pie:flow} the unbounded generator is eliminated through $\frac{\d}{\dt}(\mathcal{T}\mathbf{x})=\mathcal{T}\dot{\mathbf{x}}=\mathcal{A}\mathbf{x}$, so
\begin{equation}\label{eq:pie:Vdot}
  \tfrac{\d}{\dt}V
  =\big\langle\mathbf{x},\big(\mathcal{T}^{*}\dot{\mathcal{P}}(\tau)\mathcal{T}+\mathcal{A}^{*}\mathcal{P}(\tau)\mathcal{T}+\mathcal{T}^{*}\mathcal{P}(\tau)\mathcal{A}\big)\mathbf{x}\big\rangle_H .
\end{equation}
The following proposition turns this identity into verifiable hybrid LPIs. It is the mechanism, used throughout the PIE examples, by which the timer-dependent operator Lyapunov theory of Section~\ref{sec:hilbert} is applied to a system written in the implicit PIE form $\mathcal{T}\dot{\mathbf{x}}=\mathcal{A}\mathbf{x}$ without ever forming the unbounded generator $A$: all conditions are bounded operator inequalities on the PI operators $\mathcal{T},\mathcal{A},\mathcal{J}$ and the certificate $\mathcal{P}(\tau)$.

\begin{proposition}[PIE lift, constant dwell-time]\label{prop:pie:lift}
Let $H$ be a Hilbert space and let $\mathcal{T},\mathcal{A},\mathcal{J}\in L(H)$ be bounded operators such that $\mathcal{T}$ is injective with range $\operatorname{ran}\mathcal{T}=D(A)$, where $A$ generates a $C_0$-semigroup on $H$, and $\mathcal{A}=A\mathcal{T}$ on $H$. Consider the hybrid PIE \eqref{eq:pie:hybrid} with $\mathbb{T}_\sigma=\{kT\}$, $\sigma\in\mathcal{S}_{\mathrm{cst}}(T)$, and let $z(t):=\mathcal{T}\mathbf{x}(t)$ be the physical state. If there exist a timer-dependent self-adjoint PI operator $\mathcal{P}(\tau)$ with $\alpha_m I\preceq\mathcal{P}(\tau)\preceq c_P I$ for some $\alpha_m,c_P>0$ and all $\tau\in[0,T]$, and scalars $\gamma,\delta>0$, such that, for all $\tau\in(0,T)$,
\begin{subequations}\label{eq:pie:lift:lpi}
\begin{align}
  \mathcal{T}^{*}\dot{\mathcal{P}}(\tau)\mathcal{T}+\mathcal{A}^{*}\mathcal{P}(\tau)\mathcal{T}+\mathcal{T}^{*}\mathcal{P}(\tau)\mathcal{A}
  &\;\preceq\;-\gamma\,\mathcal{T}^{*}\mathcal{P}(\tau)\mathcal{T},\label{eq:pie:flow:lpi}\\
  \mathcal{J}^{*}\mathcal{T}^{*}\mathcal{P}(0)\mathcal{T}\mathcal{J}-\mathcal{T}^{*}\mathcal{P}(T^-)\mathcal{T}
  &\;\preceq\;-\delta\,\mathcal{T}^{*}\mathcal{T},\label{eq:pie:jump:lpi}
\end{align}
\end{subequations}
then the physical state $z$ is strongly UGES over $\mathcal{S}_{\mathrm{cst}}(T)$. The certificate is the timer-dependent quadratic functional $V(\tau,\mathbf{x})=\langle\mathcal{T}\mathbf{x},\mathcal{P}(\tau)\mathcal{T}\mathbf{x}\rangle_H=\langle z,\mathcal{P}(\tau)z\rangle_H$, equivalently $\bar P(\tau)=\mathcal{T}^{*}\mathcal{P}(\tau)\mathcal{T}$ in the operator Lyapunov conditions of Section~\ref{sec:hilbert}.
\end{proposition}

\begin{proof}
Along \eqref{eq:pie:hybrid}, $z=\mathcal{T}\mathbf{x}\in\operatorname{ran}\mathcal{T}=D(A)$ and, since $\mathcal{T}$ is bounded, $\dot z=\mathcal{T}\dot{\mathbf{x}}=\mathcal{A}\mathbf{x}=A\mathcal{T}\mathbf{x}=Az$ on each interval $(t_k,t_{k+1})$; at a jump, $z(t_k)=\mathcal{T}\mathbf{x}(t_k)=\mathcal{T}\mathcal{J}\mathbf{x}(t_k^-)$. Hence, $z$ is the state of an impulsive system of the form \eqref{eq:syst} with generator $A$, jump applied through $\mathcal{T}\mathcal{J}$, and $\sigma\in\mathcal{S}_{\mathrm{cst}}(T)$. With $V(\tau,\mathbf{x})=\langle z,\mathcal{P}(\tau)z\rangle_H$, the bounds $\alpha_m I\preceq\mathcal{P}(\tau)\preceq c_P I$ give $\alpha_m\|z\|^2\le V(\tau,\mathbf{x})\le c_P\|z\|^2$, so $V$ is coercive in $z$. Using \eqref{eq:pie:Vdot}, the flow LPI \eqref{eq:pie:flow:lpi} yields $\frac{\d}{\dt}V\le-\gamma\langle\mathbf{x},\mathcal{T}^{*}\mathcal{P}(\tau)\mathcal{T}\mathbf{x}\rangle_H=-\gamma V$ on each inter-jump interval, the unbounded generator $A$ never appearing since $\mathcal{T}\dot{\mathbf{x}}=\mathcal{A}\mathbf{x}$ replaces it by the bounded $\mathcal{A}$. At a jump, using $\mathbf{x}(t_k)=\mathcal{J}\mathbf{x}(t_k^-)$,
\[
  V(0,\mathbf{x}(t_k))-V(T^-,\mathbf{x}(t_k^-))
  =\big\langle\mathbf{x}(t_k^-),\big(\mathcal{J}^{*}\mathcal{T}^{*}\mathcal{P}(0)\mathcal{T}\mathcal{J}-\mathcal{T}^{*}\mathcal{P}(T^-)\mathcal{T}\big)\mathbf{x}(t_k^-)\big\rangle_H
  \le-\delta\|\mathcal{T}\mathbf{x}(t_k^-)\|_H^2,
\]
by \eqref{eq:pie:jump:lpi}. With the coercive weight $\mathcal{P}(\tau)$ ($\alpha_m I\preceq\mathcal{P}(\tau)\preceq c_P I$) and the functional $V(\tau,z)=\langle z,\mathcal{P}(\tau)z\rangle_H$, these are exactly the coercive flow and jump inequalities \eqref{eq:hilbert:cst:flow:LMI} for the physical state $z$; Corollary~\ref{cor:hilbert:cst:coerc} then yields UGES of $z$ over $\mathcal{S}_{\mathrm{cst}}(T)$. (Since $\mathcal{P}$ is coercive on $z$, no boundedness of $-(\dot{\mathcal{P}}+A^*\mathcal{P}+\mathcal{P}A)$ is needed, and no use is made of $\mathcal{T}^{-1}$, which is in general unbounded.)
\end{proof}

\paragraph{Computation.}
Positivity of $\mathcal{P}(\tau)$ is enforced, following \cite{Shivakumar:21}, through the cone of PI operators of the form $\mathcal{P}(\tau)=\mathcal{Z}^{*}H(\tau)\mathcal{Z}$ with $H(\tau)\succeq0$ a matrix and $\mathcal{Z}$ a fixed vector of monomial PI operators; after a polynomial parametrization of $\tau\mapsto H(\tau)$, the conditions \eqref{eq:pie:flow:lpi}--\eqref{eq:pie:jump:lpi} become a finite semidefinite program solvable in \texttt{PIETOOLS} \cite{PIETOOLS:24}. An advantage of this method over the use of simpler Lyapunov-Krasovskii functionals used in the analysis and the control of time-delay systems \cite{Niculescu:01,GuKC:03,Briat:book1} and over the Lyapunov-Krasovskii and looped-functional approaches developed specifically for sampled-data and aperiodically sampled systems \cite{Fridman:04,Fridman:10,Liu:12,Seuret:12,Seuret:10,Davo:17} is that the flow inequality \eqref{eq:pie:flow:lpi} is \emph{exact}: it carries the full Lyapunov-Krasovskii functional through the PI-operator algebra with no Jensen, Wirtinger \cite{Seuret:13b}, or integral-inequality relaxation, so the only remaining source of conservatism is the degree of the polynomial parametrization of $\mathcal{P}(\tau)$, which can be increased to approach the necessary and sufficient conditions of Theorem~\ref{thm:cst:nc}.

\paragraph{Sampled-data control of a delay system: state delay versus input delay.}
We instantiate the construction on the sampled-data stabilization of a delay system, with the held control $u(t)=v(t)$, $v(t)=Kx(t_k)$ on $[t_k,t_{k+1})$, in the two configurations where the delay $h$ acts on the state or on the control input. This sampled-data stabilization problem is classically addressed through the input-delay approach \cite{Fridman:04,Fridman:10}, in which the zero-order hold is modelled as a sawtooth delay; the impulsive reformulation used here instead keeps the held input as an explicit state component reset at each sample. In both, the held input $v\in\mathbb{R}^m$ is a finite-dimensional component, frozen along the flow ($\dot{v}=0$) and reset at each sample by $v(t_k)=Kx(t_k)$. The finite-dimensional part of the augmented state is $\zeta=\mathrm{col}(x,v)\in\mathbb{R}^{n+m}$. The difference between the two cases is which signal carries the delay history: in Case 1 the history is the plant-state derivative $\dot{x}_t$ (an $n$-vector function), in Case 2 it is the held-input derivative $\dot{v}_t$ (an $m$-vector function).

The decisive observation is that the \emph{hybrid jump operator is identical in the two cases}: it is the finite-rank reset
\begin{equation}\label{eq:pie:Jsd}
  \mathcal{J}=\mathcal{P}\!\left[\begin{smallmatrix}\left[\begin{smallmatrix}I_n & 0\\ K & 0\end{smallmatrix}\right] & 0\\[2pt] 0 & \{I,0,0\}\end{smallmatrix}\right],
  \qquad \mathcal{J}:\ \mathrm{col}(x,v,\varphi)\mapsto\mathrm{col}(x,Kx,\varphi),
\end{equation}
whose finite-dimensional block, the matrix $\left[\begin{smallmatrix}I_n & 0\\ K & 0\end{smallmatrix}\right]\in\mathbb{R}^{(n+m)\times(n+m)}$ acting on $\zeta=\mathrm{col}(x,v)$, leaves the plant state $x$ unchanged and resets the held input as $v \leftarrow Kx$ (note that the second block-column is zero, so the previous held value of $v$ is discarded); the distributed block $\{I,0,0\}$ leaves the history $\varphi$ unchanged across the sample, only its boundary value being updated through the subsequent flow. The two configurations differ \emph{only} in the flow PIE $\mathcal{T}\dot{\mathbf{x}}=\mathcal{A}\mathbf{x}$.

\emph{Case 1 (delay on the state):} $\dot{x}=A_0x(t)+A_1x(t-h)+Bv$. The carried history is the plant-state history $x_t \in L^2([-h,0];\mathbb{R}^n)$. The fundamental state is $\mathbf{x}=\mathrm{col}(x,v,\dot{x}_t)$. Applying the identity \eqref{eq:pie:identity} with integration variable $r \in [-h,0]$ gives $x_t(s) = x(t) - \int_s^0 \dot{x}_t(r)\dr$ and, with $s=-h$, $x(t-h) = x(t) - \int_{-h}^0 \dot{x}_t(r)\dr$. Substituting into the dynamics yields
\begin{equation*}
  \dot{x}(t) = (A_0 + A_1)\,x(t) + B v - A_1\!\int_{-h}^0 \dot{x}_t(r)\dr,
\end{equation*}
which in PIE form $\mathcal{T}\dot{\mathbf{x}} = \mathcal{A}\mathbf{x}$ corresponds to
\begin{equation}\label{eq:pie:case1}
  \mathcal{T}=\mathcal{P}\!\left[\begin{smallmatrix}I_{n+m} & 0\\[2pt] [\,I_n\ \ 0\,] & \{0,0,-I\}\end{smallmatrix}\right],\qquad
  \mathcal{A}=\mathcal{P}\!\left[\begin{smallmatrix}\left[\begin{smallmatrix}A_0+A_1 & B\\ 0 & 0\end{smallmatrix}\right] & \left[\begin{smallmatrix}-A_1\\ 0\end{smallmatrix}\right]\\[4pt] [\,0\ \ 0\,] & \{I,0,0\}\end{smallmatrix}\right].
\end{equation}
Reading the blocks: the finite-dimensional matrix in $\mathcal{A}$ acts on $\zeta=\mathrm{col}(x,v)$ producing $\mathrm{col}((A_0+A_1)x+Bv,\,0)$ (the second row is zero because $\dot{v}=0$ along the flow); the $Q_1$-block $\mathrm{col}(-A_1,0)$ acts on $\dot{x}_t$ through the integral $\int_{-h}^0 (-A_1)\dot{x}_t(r)\dr$, feeding the delayed-state term into the $x$-equation; the distributed block $\{I,0,0\}$ in $\mathcal{A}$ returns $\dot{x}_t(s)$ from itself, encoding the tautological identity $\dot{x}_t(s) = \dot{x}(t+s)$ as in \eqref{eq:pie:A}. The $\mathcal{T}$ operator has finite-dimensional block $I_{n+m}$ (identity on $\zeta$) and $Q_2$-block $[I_n\ \ 0]$ acting on $\zeta=\mathrm{col}(x,v)$ to return $I_n x = x$, while $\{0,0,-I\}$ converts $\dot{x}_t$ into the history offset $-\int_s^0 \dot{x}_t(r)\dr$; combined, these give $x_t(s) = x - \int_s^0 \dot{x}_t(r)\dr$, so $\mathcal{T}\mathbf{x} = \mathrm{col}(x,v,x_t)$, the original Krasovskii state augmented with the held input.

\emph{Case 2 (delay on the control input):} $\dot{x}=A_0x(t)+Bu(t-h)=A_0x(t)+Bv(t-h)$. The carried history is now the held-input history $v_t \in L^2([-h,0];\mathbb{R}^m)$, with $v_t(0)=v$ and the delayed input given, via the same identity \eqref{eq:pie:identity} applied to $v_t$, by $v(t-h)=v - \int_{-h}^0 \dot{v}_t(r)\dr$ (integration variable $r \in [-h,0]$). Substituting,
\begin{equation*}
  \dot{x}(t) = A_0\,x(t) + B v - B\!\int_{-h}^0 \dot{v}_t(r)\dr.
\end{equation*}
The fundamental state is $\mathbf{x}=\mathrm{col}(x,v,\dot{v}_t)$ and the PIE operators become
\begin{equation}\label{eq:pie:case2}
  \mathcal{T}=\mathcal{P}\!\left[\begin{smallmatrix}I_{n+m} & 0\\[2pt] [\,0\ \ I_m\,] & \{0,0,-I\}\end{smallmatrix}\right],\qquad
  \mathcal{A}=\mathcal{P}\!\left[\begin{smallmatrix}\left[\begin{smallmatrix}A_0 & B\\ 0 & 0\end{smallmatrix}\right] & \left[\begin{smallmatrix}-B\\ 0\end{smallmatrix}\right]\\[4pt] [\,0\ \ 0\,] & \{I,0,0\}\end{smallmatrix}\right].
\end{equation}
The differences from \eqref{eq:pie:case1} are exactly three: (i) the $Q_2$-block of $\mathcal{T}$ is now $[0\ \ I_m]$ instead of $[I_n\ \ 0]$, so $\mathcal{T}$ reconstructs $v_t(s)= v - \int_s^0 \dot{v}_t(r)\dr$ from the held-input component $v$ (and not from $x$); (ii) the open-loop matrix in the finite block of $\mathcal{A}$ is $A_0$ rather than $A_0+A_1$, since there is no instantaneous delayed-state term; (iii) the $Q_1$-block carrying the integral term is $\mathrm{col}(-B,0)$ rather than $\mathrm{col}(-A_1,0)$, since the integral now represents the delayed input $B v(t-h)$ rather than the delayed state $A_1 x(t-h)$.

In both cases $\mathcal{T},\mathcal{A},\mathcal{J}$ are bounded PI operators, $\sigma\in\mathcal{S}_{\mathrm{cst}}(T)$, and the hybrid LPIs \eqref{eq:pie:flow:lpi}--\eqref{eq:pie:jump:lpi} with the timer-dependent positive PI operator $\mathcal{P}(\tau)$ certify UGES over $\mathcal{S}_{\mathrm{cst}}(T)$ by Corollary~\ref{cor:hilbert:cst:coerc}. The sampled-data hybrid mechanism, namely the finite-rank jump \eqref{eq:pie:Jsd}, is therefore common to state-delayed and input-delayed plants; the location of the delay manifests solely as a different bounded $\mathcal{A}$ and a different attachment block in $\mathcal{T}$, with no change to the impulsive machinery. The input-delay case, classically the harder of the two because the delay sits inside the control loop, requires no special treatment in PIE coordinates.

\paragraph{Numerical comparison.}
For the scalar data $A_0=1$, $A_1=0.5$, $B=1$, $K=-1.6$, $h=0.3$, the exact stability limit under constant sampling is the largest $T$ for which $\rho\big(\mathcal{J}\,S(T)\big)<1$, the one-period monodromy condition of Theorem~\ref{thm:hilbert:cst}(iii). It equals $T^\star\approx2.1873$ for the state-delay configuration \eqref{eq:pie:case1} and $T^\star\approx0.8802$ for the input-delay configuration \eqref{eq:pie:case2} (the latter tolerates a smaller period because the control action sees the plant only through the delay). The hybrid LPIs \eqref{eq:pie:flow:lpi}--\eqref{eq:pie:jump:lpi}, solved with a polynomial parametrization of $\mathcal{P}(\tau)$ of degree $N$, return a guaranteed bound $T^\star_{\mathrm{LPI}}\le T^\star$; the conservatism is reduced by raising $N$, which enlarges the cone of admissible certificates toward the necessary and sufficient condition of Theorem~\ref{thm:cst:nc}. The results are summarized in Table~\ref{tab:results} where we can observe that we approach the theoretical values as the discretization order $N$ increases.

\begin{table}
\centering
\caption{Upper bound for the sampling period for the system $A_0=1$, $A_1=0.5$, $B=1$, $K=-1.6$, $h=0.3$ as a function of the discretization order $N$ of $\mathcal{P}$}\label{tab:results}
\begin{tabular}{c|cccccc}
 Case/$N$ & 1 & 2 & 4 & 8 & 16 & 32\\
 \hline
 State delay & 0.3611 & 0.6516 & 1.0433 & 1.4442 & 1.7514 & 1.9394\\
 Input delay & 0.2875 & 0.4270 & 0.5724 & 0.6910 & 0.7688 & 0.8058\\
 \hline
\end{tabular}
\end{table}

Figures~\ref{fig:ex3state} and~\ref{fig:ex3input} show the closed-loop response of the two configurations under the sampled-data feedback $u(t)=Kx(t_k)$, obtained by method-of-steps time-marching of the delay dynamics with a zero-order hold. For the state-delay plant $\dot x=x+\tfrac12 x(t-h)+u$ with $h=0.3$, the gain $K=-2.5$ at sampling period $T=0.10$ drives the state to zero, while a weak gain $K=-0.5$ at the same period leaves the closed loop unstable. For the input-delay plant $\dot x=x+u(t-h)$, the same gain $K=-2.5$ stabilizes at $T=0.10$ but fails at the slower period $T=0.60$, the held control then being too stale to counter the open-loop growth between samples.

\begin{figure}[t]
\centering
\includegraphics[width=0.48\linewidth]{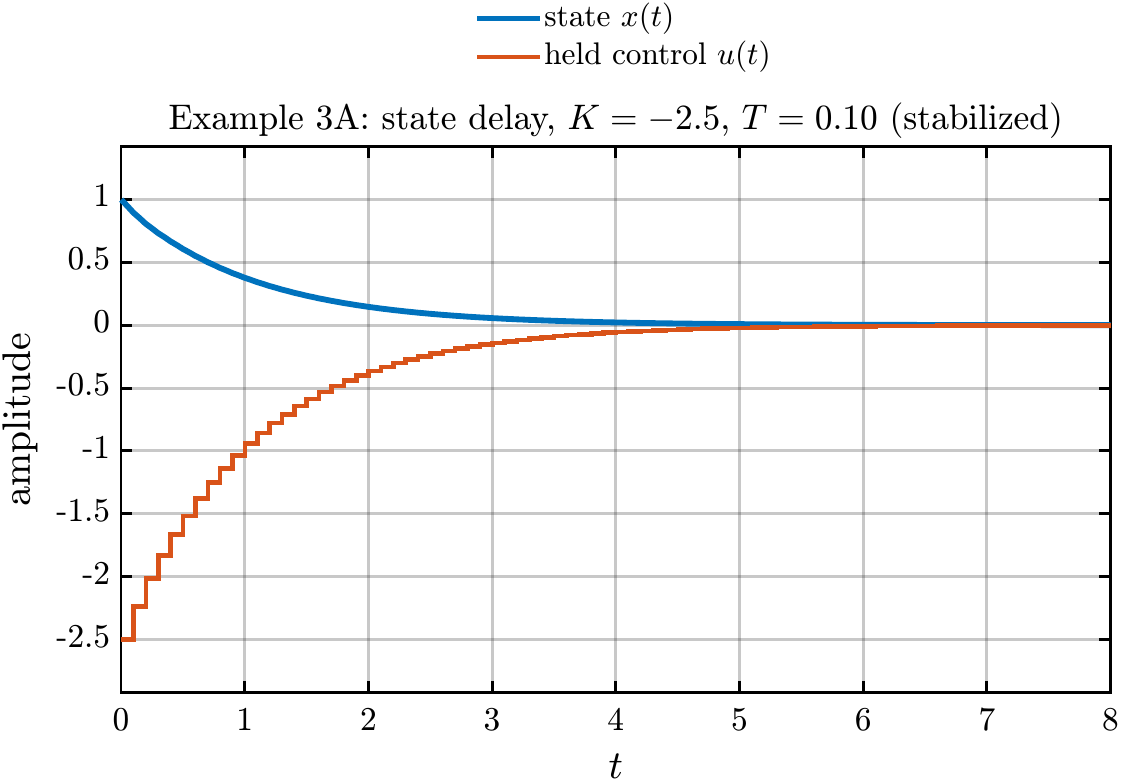}\hfill
\includegraphics[width=0.48\linewidth]{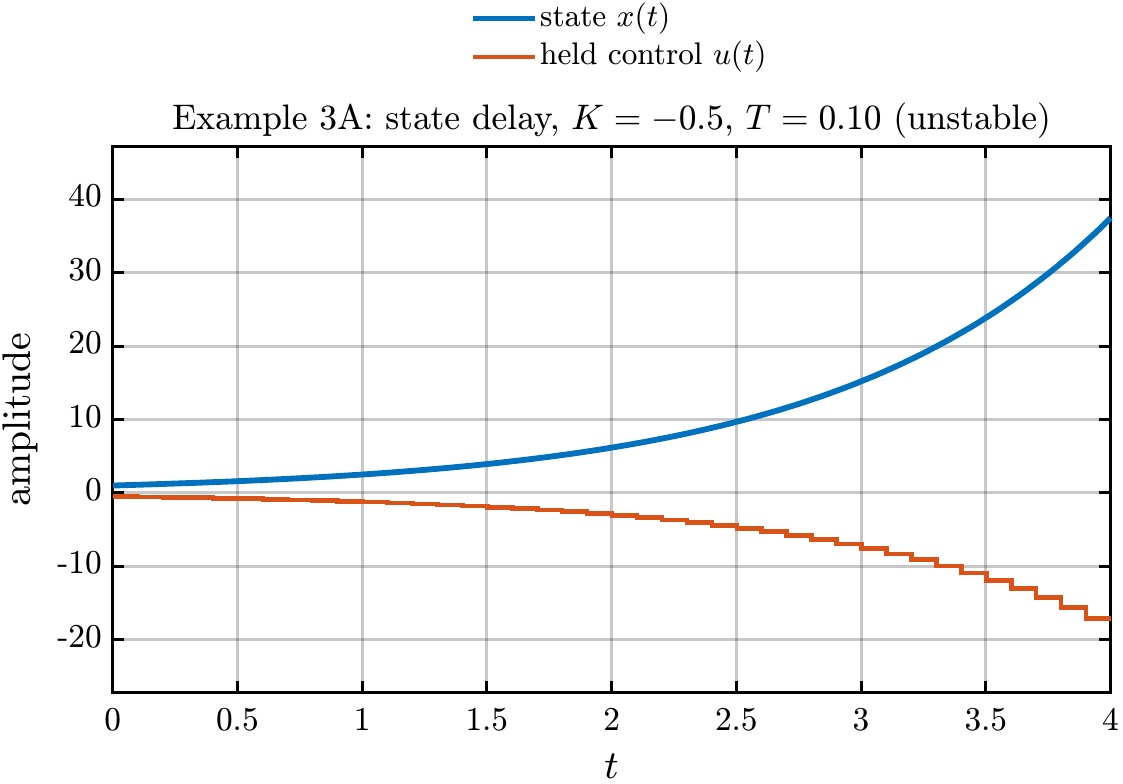}
\caption{Example~\ref{ex:pie}, state-delay configuration $\dot x=A_0x(t)+A_1x(t-h)+Bu(t)$ with $A_0=1$, $A_1=\tfrac12$, $B=1$, $h=0.3$. Each panel shows the state $x(t)$ (solid) and the sampled-data control $u(t)=Kx(t_k)$ held constant between samples (piecewise-constant staircase). Left: $K=-2.5$, $T=0.10$, the feedback drives the state to zero. Right: weak gain $K=-0.5$, $T=0.10$, the closed loop diverges.}\label{fig:ex3state}
\end{figure}

\begin{figure}[t]
\centering
\includegraphics[width=0.48\linewidth]{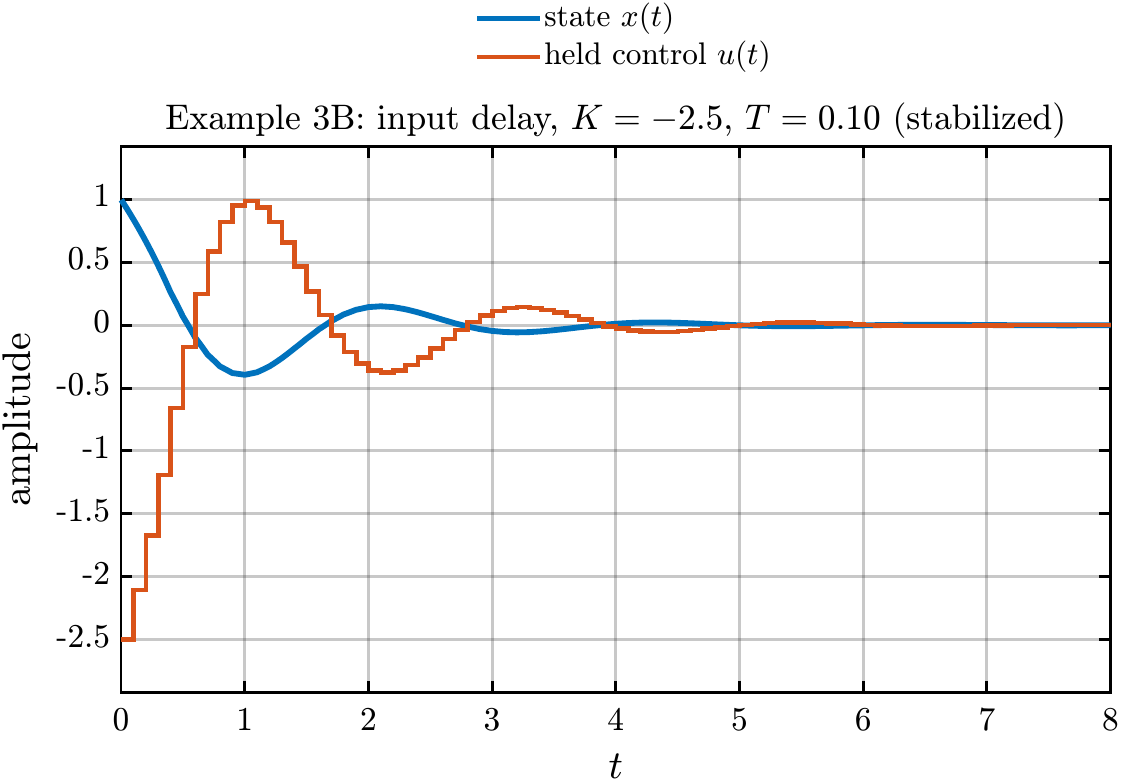}\hfill
\includegraphics[width=0.48\linewidth]{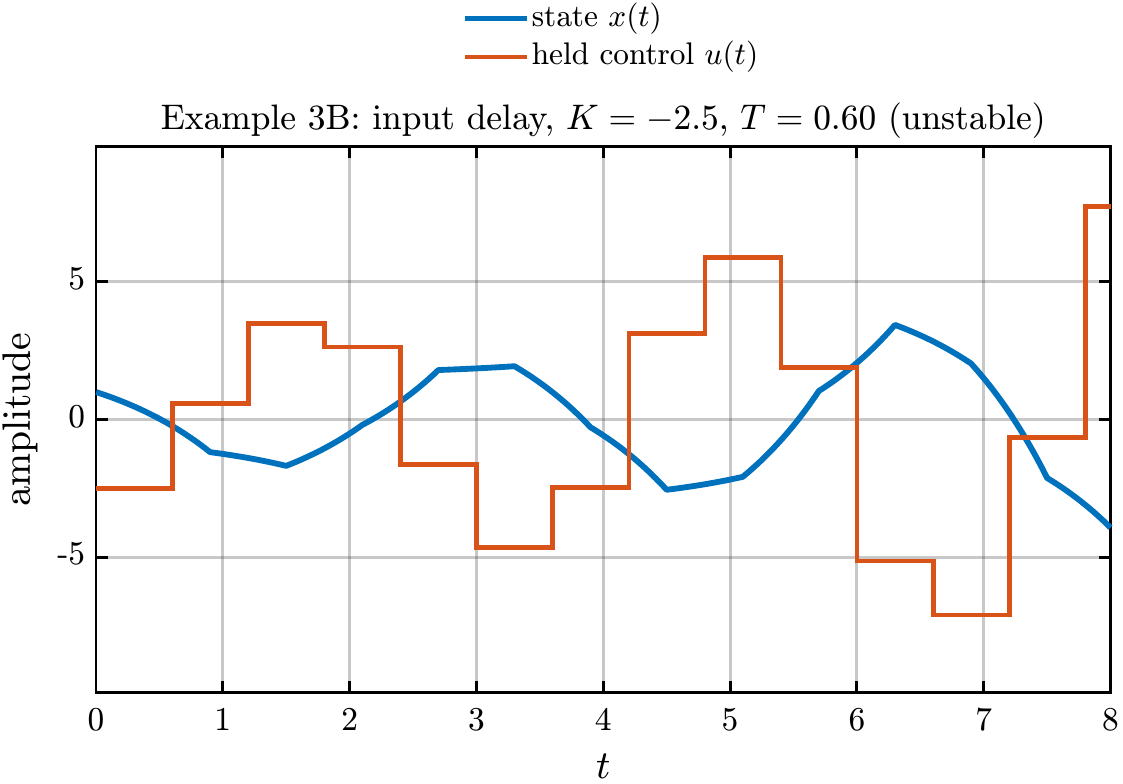}
\caption{Example~\ref{ex:pie}, input-delay configuration $\dot x=A_0x(t)+Bu(t-h)$ with $A_0=1$, $B=1$, $h=0.3$. Each panel shows the state $x(t)$ (solid) and the sampled-data control $u(t)=Kx(t_k)$ held constant between samples (piecewise-constant staircase). Left: $K=-2.5$, $T=0.10$, stabilized. Right: same gain at the slower period $T=0.60$, unstable, the held control acting on stale state information between widely spaced samples.}\label{fig:ex3input}
\end{figure}

\begin{remark}
The range dwell-time case is obtained by replacing the single jump LPI \eqref{eq:pie:jump:lpi} with the family $\mathcal{J}^{*}\mathcal{T}^{*}\mathcal{P}(0)\mathcal{T}\mathcal{J}-\mathcal{T}^{*}\mathcal{P}(\theta)\mathcal{T}\preceq-\delta\,\mathcal{T}^{*}\mathcal{T}$ for all $\theta\in[T_{\min},T_{\max}]$ and the flow LPI by $\mathcal{T}^{*}\dot{\mathcal{P}}(\tau)\mathcal{T}+\mathcal{A}^{*}\mathcal{P}(\tau)\mathcal{T}+\mathcal{T}^{*}\mathcal{P}(\tau)\mathcal{A}\preceq-\gamma\,\mathcal{T}^{*}\mathcal{P}(\tau)\mathcal{T}$, $\tau\in(0,T_{\max})$, which is the range analogue of Proposition~\ref{prop:pie:lift} routed through the coercive Corollary~\ref{cor:hilbert:rng:coerc}. More generally, every operator Lyapunov condition of Section~\ref{sec:hilbert} has a PIE counterpart obtained by the substitution $\bar P(\tau)\mapsto\mathcal{T}^{*}\mathcal{P}(\tau)\mathcal{T}$ and $A\mapsto\mathcal{T}^{-1}\mathcal{A}$, the latter never appearing explicitly because $A$ always occurs sandwiched as $A^*\bar P=\mathcal{A}^*\mathcal{P}\mathcal{T}$. This provides hybrid stability conditions for impulsive time-delay systems entirely within the PI-operator algebra.
\end{remark}

\subsection{Example 4: Sampled-data boundary control of a reaction-diffusion equation}\label{ex:boundary}

We now treat a genuinely boundary-actuated PDE, where the input operator is \emph{unbounded} in the conventional state space, and show that the PIE representation renders it bounded so that the impulsive theory applies through Proposition~\ref{prop:pie:lift}. Consider the reaction-diffusion equation on $s\in[0,1]$,
\begin{equation}\label{eq:rd:pde}
  x_t(t,s)=x_{ss}(t,s)+\lambda  x(t,s),\qquad x(t,0)=0,\quad x(t,1)=u(t),
\end{equation}
with reaction coefficient $\lambda>0$ (so that, for $\lambda>\pi^2$, the open-loop Dirichlet problem $u\equiv0$ is unstable), and a scalar boundary control $u(t)$ entering through the Dirichlet trace at $s=1$. The associated input operator is the Dirichlet map, which is unbounded on $L^2(0,1)$, the classical obstruction to a direct semigroup treatment of boundary control.

\paragraph{PIE representation.}
Following the construction of Section~\ref{ex:pie}, the PIE fundamental state is the highest spatial derivative $v:=x_{ss}\in L^2(0,1)$, which is free of boundary conditions. Integrating twice and imposing $x(t,0)=0$ gives, for $s\in[0,1]$,
\begin{equation}\label{eq:rd:reconstruct}
  x(t,s)=s  x_s(t,0)+\int_0^s(s-\theta)v(t,\theta)\dtheta,
\end{equation}
and the second boundary condition $x(t,1)=u(t)$ fixes the unknown slope $x_s(t,0)=u(t)-\int_0^1(1-\theta)v(t,\theta)\dtheta$. Substituting back,
\begin{equation}\label{eq:rd:Top}
  x(t,s)=(\mathcal{T}_0 v)(t,s)+(\mathcal{B}u)(t,s),\quad
  (\mathcal{T}_0 v)(s):=\int_0^s(s-\theta)v(\theta)\dtheta-s\!\int_0^1(1-\theta)v(\theta)\dtheta,\quad (\mathcal{B}u)(s):=s  u,
\end{equation}
where $\mathcal{T}_0\in L(L^2(0,1))$ is a bounded partial-integral operator and $\mathcal{B}\in L(\mathbb{R},L^2(0,1))$ is the bounded multiplication by $s$. The key gain is that the Dirichlet input, unbounded in the $x$-coordinate, is represented in the $v$-coordinate by the \emph{bounded} operator $\mathcal{B}$.

Differentiating \eqref{eq:rd:Top} in time, $x_t=\mathcal{T}_0\dot v+\mathcal{B}\dot u$, while the PDE \eqref{eq:rd:pde} gives $x_t=v+\lambda x=v+\lambda(\mathcal{T}_0 v+\mathcal{B}u)$. Equating the two expressions yields the boundary-input PIE
\begin{equation}\label{eq:rd:pie}
  \mathcal{T}_0\dot v=(I+\lambda\mathcal{T}_0)v+\lambda\mathcal{B}u-\mathcal{B}\dot u,
\end{equation}
in which, consistently with the general theory of boundary-actuated PIEs \cite{Shivakumar:24}, the input enters both through $u$ and through its time-derivative $\dot u$. All four operators $\mathcal{T}_0$, $I+\lambda\mathcal{T}_0$, $\lambda\mathcal{B}$, $\mathcal{B}$ are bounded.

\paragraph{Sampled-data feedback and impulsive reformulation.}
Let the boundary control be a sampled static feedback $u(t)=K  y(t_k)$ for $t\in[t_k,t_{k+1})$, held constant between samples, where $$y(t)=\langle c,x(t,\cdot)\rangle_{L^2}=\int_0^1 c(s)x(t,s)\ds$$ is a bounded averaged measurement with weight $c\in L^2(0,1)$. The zero-order hold gives $\dot u(t)=0$ on each interval $(t_k,t_{k+1})$, so the troublesome $\mathcal{B}\dot u$ term in \eqref{eq:rd:pie} \emph{vanishes along the flow}, leaving the bounded evolution $\mathcal{T}_0\dot v=(I+\lambda\mathcal{T}_0)v+\lambda\mathcal{B}u$. The input updates only at the sampling instants, which is exactly an impulsive reset.

Augmenting the held control as a state component, set $\mathbf{x}:=\mathrm{col}(v,u)\in H:=L^2(0,1)\times\mathbb{R}$. The flow and jump operators are
\begin{equation}\label{eq:rd:TAJ}
  \mathcal{T}=\begin{bmatrix}\mathcal{T}_0&\mathcal{B}\\0&1\end{bmatrix},\quad
  \mathcal{A}=\begin{bmatrix}I+\lambda\mathcal{T}_0&\lambda\mathcal{B}\\0&0\end{bmatrix},\quad
  \mathcal{J}=\begin{bmatrix}I&0\\K\langle c,\mathcal{T}_0\cdot\rangle&K\langle c,(\cdot)s\rangle\end{bmatrix},
\end{equation}
all bounded on $H$. The closed loop is the hybrid PIE
\begin{equation}\label{eq:rd:hybrid}
  \left\{
  \begin{aligned}
    \mathcal{T}\dot{\mathbf{x}}(t)&=\mathcal{A}\mathbf{x}(t), && t\ne t_k,\\
    \mathbf{x}(t_k)&=\mathcal{J}\mathbf{x}(t_k^-), && t_k=kT,
  \end{aligned}
  \right.
\end{equation}
whose flow $\mathcal{T}\dot{\mathbf{x}}=\mathcal{A}\mathbf{x}$ encodes $\mathcal{T}_0\dot v=(I+\lambda\mathcal{T}_0)v+\lambda\mathcal{B}u$ together with $\dot u=0$, and whose jump $\mathbf{x}(t_k)=\mathcal{J}\mathbf{x}(t_k^-)$ leaves the fundamental component $v$ unchanged and resets the held control $u(t_k)=K\,y(t_k^-)=K(\langle c,\mathcal{T}_0 v\rangle+u\langle c,s\rangle)$, using $x=\mathcal{T}_0 v+\mathcal{B}u$ and $(\mathcal{B}u)(s)=su$. As in Example~\ref{ex:pie}, the jump is written on the fundamental state $\mathbf{x}$, matching the abstract impulsive model \eqref{eq:syst}. Since the sampling is periodic, $\sigma\in\mathcal{S}_{\mathrm{cst}}(T)$.

\paragraph{Stability certificate.}
The operator $\mathcal{T}_0$ is injective with range $\operatorname{ran}\mathcal{T}_0=\{w\in H^2(0,1):w(0)=w(1)=0\}$, on which $\mathcal{T}_0^{-1}=\partial_s^2$. With $\mathcal{B}u=su$, the augmented reconstruction $\mathcal{T}=\left[\begin{smallmatrix}\mathcal{T}_0&\mathcal{B}\\0&1\end{smallmatrix}\right]$ is injective with range the graph $D(A):=\operatorname{ran}\mathcal{T}=\{(x,u)\in H^2(0,1)\times\mathbb{R}:x(0)=0,\ x(1)=u\}$, the physical Dirichlet-at-$0$/held-input-at-$1$ domain on which $A=\operatorname{diag}(\partial_s^2+\lambda,0)$ generates a $C_0$-semigroup; then $\mathcal{A}=A\mathcal{T}$ (using $\partial_s^2\mathcal{B}u=\partial_s^2(su)=0$) and $\mathcal{T}^{-1}\mathcal{A}=\operatorname{diag}(\partial_s^2+\lambda,0)$, recovering exactly the physical generator. Proposition~\ref{prop:pie:lift} therefore applies: if there exist a timer-dependent self-adjoint $\mathcal{P}(\tau)$ with $\alpha_m I\preceq\mathcal{P}(\tau)\preceq c_P I$ and $\gamma,\delta>0$ solving the bounded LPIs \eqref{eq:pie:lift:lpi} with $\mathcal{T},\mathcal{A},\mathcal{J}$ of \eqref{eq:rd:TAJ}, then the closed-loop reaction-diffusion state is UGES over $\mathcal{S}_{\mathrm{cst}}(T)$. As in Section~\ref{ex:pie}, positivity of $\mathcal{P}(\tau)$ is enforced through the PI-operator cone and the LPIs are solved with \texttt{PIETOOLS} \cite{PIETOOLS:24}. The unbounded Dirichlet input has been removed entirely: the analysis takes place in the bounded PI-operator algebra, and the sampling annihilates the input-derivative term that would otherwise prevent a purely bounded formulation.

\paragraph{Numerical illustration.}
We take $\lambda=11>\pi^2$, so the open-loop Dirichlet problem is unstable (its first mode grows at rate $\lambda-\pi^2\approx1.13$), the averaged measurement weight $c\equiv1$, and the scalar gain $K=-1$. Solving the bounded LPIs \eqref{eq:pie:lift:lpi} with a piecewise-linear-in-timer parametrization of $\mathcal{P}(\tau)$ in \texttt{PIETOOLS} certifies UGES of the sampled-data closed loop up to a sampling period $T^\star\approx0.22$, in agreement with the monodromy limit $\rho(\mathcal{J}S(T))<1$. Figure~\ref{fig:ex4} contrasts the open loop ($u\equiv0$), whose state diverges, with the closed loop at $T=0.10$, whose state is driven to zero by the sampled boundary feedback, and Figure~\ref{fig:ex4ctrl} shows the corresponding piecewise-constant boundary control. This example exhibits a mechanism absent from Examples~\ref{ex:jump}--\ref{ex:flow}: the flow itself is unstable and the jumps supply the entire stabilizing action through the boundary input.

\begin{figure}[t]
\centering
\includegraphics[width=0.48\linewidth]{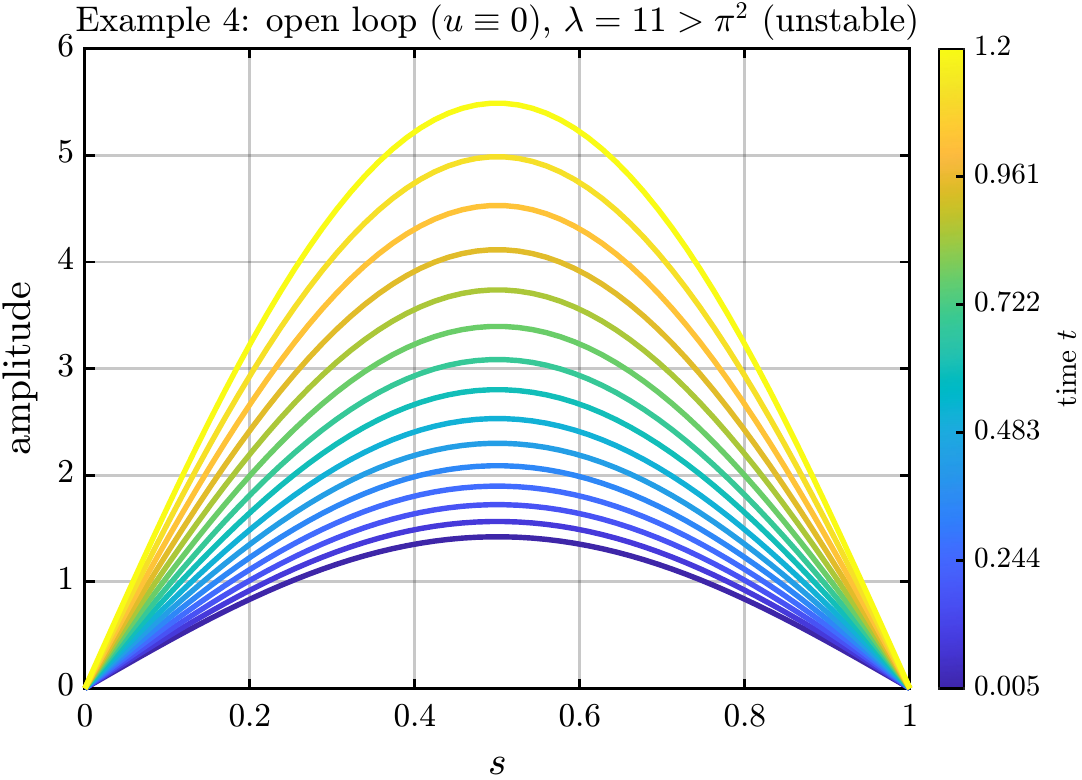}\hfill
\includegraphics[width=0.48\linewidth]{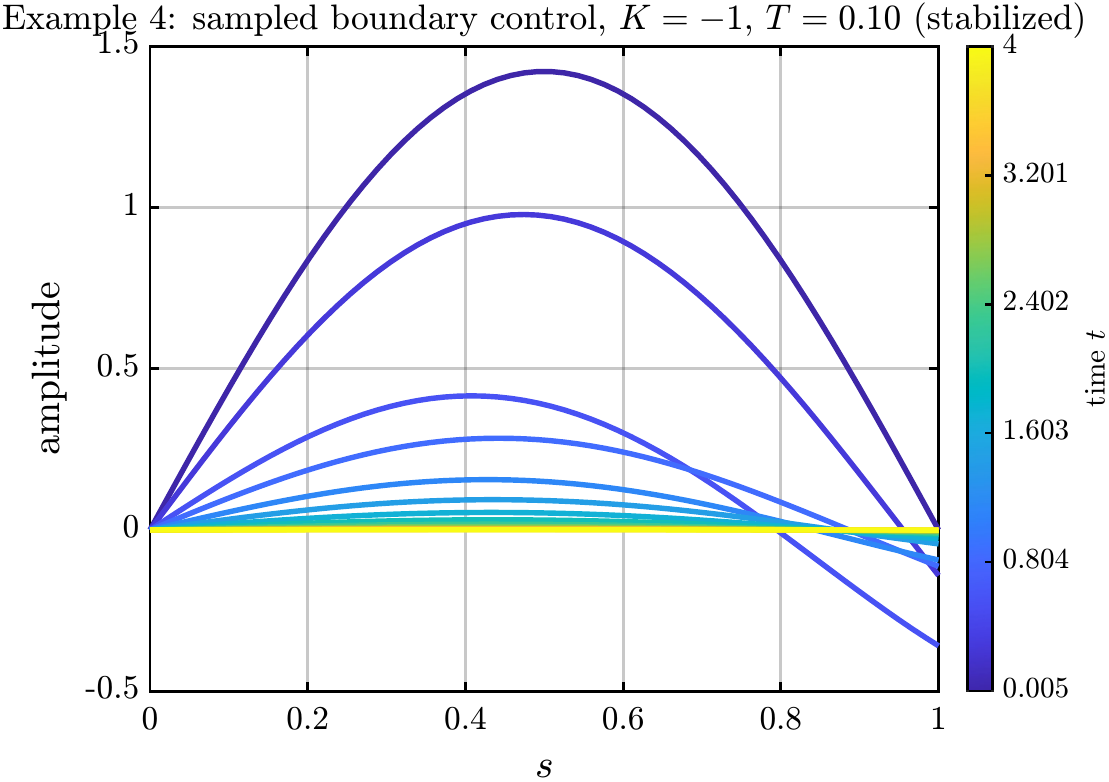}
\caption{Example~\ref{ex:boundary} (sampled-data boundary control of the reaction-diffusion equation, $\lambda=11>\pi^2$). Spatial profiles $x(\cdot,s)$ at successive times, colored by time from early (dark) to late (bright), the time scale being given by the colorbar. Left: open loop $u\equiv0$, the unstable diffusion mode grows. Right: sampled boundary feedback $u(t)=Ky(t_k)$ with $K=-1$, $T=0.10$, the profile converges to zero.}\label{fig:ex4}
\end{figure}

\begin{figure}[t]
\centering
\includegraphics[width=0.6\linewidth]{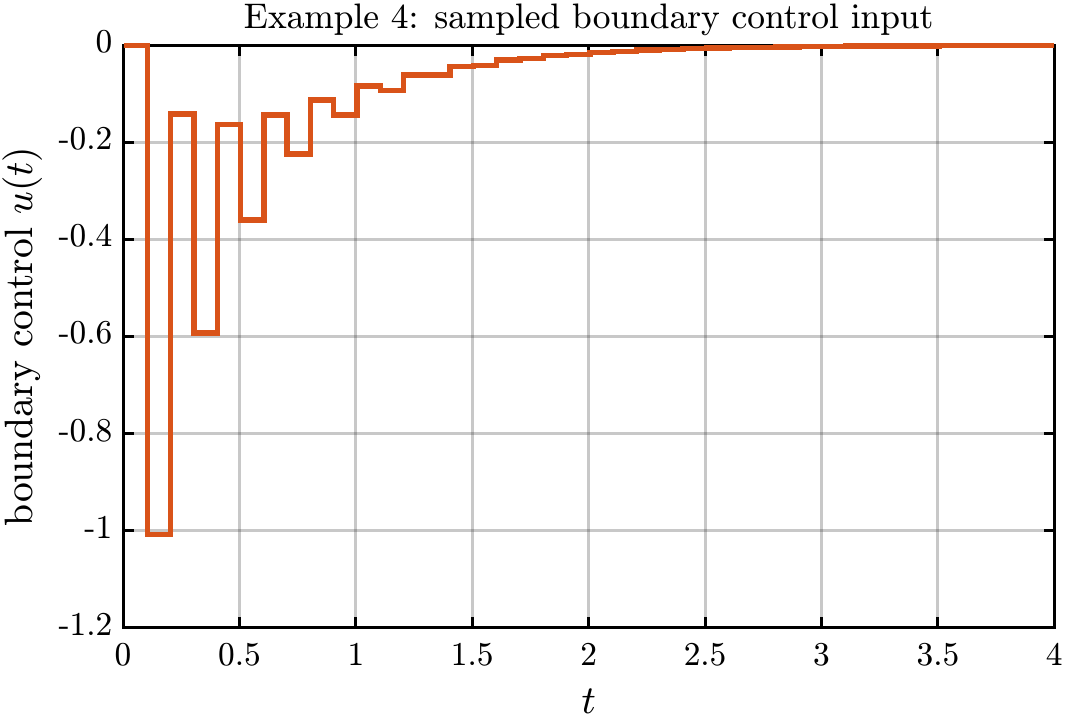}
\caption{Example~\ref{ex:boundary}, sampled-data boundary control input $u(t)=Ky(t_k)$ ($K=-1$, $T=0.10$), held constant between samples and updated at each sampling instant. The control amplitude decays as the closed-loop state converges.}\label{fig:ex4ctrl}
\end{figure}

\section{Conclusion}

We have established a complete Lyapunov characterization of uniform global exponential stability for impulsive linear time-invariant infinite-dimensional systems, organized by the structural information available on the impulse sequence. On Banach spaces, both non-coercive and Lipschitz-coercive characterizations are necessary and sufficient, with explicit Lyapunov functionals provided in each case. On Hilbert spaces, the non-coercive operator Lyapunov equations are necessary and sufficient in the fixed-sequence and constant dwell-time cases, but only sufficient in the arbitrary, minimum dwell-time, and range dwell-time cases; the more restrictive quadratic coercive conditions are sufficient in all cases. Both gaps reflect the absence of coercive quadratic certificates for exponentially stable $C_0$-semigroups \cite{Mironchenko:19}, which on infinite-dimensional Hilbert spaces makes the quadratic class strictly smaller than the non-coercive Lipschitz class whenever the family of admissible monodromies is not a singleton. Because coercive certificates are boundedly invertible, they further dualize into convex synthesis conditions through a congruence transformation and a linearizing change of variables, making the analysis machinery directly applicable to feedback design when the closed-loop system remains time-invariant.

A natural extension is the stabilization of \eqref{eq:syst} by a hybrid state feedback that controls both flow and jumps,
\begin{equation*}
  \dot{x}(t)=Ax(t)+Bu_c(t),\quad u_c(t)=K_c  x(t),\qquad x(t)=Jx(t^-)+B_d u_d(t^-),\quad u_d(t^-)=K_d  x(t^-),
\end{equation*}
with constant gains $K_c\in L(X,\mathcal{U}_c)$, $K_d\in L(X,\mathcal{U}_d)$ to be designed. Plugging the closed-loop generator $A_{\mathrm{cl}}:=A+BK_c$ and jump map $J_{\mathrm{cl}}:=J+B_d K_d$ into a coercive Hilbert condition (for instance, the constant dwell-time corollary in Section~\ref{subsec:hilbert:cst}) gives a flow LMI $\dot{\bar P}(\tau)+A_{\mathrm{cl}}^*\bar P(\tau)+\bar P(\tau)A_{\mathrm{cl}}\preceq -\epsilon I$ and a jump LMI $J_{\mathrm{cl}}^*\bar P(0)J_{\mathrm{cl}}\preceq \bar P(T)-\epsilon I$, both bilinear in $(\bar P,K_c,K_d)$. Assuming $\bar P$ coercive and setting $\bar Q(\tau):=\bar P(\tau)^{-1}$, $Y_c(\tau):=K_c\bar Q(\tau)$, $Y_d:=K_d\bar Q(T)$, the standard pre/post multiplication by $\bar Q$ and Schur complement on the jump inequality yield the LMI
\begin{equation*}
  -\dot{\bar Q}(\tau)+\bar Q(\tau)A^*+A\bar Q(\tau)+Y_c(\tau)^*B^*+BY_c(\tau)\preceq-\epsilon' I,\qquad\begin{pmatrix}-\bar Q(T)&\bar Q(T)J^*+Y_d^*B_d^*\\J\bar Q(T)+B_d Y_d&-\bar Q(0)+\epsilon'' I\end{pmatrix}\prec0,
\end{equation*}
which is linear in $(\bar Q,Y_c,Y_d)$ with $K_c=Y_c(\tau)\bar Q(\tau)^{-1}$ and $K_d=Y_d\bar Q(T)^{-1}$ recovered after solution.

However, the conditions developed in this paper \emph{do not directly apply} to this closed-loop system. The change of variable $Y_c(\tau)=K_c\bar Q(\tau)$ used to linearize the flow LMI introduces a timer-dependent gain reconstruction $K_c=Y_c(\tau)\bar Q(\tau)^{-1}$, and forcing $K_c$ constant requires $Y_c$ and $\bar Q$ to be compatible across the entire timer interval. More fundamentally, even when $K_c, K_d$ are constants, the closed-loop generator $A_{\mathrm{cl}}=A+BK_c$ remains time-invariant during flow, but the controlled system with sampled input or memoryful structure (as arises when $u_c$ is itself held between jumps, or when $u_d$ feeds back through delays) becomes \emph{genuinely time-varying}. The Lyapunov characterizations developed here are formulated for time-invariant generators and jump maps, and their extension to time-varying $A(t)$ and $J(t)$ is a substantive open problem. We leave for future work the development of time-varying versions of the present results, which would close the gap between analysis and synthesis for the general feedback configuration above.

\section*{Acknowledgments}

The authors are grateful to Mario Sigalotti for his useful comments while preparing this manuscript.

\section*{{Disclosure regarding the use of AI tools}}

{Claude (Anthropic) was used in the preparation of this version of the manuscript, for proofreading, for polishing the language, and for assistance in drafting parts of the added material. All results and proofs were directed, reviewed, and verified by the authors, who take full responsibility for the content of the paper.}

\bibliographystyle{unsrt}
\bibliography{./bib_francesco,./briat_IDS,./global_IDS}

@article{liu2015lyapunov,
  title={Lyapunov-based sufficient conditions for stability of hybrid systems with memory},
  author={Liu, J. and Teel, A. R.},
  journal={IEEE Transactions on Automatic Control},
  volume={61},
  number={4},
  pages={1057--1062},
  year={2015},
  publisher={IEEE}
}

@string{ieeetac = {IEEE Transactions on Automatic
Control}}

@string{siam = {SIAM Journal of Control and Optimization}}

@article{Briat:11l,
  author = "C. Briat and A. Seuret",
  title = "A looped-functional approach for robust stability analysis of linear impulsive systems",
  journal = "Systems \& Control Letters",
  year = "2012",
  volume = "61(10)",
  pages="980--988",
}

@article{Briat:12h,
  author = "C. Briat and A. Seuret",
  title = "Convex dwell-time characterizations for uncertain linear impulsive systems",
  journal = "{IEEE} Transactions on Automatic Control",
  year = "2012",
  volume = "57(12)",
  pages="3241--3246",
}

@article{Briat:13b,
  author = "C. Briat and A. Seuret",
  title = "Affine minimal and mode-dependent dwell-time characterization for uncertain switched linear systems",
  journal = "{IEEE} Transactions on Automatic Control",
year={2013},
volume={58},
number={5},
pages={1304--1310},
}

@article{Briat:13d,
  author = "C. Briat",
  title = "Convex conditions for robust stability analysis and stabilization of linear aperiodic impulsive and sampled-data systems under dwell-time constraints",
  journal = "Automatica",
  year = "2013",
  volume = "49(11)",
  pages="3449--3457",
}

@book{Briat:book1,
  author={C. Briat},
  title={Linear Parameter-Varying and Time-Delay Systems -- Analysis, Observation, Filtering \& Control},
  publisher={Springer-Verlag},
  series={Advances on Delays and Dynamics},
  volume ={3},
  year={2015},
  address={Heidelberg, Germany},
}

@article{Briat:15i,
  author = "C. Briat",
  title = "Stability analysis and stabilization of stochastic linear impulsive, switched and sampled-data systems under dwell-time constraints",
  year = "2016",
  volume = "74",
  pages="279--287",
  journal = "Automatica",
}

@article{Briat:16c,
  author = "C. Briat",
  title = "Dwell-time stability and stabilization conditions for linear positive impulsive and switched systems",
  year = "2017",
  volume = "24",
  pages="198--226",
  journal = "Nonlinear Analysis: Hybrid Systems",
}

@string{ieeetac = {IEEE Transactions on Automatic Control}}

@string{automatica = {Automatica}}

@article{PrieurAstolfi03,
	Author = {C. Prieur and A. Astolfi},
	Journal = ieeetac,
	Number = 10,
	Pages = {1768-1772},
	Title = {Robust stabilization of chained systems via hybrid control},
	Volume = 48,
	Year = 2003}

@BOOK{Boyd:94a,
    author={S.  Boyd and L.  {El-Ghaoui} and E.  Feron and V.  Balakrishnan},
    title={Linear Matrix Inequalities in Systems and Control Theory},
    publisher={PA, SIAM, Philadelphia},
    year={1994},
}

@article{Scherer:97a,
    author={C.  W.  Scherer and P.  Gahinet and M.  Chilali},
    title={Multiobjective Output-Feedback Control via {LMI} Optimization},
    journal={IEEE Transaction on Automatic Control},
    year={1997},
    volume={42},
    number={7},
    pages={896--911},
}

@BOOK{GuKC:03,
  AUTHOR =       {K.  Gu and V. L.  Kharitonov and J.  Chen},
  TITLE =        {Stability of Time-Delay Systems},
  PUBLISHER =    {Birkh{\"a}user, Boston},
  year =         {2003},
}

@BOOK{Niculescu:01,
  author    = {S. -I.  Niculescu},
  title     = {Delay effects on stability.  A robust control approach},
  volume    = {269},
  publisher = {Springer-Verlag: Heidelbeg},
  year      = {2001},
}

@unpublished{Scherer:05a,
    author={C.  W.  Scherer and S.  Weiland},
    title={Linear Matrix Inequalities in Control},
    year={2005},
    note={Lecture Notes},
}

@book{Bensoussan:06,
  author={A.  Bensoussan and G.  {Da~Prato} and M.  C.  Delfour and S.  K.  Mitter},
  title={Representation and Control of Infinite Dimensional Systems},
  publisher={Birkh{\"{a}}user},
  address = {Boston, USA},
  edition = {2nd},
  year={2007},
}

@article{Fridman:04,
  author ={E.  Fridman and A.  Seuret and J.  P.  Richard},
  title ={Robust Sampled-Data Stabilization of Linear Systems: An Input Delay Approach},
  journal ={Automatica},
  year ={2004},
  volume ={40},
  pages ={1441-1446},
}

@article{Barmish:85,
  author ={B.  R.  Barmish},
  title ={Necessary and sufficient conditions for quadratic stabilizability of an uncertain system},
  journal ={J.  Optim.  Theory Appl. },
  year ={1985},
  volume ={46},
  pages ={399-408},
}

@book{Lakshmikantham:89,
  author={V.  Lakshmikantham and D. D.  Bainov and P. S.  Simeonov},
  title={Theory of Impulsive Differential Equations},
  publisher={World Scientific},
  series={Series in Modern Applied Mathematics},
  year={1989},
}

@book{Liberzon:03,
  author={D.  Liberzon},
  title={Switching in Systems and Control},
  publisher={Birkh{\"{a}}user},
  address = {New York},
  year={2003},
}

@article{Geromel:06b,
  author={J. C.  Geromel and P.  Colaneri},
  title={Stability and Stabilization of Continuous-Time Switched Linear Systems},
    journal={{SIAM} Journal on Control and Optimization},
    year ={2006},
    volume ={45(5)},
    pages ={1915-1930},
}

@article{Hagood:06,
  author={J. W.  Hagood and B. S.  Thomson},
  title={Recovering a function from a {D}ini derivative},
    journal={Amer.  Math.  Monthly},
    year ={2006},
    volume ={113},
    pages ={34-46},
}

@article{Naghshtabrizi:08,
  author={P.  Naghshtabrizi and J.  P.  Hespanha and A.  R.  Teel},
  title={Exponential stability of impulsive systems with application to uncertain sampled-data systems},
    journal={Systems \& Control Letters},
    year ={2008},
    volume ={57},
    pages ={378-385},
}

@inproceedings{Seuret:10,
   author ={A.  Seuret},
   title ={Exponential stability and stabilization of sampled-data systems with time-varying period},
   year ={2010},
   booktitle = {{IFAC} Conference on Time-Delay System},
   address ={Prague, Czech Republic},
   pages={301--306},
}

@article{Bamieh:91,
  author ={B.  Bamieh and J. B.  Pearson and B. A.  Francis and A.  Tannenbaum},
  title ={A lifting technique for linear periodic systems with applications to sampled-data control},
  journal ={Systems \& Control Letters},
  year ={1991},
  volume ={17},
  pages={ 79--88},
}

@ARTICLE{Seuret:12,
    author={A.  Seuret},
    title={A novel stability analysis of linear systems under asynchronous samplings},
    journal={Automatica},
    year ={2012},
    volume ={48(1)},
    pages={177--182},
}

@article{Fridman:10,
  author ={E.  Fridman},
  title ={A refined input delay approach to sampled-data control},
  journal ={Automatica},
  year ={2010},
  volume ={46(2)},
  pages={421--427}, }

@book{Yang:01b,
  author={T.  Yang},
  title={Impulsive control theory},
  publisher={Springer-Verlag},
  address = {Berlin Heidelberg},
  year={2001},
}

@book{Bainov:89,
  author={D. D.  Bainov and P. S.  Simeonov},
  title={Systems with impulse effects: Stability, theory and applications},
  publisher={Academy Press},
  address = {Chichester, UK},
  year={1989},
}

@article{Hespanha:08,
  author ={J.  P.  Hespanha and D.  Liberzon and A.  R.  Teel},
  title ={{L}yapunov conditions for input-to-state stability of impulsive systems},
  journal ={Automatica},
  year ={2008},
  volume ={44(11)},
  pages={2735--2744},
}

@book{Haddad:06,
  author={W.  M.  Haddad and V.  Chellaboina and S.  G.  Nersesov},
  title={Impulsive and Hybrid Dynamical Systems},
  publisher={Princeton University Press},
  year={2006},
}

@article{Wirth:05,
  author ={F. Wirth},
  title ={A converse {L}yapunov theorem for linear parameter-varying and linear switching systems},
  journal ={{SIAM} Journal on Control and Optimization},
  year ={2005},
  volume ={44(1)},
  pages={210--239},
}

@article{Liu:12,
 author ={K.  Liu and E.  Fridman},
 title ={Wirtinger's inequality and {L}yapunov-based sampled-data stabilization},
 journal ={Automatica},
 volume ={48(1)},
 year ={2012},
 pages ={102--108},
}

@article{Sivashankar:94,
author = {N.  Sivashankar and P.  P.  Khargonekar},
title = {Characterization of the {${\mathcal{L}}_2$}-Induced Norm for Linear Systems with Jumps with Applications to Sampled-Data Systems},
year = {1994},
journal = {SIAM Journal on Control and Optimization},
volume = {32(4)},
pages = {1128--1150},
}

@article{Khammash:93,
title ={Necessary and sufficient conditions for the robustness of time-varying systems with applications to sampled-data systems},
author ={M.  H. Khammash},
journal ={{IEEE} Transactions on Automatic Control},
volume ={38(1)},
pages ={49--57},
year ={1993},
}

@inproceedings{Seuret:12b,
 author ={A.  Seuret and F.  Gouaisbaut},
 title ={On the use of the {W}irtinger inequalities for time-delay systems},
 booktitle ={10th {IFAC} Worhshop on Time Delay Systems},
 year ={2012},
 location ={Boston, USA},
 pages ={260--265},
}

@article{Chesi:12,
title ={A Nonconservative {LMI} condition for stability of switched systems with guaranteed dwell-time},
author ={G.  Chesi and P.  Colaneri and J.  C.  Geromel and R.  Middleton and R.  Shorten},
journal = IEEETAC,
volume ={57(5)},
pages ={1297--1302},
year ={2012},
}

@article{Bernussou:89,
  author ={J. Bernussou and P. L. D. Peres and J. C. Geromel},
  title ={{A linear programming oriented procedure for quadratic stabilization of uncertain systems}},
  journal = {Systems \& Control Letters},
  year ={1989},
  volume ={13},
  pages={65--72},
}

@article{Dashkovskiy:12,
  author ={S.  Dashkovskiy and M. Kosmykov and A. Mironchenko and L. Naujok},
  title ={Stability of interconnected impulsive systems with and without time-delays, using {L}yapunov methods},
  journal = {Nonlinear Analysis: Hybrid Systems},
  year ={2012},
  volume ={6},
  pages ={899--915},
}

@article{Oliveira:99,
  author ={M. C. {de Oliveira} and J. Bernussou and J. C. Geromel},
  title ={A new discrete-time robust stability criterion},
  journal = {Systems \& Control Letters},
  year ={1999},
  volume ={37(4)},
  pages ={261--265},
}

@incollection {Datko:72,
   author ={R. Datko},
   title ={An algorithm for computing {L}iapunov functionals for some differential difference equations},
   booktitle ={Ordinary differential equations},
   editor ={L. Weiss},
   publisher = {Academic Press, New York},
   pages = {387--398},
   year = {1972}
}

@article{Dashkovskiy:13,
   author = {S. Dashkovskiy and A. Mironchenko},
    title = "{Input-to-state stability of nonlinear impulsive systems}",
    journal = {SIAM Journal on Control and Optimization},
    volume={51(3)},
    pages={1962--1987},
     year = 2013,
}

@article{Allerhand:13,
   author = {L. I. Allerhand and U. Shaked},
    title = "{Robust state-dependent switching of linear systems with dwell-time}",
    journal = ieeetac,
    volume ={58(4)},
    year ={2013},
    pages ={994--1001},
}

@article{Seuret:13b,
 author ={A. Seuret and F. Gouaisbaut},
 title ={Wirtinger-based integral inequality: Application to time-delay systems},
  journal = {Automatica},
  year ={2013},
  volume ={49(9)},
  pages={2860--2866},
}

@article{Allerhand:11,
author = {L. I. Allerhand and U. Shaked},
title = {Robust stability and stabilization of linear switched systems with dwell time},
journal = {{IEEE} Transactions on Automatic Control},
volume = {56(2)},
pages = {381--386},
year = {2011},
}

@article{Teel:14,
  author ={A. R. Teel and A. Subbaraman and A. Sferlazza},
  title ={Stability analysis for stochastic hybrid systems: {A} survey},
  journal ="Automatica",
  year ={2014},
  volume ={50(10)},
  pages={2435--2456},
}

@book{Goebel:12,
  author={R. Goebel and R. G. Sanfelice and A. R. Teel},
  title={Hybrid Dynamical Systems. Modeling, Stability, and Robustness},
  publisher={Princeton University Press},
  year={2012},
}

@article{Xiang:15a,
  author ={W. Xiang},
  title ={On equivalence of two stability criteria for continuous-time switched systems with dwell time constraint},
  journal ={Automatica},
  year ={2015},
  volume ={54},
  pages ={36--40},
}

@book{Samoilenko:95,
  author={A. M. Samoilenko and N. A. Perestyuk},
  title={Impulsive differential equations (translated from Russian by Y. Chapovsky)},
  publisher={World Scientific},
  year={1995},
  address={Singapore},
}

@book{Curtain:95,
  author ={R. F. Curtain and H. J. Zwart},
  title ={An introduction to infinite-dimensional linear systems theory},
  publisher={Springer-Verlag},
  year={1995},
  address={New York, USA},
}

@book{Fridman:14,
  author ={E. Fridman},
  title ={Introduction to Time-Delay Systems},
  publisher={Birkh{\"a}user},
  year={2014},
  address={Springer International Publishing Basel Switzerland},
}

@article{Hetel:13,
author = {L. Hetel and E. Fridman},
journal     = ieeetac,
title        = {Robust sampled-data control of switched affine systems},
year        = {2013},
volume = {58(11)},
pages = {2922--2928},
}

@book{Lakshmikantham:91,
  author={V. Lakshmikantham and V. M. Matrosov and S. Sivasundaram},
  title={Vector {L}yapunov Functions and Stability Analysis of Nonlinear Systems},
   booktitle={Mathematics and Its Applications},
   year ={1991},
   publisher = {Springer Science+Business Media Dordrecht},
  }

@article{Davo:17,
  author ={M. A Davo and A. Ba\~nos and F. Gouaisbaut and S. Tarbouriech and A. Seuret},
  title = {Stability analysis of linear impulsive delay dynamical systems via looped-functionals},
  journal = "Automatica",
  year ={2017},
  volume={81},
  pages ={107--114},
}

@book{Pazy:83,
 author ={A. Pazy},
  title ={Semigroups of Linear Operators and Applications to Partial Differential Equations},
  year ={1983},
  publisher = {Springer-Verlag New York},
}

@article{Dashkhovskiy:23,
author ={S. Dashkovskiy and V. Slynko},
title = {Dwell-time stability conditions for infinite dimensional impulsive systems},
journal = {Automatica},
year = {2023},
volume = {147},
pages = {110695},
}

@article{Bivziuk:23,
author = {V. Bivziuk and S. Dashkovskiy and V. Slynko},
title = {Comparison theorem for infinite-dimensional linear impulsive systems},
journal = {arXiv preprint arXiv:2308.05615},
year = {2023},
}

@article{Mancilla:20,
author = {J. L. Mancilla-Aguilar and H. Haimovich},
title = {Uniform input-to-state stability for switched and time-varying impulsive systems},
journal = {IEEE Transactions on Automatic Control},
volume = {65},
number = {12},
pages = {5028--5042},
year = {2020},
}

@article{Peet:20,
author ={M. M. Peet},
title = {A Convex Solution of the $H_\infty$-Optimal Controller Synthetic Problem For Multidelay Systems},
journal = {SIAM Journal on Control and Optimization},
year = {2020},
volume = {58(3)},
pages = {1547--1578},
}

@article{Haidar:22,
 author ={I. Haidar and Y. Chitour and P. Mason and M. Sigalotti},
 title ={Lyapunov Characterization of Uniform Exponential Stability for Nonlinear Infinite-Dimensional Systems},
journal = ieeetac,
year = {2022},
volume = {67(4)},
pages = {1685--1697},
}

@article{Mironchenko:18,
author ={A. Mironchenko and F. Wirth},
title = {Lyapunov characterization of input-to-state stability for semilinear control systems over {B}anach spaces},
journal = {Systems \& Control Letters},
year = {2018},
volume = {119},
pages = {64--70},
}

@article{Datko:70,
author ={R. Datko},
title = {Extending a theorem of {A. M. Liapunov} to {H}ilbert space},
journal = { Journal of Mathematical Analysis and Applications},
year = {1970},
volume = {32},
pages = {610--616},
}

@manual{PIETOOLS:24,
    author={S. Shivakumar and A. Das and D. Braghini and D. Jagt and Y. Peet and M. Peet},
    title={PIETOOLS 2024: User Manual},
    year={2004},
}

@article{Peet:21pie,
  author  = {M. M. Peet},
  title   = {Representation of Networks and Systems with Delay: {DDE}s, {DDF}s, {ODE}-{PDE}s and {PIE}s},
  journal = {Automatica},
  year    = {2021},
  volume  = {127},
  pages   = {109508},
  note    = {arXiv:1910.03881},
}

@article{Shivakumar:21,
  author  = {S. Shivakumar and A. Das and S. Weiland and M. M. Peet},
  title   = {A Partial Integral Equation ({PIE}) Representation of Coupled Linear {PDE}s and Scalable Stability Analysis using {LMI}s},
  journal = {Automatica},
  year    = {2021},
  volume  = {129},
  pages   = {109674},
}

@article{Chitour:25,
  author  = {Y. Chitour and J. Daafouz and I. Haidar and P. Mason and M. Sigalotti},
  title   = {A {Berger}-{Wang} Formula for Impulsive Switched Systems},
  journal = {arXiv preprint},
  year    = {2025},
  note    = {arXiv:2507.02434},
}

@article{Hante:11,
author = {Hante, Falk M. and Sigalotti, Mario},
title = {Converse {L}yapunov Theorems for Switched Systems in {B}anach and {H}ilbert Spaces},
journal = {SIAM Journal on Control and Optimization},
volume = {49},
number = {2},
pages = {752-770},
year = {2011},
}

@article{Mironchenko:19,
  author  = {A. Mironchenko and F. Wirth},
  title   = {Non-coercive {L}yapunov functions for infinite-dimensional systems},
  journal = {Journal of Differential Equations},
  year    = {2019},
  volume  = {266},
  number  = {11},
  pages   = {7038--7072},
  doi     = {10.1016/j.jde.2018.11.026},
}

@article{Shivakumar:24,
  author  = {S. Shivakumar and A. Das and S. Weiland and M. M. Peet},
  title   = {Extension of the {P}artial {I}ntegral {E}quation Representation to {GPDE} Input-Output Systems},
  journal = {{IEEE} Transactions on Automatic Control},
  year    = {2024},
  volume={70(5)},
  pages={3240--3255},
}

\end{document}